\documentclass[11pt]{amsart}

\usepackage{epsf,amssymb,amsmath,overpic,amscd,psfrag,stmaryrd,times,diagrams,mathrsfs, soul}
\usepackage[all]{xy}
\usepackage{hyperref}
\usepackage[dvipsnames]{xcolor}
\usepackage{enumitem}
\usepackage{tikz,tikz-cd}

\makeatletter
\newcommand{\labitem}[2]{%
\def\@itemlabel{#1}
\item
\def\@currentlabel{#1}\label{#2}}
\makeatother

\newtheorem{thm}{Theorem}[section]
\newtheorem{prop}[thm]{Proposition}
\newtheorem{lemma}[thm]{Lemma}

\newtheorem{claim}[thm]{Claim}
\newtheorem{fact}[thm]{Fact}

\numberwithin{equation}{subsection}
\numberwithin{thm}{subsection}

\theoremstyle{definition}
\newtheorem{defn}[thm]{Definition}
\newtheorem{assumption}[thm]{Assumption}
\newtheorem{notation}[thm]{Notation}
\newtheorem{terminology}[thm]{Terminology}
\newtheorem{warning}[thm]{Warning}

\theoremstyle{remark}

\newtheorem{rmk}[thm]{Remark}
\newtheorem{example}[thm]{Example}

\newtheorem{convention}[thm]{Convention}
\newtheorem{changeofnotation}[thm]{Change of Notation}

\DeclareMathAlphabet{\mathpzc}{OT1}{pzc}{m}{it}

\newcommand{\C}{\mathbb{C}}
\newcommand{\R}{\mathbb{R}}
\newcommand{\Z}{\mathbb{Z}}
\newcommand{\Q}{\mathbb{Q}}
\newcommand{\N}{\mathbb{N}}

\newcommand{\bdry}{\partial}
\newcommand{\s}{\vskip.1in}
\newcommand{\n}{\noindent}

\newcommand{\F}{\mathcal{F}}
\newcommand{\E}{E}
\newcommand{\EE}{\mathbb{E}}
\newcommand{\V}{V}
\newcommand{\W}{W}
\newcommand{\K}{K}
\newcommand{\Kur}{\mathscr{K}}
\newcommand{\be}{\begin{enumerate}}
\newcommand{\ee}{\end{enumerate}}
\newcommand{\op}{\operatorname}
\newcommand{\bs}{\boldsymbol}

\newcommand{\nl}{\mathfrak{nl}}
\newcommand{\git}{/\!\!/}
\newcommand{\I}{\mathcal{I}_+}

\newcommand{\tp}{\mathsf}
\newcommand{\X}{\mathbb X}
\newcommand{\arr}{\overrightarrow}
\newcommand{\en}{{\op{en}}}

\newcommand{\coblue}{\color{black}}
\newcommand{\coblu}{\color{black}}

\newcommand{\cb}{\color{black}}

\newcommand{\rev}{\dagger}

\newcommand\NoIndent[1]{%
  \begingroup
  \par
  \parshape0
  #1\par
  \endgroup
}

\usepackage[refpage,notocbasic]{nomencl}

\makenomenclature

\newcommand{\nom}{\nomenclature}

\begin{document}

\title[Definition of contact homology]
{Semi-global Kuranishi charts and the definition of contact homology}

\author{Erkao Bao}
\address{School of Mathematics, University of Minnesota, Minneapolis, MN 55455}
\email{baoerkao@gmail.com} \urladdr{https://erkaobao.github.io/math/}

\author{Ko Honda}
\address{University of California, Los Angeles, Los Angeles, CA 90095}
\email{honda@math.ucla.edu} \urladdr{http://www.math.ucla.edu/\char126 honda}

\date{This version: December 31, 2022.}



\thanks{KH supported by NSF Grants DMS-1406564 and DMS-1549147.}

\begin{abstract}
We define the contact homology algebra for any contact manifold and show that it is an invariant of the contact manifold. More precisely, given a contact manifold $(M,\xi)$ and some auxiliary data $\mathcal{D}$, we define an algebra $HC(\mathcal{D})$.
If $\mathcal{D}_1$ and $\mathcal{D}_2$ are two choices of auxiliary data for $(M,\xi)$, then $HC(\mathcal{D}_1)$ and $HC(\mathcal{D}_2)$ are isomorphic. We use a simplified version of Kuranishi perturbation theory, consisting of {\em semi-global Kuranishi charts}.
\end{abstract}

\maketitle

\setcounter{tocdepth}{2}
\tableofcontents

\newpage  
\printnomenclature[5em] \cb
\newpage

\section{Introduction}

Symplectic field theory (SFT), proposed about 20 years ago by Eliashberg-Givental-Hofer \cite{EGH}, is a package which gives invariants of contact manifolds and symplectic manifolds with boundary as well as gluing formulas for Gromov-Witten invariants on closed symplectic manifolds. The transversality theory in the somewhere injective case and the Fredholm theory were carried out by Dragnev~\cite{Dr} and the SFT compactness was worked out by \cite{BEHWZ}, both very early on in the development of the theory. However, the rigorous foundations of the full theory are still under development, although parts of the theory are gradually becoming more rigorous, thanks to the efforts of many authors, primarily by Hofer-Wysocki-Zehnder~\cite{HWZ3}.  Partial work has been done by Hutchings-Nelson~\cite{HN} and Bao-Honda~\cite{BH} for cylindrical contact homology in dimension three and Bourgeois-Oancea~\cite{BO} on $S^1$-equivariant symplectic homology.  Also the Kuranishi perturbation theory of Fukaya-Ono~\cite{FO} and Fukaya-Oh-Ohta-Ono \cite{FO3} (which is closest in spirit to this work) and the approaches of Liu-Tian~\cite{LT} and \cite{Ru} may be used to define SFT, although this has not been done yet.

The goal of this paper is to make a further contribution to this effort and define the full contact homology differential graded algebra (dga) for any closed contact manifold in any dimension and show that its isomorphism class is an invariant of the contact manifold. More precisely, given a closed cooriented contact $(2n+1)$-dimensional manifold $(M,\xi)$ and some auxiliary data $\mathcal{D}$ which includes a nondegenerate contact form $\alpha$ for $\xi$, we define a direct limit dga $\mathfrak{A}(\mathcal{D})$ whose homology is called the {\em full contact homology algebra} $HC(\mathcal{D})$. We then prove that:

\begin{thm} \label{thm: invariance}
If $\mathcal{D}_1$ and $\mathcal{D}_2$ are two choices of auxiliary data for $(M,\xi)$, then $HC(\mathcal{D}_1)$ and $HC(\mathcal{D}_2)$ are isomorphic. Hence the isomorphism class of the algebra $HC(\mathcal{D})$ is an invariant of $(M,\xi)$.
\end{thm}

We denote the isomorphism class of $HC(\mathcal D)$ by $HC(\xi)$.  We also prove the following:

\begin{thm} \label{thm: cobordisms}
Let $(W,\alpha)$   be a compact Liouville cobordism which restricts to $(M_+,\xi_+=\ker\alpha|_{M_+})$ at the positive boundary and to $(M_-,\xi_-=\ker\alpha|_{M_-})$ at the negative boundary. \cb
Then there is an algebra homomorphism
$$\Phi_*:HC(\xi_+)\to HC(\xi_-).$$
In the case when $\xi_+=\xi_-$ and $(W,\alpha)$ is the trivial cobordism, we have $\Phi_*=\op{id}$.\footnote{  We will often suppress the $\alpha$ from the notation of a Liouville cobordism.}
\end{thm}

The construction of the direct limit dga $\mathfrak{A}(\mathcal{D})$ and the proofs of Theorems~\ref{thm: invariance} and \ref{thm: cobordisms} use a particularly simple version of the Kuranishi multivalued perturbation theory of \cite{FO,FO3} that we call {\em semi-global Kuranishi structures}.  The perturbations that we use are supported near the ends of finite energy $J$-holomorphic curves and are described using the asymptotic eigenfunctions of the asymptotic operators.   Let $\widehat{W}$ be a completed Liouville cobordism,
\nom[What]{$W, \widehat{W}$}{Compact Liouville cobordism $W$ and its completion}
$$\mathcal{M}=\mathcal{M}^{\op{ind}=k}_J(\dot F,\widehat{W};\bs\gamma_+;\bs\gamma_-)$$
be the moduli space of Fredholm index $\op{ind}=k$ $J$-holomorphic maps $\dot F\to \widehat{W}$ that are asymptotic to collections $\bs\gamma_+$ and $\bs\gamma_-$ of Reeb orbits at the positive and negative ends, and $\overline{\mathcal{M}}$ be the space of SFT buildings in $\widehat{W}$ of total Fredholm index $\op{ind}=k$ that are asymptotic to $\bs\gamma_+$ and $\bs\gamma_-$ at the positive and negative ends and whose domains, when glued up, have topological type $\dot F$.  If $\overline{\mathcal{M}}=\mathcal{M}$, \cb then {\em we only need one Kuranishi chart, which we call a semi-global Kuranishi chart}.  

\begin{rmk}
Throughout this paper we will freely use the SFT compactness theorem of \cite{BEHWZ} without explicit mention.
\end{rmk}

\begin{rmk}
We emphasize that our simplification {\em crucially uses the existence of at least one end limiting to a Reeb orbit}   (See Section~\ref{subsubsection: nonnegative Euler char}) \cb and can be used to define SFT in the case when $\widehat{W}$ is an exact symplectic cobordism, although we only deal with contact homology in this paper.  Our simplified methods are not readily applicable to the case when there is a {\em closed} multiply-covered  $J$-holomorphic curve. However, we expect the combination of our methods and those of \cite{FO} and \cite{FO3} to yield a rigorous definition of SFT.
\end{rmk}

\begin{warning}
In this paper we do {\em not} prove that: 
\be
    \item dga morphisms induced by homotopic Liouville cobordisms are dga homotopic (for any definition of dga homotopy);
    \item the set of linearized contact homologies over the set of all augmentations of $\mathfrak{A}(\mathcal{D})$ is an invariant.
\ee

\end{warning}

Finally, during the preparation of this paper, Pardon~\cite{Pa} posted a paper proving the existence and invariance of contact homology in arbitrary dimensions.

\s\n
{\em Outline of the paper.}  We will be using the language of orbifolds and multisections, following Adem-Leida-Ruan~\cite{ALR} and Fukaya-Ono~\cite{FO}; this is reviewed in  Section~\ref{section: orbifolds and multisections}. After some preparation in Sections~\ref{section: almost complex structures and moduli spaces} and \ref{section: Fredholm theory}, we construct the interior semi-global Kuranishi chart in Section~\ref{section: semi-global Kuranishi charts}, with some modifications in Section~\ref{subsection: trimming}.   The semi-global Kuranishi structures are constructed in Section~\ref{section: construction of semi-global}, using the gluing results which are stated in Section~\ref{section: gluing}; they are average specimens of the ``gluing theorem'' type and are proved in Section~\ref{section: details of gluing}. Once the semi-global Kuranishi structures are constructed, the definition of contact homology, the definition of chain maps, and the invariance of contact homology (Theorems~\ref{thm: invariance} and \ref{thm: cobordisms}) follow the usual lines of argument and are carried out in Section~\ref{section: contact homology}. 

We note that the discussion until the end of Section~\ref{section: gluing} is valid for curves of any genus; starting from Section~\ref{section: construction of semi-global} we specialize to contact homology.

\s\n
 
{\em Acknowledgements.} We thank Francis Bonahon, Tobias Ekholm, Michael Hutchings, Eleny Ionel, Gang Liu, Kaoru Ono, Russell Avdek, and Garrett Alston for very helpful discussions.  We also thank Tobias Ekholm, John Pardon, and Russell Avdek for pointing out some errors, Kenji Fukaya for suggesting that we instead use the $\alpha$-action to define $I_\gamma$ in Section~\ref{subsubsection: map I gamma}, and the anonymous referees for extensive lists of comments.
\cb

\section{Orbifolds, orbibundles, and multisections} \label{section: orbifolds and multisections}

In this section we review the basics of (effective) orbifolds, orbibundles, and multisections from \cite{ALR}, \cite{FO}, \cite{FT}, and \cite{FO33}. The definitions in this section, while elementary, are a bit cumbersome to state.  Also, at the end of the day, all we do is replace the words ``manifold", ``vector bundle'', and ``section'' by ``orbifold'', ``orbibundle'', and ``multisection'', and treat them in exactly the same way for our purposes.  For the above reasons the reader may want to skip this section and return to it as the relevant concepts gradually start appearing in the rest of the paper. 

  {\em In this paper all our orbifolds are effective, unless stated otherwise.}

\subsection{Orbifolds and orbibundles}\label{subsection: orbifolds and orbibundles}

We state the definitions in the smooth category but they can easily be adapted to the $C^k$-category.\cb 

\begin{defn} [Orbifold charts] \label{defn: orbifold charts}
Given a topological space $\tp X$, a triple $(V,\Gamma, \phi)$ is an {\em orbifold chart of $\tp X$} if
\begin{enumerate}
\item $V$ is a connected open subset of $\mathbb R^n$;
\item $\Gamma$ is a finite group that acts smoothly and   effectively \cb on $V$; and
\item $\phi:V\to \tp X$ is a $\Gamma$-invariant continuous map such that the quotient map $V/\Gamma \to \tp X$, and  is homeomorphism onto an open subset $U$ of $\tp X$. 
\end{enumerate}
  We will often abuse notation and view $V/\Gamma$ as an open subset of $\tp X$.
If the ``effective" condition is not satisfied in (2), then we have a {\em non-effective} orbifold chart.\cb
\end{defn}


\begin{defn} [Embedding of orbifold charts]
If $(V_i,\Gamma_i,\phi_i)$ and $(V_j,\Gamma_j,\phi_j)$ are orbifold charts of $\tp X$, we say that $(V_i,\Gamma_i,\phi_i)$ {\em embeds into $(V_j,\Gamma_j,\phi_j)$} if there exists a smooth embedding $\psi_{ji}:V_i\hookrightarrow V_j$ such that $\phi_j\circ \psi_{ji}=\phi_i$.
\end{defn}


\coblu
\begin{rmk} \label{orbifold chart group map}
If $(V_i,\Gamma_i,\phi_i)$ embeds into $(V_j,\Gamma_j,\phi_j)$, then there exists \coblu an injective group homomorphism \cb $\theta_{ji}: \Gamma_i\to \Gamma_j$ (see \cite{ALR} for details) such that 
\be 
    \item $\psi_{ji}$ is $\theta_{ji}$-equivariant,
    \item $\theta_{ii} = \op{id}$, and
    \item $\theta_{kj} \circ \theta_{ji} = \theta_{ki}$, if $(V_i,\Gamma_i,\phi_i)$ embeds into $(V_j,\Gamma_j,\phi_j)$ and $(V_j,\Gamma_j,\phi_j)$ embeds into $(V_k,\Gamma_k,\phi_k)$.
\ee 
\end{rmk}
\cb

\begin{defn} [Orbifolds]  \label{def: orbifold}
Let $\tp X$ be a Hausdorff, second countable topological space.  An {\em (smooth) orbifold atlas of $\tp X$} is a family $\mathcal{O}=\{(V_i,\Gamma_i, \phi_i)\}_{i\in\mathcal{I}}$ of orbifold charts such that:
\begin{enumerate}
\item $\{ U_i= V_i/\Gamma_i \}_{i\in\mathcal{I}}$ forms an open cover of $\tp X$; and
\item for any $x\in U_i \cap U_j \subseteq \tp X$, there exist a neighborhood $U_k\subseteq U_i\cap U_j$ of $x$, a chart $(V_k,\Gamma_k,\phi_k)$ for $U_k$, and embeddings of orbifold charts $\psi_{ik}:(V_k,\Gamma_k,\phi_k) \hookrightarrow (V_i,\Gamma_i,\phi_i)$ and $\psi_{jk}:(V_k,\Gamma_k,\phi_k)\hookrightarrow (V_j,\Gamma_j,\phi_j)$.
\end{enumerate}
An orbifold atlas $\mathcal O$ {\em refines} another atlas $\mathcal O'$ if for every chart in $\mathcal O$ there exists an embedding into some chart of $\mathcal O'$. Two orbifold atlases are {\em equivalent} if they have a common refinement.  We denote by $[\mathcal O]$ the equivalence class of $\mathcal O$ and 
 call the pair $\X=(\tp X, [\mathcal{O}])$ a {\em smooth (effective) orbifold}.
\end{defn}

\coblu
\begin{notation}
We sometimes abuse notation and write $x\in \X$ instead of $x\in \tp X$.
\end{notation}

For any $i,j\in \mathcal I$, we denote $i\geq j$, if $(V_i,\Gamma_i,\phi_i)$ embeds into $(V_j,\Gamma_j,\phi_j)$.
For any $x \in \tp X$, the {\em stabilizer} $\Gamma^x$ of $x$ is the inverse limit of $(\{\Gamma_i\}_{i\in \mathcal I_x}, \{\theta_{ij}\}_{i, j \in \mathcal I_x})$, where $\mathcal I_x = \{i \in \mathcal I ~|~ x \in \op{Im}\phi_i\}.$
Then for any $i \in \mathcal I_x$, we have an injective map $\Theta_i^x: \Gamma^x \to \Gamma_i.$
Since the $\Gamma_i$ are finite groups, for sufficiently large $i$, $\Theta_i^x$ are isomorphisms.

\cb 


\begin{defn}[Orbifold maps] \label{def: orbifold maps}
A {\em (smooth) orbifold map}
$$\lambda: \X' = (\tp X', [\mathcal O']) \to \X= (\tp X, [\mathcal O])$$
consists of
\be
\item an underlying continuous map $\tp X'\to \tp X$, usually denoted by $\lambda$ by abuse of notation,
\ee
such that:
\be
\item[(2)] there exist orbifold atlases $\mathcal O'$ and $\mathcal O$ for $\tp X'$ and $\tp X$, so that for any $x'\in \tp X'$, we have  charts $(V',\Gamma',\phi') \in \mathcal O'$ around $x'$ and $(V,\Gamma,\phi)\in \mathcal O$ around $\lambda(x')$ such that $\lambda(V'/\Gamma')\subset V/\Gamma$, a homomorphism $\rho:\Gamma'\to \Gamma$,  and a $\rho$-equivariant smooth lift $\widehat \lambda : V'\to V$ of $\lambda|_{V'/\Gamma'}$ such that $\phi'\circ \widehat \lambda=\lambda \circ \phi$.
\ee
 
\end{defn}

\coblu 
For any $x'\in \tp X'$, the orbifold map $\lambda: \X' \to \X$ induces a homomorphism $\lambda^{x'}: \Gamma'^{x'} \to \Gamma^{\lambda (x')}$ between the stabilizers.

\begin{defn}[Orbifold embedding] \label{def: orbifold embedding}
A (smooth) orbifold map $\lambda: \X'  \to \X$ is a {\em (smooth) embedding} if 
\be 
\item $\lambda : \tp X' \to \tp X$ is a homeomorphism onto its image;
\item $\lambda^{x'}$ is an isomorphism for any $x'\in \tp X'$; and
\item Definition~\ref{def: orbifold maps} (2) holds with the additional requirements
    \be 
        \item $\widehat \lambda$ is an embedding, and
        \item $\Gamma' \simeq \Gamma'^{x'}$ and  $\Gamma \simeq \Gamma^{\lambda(x')}$.
    \ee
\ee 
\end{defn}

\cb 

\begin{defn}[Composition of orbifold maps] \label{def: composition of orbifold maps}
Given orbifold maps
$$\X'' = (\tp X'', [\mathcal O'']) \xrightarrow\mu \X'= (\tp X', [\mathcal O']) \quad \mbox{and} \quad \X' = (\tp X', [\mathcal O']) \xrightarrow\lambda \X= (\tp X, [\mathcal O]),$$
their composition  $\lambda\circ \mu$ is given as follows:
\be
\item the underlying continuous map $\lambda\circ \mu: \tp X''\to \tp X$ is the composition of the underlying continuous maps; and
\item Condition (2) in Definition~\ref{def: orbifold maps} is satisfied for the composition $\lambda\circ \mu$ since Condition (2) holds for both $\lambda$ and $\mu$ and the equivariant smooth lifts $\widehat \lambda$ and $\widehat \mu$ can be composed.
\ee
\end{defn} 

To pull back orbibundles in general, one requires the maps between orbifolds to be strong maps in \cite{MP} or (equivalently) good maps in \cite{CR} instead of just smooth maps as above.

\begin{defn}[Good maps \cite{CR}]
An orbifold map $$\lambda: \X' = (\tp X', [\mathcal O']) \to \X= (\tp X, [\mathcal O])$$
is called a {\em good map} if for any orbifold atlas $\widetilde{\mathcal{O}}$ of $\tp X$ there exist an orbifold atlas $\mathcal O'$ of $\tp X'$, a refinement $\mathcal O$ of $\widetilde{\mathcal{O}}$, and a map $i: \mathcal O' \to \mathcal O$ such that:
\be
\item for any orbifold chart $o'\in \mathcal O'$, $\lambda$ maps $o'$ to $i(o')$ so that Condition (2) in Definition~\ref{def: orbifold maps} is satisfied;
\item for any two orbifold charts $o'_1, o'_2 \in \mathcal O'$, if $o'_1$ embeds into $o'_2$, then $i(o'_1)$ embeds into $i(o'_2)$.
\ee 
\end{defn}

\begin{example}
Let $\X_1$ and $\X_2$ be orbifolds. The product $\X_1 \times \X_2$ has an obvious orbifold structure. The projection map $\op{pr}_i: \X_1 \times \X_2 \to \X_i$ for $i= 1, 2$ is a good map.
\end{example}

\begin{example}
Let $\X$ be an orbifold, and $\tp X'$ be an open subset of $\tp X$. There is an obvious orbifold structure on $\tp X'$ denoted by $\X'$. Then the inclusion map $\X' \to \X$ is a good map.
\end{example}

\begin{defn} [Orbibundle charts]
Let $\pi: \tp E \to \tp X$ be a continuous map between topological spaces. The triple $(V \times \R^n, \Gamma, \widetilde \phi)$ is an {\em orbibundle chart of $\pi:\tp E \to \tp X$ of rank $n$ over $(V,\Gamma,\phi)$,} if:\footnote{We will often suppress $(V,\Gamma,\phi)$ from the notation.}
\begin{enumerate}
\item $(V,\Gamma,\phi)$ is an (effective) orbifold chart of $\tp X$;
\item $\Gamma$ acts diagonally on $V\times \R^n$ with a linear action on $\R^n$;
\item $\widetilde \phi: V \times \R^n \to \tp E$ is a $\Gamma$-invariant continuous map such that the quotient map  is a homeomorphism $(V\times \R^n) /\Gamma \xrightarrow{\sim}  \pi^{-1}(\phi(V)) $; and
\item $\pi \circ \widetilde \phi=\phi \circ \op{pr}_V$, where $\op{pr}_V:V\times \R^n \to V$ is the projection onto the first factor.
\end{enumerate}

A {\em local orbibundle model} is a triple $(V\times \R^n, \Gamma,\op{pr}_V)$, where $\op{pr}_V$ is often suppressed from the notation.
\end{defn}

\begin{defn} [Embedding of orbibundle charts]
If
$$(V_i \times \R^n ,\Gamma_i,\widetilde\phi_i) \quad \mbox{and} \quad (V_j \times \R^n,\Gamma_j,\widetilde\phi_j)$$
are orbibundle charts of $\pi: \tp E \to \tp X$ of rank $n$ over $(V_i, \Gamma_i, \phi_i)$ and $(V_j, \Gamma_j, \phi_j)$, we say that $(V_i \times \R^n ,\Gamma_i,\widetilde\phi_i)$ {\em embeds into $(V_j \times \R^n,\Gamma_j,\widetilde\phi_j)$} if there exist
embeddings of orbifold charts 
\begin{gather*}
    \psi_{ji} : (V_i, \Gamma_i, \phi_i) \to (V_j, \Gamma_j, \phi_j),\\
    \widetilde{\psi}_{ji}: (V_i \times \R^n ,\Gamma_i,\widetilde\phi_i) \to (V_j \times \R^n,\Gamma_j,\widetilde\phi_j),
\end{gather*}
such that  
$$\op{pr}_{V_j}\circ \widetilde \psi_{ji}= \psi_{ji}\circ \op{pr}_{V_i}$$
and $\widetilde{\psi}_{ji}$ is a linear isomorphism on each fiber. 
We often refer to the embedding $\widetilde{\psi}_{ji}$ with the understanding that $\psi_{ji}$ is part of the data.
\end{defn}

\begin{defn} [Orbibundles] \label{defn: orbibundles}
Let $\pi: \tp E \to \tp X$ be a continuous map between Hausdorff, second countable topological spaces.
A {\em (smooth) rank $n$ orbibundle atlas of $ \pi:\tp E\to  \tp X$} is a family $\mathcal Y = \{(V_i\times \R^n,\Gamma_i,\widetilde\phi_i)\}_{i\in\mathcal{I}}$ of rank $n$ orbibundle charts of $E$ over $\mathcal O = \{(V_i, \Gamma_i, \phi_i)\}_{i\in \mathcal I}$ such that:
\begin{enumerate}
\item   $\{ U_i= V_i/\Gamma_i \}_{i\in\mathcal{I}}$ forms an open cover of $\tp X$; and
\item for any $x\in U_i \cap U_j$ where $U_i=V_i/\Gamma_i$ and $U_j=V_j/\Gamma_j$, there exist a neighborhood $U_k\subseteq U_i\cap U_j$ of $x$, an orbibundle chart $(V_k \times \R^n,\Gamma_k,\widetilde\phi_k)$ such that $U_k=V_k/\Gamma_k$, and embeddings of orbibundle charts
\begin{gather*}
    \widetilde \psi_{ik}:(V_k\times \R^n,\Gamma_k,\widetilde\phi_k)\hookrightarrow ( V_i\times \R^n,\Gamma_i,\widetilde\phi_i),\\
    \widetilde \psi_{jk}:(V_k\times \R^n,\Gamma_k,\widetilde\phi_k)\hookrightarrow ( V_j\times \R^n,\Gamma_j,\widetilde\phi_j).
\end{gather*}
\end{enumerate}
Let $[\mathcal Y]$ be the equivalence class of the orbibundle atlas $\mathcal Y$ under refinement.

The triple $(\EE= (\tp E, [\mathcal Y]),\X= (\tp X, [\mathcal O]),\pi: \EE \to \X)$, where $\pi: \EE\to \X$ is a smooth orbifold map and $\mathcal{Y}, \mathcal{O}$ form an orbibundle atlas of the underlying continuous map $\pi: \tp E\to \tp X$, is an {\em orbibundle}.   We refer to $\EE$ as the {\em total space} and $\X$ as the {\em base} of the orbibundle. 
\end{defn}

\begin{rmk}
{\coblu In practice, our orbifolds and orbibundles will be particularly simple and can be described by a single single orbifold or orbibundle chart, i.e., are {\em global quotient} orbifolds or orbibundles.}
\end{rmk}

\begin{defn} [Orbibundle maps] \label{defn: orbibundle maps}
A {\em (smooth) orbibundle map}
$$f = (f^\sharp, f^\flat) : (\EE', \X', \pi': \EE' \to \X') \to (\EE, \X, \pi: \EE \to \X)$$ 
between orbibundles consists of orbifold maps $f^\sharp: \EE' \to \EE$ and $f^\flat : \X' \to \X$  such that 
\be 
    \item the underlying continuous maps $f^\sharp : \tp E' \to \tp E$ and $f^\flat: \tp X' \to \tp X$ satisfy $f^\flat \circ \pi' = \pi \circ f^\sharp$,

    \item there exist orbibundle atlases $\mathcal Y'$ for $\EE'$ and $\mathcal Y$ for $\EE$, orbibundle charts $(V'\times \R^{n'}, \Gamma', \widetilde\phi')\in \mathcal Y'$ and $(V \times \R^n, \Gamma, \widetilde\phi) \in \mathcal Y$ with $x' \in \phi'(V')$ and $f^\flat (x') \in \phi(V)$ for each $x' \in \tp X'$, and equivariant smooth lifts 
    $$\widehat f^\flat: V' \to V, \quad \widehat f^\sharp: V'\times \R^{n'} \to V \times \R^n$$ such that
    $$\op{pr}_V \circ \widehat f^\sharp = \widehat f^\flat \circ \op{pr}_{V'},\quad f^\flat \circ \phi' = \phi \circ \widehat f^\flat,\quad f^\sharp \circ \widetilde \phi' = \widetilde \phi \circ \widehat f^\sharp,$$
    and for any $v'\in V'$ the map $\op{pr}_{\R^n} \circ \widehat f^\sharp (v', \cdot): \R^{n'} \to \R^n$ is linear. (Here $\op{pr}_{\R^n}$ is the projection onto the $\R^n$ factor.)
\ee
 
\end{defn}

\coblu
\begin{defn}
A (smooth) orbibundle map $$f = (f^\sharp, f^\flat): (\EE', \X', \pi': \EE' \to \X') \to (\EE, \X, \pi: \EE \to \X)$$  is a {\em (smooth) orbibundle embedding} if 
    \be 
        \item $f^\sharp$ and $f^\flat$ are homeomorphisms onto their images;
        \item the induced morphisms $(f^\sharp)^{e'}$ and $(f^\flat)^{\pi'(e')}$  between the stabilizers are isomorphisms for any $e'\in \tp E'$; and
        \item Definition~\ref{defn: orbibundle maps}(2) holds with the following additional requirements:
            \be 
                \item $\widehat{f}^\sharp$ and $\widehat{f}^\flat$ are embeddings, and
                \item $\Gamma' \simeq \Gamma'^{x'}$ and $\Gamma \simeq \Gamma^{f^\flat(x')}$.
            \ee
    \ee
\end{defn}
\cb 

\begin{fact}
Given an orbibundle $\pi: \EE\to \X$ and a good map $\lambda: \X'\to \X$, there exists a naturally defined pullback orbibundle $\pi':\lambda^*\EE\to \X'$ and an orbibundle map 
$$ (\pi':\lambda^*\EE\to \X') \to (\pi: \EE\to\X).$$
\end{fact}
\cb

\subsection{Multisections} \label{subsection: multisections}

Given a topological space $W$, we denote its $n$-fold symmetric product by
$$\op{Sym}^n(W):=W^n/S_n.$$
Here $S_n$ is the symmetric group of $n$ elements which acts on $W^n$ by
$$\sigma(x_1,\dots,x_n)=(x_{\sigma(1)},\dots,x_{\sigma(n)}),$$
where $\sigma\in S_n$ and $(x_1,\dots,x_n)\in W^n$.
We denote the equivalence class of $(x_1,\dots,x_n)$ in $\op{Sym}^n(W)$ by $[x_1,\dots,x_n]$.
If a group $\Gamma$ acts on $W$, then $\Gamma$ also acts on $\op{Sym}^n(W)$ diagonally by
$$g[x_1,\dots,x_n]=[g(x_1),\dots,g(x_n)],$$
where $g\in \Gamma$ and $[x_1,\dots,x_n]\in \op{Sym}^n(W)$.
There is an inclusion
$$\iota^{n}:\op{Sym}^{m}(W)\to \op{Sym}^{mn}(W),$$
\[ \label{eqn: splitting}
[w_{1},\dots,w_{m}]\mapsto [\underbrace{w_{1},\dots,w_{1}}_{n\text{ copies}},\dots,\underbrace{w_{m},\dots,w_{m}}_{n\text{ copies}}].
\]

Let $(V\times \R^n, \Gamma)$ be a local orbibundle model of rank $n$.
{\em For convenience, we write $L = V \times \R^n$ and use the notation $(L \to V, \Gamma)$ instead for the local model.}
Let $\op{Sym}^m (L)$ be the fiberwise symmetric product, i.e., $ \op{Sym}^m (L)=\op{Sym}^m(\mathbb R^n) \times V$.
Then $\op{Sym}^m(L)\to V$ is a fiber bundle with an equivariant $\Gamma$-action. 
\cb

\begin{defn}[Multisections of a local model]
A {\em degree $m$ multisection (or $m$-multisection) of the local orbibundle model $(L\to V,\Gamma)$}
is a $\Gamma$-equivariant section of $\op{Sym}^m(L)\to V$.  A {\em section} is a degree $1$ multisection.
\end{defn}

\begin{defn} $\mbox{}$
\be
\item
An $m$-multisection $\mathbf s$ of $(L\to V,\Gamma)$ is {\em liftable} if there exists a section $\widetilde {\mathbf s}=(s_1,s_2,\dots,s_m):V\to L^m$ which is not necessarily $\Gamma$-equivariant, such that $\mathbf s=p \circ \widetilde{\mathbf s}$, where $L^m=\mathbb R^{mn}\times V $ and $p:L^m\to \op{Sym}^m(L)$ is the quotient map. We call each such $s_i$ a {\em branch of $\mathbf s$}. 
\item \coblu A choice of $\widetilde {\mathbf s}$ is a {\em lifting} of $\mathbf s$; a multisection with a choice of lifting is said to be {\em lifted}.  Given $\widetilde{\mathbf s}=(s_1,s_2,\dots,s_m)$, its {\em $n$-fold splitting is} $$\iota^{n}\widetilde{\mathbf s}=(\underbrace{s_{1},\dots,s_{1}}_{n\text{ copies}},\dots,\underbrace{s_{m},\dots,s_{m}}_{n\text{ copies}}).$$
\ee
\end{defn}

We assign the weight $\frac{1}{m}$ to each branch $s_i$ of ${\bf s}$, if ${\bf s}$ is a degree $m$ multisection.

\begin{defn}
A liftable $m$-multisection $\mathbf s$ of $(L\to V,\Gamma)$ and a liftable $n$-multisection $\mathbf t$ of $(L\to V,\Gamma)$ are {\em transverse} (denoted by $\mathbf s \pitchfork \mathbf t$) if $s_i$ is transverse to $t_j$ for all $1\leq i \leq m$ and $1\leq j \leq n$.
\end{defn}

In the above definition, we also refer to $s_i^{-1}(t_j)\subset V$, where $i\in\{1,\dots,m\}$ and $j\in\{1,\dots,n\}$, as a {\em branch}, and assign the weight $\frac{1}{mn}$ to it.
We denote by 
$$\mathbf s^{-1}(\mathbf t):= \cup_{1\leq i \leq m, 1 \leq j \leq n} s_i^{-1}(t_j)$$ 
the weighted space (i.e., a space together with a map to $\Q^{\geq 0}$) where the weight of each point is defined by summing up the weights of all branches containing it. 

\cb

An $m$-multisection of $(L\to V,\Gamma)$ can be mapped to an $\ell m$-multisection of $(L\to V,\Gamma)$ by postcomposing with $\iota^\ell$, which is interpreted as acting on each fiber separately.

\begin{defn}
An $\ell$-multisection $\mathbf s$ and an $m$-multisection $\mathbf s'$ of $(L\to V, \Gamma)$ are said to be {\em equivalent} if $\iota^m \circ \mathbf s=\iota^\ell \circ \mathbf s'$.
\end{defn}

If ${\bf s}$ and ${\bf s}'$ are equivalent, then $\mathbf s^{-1}(\mathbf t)$ and $\mathbf {s'}^{-1}(\mathbf t)$ represent the same weighted space.

\begin{defn}[Multisections of an orbibundle] \label{defn: multisections}
A {\em multisection $\mathbf s$ of an orbibundle $\mathbb E \to \X$} is a family $\{(L_i,\Gamma_i,\widetilde\phi_i,\mathbf s_i)\}_{i\in\mathcal{I}}$ \coblu over an orbibundle atlas indexed by $\mathcal{I}$ \cb such that
\begin{enumerate}
\item $L_i = V_i \times \R^n$ and  $\{(L_i,\Gamma_i,\widetilde\phi_i)\}_{i\in\mathcal{I}}$ is an orbibundle atlas of $\mathbb E \to \X$;
\item $\mathbf s_i$ is an $m_i$-multisection of $(L_i\to V_i,\Gamma_i)$;
\item for any embedding 
$$\widetilde\psi_{ji}:(L_i,\Gamma_i,\widetilde\phi_i)\hookrightarrow (L_j,\Gamma_j,\widetilde\phi_j)$$
of orbibundle charts, the pullback $m_j$-multisection 
$$\mathbf s'_{i}:=({\widetilde\psi_{ji}}^{-1})^{m_j} \circ \mathbf s_j\circ \psi_{ji}$$ and the $m_i$-multisection $\mathbf s_i$ of $(L_i\to V_i,\Gamma_i)$ are equivalent. Here
$$\widetilde\psi_{ji}^{m_j}:\op{Sym}^{m_j}(L_i)\to \op{Sym}^{m_j}(L_j),$$
$$({\widetilde\psi_{ji}}^{-1})^{m_j}:\op{Sym}^{m_j}(L_j)|_{{\widetilde\psi_{ji}}^{m_j}(\op{Sym}^{m_j}(L_i))}\to \op{Sym}^{m_j}(L_i)$$
are the obvious maps induced by
$$\widetilde\psi_{ji}:L_i\to L_j \quad \mbox{and} \quad \widetilde \psi_{ji}^{-1}:L_j|_{\widetilde \psi_{ji}(L_i)}\to L_i.$$
\end{enumerate}
A {\em section $\mathbf s$} is a multisection which has degree $1$ on each local orbibundle model.

The multisection ${\bf s}$ is {\em liftable} if each ${\bf s_i}$ is liftable.

\coblu The multisection ${\bf s}$ is {\em lifted} if there exist an orbibundle atlas (still indexed by $\mathcal{I}$) and a lifting $\widetilde{\bf s}_i$ of each ${\bf s}_i$ such that
$$\widetilde{\mathbf s}'_{i}:={\widetilde\psi_{ji}}^{-1} \circ \widetilde{\mathbf s}_j\circ \psi_{ji}$$
and $\iota^{m_j/m_i}\widetilde{\bf s}_i$ agree as unordered sets with multiplicities.
\end{defn}

{\em In this paper, we restrict ourselves to liftable or lifted multisections.}

To better describe the intersections of multisections of orbifolds and semi-global Kuranishi structures, we introduce the notion of a weighted branched manifold. Roughly speaking, a branched manifold is a complex embedded in a Euclidean space such that at each point there is a unique tangent space. For convenience we also enlarge the category of underlying objects to include images of immersions, and not just embeddings. Intuitively, it is what $\mathbf s^{-1}(\mathbf t)$ should be, if $\mathbf s$ and $\mathbf t$ are multisections of $\EE \to \X$ such that $\mathbf s \pitchfork \mathbf t$.

\begin{defn}
A {\em weighted branched manifold structure of dimension $n$ with weight $w$} on a Hausdorff, second countable topological space $Z$ is a pair 
$$(\mathfrak F = \{(U_i, \mathfrak B_i, \Gamma_i) ~|~ i \in \mathcal I\},w)$$ 
such that:
\begin{enumerate}
    \item $w: Z \to \Q^{\geq 0}$ is a continuous function;
    \item $\{U_i\}_{i \in \mathcal I}$ is an open cover of $Z$;
    \item $\Gamma_i$ is a group which acts effectively on $U_i$; 
    \item $\mathfrak B_i = \{(Y_i^j, \varphi_i^j, w_i^j) ~|~ j \in \mathcal J_i\}$ with $\mathcal J_i$ finite, where:
    \begin{enumerate}[label=(\alph*)]
        \item $Y_i^j$ is a connected, relatively closed subset of $U_i$ and $\cup_{j\in \mathcal J_i} Y_i^j = U_i$;
        \item $\varphi_i^j: Y_i^j \to \R^n$ is a homeomorphism onto some open disk;
        \item $\varphi_i^j \circ (\varphi_{i'}^{j'})^{-1}$, where defined, is smooth for any $i,i' \in \mathcal I$, $j \in \mathcal J_i$, and $j' \in \mathcal J_{i'}$;
        \item $w_i^j \in \Q^{\geq 0}$ and for any $z\in U_i$, $w(z) = \textstyle\sum w_i^j$,
        where the sum is taken over all $j\in \mathcal J_i$ such that $z \in Y_i^j$;
        \item $\Gamma_i$ permutes the set $\{Y_i^j ~|~ j \in \mathcal J_i\}$ and $w_i^j = w_i^{j'}$ if $\Gamma_i$ maps $Y_i^j$ to $Y_i^{j'}$ (in particular, $w|_{U_i}$ is $\Gamma_i$-invariant); and
    \end{enumerate}
    \item $\mathfrak F$ is a maximal collection satisfing (2)--(4).
\end{enumerate}
\end{defn}

On each orbibundle chart $(L_i\to V_i, \Gamma_i, \widetilde \phi_i)$ of $\mathbb E\to \X$ over $(V_i,\Gamma_i,\phi_i)$,
if $\mathbf s_i$ is an $m_i$-multisection and $\mathbf t_i$ is an $n_i$-multisection,
then we define 
$$Y_i^j =\phi_i( s_\ell^{-1}(t_k)), \quad U_i:= \cup _{1\leq \ell \leq m_i, 1 \leq k \leq n_i}Y_i^j,$$
and $w_i^j = \frac{1}{m_i n_i}$, where $j = m_i (k-1) + \ell$.
This gives $\mathbf s^{-1}(\mathbf t)$ the structure of a weighted branched manifold by patching together $\{\mathbf s_i^{-1}(\mathbf t_i)\}_{i\in \mathcal I}$.
\cb

\coblu
\begin{defn}
A sequence of multisections $\{\mathbf s^n\}_{n\in \N}$ of $\mathbb E\to \X$ {\em converges to a multisection $\mathbf s$ of $\mathbb E \to \X$ in the $C^k$-topology} (where $k$ may be $\infty$), if for each fixed orbibundle chart $(L_i,\Gamma_i,\widetilde\phi_i)$ of $\mathbb E\to \X$, which is independent of $n$,
there exist $\ell_i$-multisections $\{\mathbf s'^n_i\}_{n\in \N}$ and $\mathbf s'_i$ of $(L_i,\Gamma_i)$ such that
    \be 
        \item $\ell_i$ is independent of $n$,
        \item $\mathbf s'_i$ is independent of $n$,
        \item $\mathbf s'^n_i$ is equivalent to $\mathbf s^n_i$, 
        \item $\mathbf s'_i$ is equivalent to $\mathbf s_i$, and 
        \item $\mathbf s'^n_i$ converges  to $\mathbf s'_i$ in the $C^k$-topology.
    \ee 
\end{defn}
\cb

The following is proved in \coblu \cite[Prop.\ 2.4.2]{FT}  (also see \cite[Theorem 3.11]{FO}). \cb

\begin{lemma}[Multisection perturbation lemma] \label{lemma: multisection perturbation lemma}
Let $\mathbf s,\mathbf t$ be two multisections of an orbibundle $\mathbb E \to \X$, {\coblu where $\mathbf s$ is lifted.} If the base $\tp X$ is compact, then there exists a sequence $\{\mathbf s^n\}_{n\in \N}$ of {\coblu lifted} multisections of $\mathbb E \to \X$ such that $\{\mathbf s^n\}_{n\in\N}$ converges to $\mathbf s$ in the $C^\infty$-topology and $\mathbf s^n$ is transverse to  $\mathbf t$.
\end{lemma}

{\coblu In practice, we will take ${\bf s}$ to be the zero section and ${\bf t}$ to be the linearized $\overline\bdry$-section.}

\section{Almost complex structures and moduli spaces} 
\label{section: almost complex structures and moduli spaces}

\nom[Mzxi]{$(M,\xi)$}{Contact manifold; $(M_+,\xi_+)$ and $(M_-,\xi_-)$ are contact manifolds at the positive and negative ends of a Liouville cobordism $\widehat W$}
\nom[Ralpha]{${\bf R}_\alpha$}{Reeb vector field for the contact form $\alpha$}
\nom[Palpha]{$\mathcal{P}_\alpha^L$}{Set of possibly multiply-covered Reeb orbits of $\alpha$ with action $<L$}
\nom[Palphagood]{$\mathcal{P}_\alpha^{L,\op{good}}$}{Good Reeb orbits of $\mathcal{P}_\alpha^L$}
\nom[l0]{$L>0$}{A large constant which serves as the action cutoff in $\mathcal{P}_\alpha^L$ and the notions of {\em $L$-simple} and {\em $L$-nondegenerate}}

Let $(M^{2n+1},\xi)$ be a contact manifold and $\alpha$ a nondegenerate contact form for $\xi$. \coblu Let $L$ be a positive real number. \cb Denote the Reeb vector field of $\alpha$ by $\mathbf R_\alpha$ and the set of possibly multiply-covered closed orbits $\gamma$ of $\mathbf R_\alpha$ with action $\mathcal A_\alpha(\gamma)\leq L$ by $\mathcal P_\alpha^L$.  Also let $\mathcal P_\alpha^{L,\op{good}}\subset \mathcal P_\alpha^L$ be the subset of good Reeb orbits.

{\em The discussion until the end of Section~\ref{section: gluing} is valid for curves of any genus.}

\subsection{Almost complex structures} \label{subsection: almost complex structures}

We will review some notions from \cite{BH}.

\begin{defn} \label{Lsimple form}
A contact form $\alpha$ is called {\em $L$-simple}\footnote{In \cite{BH} there is a notion called ``$L$-supersimple", the difference between ``$L$-simple" and ``$L$-supersimple" is that ``$L$-simple" allows the existence of elliptic orbits of action $\leq L$.}
if 
\be
\item each simple orbit $\gamma\in \mathcal P_\alpha^L$ of $\mathbf R_\alpha$ is nondegenerate;

\item around $\gamma$ there exists an embedding
$$\Phi_\gamma:[0, \mathcal{A}_\alpha(\gamma)] \times D_\delta^{2\overline m} \times D_\delta ^{2 \underline m}/_{\sim} \to M,$$
such that $[0, \mathcal{A}_\alpha(\gamma)]$ has coordinate $t$, 
$\overline m, \underline m \in \Z^{\geq 0}$ satisfy $\overline m + \underline m = n$,  $\delta>0$ is small, 
\begin{gather*}
    D_\delta^{2\overline m} =\Big\{(\overline x,\overline y)=(\overline x_1,\dots,\overline x_{\overline m}, \overline y_1,\dots, \overline y_{\overline m})\in \R^{2\overline m} ~|~\sum_{i=1}^{\overline m} \overline x_i^2+ \overline y_i^2\leq \delta^2 \Big\},\\
    D_\delta^{2\underline m} =\Big\{(\underline x,\underline y)=(\underline x_1,\dots, \underline x_{\underline m}, \underline y_1,\dots, \underline{y}_{\underline m})\in \R^{2 \underline m} ~|~\sum_{i=1}^{\underline m} \underline x_i^2+ \underline y_i^2\leq \delta^2 \Big\},
\end{gather*}
$(0,\overline x, \overline y, \underline x, \underline y) \sim (\mathcal{A}_\alpha(\gamma), \overline x, \overline y, -\underline x, -\underline y)$, $\gamma=\Phi_\gamma([0, \mathcal{A}_\alpha(\gamma)] \times\{0\} \times\{0\})$, and
$$\Phi_\gamma^*\alpha=(1+\overline Q + \underline Q)dt+\beta; \quad \mbox{and}$$ 

\item the images of $\Phi_\gamma$ are disjoint for all $\gamma\in \mathcal P_\alpha^L$.

\ee

Here $\overline Q$ is quadratic in $(\overline x,\overline y)$
(without constant or linear terms), 
$\underline Q$ is quadratic in $(\underline x, \underline y)$, and
\begin{align} \label{formula for beta}
\beta=\sum_{i=1}^{\overline m} \frac{1}{2}(\overline x_i d\overline y_i - \overline y_i d\overline x_i) + \sum_{i=1}^{\underline m} \frac{1}{2}(\underline x_i d\underline y_i - \underline y_i d\underline x_i).
\end{align}
\end{defn}

We denote 
$$(\mathcal N_0, \alpha_0):=([0,\mathcal{A}_\alpha(\gamma)] \times D_\delta^{2\overline m} \times D_\delta ^{2 \underline m}/_{\sim}, (1+\overline Q + \underline Q)dt +\beta).$$ 
Russell Advek pointed out a mistake in the definition of $(\mathcal N_0, \alpha_0)$ in the previous versions and suggested the current version. 

\begin{lemma} \label{lemma: perturb alpha}
Given an $L$-nondegenerate\footnote{This means all the orbits with action $\leq L$ are nondegenerate} contact form $\alpha'$ on $M$, there exists a perturbation $\alpha$ of $\alpha'$ on small neighborhoods of simple orbits of $\mathcal{P}_{\alpha'}^L$ such that $\ker \alpha = \ker \alpha'$, $\alpha$ is $L$-nondegenerate, $\mathcal{P}_\alpha^L=\mathcal{P}_{\alpha'}^L$, the orbits have the same $\alpha'$ and $\alpha$ actions, and $\alpha$ is $L$-simple.
\end{lemma}

\begin{proof}
For each simple orbit $\gamma\in \mathcal{P}_{\alpha'}^L$, let $d\varphi^t: \xi_{\gamma(0)} \to \xi_{\gamma(t)}$ be the linearized Reeb flow along $\gamma$. Without loss of generality, we assume that $\mathcal{A}_\alpha(\gamma) = 1.$
Writing $p=\gamma(0)$, $d\varphi^1: \xi_p \stackrel\sim \to \xi_p$ is a linear symplectic map with symplectic form $\omega'_p = d\alpha' |_{\xi_p}$. Since $\gamma$ is nondegenerate, $1$ is not an eigenvalue of $d\varphi^1$.
The contact plane $\xi_p$ splits as the direct sum $\overline{\xi}_p \oplus \underline{\xi}_p$ of two $d\varphi^1$-invariant symplectic subspaces such that $d\varphi^1|_{\overline \xi_p }$ has no negative eigenvalue
and $d\varphi^1|_{\underline \xi_p}$ has only negative eigenvalues.

Choose an identification $L_0: \R^{2n} \stackrel\sim \to \xi_p$,
such that 
\be
\item $L_0(\R^{2\overline m} \oplus 0) = \overline \xi_p$,
\item $L_0(0 \oplus \R^{2\underline{m}}) = \underline \xi_p,$ and 
\item $L_0^* \omega'_p = \sum_{i=1}^{\overline m} d\overline x_i \wedge d\overline y_i + \sum_{i=1}^{\underline m} d\underline x_i \wedge d\underline y_i=:\omega.$
\ee 
The symplectic matrix $A = L_0^*(d\varphi^1): \R^{2n} \to \R^{2n}$ splits as $A = \overline A \oplus \underline A$,
where $\overline A \in \op{Sp}(2\overline m, \R)$ and $\underline A \in \op{Sp}(2\underline m, \R).$
Then there exist matrices $\overline X \in \mathfrak{sp}(2\overline m, \R)$ and $\underline X \in \mathfrak{sp}(2\underline m, \R)$ such that $e^{\overline X} = \overline A$ and $e^{\underline X} = - \underline A$: Indeed, let $\overline X, \underline X$ be the principal logarithms of $\overline A$ and $\underline A$ (see \cite[Theorem 1.23]{Hi}). By the standard properties of principal logarithms (see \cite[Theorems 1.13 and 11.2]{Hi}), $\overline X$, $\underline X$ lie in the Lie algebras of the corresponding Lie groups.

Consider the vector field $Y = X \cdot (\overline x,\overline y, \underline x, \underline y)^T$ on $D_\delta^{2\overline m} \times D_\delta ^{2 \underline m}$ with $X = \overline X \oplus \underline X$.
The Lie derivative $\mathcal L_{Y}\omega = 0$, since $\overline X\in \mathfrak{sp}(2\overline m, \R)$ and $\underline X\in \mathfrak{sp}(2 \underline m, \R)$.
This implies that $d (Y\lrcorner~\omega) = 0$. Let $H: D_\delta^{2\overline m} \times D_\delta ^{2 \underline m} \to \R$ be a Hamiltonian function such that $Y \lrcorner ~\omega = dH$. Since $Y$ is linear in $\overline x, \overline y, \underline x, \underline y$, we  may take $H$ to be $\overline Q +\underline Q$, where $\overline Q$ is quadratic in $(\overline x,\overline y)$ and $\underline Q$ is quadratic in $(\underline x, \underline y)$. Consider the $1$-form
$$\alpha_0:=(1+H)dt + \beta \quad \mbox{on} \quad [0,1]\times D_\delta^{2\overline m} \times D_\delta ^{2 \underline m},$$ where $\beta$ is as in Equation~\eqref{formula for beta}. It is contact for $\delta>0$ sufficiently small and its Reeb vector field is directed by $\bdry_t+ Y$.

For $t\in [0,1]$ we extend the definition of $L_0$ to:
$$L_t: \R^{2n} \stackrel\sim\to \xi_{\gamma(t)}, \quad \zeta\mapsto d\varphi^t \circ L_0 \circ (e^{-t\overline X}\oplus e^{-t\underline X}) \zeta.$$
Identifying $T_{(t, 0, 0)}( \{t\} \times D_\delta^{2 \overline m} \times D_\delta ^{2 \underline m})$ with $\R^{2n}$ using the basis $\{\bdry_{\overline x}, \bdry_{\overline{y}}, \bdry_{\underline x}, \bdry_{\underline y} \}$, we can form
\begin{gather*}
    \bold L: \coprod_{t \in [0,1]} T_{(t, 0, 0)} (\{t\} \times D_\delta^{2 \overline m} \times D_\delta ^{2 \underline m})  \to 
    \coprod_{t \in [0,1]} T_{\gamma(t)}M,\quad
    \zeta  \mapsto  L_t\zeta,
\end{gather*}

Now let $\Phi$ be a map from $\mathcal N_0$ to a neighborhood of $\gamma\subset M$ such that its tangent map along $[0,1]\times \{0\} \times \{0\}$ agrees with $\bold L$. On $\mathcal N_0$ we have two contact forms $\alpha_0$ and $\Phi^*\alpha'$ that agree along $[0,1]\times \{0\} \times \{0\}$. By a standard Moser-type argument, there exists a diffeomorphism $\Psi$ of $\mathcal N_0$ (after shrinking $\mathcal{N}_0$ if necessary), such that $\Psi$ maps $\ker \alpha_0$ to $\ker \Phi^*\alpha'$. Therefore, $\Phi \circ \Psi: \mathcal N_0 \to M$ maps $\ker \alpha_0$ to $\ker \alpha'.$ 
Near $\gamma$, we replace the contact form $\alpha'$ by $(\Phi \circ \Psi)_* \alpha_0$, and in an annulus region we interpolate between the two.
After the modification, the contact condition, the Reeb orbit $\gamma$, and the linearized return map are preserved, and we do not create additional orbits of action $\leq L$.
\end{proof}

Let $Y$ be the vector field on $D_\delta^{2\overline m} \times D_\delta^{2\underline m}$ satisfying $$Y \lrcorner ~ (\textstyle\sum_{i=1}^{2\overline m}  d\overline x_i\wedge d\overline y_i + \textstyle\sum_{i=1}^{2\underline m} d\underline x_i \wedge d\underline y_i)= dQ,$$ where $Q = \overline Q + \underline Q$.
Since $Q$ is quadratic, $Y$ is also a vector field on $\mathcal N_0$.
Because $\partial_t+ Y \in \ker d\alpha$, 
there exists a positive function $g$ on $\mathcal N_0$ such that
\begin{itemize}
\item[(g)] $g \mathbf R_\alpha=\partial_t+Y$.
\end{itemize}
Since $Q$ is quadratic in 
$$(x_1,\dots, x_n, y_1,\dots,y_n)= (\overline x_1,\dots, \overline x_{\overline m}, \underline x_1,\dots,\underline x_{\underline m}, \overline y_1,\dots, \overline y_{\overline m}, \underline y_1,\dots, \underline y_{\underline m}),$$
the coefficients of $Y$ with respect to $(\bdry_{x_1},\dots,\bdry_{x_n},\bdry_{y_1},\dots,\bdry_{y_n})$ can be written as a column vector $X \cdot (x_1,\dots, x_n, y_1,\dots,y_n)^T$, where $X$ is a $2n\times 2n$ matrix with constant real coefficients. 
\cb 

\begin{defn}
Given a contact form $\alpha$, an almost complex structure $J$ on $\R \times M$ is {\em $\alpha$-tame} if
\begin{enumerate}
\item $J$ is independent of $s$, where $s$ is the $\R$-coordinate;
\item $J(\partial_s)=g\mathbf R_\alpha$ for some positive function $g$ on $M$; and
\item there exists a $2n$-plane field $\xi'$ on $M$ such that $J\xi'=\xi'$ and $d\alpha(v,Jv)>0$ for all nonzero $v\in \xi'$.
\end{enumerate}
\end{defn}

\nom[gzammas]{$\gamma^s$}{Simple orbit underlying the Reeb orbit $\gamma$}
\nom[mzgamma]{$m(\gamma)$}{Multiplicity of the Reeb orbit $\gamma$}
\nom[xgammas]{$x_{\gamma^s}$}{Point (marker) on the simple orbit $\gamma^s$ underlying $\gamma$}
Given a Reeb orbit $\gamma$ of $\mathbf R_\alpha$, we write $\gamma^s$ for the simple orbit underlying $\gamma$ and $m(\gamma)$ for the multiplicity of $\gamma$ over $\gamma^s$. Choose a point $x_{\gamma^s}$ (a {\em marker}) on each $\gamma^s$.

\begin{defn}[$L$-simple for a symplectization] \label{defn: L-supersimple 1} Let $L>0$ and $\alpha$ be an $L$-simple contact form.

(1) An almost complex structure $J$ on $\R\times M$ is {\em $L$-simple for $\alpha$} if $J$ is $\alpha$-tame and inside a sufficiently small neighborhood $\Phi_\gamma(\mathcal N_0)$ (recall $\mathcal N_0 = [0, \mathcal{A}_\alpha(\gamma)] \times D_\delta^{2\overline m} \times D_\delta ^{2 \underline m}/_{\sim} $) of each simple Reeb orbit $\gamma \in \mathcal P_\alpha^L$ we have
\cb
\begin{enumerate}
\item[(i)] $\xi'=\R\langle\frac{\partial}{\partial x_1},\dots, \frac{\partial}{\partial x_n}, \frac{\partial}{\partial y_1},\dots, \frac{\partial}{\partial y_n}\rangle$;
\item[(ii)] $J: \frac{\partial}{\partial x_i} \mapsto \frac{\partial}{\partial y_i}$ for $1\leq i \leq n$; and
\item[(iii)] the function $g$ satisfying $J(\partial_s)=g\mathbf R_\alpha$ is given by (g) above.
\end{enumerate}
A pair $(\alpha,J)$ is said to be {\em $L$-simple pair}, if $\alpha$ is $L$-simple and $J$ is $L$-simple for $\alpha$.

(2) If $\gamma$ is simple, then the above coordinates
$$(t,x_1,\dots,x_n,y_1,\dots,y_n),$$
subject to the normalization
\begin{itemize}
\item[(N)] $\Phi_\gamma(0,0)=x_{\gamma^s}$.
\end{itemize}
are called {\em simple coordinates for $\gamma$}.

(3) If $\gamma$ is not simple, then, by abuse of notation, the coordinates 
$$(t,x_1,\dots,x_n,y_1,\dots,y_n)$$
on $[0, \mathcal{A}_\alpha(\gamma)] \times D_\delta^{2\overline m} \times D_\delta ^{2 \underline m}/_{\sim}$ are called {\em simple coordinates for $\gamma$} where
$$[0, \mathcal{A}_\alpha(\gamma)] \times D_\delta^{2\overline m} \times D_\delta ^{2 \underline m}/_{\sim} \to [0, \mathcal{A}_\alpha(\gamma^s)] \times D_\delta^{2\overline m} \times D_\delta ^{2 \underline m}/_{\sim}$$
is the $m(\gamma)$-fold covering map.
\end{defn}

Given an $L$-simple $\alpha$, we can construct an $L$-simple $J$ as follows: Let $\xi'\subset TM$ be a $2n$-plane field such that $\xi'=\R\langle\frac{\partial}{\partial x_1},\dots, \frac{\partial}{\partial x_n}, \frac{\partial}{\partial y_1},\dots, \frac{\partial}{\partial y_n}\rangle$ on a small neighborhood $V_\gamma$ of each $\gamma\in \mathcal P_\alpha^L$; $\xi'=\xi$ outside a slightly larger neighborhood $V'_\gamma$ of each $\gamma\in \mathcal P_\alpha^L$; and $\xi'$ interpolates between $\R\langle\frac{\partial}{\partial x_1},\dots, \frac{\partial}{\partial x_n}, \frac{\partial}{\partial y_1},\dots, \frac{\partial}{\partial y_n}\rangle$ and $\xi$ on the remaining annular regions.  Let $J$ be a complex structure of $\xi'$ such that $d\alpha(v,Jv)>0$ for all $0\neq v\in \xi'$ and $J(\frac{\partial}{\partial x_i})= \frac{\partial}{\partial y_i}$ for $1\leq i \leq n$ on $V_\gamma$. We then extend $J$ to $\mathbb R\times M$ by $J(\partial_s)=g\mathbf R_\alpha$, where $g$ is given by (g) on $V_\gamma$, $g=1$ outside $V'_\gamma$, and $g$ interpolates between the two in the remaining annular regions.

The main reason for using $L$-simple $J$ is the following: Let $u$ be a $J$-holomorphic curve in $\R \times M$ and suppose $u$ has an end which converges to $\R \times \gamma$ with coordinates $(s,t)$. Then, with respect to the simple coordinates, this end of $u$ can be written as $(s,t,\eta(s,t))$, where $\eta$ satisfies  
\begin{equation}\label{normal equation}
\frac{\partial \eta }{\partial s}+j_0\frac{\partial \eta}{\partial t}+S\eta=0,
\end{equation}
$j_0$ is the standard complex structure, and $S:=-j_0X$ is a $2n\times 2n$ symmetric matrix.  \cb In other words, {\em the Cauchy-Riemann equations become linear near the ends for curves that are graphical over $\R\times \gamma$.}

\begin{convention}
We are using the convention that when a matrix (e.g., $S$) acts on $\eta$ or its derivatives from the left, $\eta$ or its derivatives are viewed as column vectors.
\end{convention}

More generally, let $(W,\alpha)$ be a $(2n+2)$-dimensional   Liouville cobordism from $(M_+,\alpha_+)$ to $(M_-,\alpha_-)$, i.e., $W$ is compact, \cb $\bdry W= M_+-M_-$, $\alpha$ is a $1$-form on $W$ such that $d\alpha$ is symplectic, and $\alpha_\pm=\alpha|_{M_\pm}$ is a contact form on $M_\pm$.  Let $(\widehat{W},\widehat\alpha)$ be the completion of $(W,\alpha)$, obtained by smoothly attaching the symplectization ends $[1,\infty)\times M_+$ and $(-\infty,-1]\times M_-$.  Let $s$ be the $[1,\infty)$- or $(-\infty,-1]$-coordinate at the positive and negative ends of $\widehat{W}$.

\begin{defn}[$(L_+,L_-)$-simple for a cobordism]
A $1$-form $\widehat \alpha$ on $\widehat W$ is said to be {\em $(L_+,L_-)$-simple} if it restricts to $L_\pm$-simple contact form $\alpha_\pm$ at the positive and negative ends.
Given an $(L_+,L_-)$-simple $\widehat \alpha$, an almost complex structure $J$ on $\widehat{W}$ said to be  {\em $(L_+,L_-)$-simple for $\widehat \alpha$},
if it is $d\widehat{\alpha}$-tame and restricts to $J_\pm$ which is $L_\pm$-simple for $\alpha_\pm$ at the positive and negative ends.
A pair $(\widehat \alpha, J)$ is said to be $(L_+,L_-)$-simple, if $\widehat \alpha$ is $(L_+,L_-)$-simple and $J$ is $(L_+,L_-)$-simple for $\widehat \alpha$.
\end{defn}

{\em In this paper we assume that all almost complex structures in symplectizations (resp. cobordisms) are $L$-simple (resp. $(L_+,L_-)$-simple) for some appropriate $L$ (resp. $(L_+,L_-)$) and some appropriate contact form $\alpha$ (resp. $\widehat \alpha$). }

\begin{rmk}
Strictly speaking, it is not necessary to use simple $J$ in this paper, as most of our discussion carries over to the general case.  Simple almost complex structures allow for better control of the ends and also simplify gluing.
\end{rmk}

\subsection{Riemann surfaces and holomorphic maps}

Let $(F,j)$ be a closed connected Riemann surface.

\nom[p]{${\bf p}$}{Ordered tuple $\mathbf{p}=\mathbf{p}_+ \sqcup \mathbf{p}_-$ of marked points on a closed Riemann surface $(F,j)$}
\nom[r]{${\bf r}$}{Asymptotic markers}

\nom[Fd]{$\dot F$}{The punctured surface $F-{\bf p}$}

\begin{defn} \label{defn: marked Riemann surface}
A {\em marked Riemann surface}
$$\F=(F,j,\mathbf{p},\mathbf{r})$$
is a quadruple which consists of $(F,j)$, an {\em ordered} tuple $\mathbf{p}=\mathbf{p}_+ \sqcup \mathbf{p}_-$, $\mathbf{p}_\pm=(p_{\pm,1},\dots,p_{\pm,l_\pm})$, of distinct points on $F$ and an {\em ordered} tuple $\mathbf{r} = \mathbf{r}_+ \sqcup \mathbf{r}_-$, $\mathbf{r}_\pm=(r_{\pm,1},\dots,r_{\pm,l_\pm})$, of asymptotic markers at $\mathbf{p}$.
Here an {\em asymptotic marker} $r_{\pm,i}$ at a puncture $p_{\pm,i}$ is an element of $(T_{p_{\pm,i}} F-\{0\})/\R^+$.
We denote $\dot F=F-\mathbf{p}$.
\end{defn}

\nom[Simzero]{$\sim_0$}{Equivalence relation between two marked Riemann surfaces given by Definition~\ref{defn: equiv relation sim zero}; also see Change of Notation~\ref{changeofnotation}}

\begin{defn}[Equivalence relation $\sim_0$] \label{defn: equiv relation sim zero}
Given two marked Riemann surfaces $\F=(F,j,\mathbf{p},\mathbf{r})$ and $\F'=(F',j',\mathbf{p}',\mathbf{r}')$,  we write $\F\sim_0\F'$ if there is a diffeomorphism $\phi:F\xrightarrow\sim F'$ satisfying $\phi_*(\F)=\F'$ (that is, $\phi_{*}j=j'$, $\phi(\mathbf{p}_\pm)=\mathbf{p}'_\pm$, and $\phi_{*}(\mathbf{r}_\pm)=\mathbf{r}'_\pm$, where the latter two maps are maps of ordered tuples).
\end{defn}

Let $(W,\alpha)$ be the exact symplectic cobordism from Section~\ref{subsection: almost complex structures}.

\nom[Sim]{$\sim$}{Equivalence relation between two maps from Riemann surfaces given by Definition~\ref{defn: equiv relation sim}; also see Change of Notation~\ref{changeofnotation}}

\begin{defn}[Equivalence relation $\sim$]\label{defn: equiv relation sim}
Given marked Riemann surfaces $\F=(F,j,\mathbf{p},\mathbf{r})$, $\F'=(F',j',\mathbf{p}',\mathbf{r}')$ and maps $u: \dot F\to \widehat{W}$, $u': \dot F'\to \widehat{W}$, we write $(\F,u)\sim (\F',u')$ if there is a diffeomorphism $\phi :  F\xrightarrow\sim F'$ satisfying $\phi_*(\F,u)=(\F',u')$ (that is, $\phi_{*}j=j'$, $\phi(\mathbf{p}_\pm)=\mathbf{p}'_\pm$ and $\phi_{*}(\mathbf{r}_\pm)=\mathbf{r}'_\pm$, again as ordered tuples, and $u' \circ \phi=u$).
\end{defn}

\nom[Fu1]{$[\mathcal{F},u]$}{Equivalence class of $(\mathcal{F},u)$ under the equivalence relation $\sim$ given by Definition~\ref{defn: equiv relation sim}}

\begin{defn}[Equivalence classes] \label{defn: llbracket}
The equivalence class of $(\F,u)$ under $\sim$ will be denoted by $[\F,u]$. If $\widehat{W}=\R\times M$, then the equivalence class of $(\F,u)$ under $\R$-translations   in the target \cb will be denoted by $(\F,u)_{\R}$ and the equivalence class under $\sim$ and $\R$-translations in the target will be denoted by $\llbracket \F,u\rrbracket$.
\end{defn}

\nom[Fu2]{$\llbracket \F,u\rrbracket$}{Equivalence class of $(\mathcal{F},u)$ under $\sim$ and $\R$-translations in the target given by Definition~\ref{defn: llbracket}}

Let $\bs\gamma_+=(\gamma_{+,1},\dots,\gamma_{+,l_+})$ and $\bs\gamma_-=(\gamma_{-,1},\dots,\gamma_{-,l_-})$ be ordered tuples of Reeb orbits for $\alpha_+$ and $\alpha_-$.  Then $\mathcal{M}_J^{\op{ind}=k}(\dot F,\widehat{W};\bs\gamma_+;\bs\gamma_-)$ is the space of equivalence classes $[\F,u]$ of pairs, where $u$ is a $J$-holomorphic map $\dot F \to \widehat{W}$ from $\bs\gamma_+$ to $\bs\gamma_-$ (i.e., $u$ is asymptotic to $\gamma_{+,i}$ at the positive end near $p_{+,i}$ and to $\gamma_{-,i}$ at the negative end near $p_{-,i}$) and $u$ ``maps" the asymptotic markers $r_{\pm,i}$ to the markers $x_{\gamma^s_{\pm,i}}$. \coblu  Also the superscript $\op{ind}=k$ as in $\mathcal{M}_J^{\op{ind}=k}(\dot F,\widehat{W};\bs\gamma_+;\bs\gamma_-)$ indicates the subset consisting of $[\F,u]$ of Fredholm index $\op{ind}(u)=k$. \cb

If $\widehat{W}=\R\times M$, then $\mathcal{M}_J(\dot F,\R\times M;\bs\gamma_+;\bs\gamma_-)/\R$ is the usual quotient of $\mathcal{M}_J(\dot F,\R\times M;\bs\gamma_+;\bs\gamma_-)$ by $\R$-translations in the target.

\begin{defn}
The {\em $\alpha$-energy} of a map $u$ from $\bs\gamma_+$ to $\bs\gamma_-$ is given by $$E_\alpha(u):=\mathcal{A}_\alpha(\bs\gamma_+)-\mathcal{A}_\alpha(\bs\gamma_-),$$
where $\mathcal{A}_\alpha(\bs\gamma_\pm)=\sum_{i=1}^{l_\pm}\mathcal{A}_\alpha(\gamma_{\pm,i})$.
\end{defn}

\nom[A]{$\mathcal{A}_\alpha(\bs\gamma)$}{The $\alpha$-action/energy of the collection of orbits $\bs\gamma$}

\subsection{Sorting}

\coblu
We fix an ordering of $\mathcal P^{L}_\alpha$ which we denote by $\vartheta$. 
In practice, we choose a generic $\alpha$ such that the map $\mathcal{P}^L_\alpha\to \R^+$, $\gamma\mapsto \mathcal{A}_\alpha(\gamma)$, is injective and order the orbits $\mathcal P^{L}_\alpha$ by increasing action.
\cb 
\nom[tZheta]{$\vartheta$}{An ordering of $\mathcal P_\alpha^L$}

\begin{defn}
An ordered tuple of Reeb orbits $\bs \gamma$ is {\em sorted} if
$$\bs \gamma=(\underbrace{\gamma_{1},\dots, \gamma_{1}}_{i_{1} \text{ copies }},\dots, \underbrace{\gamma_{k},\dots, \gamma_{k}}_{i_{k} \text{ copies }}),$$
where $\gamma_{i} \neq \gamma_{j}$ as long as $i\neq j$.  We also allow $k=0$, i.e., $\bs \gamma=\varnothing$.
If $\bs \gamma$ is sorted according to $\vartheta$,   i.e., we have $\vartheta(\gamma_i) < \vartheta(\gamma_j)$ for any $1 \leq i<j \leq k$, \cb then $\bs \gamma$ is said to be {\em $\vartheta$-sorted}.
\end{defn}

{\em From now on we assume that all the positive (and likewise all the negative) ends of all the moduli spaces are $\vartheta$-sorted   using the action, \cb unless stated otherwise.}

\subsection{Puncture reorderings and marker rotations} \label{subsection: puncture reorderings and marker rotations}


\coblu Let $[(F,j,{\bf p},{\bf r}),u]$ be an element of $\mathcal{M}_J(\dot F,\widehat{W};\bs\gamma_+;\bs\gamma_-)$, or more generally of the weighted Sobolev space $W_\delta^{l+1,p}(\dot F, \widehat{W};\bs\gamma_+;\bs\gamma_-)$ given in Definition~\ref{def: weighted Sobolev space}. As before $\mathbf{p}=\mathbf{p}_+ \sqcup \mathbf{p}_-$, $\mathbf{p}_\pm=(p_{\pm,1},\dots,p_{\pm,l_\pm})$, $\mathbf{r} = \mathbf{r}_+ \sqcup \mathbf{r}_-$, and $\mathbf{r}_\pm=(r_{\pm,1},\dots,r_{\pm,l_\pm})$.

\begin{defn} \label{defn: puncture reordering/marker rotation} $\mbox{}$
\be
\item A {\em puncture reordering} takes $[(F,j,{\bf p},{\bf r}),u]$ to 
$$[(F,j,\sigma_+({\bf p}_+)\sqcup \sigma_-({\bf p}_-),\sigma_+({\bf r}_+)\sqcup \sigma_-({\bf r}_-)),u],$$
where $\sigma_\pm\in S_{l_\pm}$,  
\begin{align*}
\sigma_\pm({\bf p}_\pm)& =(p_{\pm, \sigma_\pm(1)},\dots, p_{\pm, \sigma_\pm(l_\pm)})\\
\sigma_\pm({\bf r}_\pm )& =(r_{\pm, \sigma_\pm(1)},\dots, r_{\pm, \sigma_\pm(l_\pm)}),
\end{align*}
and the reordering preserves $\bs\gamma_\pm$ (i.e., $\gamma_{\pm,\sigma_\pm(i)}= \gamma_{\pm,i}$).

\item A {\em marker rotation} takes $[(F,j,{\bf p},{\bf r}),u]$ to $[(F,j,{\bf p}, {\bf r}'),u]$, where we write $\mathbf{r}' = \mathbf{r}'_+ \sqcup \mathbf{r}'_-$ and $\mathbf{r}'_\pm=(r'_{\pm,1},\dots,r'_{\pm,l_\pm})$ and the asymptotic markers $r_{\pm,i}$ and $r'_{\pm,i}$ are mapped to the same marker $x_{\gamma_{\pm,i}^s}$.

\item Let $G(\dot F,\widehat{W};\bs\gamma_+;\bs\gamma_-)$ be the automorphism group of $W_\delta^{l+1,p}(\dot F, \widehat{W};\bs\gamma_+;\bs\gamma_-)$ generated by the puncture reorderings and marker rotations.
\ee
\end{defn}
\cb

\subsection{$\op{(Symp)}$ vs $\op{(Cob)}$}  \label{subsection: symp vs cob}

Consider the moduli space
$$\mathcal{M}=\mathcal{M}_J^{\op{ind}=k}(\dot F,\widehat W; \bs\gamma_+;\bs\gamma_-).$$
There will always be two cases to consider:
\begin{enumerate}
\item[(Symp)] $(\widehat W=\R\times M,J)$ is $\R$-invariant and is viewed as a symplectization.  This is the case when defining the differential $d$ and proving $d^2=0$. We usually take $\mathcal{M}/\R$.
\item[$\op{(Cob)}$] $(\widehat W,J)$ is viewed as a cobordism.  This is the case when defining chain maps.  We allow $(\widehat W,J)$ to be $\R$-invariant (e.g., when $(\widehat W,J)$ is the ``identity cobordism''). In the case of a cobordism there is no quotienting by $\R$, even if there is an $\R$-action.
\end{enumerate}
We will write $(\widehat W,J)\in \op{(Symp)}$ or $(\widehat W,J)\in\op{(Cob)}$ to indicate whether $(\widehat W,J)$ is viewed as a symplectization or a cobordism. There are two reasons for this distinction:
\begin{itemize}
\item We do not quotient by the $\R$-action for $\op{(Cob)}$.
\item Trivial cylinders are unperturbed (i.e., are more or less ignored) when considered in (Symp), but are perturbed (i.e., we construct an obstruction bundle $\E\to \V$ where the trivial cylinder is in $\V$) when considered in $\op{(Cob)}$.
\end{itemize}

For the $\op{(Cob)}$ case, we also require $J$ to be

\begin{defn} [$(L_+,L_-)$-end-generic]
Let $J$ be an almost complex structure on $(\widehat W,\widehat \alpha)$ that is $(L_+,L_-)$-simple for $\widehat \alpha$. We say $J$ is {\em $(L_+,L_-)$-end-generic},
if for all $\bs\gamma_\pm$ with $\mathcal A_{\alpha_\pm}(\bs\gamma_\pm)\leq L_\pm$ and all $k\in \mathbb Z$, there is no $J$-holomorphic curve $[\F, u]\in \mathcal M^{\op{ind}=k}_{J} (\dot F, \widehat W; \bs\gamma_+; \bs\gamma_-)$ that agrees with the trivial half cylinder near some puncture of $\dot F$ in the region $\widehat W_{s\geq T}$  or $ \widehat W_{s\leq -T}$, for some $T>1$.
\end{defn}

\begin{rmk}
We require $J$ to be $(L_+,L_-)$-end-generic, because with it we can add additional marked points to stabilize the domain of a $J$-holomorphic cylinder or plane in a more canonical way.
\end{rmk}

\begin{lemma}\label{end generic}
Given an $(L_+,L_-)$-simple pair $(\widehat \alpha, J)$, we can perturb $J$ to $J'$ such that $(\widehat \alpha, J')$ is still $(L_+,L_-)$-simple and $J'$ is $(L_+,L_-)$-end-generic.
\end{lemma}


\begin{proof}
We restrict to the simpler case where $\bs\gamma_+$ consists of only one Reeb orbit $\gamma_+$ of $\alpha_+$ with $\mathcal A_{\alpha_+}(\gamma_+) \leq L_+$. Technically speaking, this suffices for the current paper; the proof of the general case is quite similar to this special case and is left to the reader.

We perturb $J$ inductively, using a triple which we call the {\em complexity}
$$c(\F,u)=(\mathcal{A}_{\alpha_+}(\gamma_+),E_\alpha(u),-\chi(\dot F)),$$
where we are using the lexicographic ordering, to order the moduli spaces of curves. More precisely, we perturb $J$ so that all the curves of
$$\mathcal M(\gamma_+;\bs\gamma_+):=\mathcal M(\dot F, \widehat W;\gamma_+;\bs\gamma_-)$$
have nontrivial positive ends, assuming all curves in $\mathcal{M}(\gamma_+';\bs\gamma_-')$ with complexity less than that of $\mathcal{M}(\gamma_+;\bs\gamma_-)$ have nontrivial positive ends.

Let us write $\mathcal{M}:= \mathcal{M}(\gamma_+;\bs\gamma_-)$ and $\bdry\mathcal M=\overline{\mathcal{M}}-\mathcal{M}$.

\s\n
{\em Step 1.}
We claim that there exists an open neighborhood $V\subset \overline{\mathcal M}$ of $\partial \mathcal M$ such that no holomorphic curve in $V-\bdry \mathcal M$ has trivial positive end. \coblu During the proof of the lemma we call a holomorphic building (See Section 7 of \cite{BEHWZ} for the definition of holomorphic buildings) whose topmost level is a branched cover with at least one branch point over a trivial cylinder a {\em {\bf b}-building}. \cb

Consider a holomorphic building in $\partial \mathcal M$ of the form $u_0\cup u_1$, where $u_0$ maps to $\widehat W$ and $u_1$ maps to $\mathbb R \times M_+$.  (The other types of holomorphic buildings in $\bdry \mathcal M$ can be treated similarly and are left to the reader.) We choose a small neighborhood $U\subset \overline{\mathcal{M}}$ of $u_0\cup u_1$ such that no holomorphic curve $u \in U-\bdry\mathcal{M}$ has a trivial positive end as follows: If $u_0\cup u_1$ is not a {\bf b}-building, then $u_1$ does not have a trivial positive end by unique continuation (here we are assuming that $u_1$ is not a trivial cylinder) and there is a sufficiently small neighborhood $U$ of $u_0\cup u_1$ with the desired property. Next suppose that $u_0\cup u_1$ is a {\bf b}-building and there exists a sequence $\{u^k\}$ of curves in $\mathcal M$ that limit to $u_0\cup u_1$ and have trivial positive ends.   Then there exist constants $s_0$ and $k_0$ such that:
\begin{itemize}
\item $J|_{[s_0,\infty)\times M_+}= J_+$; and
\item for $k\geq k_0$,  $u^k$ restricted to $ [s_0,\infty)\times M_+$ is a branched cover of $[s_0,\infty)\times\gamma^s_+$.
\end{itemize}
If we replace the portion of $u^k$ in $[s_0,\infty)\times M_+$ by a trivial half-cylinder without branch points, we obtain a holomorphic curve with a trivial positive end in $\mathcal{M}(\gamma_+',\bs\gamma_-')$ of lower complexity, which contradicts the inductive assumption.  Since $\partial \mathcal M$ is compact, we can cover $\bdry\mathcal M$ by finitely many open sets $U$ and the claim follows.

\s\n
{\em Step 2.}
Recall from Section~\ref{subsection: almost complex structures} that, for any $[\mathcal F, u]\in \mathcal M$, on a neighborhood of the puncture $p_+$ of $\dot F$ where $u$ converges to $\gamma_+$, we have $u(s,t)=(s,t,\eta(s,t))$, where $\eta$ satisfies Equation~\eqref{normal equation}.

We denote by ${\mathfrak C}$ the space of functions $f = (\overline f, \underline f):[0, \mathcal A_\alpha (\gamma)] \to \R^{2\overline m}\times \R^{2\underline m}$ that satisfies  $\overline f(\mathcal A_\alpha (\gamma)) = \overline f(0)$ and $\underline f(\mathcal A_\alpha (\gamma)) = -\underline f(0).$
Define a self-adjoint operator, called the {\em normal asymptotic operator},\footnote{This is slightly different from the full asymptotic operator which appears in Section~\ref{subsection: asymptotic operator}.}
$$A:=-j_0\frac{\partial}{\partial t}-S:W^{1,2}(\mathfrak C)\to L^2(\mathfrak C),$$ \cb 
where $S=-j_0X$.
Then $\eta$ satisfies
\begin{equation} \label{normal equation bis}
\frac{\partial \eta}{\partial s}=A \eta.
\end{equation}
Let
$$\dots\leq \theta_{-2}\leq \theta_{-1}<0<\theta_1\leq \theta_2\leq \dots$$
be the eigenvalues of $A$ and
$$\dots,g_{-2},g_{-1},g_1,g_2,\dots$$
be the associated complete set of orthonormal eigenfunctions.
We can write
$$\eta(s,t)=\sum_{i \leq -1}c_i e^{\theta_i s}g_i(t),$$
where the $c_i:= c_i^{\gamma_+}(u)$ are constants.  Observe that $u$ has trivial positive end if and only if all the $c_i$ are zero.

\s\n
{\em Step 3.}
Choose a compact subset  $K$ of $\mathcal M$ such that $K\cup V = \overline{\mathcal M}$. Below we explain how to make a  perturbation of $J$ such that all the curves in $K$ have no trivial positive end and yet preserve this property for curves near $V$ and in lower inductive strata.

\s\n
{\em Case I} ($\gamma_+$ is simple).  Since $K$ is compact, there exists $s_1> 0$ such that:
\be
\item $J|_{[s_1,\infty)\times M_+}= J_+$; and
\item for any $[\mathcal F, u] \in K$ the restriction of $u$ to $(s_1,\infty)\times M_+$ has image in $(s_1,\infty)\times \gamma_+\times D_{\delta'}$ and can be expressed as $(s,t,\eta(s,t))$, where $\eta$ satisfies Equation \eqref{normal equation bis}.
\ee
Here $0<\delta'< \frac{\delta}{2}$, where $\delta$ is the constant that appears in Definition~\ref{Lsimple form}.

Next choose $s_1< R$ and $\varrho_1,\dots,\varrho_k>0$ sufficiently small; we assume that $R$ is much larger than the $R$ used when perturbing the $J$ for the lower inductive strata. Let $\mu: \R\to[0,1]$ be a smooth bump function that is supported on $(0,3)$ and satisfies $\mu([1,2])=1$, and let $\mu_T:\R\to[0,1]$ be its translate $\mu_T(s)=\mu(s-T)$.

We perturb $J$ to $J'=J'_{\varrho_1,\dots,\varrho_k}$ on the region $[R, R+3]\times \gamma_+ \times D_{\delta/2} \subset \widehat W$ by setting
$$J'(\partial_s) = g\mathbf R_\alpha -\sum_{j=1}^kj_0 \varrho_j\mu_{R}(s)g_{-j}(t), \quad J'(\bdry_{x_i})=\bdry_{y_i}$$
on $[R,R+3]\times \gamma_+\times D_{\delta'}$. The perturbation $\eta'$ of $\eta$ satisfies the equation
\begin{equation}
\frac{\bdry\eta'} {\bdry s}= A\eta' +\sum_{j=1}^k \varrho_j\mu_{R}(s)g_{-j}(t).
\end{equation}
All but finitely many Fourier coefficients of $\eta$ and $\eta'$ are the same, i.e., those of $g_{-j}(t)$, $j=1,\dots,k$, which we denote by $f_{-j}(s)g_{-j}(t)$.  The coefficient $f_{-j}(s)$ satisfies
$$\frac{df_{-j}}{ds}=\theta_{-j} f_{-j}(s) + \varrho_j\mu_{R}(s),$$
subject to the condition $f_{-j}(s) =c_{-j}e^{\theta_{-j}s}$ for $s< R$. We compute
$$f_{-j}(s)=e^{\theta_{-j}s}\left( c_{-j}+ \int_0^s e^{-\theta_{-j}\sigma}\varrho_j \mu_{R}(\sigma)d\sigma\right).$$
This shows that the Fourier coefficients of all the curves in $K$ are perturbed in the $g_{-j}(t)$ component by a fixed small constant that depends on $\varrho_j$.

We now assume without loss of generality that $J$ regular.  Since $\gamma_+$ is simple, $K$ is part of a transversely cut out moduli space and the map
$$K\to \R^k, \quad[\F,u]\mapsto (c_{-1}^{\gamma_+}(u),\dots,c_{-k}^{\gamma_+}(u))$$
is smooth.  Hence, by Sard's theorem, for $k>0$ sufficiently large, there exist $\varrho_1,\dots,\varrho_k$ arbitrarily small such that no holomorphic curve of $J'_{\varrho_1,\dots,\varrho_k}$ that is close to a curve in $K$ has trivial positive end.

\s\n
{\em Case II} ($\gamma_+$ is not simple). If $\gamma_+$ is a $k$-fold cover of some simple $\gamma^s_+$, then the $k$-fold multiple cover of the $i$th eigenfunction associated to $\gamma^s_+$ is more or less the $ki$th eigenfunction associated to $\gamma_+$. The only modification we need to make is to  perturb $J$ using the eigenfunctions $g_{-j}(t)$ associated to $\gamma^s_+$.

\s
Finally we discuss the effect of perturbing $J$ on the ``end-genericity" of lower inductive strata and curves close to $V$:
\begin{itemize}
\item The perturbation of $J$ does not affect $J_+$ and the ends of curves in the symplectization $\R\times M_+$.
\item Since $\varrho_1,\dots,\rho_k$ are arbitrarily small, for moduli spaces $\mathcal{M}(\gamma_+'; \bs\gamma'_-)$ of lower complexity, the curves that are not close to {\bf b}-buildings still do not have trivial positive ends.
\item The curves of $\mathcal{M}(\gamma_+'; \bs\gamma'_-)$ that are close to {\bf b}-buildings also do not have trivial positive ends since the considerations of Step 1 reduce the issue to the ``end-genericity" of moduli spaces of even lower complexity.
\item Similarly, the curves that are close to $V$ still do not have trivial positive ends.
\end{itemize}
This completes the proof of Lemma~\ref{end generic}. \cb
\end{proof}

{\em In this paper, we always assume $J$ is $(L_+,L_-)$-end-generic for some appropriate $(L_+,L_-)$ when $(\widehat W, J)\in (\op{Cob})$.}

\subsection{Teichm\"uller space and mapping class group} \label{subsection: Teichmuller space and mapping class group}

\coblu In this subsection let us fix $(F,\mathbf{p}, \mathbf{r})$ and assume that $\dot F=F-{\bf p}$ has negative Euler characteristic. \cb For each $p\in \mathbf{p}$ fix an identification  $N_p-\{p\}\simeq (0,\infty)\times S^1$, where $N_p\subset F$ is a small neighborhood of $p$.

\coblu Let $\op{Diff}(F,{\bf p},{\bf r})$ be the group of diffeomorphisms of $F$ that fix ${\bf p}$ and ${\bf r}$ as ordered sets and let $\op{Diff}'(F,{\bf p},{\bf r})$ be the subgroup generated by the identity component together with Dehn twists $\tau_{p}$ along loops around all the $p\in {\bf p}$. 

We make the following {\em slightly nonstandard definition:}

\begin{defn} \label{defn: Teich}
$\op{Teich}(F,{\bf p},{\bf r})$ is the Teichm\"uller space of equivalence classes $[j]$ of complex
\nom[Teich]{$\op{Teich}(F,{\bf p},{\bf r})$}{Teichm\"uller space of $(F,{\bf p},{\bf r})$; see Definition~\ref{defn: Teich} which is slightly nonstandard}
structures on $F$, where two complex structures $j$ and $j'$ are equivalent if and only if there exists $g\in \op{Diff}'(F,{\bf p},{\bf r})$ such that $g_* j = j'$.
\end{defn}

Observe that $\op{Teich}(F,{\bf p},{\bf r})$ is diffeomorphic the product of the usual Teichm\"uller space $\op{Teich}(F,{\bf p})$ with $\times_{p\in {\bf p}} ((T_pF-\{0\})/\R^+)$ (the set of tuples of markers indexed by ${\bf p}$); this is how we will view $\op{Teich}(F,{\bf p},{\bf r})$ from now on.
\cb

Given an open set $\mathcal{U}$ of $\op{Teich}(F,{\bf p},{\bf r})$, we write $\mathcal{U}=\mathcal{U}_{[j]}$ to indicate that $[j]\in \mathcal{U}$. \cb Also let $\widetilde{\mathcal{U}}$ be a {\em Teichm\"uller slice} over $\mathcal{U}$, i.e., a smooth choice of complex structure $j_x$ on $F$ in each equivalence class $x\in \mathcal{U}$   subject to the condition that: 
\be
\item[($\#$)] for each $p\in \mathbf{p}$, $j_x$ agrees with the standard complex structure $j$ on some $[R(x),\infty)\times S^1$, where $R(x)>0$ depends on $x$.
\ee
\nom[Mod]{$\op{Mod}(F,{\bf p},{\bf r})$}{Mapping class group of $F$ fixing ${\bf p}$ and ${\bf r}$ as ordered sets}

\coblu The mapping class group\footnote{Again, slightly nonstandardly defined.}
$$\op{Mod}(F,{\bf p},{\bf r}):=\op{Diff}(F,{\bf p},{\bf r})/\op{Diff}'(F,{\bf p},{\bf r})$$ 
acts properly discontinuously on $\op{Teich}(F,{\bf p},{\bf r})$.  
\nom[Modj]{$\op{Mod}(F,{\bf p},{\bf r})_{[j]}$}{Stabilizer of $[j]$ in $\op{Mod}(F,{\bf p},{\bf r})$ given by \eqref{eqn: Modj}}
We denote the stabilizer of $[j]$ by
\begin{equation}\label{eqn: Modj}
    \op{Mod}(F,{\bf p},{\bf r})_{[j]}=\{ [g]\in \text{Mod}(F,{\bf p},{\bf r})~|~ [g]([j])=[j]\}.
\end{equation}

\begin{lemma} \label{lemma: stabilizer}
The stabilizer $\op{Mod}(F,{\bf p},{\bf r})_{[j]}$ is trivial.
\end{lemma}

\begin{proof}
$\op{Mod}(F,{\bf p},{\bf r})_{[j]}$ is isomorphic to the automorphism group $\op{Aut}(F,j,{\bf p},{\bf r})$ of \nom[Aut]{$\op{Aut}(F,j, {\bf p},{\bf r})$}{Automorphism group of $(F,j)$ fixing ${\bf p}$ and ${\bf r}$ pointwise}
$(F,j)$ which takes ${\bf p}$ to itself and ${\bf r}$ to itself as ordered sets. The automorphism group of $(F,j)$ taking ${\bf p}$ to itself is finite since $\chi(\dot F)<0$ and there are no nontrivial automorphisms of finite order that preserve ${\bf r}$.
\end{proof} 
\cb

\section{Fredholm theory} \label{section: Fredholm theory}

In this section we fix $(F,{\bf p})$ and hence $\dot F$.

\subsection{Canonical cylindrical coordinates} \label{subsection: canonical cylindrical coordinates}

Let $(\F,u)$ be a pair where $u:\dot F\to  \widehat{W}$ is a sufficiently differentiable map which is asymptotic to $\gamma_{+,i}$ at the positive end near $p_{+,i}$ and to $\gamma_{-,i}$ at the negative end near $p_{-,i}$.    Since $J$ is $L$-simple, there exist simple coordinates $(t,x,y)$ on $\R/\mathcal{A}_\alpha(\gamma)\Z\times D$ about every orbit $\gamma\in \mathcal{P}^L_\alpha$. \cb If $u$ can be written as a graph $(s,t)\mapsto (s,t,\eta_{\pm,i}(s,t))$ on $\R\times (\R/\mathcal{A}_\alpha(\gamma_{\pm,i}))\times D$ with respect to simple coordinates near the end corresponding to the puncture $p_{\pm,i}\in \mathbf{p}$, then $(s,t)$ is said to be the {\em canonical cylindrical coordinates on $\dot F$ near $p_{\pm,i}$ with respect to $u$}.

By the simplicity of $J$, if $[\F,u]\in \mathcal{M}_J^{\op{ind}=k}(\dot F,\widehat{W};\bs\gamma_+;\bs\gamma_-)$, then the canonical cylindrical coordinates for a representative $(\F,u)$ near the punctures are holomorphic coordinates on $\dot F$.  On the other hand, $(s,t)$ is not necessarily holomorphic if $u$ is not $J$-holomorphic.

If $\llbracket\F,u\rrbracket\in \mathcal{M}_J^{\op{ind}=k}(\dot F,\R\times M;\bs\gamma_+;\bs\gamma_-)/\R$, the coordinates $(s,t)$ for $(\F,u)$ are canonical up to translations in the $s$-direction.

\subsection{Fredholm setup} \label{subsection: Fredholm setup}

Fix an $L$-simple $J$ on the symplectization $\widehat{W}=\R\times M$ or an $(L_+,L_-)$-simple $J$ on the cobordism $\widehat{W}$ in the class $C^\infty$.

Recall the coordinate $s$ at the ends of $\widehat{W}$. Fix a Riemannian metric $g_0$ on $\widehat{W}$ which is $s$-invariant at the ends. For simplicity we assume that $g_0$ restricts to the standard flat metric
$$g_0=ds^2+dt^2+\sum_idx_i^2+\sum_i dy_i^2$$
on the $L$-simple coordinate charts (and hence $g_0=g_0^L$, i.e., depends on $L$). Let $\exp: T\widehat W \to \widehat{W}$ be the exponential map with respect to $g_0$, $\epsilon>0$ be a constant smaller than the injectivity radius of $g_0$, and
$$D_\epsilon\subset T\widehat W= \{(w,\xi)~|~w\in W, |\xi|_{g_0}< \epsilon\}$$
be the $\epsilon$-disk bundle of $T\widehat{W}$.

{\em For the rest of the paper we assume that $\bs\gamma_+$ satisfies $\mathcal{A}_\alpha(\bs\gamma_+)<L$, unless indicated otherwise.}

\subsubsection{Weighted Sobolev spaces} \label{subsubsection: weighted sobolev spaces}

  As in Section~\ref{subsection: Teichmuller space and mapping class group}, we fix an identification 
$$N_p-\{p\} \simeq (0,\infty)\times S^1$$ 
for each $p\in \mathbf{p}$, where $N_p\subset F$ is a small neighborhood of $p$, and assume that all $j$ on $F$ or $\dot F$ satisfy $(\#)$.\footnote{This identification is different from the canonical cylindrical coordinates from Section~\ref{subsection: canonical cylindrical coordinates}.}  We also choose a Riemannian metric on $\dot F$ that restricts to a product metric on all the cylindrical ends $(0,\infty)\times S^1$. \cb

Let $\delta>0$ be a small positive number.  We define the ``weighted Sobolev space'' $W_\delta ^{k+1,p}( \dot F, \widehat{W};\bs\gamma_+;\bs\gamma_-)$,   which is similar to but slighlty different from Dragnev's~\cite[Definition 2.6]{Dr}:

\begin{defn}[Weighted Sobolev space of maps $\dot F\to\widehat{W}$] \label{def: weighted Sobolev space}
  Choose an auxiliary complex structure $j$ on $F$. Let $\mathcal{C}$ be the space of smooth maps $u:\dot F\to \widehat{W}$ that agree with $(j,J)$-holomorphic maps parametrizing trivial holomorphic half-cylinders $[R,+\infty)\times \gamma_{+,i}$ near $p_{+,i}$ and $(-\infty,-R]\times\gamma_{-,i}$ near $p_{-,i}$ for some $R$ that depends on $u$. \cb
Then the {\em weighted Sobolev space of maps $\dot F\to\widehat{W}$} is
$$W_\delta ^{k+1,p}( \dot F, \widehat{W};\bs\gamma_+;\bs\gamma_-):=\{\exp(u,\xi)~|~u\in \mathcal{C},\xi\in W_\delta^{k+1,p}(\dot F, u^*D_\epsilon)\},$$
where $W_\delta^{k+1,p}(\dot F, u^*D_\epsilon)$ is the usual weighted Sobolev space of sections of $u^*D_\epsilon$, defined using $g_0$, and we are using a smooth weight function $f_\delta:\widehat{W}\to \R^+$ which agrees with $e^{\delta |s|}$ at the ends of $\widehat{W}$.
\end{defn}


\begin{rmk} \label{rmk: indep of choice of j}
The definition of $W_\delta ^{k+1,p}( \dot F, \widehat{W};\bs\gamma_+;\bs\gamma_-)$ is independent of the choice of $j$: Let $\phi: D^2 \to D^2$, where $D^2\subset \C$, be a holomorphic map such that $\phi(0)=0$ and $\phi'(0)\not=0$. Identifying $D^2-\{0\}$ with the holomorphic half-cylinder $[0,\infty)\times S^1$, the independence is a consequence of the fact that $\phi:[0,\infty)\times S^1\to [0,\infty)\times S^1$ exponentially approaches $(s,t)\mapsto (s+a,t+b)$ for some $a,b$ as $s\to \infty$. 
\end{rmk}

Given $u\in W_\delta ^{k+1,p}( \dot F, \widehat{W};\bs\gamma_+;\bs\gamma_-)$, let $\R^{2(l_++l_-)}$ be the vector space spanned by sections $u^*\tilde \bdry_{s,\pm,i}, u^*\tilde \bdry_{t,\pm,i}$ of $C^\infty(\dot F, u^*D_\epsilon)$, where $i=1,\dots,l_\pm$; $R\gg 0$; $f:[R,\infty)\to [0,1]$ is a smooth map satisfying 
$$f(s)=\left\{  
\begin{array}{cl} 1 & \mbox{for $s\geq 2R$}, \\ 
0 & \mbox{for $s\leq 2R-1$};
\end{array}  
\right.$$
$\tilde \bdry_{s,+,i}=f(s) \bdry_s$ and $\tilde \bdry_{t,+,i}=f(s)\bdry_t$ on $[R,\infty)$ times an $L$-simple coordinate chart for $\gamma_{+,i}$; and $\tilde \bdry_{s,-,i}, \tilde \bdry_{t,-,i}$ are defined analogously using $f(-s)$.  The topology on $W_\delta ^{k+1,p}( \dot F, \widehat{W};\bs\gamma_+;\bs\gamma_-)$ comes from the topology on $\R^{2(l_++l_-)}\times W_\delta^{k+1,p}(\dot F, u^*D_\epsilon)$ via the exponential map $\exp(u,\mathfrak{a}+\xi)$, where $(\mathfrak{a},\xi)\in \R^{2(l_++l_-)}\times W_\delta^{k+1,p}(\dot F, u^*D_\epsilon)$.
\cb

Also  note that the Sobolev norm depends on $g_0$ and $f_\delta$, but the topology on this Sobolev space does not depend on the particular choices of $g_0$ and $f_\delta$.

\begin{rmk}[Remark on values of $k$ and $p$] \label{rmk: values of k and p}
By the Sobolev embedding theorem, if $k>r+\frac{m}{p}$, then $W^{k,p}(\Omega)\subset C^r(\Omega)$ for any compact domain $\Omega$ of $\R^m$ with smooth boundary; see Remark~\ref{rmk:sobolev embedding} for obtaining the Sobolev embedding theorem for $\dot F$.  In our case $m=\dim \dot F= 2$ and we take $r\geq 2$, since we want at least two continuous derivatives (we need one continuous derivative for defining $s_{\pm,i}$ in Step 1B in the proof of Theorem~\ref{thm: construction of L-transverse subbundle} and two continuous derivatives to ensure that Lemma~\ref{lemma: I smoothness} holds). Hence from now on we assume that  $k> r+1$, $p= 2$, and $r\geq 2$.
\end{rmk}

The following is well-known:

\begin{rmk} \label{rmk:sobolev embedding}
We can use the Sobolev embedding theorem for compact domains to show that there exists a global constant $C>0$ such that for all
$$\xi\in W_\delta^{k+1,p}(\dot F, u^*D_\epsilon)$$
we have
\begin{equation}
|f_\delta\xi|_{C^r}\leq C\|f_\delta\xi\|_{W^{k+1,p}}=C \| \xi\|_{W^{k+1,p}_\delta}.
\end{equation}
Suppose for simplicity that $\dot F$ has only one end and $\dot F$ is written as the union $\Omega_0\cup\Omega_1\cup\Omega_2\dots$ of compact domains, where each $\Omega_i\cap \Omega_{i+1}$ is a circle, $\Omega_0$ is the ``thick part'', and $\Omega_i$, $i>0$, is biholomorphic to the standard $S^1\times[0,1]$.  Then we apply the Sobolev embedding theorem for each $\Omega_i$, noting that
$$\| \xi\|_{W^{k+1,p}_\delta(\Omega_i)}\leq \| \xi\|_{W^{k+1,p}_\delta(\dot F)}$$
and that the Sobolev spaces $W^{k+1,p}_\delta(\Omega_i)$ and $W^{k+1,p}_\delta(\Omega_{i+1})$ can be identified for $i\gg 0$ because $u^*D_\epsilon$ is asymptotically cylindrical at the ends of $\dot F$.  Hence the same constant $C>0$ can be used when we apply the Sobolev embedding theorem to all $\Omega_i$.
\end{rmk}

Let $W_\delta^{k,p}(\dot F,\wedge^{0,1}u^{*}T\widehat{W})$ be the usual weighted Sobolev space of sections of
$$\wedge^{0,1}u^*T\widehat{W}=\wedge^{0,1}\dot F \otimes_J u^{*}T\widehat W\to\dot F$$
in the class $(k,p)$ with weight $f_\delta$.

\begin{rmk}
In the functional analysis setup we use $W^{k+1,p}_\delta$ and $W^{k,p}_\delta$ spaces for large $k$ to extract a collection of finite-dimensional bundles $\E\to \V$. Once this is done we throw away the Sobolev space setup and work in the category of finite-dimensional manifolds or orbifolds.  (This is analogous to the \cite{FO} and \cite{FO3} setup and differs from that of \cite{HWZ3}.)
\end{rmk}

\subsubsection{$\chi(\dot F)<0$} \label{subsubsection: negative Euler char}

We first treat the case where $\dot F$ has negative Euler characteristic.
Given a Teichm\"uller slice $\widetilde{\mathcal{U}}$ of an open subset \coblu $\mathcal{U}\subset \op{Teich}(F,{\bf p},{\bf r})$, \cb define
\begin{align*}
\mathcal{B}_{\widetilde{\mathcal{U}}}= &\mathcal{B}_{\widetilde{\mathcal{U}}}(\dot F, \widehat{W};\bs\gamma_+;\bs\gamma_-)\\
=&\left\{ (\F,u)=((F,j,\mathbf{p},\mathbf{r}),u) \left|
\begin{array}{cc}
j\in \widetilde{\mathcal{U}},u\in W_\delta ^{k+1,p}( \dot F, \widehat{W};\bs\gamma_+;\bs\gamma_-), \\  u(r_{\pm, i})=x_{\gamma_{\pm,i}^s}
\end{array}
\right.\right\},
\end{align*}
where the pair $(F,\mathbf{p})$ is fixed (as stipulated at the beginning of this section).
We define the bundle
\begin{equation}
\pi_{\widetilde{\mathcal{U}}}:\mathcal{E}_{\widetilde{\mathcal{U}}}=\mathcal{E}_{\widetilde{\mathcal{U}}}(\dot F, \widehat{W};\bs\gamma_+;\bs\gamma_-)\to \mathcal{B}_{\widetilde{\mathcal{U}}},
\end{equation}
whose fiber over $(\F,u)$ is
$$(\mathcal{E}_{\widetilde{\mathcal{U}}})_{(\F,u)} = W_\delta^{k,p}(\dot F,\wedge^{0,1}u^{*}T\widehat{W}).$$
Then $\overline{\partial}_J^{\widetilde{\mathcal{U}}}$ is a section of $\mathcal{E}_{\widetilde{\mathcal{U}}}$ defined by
$$\overline{\partial}_J^{\widetilde{\mathcal{U}}} (\F, u)=(\F,u,\tfrac{1}{2}(du+J(u)du\circ j)).$$
Also let
\begin{equation} \label{eqn: proj to U}
\Pi_{\widetilde{\mathcal{U}}}: \mathcal{B}_{\widetilde{\mathcal{U}}} \to \widetilde{\mathcal{U}}, \quad (\F,u)=((F,j,{\bf p},{\bf r}),u)\mapsto j
\end{equation}
be the projection to $\widetilde{\mathcal{U}}$.
{\em When the Teichm\"uller slice $\widetilde{\mathcal{U}}$ is understood, it will often be omitted from the notation, e.g., $\overline{\partial}_J$, $\mathcal{B}$, $\mathcal{E}$.}

Let $\nabla$ be the Levi-Civita connection on $T\widehat W$ with respect to the metric $g_0$.  Let $L_{(\F,u)}$ be the differential
$$(\overline{\partial}_J)_*: T_{(\F,u)} \mathcal{B}\to T_{(\F,u,\overline\bdry_J u)}\mathcal{E}$$
postcomposed with the projection to $W_\delta^{k,p}(\dot F,\wedge^{0,1}u^{*}T\widehat{W}).$  Then $L_{(\F,u)}$ can be written as:
\begin{equation} \label{eqn: linearization}
L_{(\F,u)}: T\widetilde{\mathcal{U}}(j)\oplus\R^{2(l_++l_-)}\oplus W^{k+1,p}_\delta (\dot F,u^{*}T\widehat{W}) \to W^{k,p}_\delta (\dot F,\wedge ^{0,1} u^{*}T\widehat{W}),
\end{equation}
$$(\mathfrak{j},\mathfrak{a},\xi)\mapsto Y(\mathfrak{j})+D_u(\mathfrak{a}+\xi),$$
where $Y(\mathfrak{j})$ is induced from the variation of $j$ in the direction on $\mathfrak{j}$, $D_u$ is the usual linearized $\overline\bdry_J$-operator   with fixed $j$ \cb (which depends on the choice of $\nabla$ if $u$ is not $J$-holomorphic), and   $\R^{2(l_++l_-)}$ is as given in Section~\ref{subsubsection: weighted sobolev spaces}. \cb
Here $l_++l_-$ is the total number of punctures of $\dot F$.

\subsubsection{Topology for $\mathcal{M}_J(\dot F,\widehat{W};\bs\gamma_+;\bs\gamma_-)$} \label{subsubsection: topology}

The topology for $\mathcal{M}_J(\dot F,\widehat{W};\bs\gamma_+;\bs\gamma_-)$ is generated by the topologies $\mathcal{T}_{\widetilde{\mathcal{U}}}$ from the right-hand side of \coblu
\begin{equation} \label{eqn: topology of moduli space}
    \mathcal{M}_J(\dot F,\widehat{W};\bs\gamma_+;\bs\gamma_-)\cap (\mathcal{B}_{\widetilde{\mathcal{U}}}/\sim ) =(\overline{\partial}_J^{\widetilde{\mathcal{U}}})^{-1}(0),
\end{equation}
where we are ranging over all open sets $\mathcal{U}=\mathcal{U}_{[j]}$; the topology on $(\overline{\partial}_J^{\widetilde{\mathcal{U}}})^{-1}(0)$ is induced from $\mathcal{B}_{\widetilde{\mathcal{U}}}$; and we note that $\op{Mod}(F,{\bf p},{\bf r})_{[j]}=1$ by Lemma~\ref{lemma: stabilizer}. \cb

We claim that $\mathcal{T}_{\widetilde{\mathcal{U}}}$ does not depend on the choice of Teichm\"uller slice: Given Teichm\"uller slices $\widetilde{\mathcal{U}}$ and $\widetilde{\mathcal{U}}'$ over $\mathcal{U}$, there is a smooth family of diffeomorphisms $\phi_x: F\xrightarrow{\sim} F$ isotopic to the identity parametrized by $x\in \mathcal{U}$ such that $(\phi_x)_* j_x= j'_x$, where $j_x\in \widetilde{\mathcal{U}}$ and $j'_x\in \widetilde{\mathcal{U}}'$ are lifts of $x\in \mathcal{U}$. Let $x_0,x_1\in \mathcal{U}$ be close with respect to some metric (e.g., the Teichm\"uller metric) $d$ on $\mathcal{U}$ and $u_0,u_1\in W_\delta ^{k+1,p}( \dot F, \widehat{W};\bs\gamma_+;\bs\gamma_-)$ be $(j'_{x_i}, J)$-holomorphic maps that are close.  By Section~\ref{subsubsection: weighted sobolev spaces}, $u_1=\exp (u_0, \mathfrak{a}+ \xi)$ for $(\mathfrak{a},\xi)$ with small norm $\|\cdot \|$ in $\R^{2(l_++l_-)}\times W_\delta^{k+1,p}(\dot F, u_0^*D_\epsilon)$. 

\begin{lemma} \label{lemma: estimate}
Writing $u_1\circ \phi_{x_1} = \exp (u_0\circ \phi_{x_0}, \mathfrak{a}' + \xi')$, where 
$$(\mathfrak{a}',\xi')\in \R^{2(l_++l_-)}\times W_\delta^{k+1,p}(\dot F, (u_0\circ \phi_{x_0})^*D_\epsilon),$$ 
and $\phi_{x_1}=\exp(\phi_{x_0},\eta)$, we have:
\begin{equation} \label{eqn: estimate for moduli space}
   \|\mathfrak{a}' + \xi'\|< C(d(x_1,x_0) +\|(\mathfrak{a}, \xi)\|)
\end{equation}
for some constant $C>0$ which depends on $C^m$-norms of derivatives of $\phi_{x_0}$, $\phi_{x_0}^{-1}$ and $C^m$-norms of $\mathfrak{a}+\xi$ and $\eta$. 
\end{lemma}

Lemma~\ref{lemma: estimate} then implies the claim.

\begin{proof}[Proof of Lemma~\ref{lemma: estimate} ]
The proof uses Remark~\ref{rmk: indep of choice of j}.  We first do some model calculations. 

\s\n
{\bf Model calculations.}

\s\n 
(1) Given $f\in W^{1,3}(\R)$, we want to compare $\|f\|_{W^{1,3}}=\|f\| + \|f'\|+\|f''\|$, where $\|f\|$ is the $L^2$-norm, with $\|f\circ \phi\|_{W^{1,3}}$, where 
\begin{enumerate}
    \item[(*)] $\phi:\R\to \R$ is a diffeomorphism such that $\phi$ limits to $x\mapsto x+c_\pm$ and $\phi'$ to $1$ at the positive and negative ends.
\end{enumerate} 
(Note that the maps $\phi_x$ have this form.) By the change of variables formula and our requirement on $\phi'$,
\begin{align*}
    \|f\circ \phi\|^2 &= \int |f\circ \phi(x)|^2 dx \leq |\tfrac{1}{\phi'}|_{C^0} \int |f\circ \phi(x)|^2|\phi'(x)| dx\\
    & =|\tfrac{1}{\phi'}|_{C^0} \int |f(y)|^2 dy  \leq C^2 \|f\|^2
\end{align*}
for some constant $C>0$ which depends on $|(\phi^{-1})'|_{C^0}$. Similarly, 
\begin{align*}
    \|(f\circ\phi)'\|^2 &= \int |(f\circ \phi)'(x)|^2 dx = \int |f'\circ \phi(x)|^2 |\phi'(x)| ^2 dx\\
    &\leq |\phi'|_{C^0} \int |f'\circ \phi(x)|^2 |\phi'(x)| dx  \\
    &= |\phi'|_{C^0}\int |f'(y)|^2dy \leq C^2 \|f'\|^2,
\end{align*}
for some constant $C>0$ which depends on $|\phi'|_{C^0}$ and
\begin{align*}
   & \|(f\circ\phi)''\|^2 = \int |(f\circ \phi)''(x)|^2 dx\\
    & = \int (|f''( \phi(x))|^2 |\phi'(x)| ^4 + 2 | f''(\phi(x)) f'(\phi(x))| |\phi'(x)|^2 |\phi''(x)| + |f'(\phi(x))|^2 |\phi''(x)|^2 )dx\\
    &\leq  |\phi'|_{C^0}^3 \int |f''(y)|^2 dy + 2 |\phi''|_{C^0} | \phi'|_{C^0} \int |f''(y) f'(y)| dy + |\phi''|_{C^0} |\tfrac{1}{\phi'}|_{C^0} \int|f'(y)|^2 dy\\
&\leq C(\|f''\|^2 + \|f'\|^2),
\end{align*}
where $C>0$ depends on $C^m$-norms of derivatives of $\phi$ and $\phi^{-1}$.  Hence
$$\|f\circ \phi\|_{W^{1,3}} \leq C \|f\|_{W^{1,3}}.$$

\s\n
(2) We also estimate $\|(f^2\circ \phi)'\|^2$:
\begin{align*}
    (f^2\circ \phi)'&= (2f f')(\phi(x)) \phi'(x),\\
    (f^2\circ \phi)''&= (2ff')(\phi(x)) \phi''(x) + (2ff''+ 2(f')^2)(\phi(x))  \phi'(x),\\
    \|(f^2\circ \phi)'\|^2 & \leq C |f|_{C^1} ^2(\|f''\|^2 + \|f'\|^2),
\end{align*}
where $C>0$ depends on $C^m$-norms of derivatives of $\phi$ and $\phi^{-1}$. 

\s\n
(3) Next consider $W^{1,3}_{g_\delta}(\R)$, where $g_\delta:\R\to \R^+$ is a smooth weight function that has the form $e^{\delta|x|}$ outside of $[-1,1]$ with $\delta>0$. We treat the first derivative:
\begin{align*}
    \|f\circ \phi\|^2_{g_\delta} &= \int |f\circ \phi(x)|^2 |g_\delta(x)|^2 dx \\
    & \leq |\tfrac{1}{\phi'}|_{C^0} |\tfrac{g_\delta}{g_\delta\circ \phi}|^2_{C^0} \int |f\circ \phi(x)|^2 |g_\delta\circ \phi(x)|^2 |\phi'(x)| dx\leq C \|f\|^2_{g_\delta},
\end{align*}
in view of the fact that $\tfrac{g_\delta}{g_\delta\circ \phi}$ is bounded (at the positive end limits to $\tfrac{e^{\delta x}}{e^{\delta (x+c_+)}}=e^{-\delta c_+}$).  The estimates easily generalize to $W^{k,p}_{g_\delta}(\R)$.

\s\n
{\bf Verification of Estimate~\eqref{eqn: estimate for moduli space}.} Since $u_1=\exp (u_0, \mathfrak{a}+ \xi)$, we have
\begin{align*}
    u_1\circ \phi_{x_1} &= \exp (u_0\circ \phi_{x_1}, \mathfrak{a}\circ \phi_{x_1} + \xi\circ \phi_{x_1})\\
    &= \exp (u_0\circ \phi_{x_0}, \mathfrak{a}\circ \phi_{x_0} + \xi\circ \phi_{x_0} + du_0(\phi_{x_0})\eta +Q_{u_0}),
\end{align*}
where $Q_{u_0}$ is quadratic in $\mathfrak{a}, \xi, \eta$ and depends on $u_0$. 
The same argument as in (1) and (3) of the model calculation  --- the key point is that, by Remark~\ref{rmk: indep of choice of j}, at the ends of $\dot F$ with coordinates $(s,t)$, $\phi_{x_0}$ and $\phi_{x_1}$ asymptotically have the form $(s,t)\mapsto (s+c,t+d)$, where $c,d$ are constant --- implies that:
\begin{align*}
    \|\xi\circ \phi_{x_0}\| &\leq C \|\xi\|,
\end{align*}
where $C$ depends on the $C^m$-norms of derivatives of the derivatives of $\phi_{x_0}$ and $\phi^{-1}_{x_0}$. We also have:
\begin{gather*}
    \|\mathfrak{a}\circ \phi_{x_0}\| \leq C \|\mathfrak{a}\|,\\
    \|du_0(\phi_{x_0})\eta\|  \leq C d(x_1,x_0),\\
    \|Q_{u_0}\|  \leq C (d(x_1,x_0) + \|(\mathfrak{a}, \xi)\|),
\end{gather*}
where $C$ in the last estimate also depends on the $C^m$-norms of $\mathfrak{a}+\xi$ and $\eta$. Note that the $C^m$-norm of $\mathfrak{a}+\xi$ is controlled by the $C^0$-norm by elliptic regularity and the exponential decay of the ends.
 Putting the inequalities together yields Estimate~\eqref{eqn: estimate for moduli space}.
\cb
\end{proof}

\subsubsection{The map $I_\gamma$} \label{subsubsection: map I gamma}

Let $(\F_0,u_0)$ be a representative of
$$[\F_0,u_0]\in \mathcal{M}_J(\dot F,\widehat{W};\bs\gamma_+;\bs\gamma_-)$$
and let $\mathcal{B}'_{(\F_0,u_0)}$ be a small neighborhood of $(\F_0,u_0)$ inside
\begin{align*}
\mathcal{B}'(\dot F, \widehat{W};\bs\gamma_+;\bs\gamma_-)=&\left\{ (\F,u)=((F,\mathbf{p},\mathbf{r}),u) \left|
\begin{array}{cc}
u\in W_\delta ^{k+1,p}( \dot F, \widehat{W};\bs\gamma_+;\bs\gamma_-), \\  u(r_{\pm, i})=x_{\gamma_{\pm,i}^s}
\end{array}
\right.\right\}
\end{align*}

Consider the positive end of $(\F,u)\in \mathcal{B}'_{(\F_0,u_0)}$ asymptotic to $\gamma\in \bs\gamma_+$.  Then on this end $u$ is graphical, i.e., can be written in canonical cylindrical coordinates on $\dot F$ and simple coordinates on the target as
$$u(s,t)=(s,t,\eta(s,t)),$$
after lifting to the $m(\gamma)$-fold cover $\R\times (\R/\mathcal{A}_\alpha(\gamma)\Z)\times D$. We assume that the ray $\{s\in [R,\infty), t=0\}\subset \dot F$ corresponds to the asymptotic marker $r_{+}$ which ``maps'' to $\gamma$.

We then define a map
\begin{gather*}
I=I_\gamma: \mathcal B'_{(\mathcal F_0, u_0)} \times [T',+\infty) \to \mathbb R,\\
 ((\mathcal F, u),s')\mapsto \mathcal{A}_\alpha((t,\eta(s,t))|_{s=s'}),
\end{gather*}
where $T'$ is a constant such that we can write $u(s,t)=(s,t,\eta(s,t))$ for any $(\mathcal F, u)\in \mathcal B'_{(\mathcal F_0, u_0)}$ and $s\geq T'$.
Clearly $I((\mathcal F, u),s')$ is independent of the choice of sufficiently close $(\mathcal F_0, u_0)$.

\begin{lemma}\label{shenyang}
$I$ is a $C^1$-map.
\end{lemma}

Lemma~\ref{shenyang} will be proved in the Appendix.

\subsubsection{$\chi(\dot F)\geq 0$} \label{subsubsection: nonnegative Euler char}

We now consider the case where $\chi(\dot F)\geq 0$, namely when $\chi(\dot F)= 0$ or $1$. 

\nom[E1]{$\varepsilon'>0$}{Small positive constant used to define the punctures ${\bf q}$ when $\chi(\dot F)=0$ or $1$}
\nom[s2]{$s'_{\gamma,+}(\mathcal{F},u)$}{The $s$-coordinate defined in Lemma~\ref{tenren} and used in the definition of the punctures ${\bf q}$ when $\chi(\dot F)=0$ or $1$; here $\gamma$ is a limiting orbit of $u$ at the positive end}

\begin{lemma}\label{tenren}
In the $\op{(Symp)}$ case, suppose $\bs\gamma_+\not=\bs\gamma_-$.  Given a compact subset $K\subset  \mathcal{M}_J(\dot F,\widehat{W};\bs\gamma_+;\bs\gamma_-)/\R$, there exists $\varepsilon'>0$ small such that for each representative of $(\F,u)$ of each $[\F,u]\in K$ and  $\gamma\subset \bs\gamma_+$ there exists a value
\begin{gather*}
s'_{\gamma,+}(\F,u) :=  \max \{ s'\in \R~|~ I_\gamma((\F,u),s')=\mathcal{A}_\alpha(\gamma)-\varepsilon'\}.
\end{gather*}
Moreover, if $u_T$ is $u$ translated up by $s=T$ units, then
$$s'_{\gamma,+}(\F,u_T)= s'_{\gamma,+}(\F,u)+T.$$
The same holds for $(\F',u')$ which is $C^1$-close to some representative $(\F,u)$ of $K$.
\end{lemma}
\nom[u_T]{$u_T$}{$u$ translated up by $s=T$ units}

\begin{proof}
The first assertion follows from observing that $(\F,u)$ cannot be a branched cover of a trivial cylinder by Euler characteristic reasons (note that $u$ also cannot be a trivial cylinder since $\bs\gamma_+\not=\bs\gamma_-$) and that $I_\gamma((\mathcal F, u),s)$ is a strictly increasing function of $s$ if $u$ is not a branched cover of a trivial cylinder. The second and third assertions are immediate.
\end{proof}

\begin{lemma} \label{tenren2}
In the $\op{(Cob)}$ case, there exists a small compactly supported perturbation $J'$ of $J$ such that, given a compact subset $K\subset  \mathcal{M}_J(\dot F,\widehat{W};\bs\gamma_+;\bs\gamma_-)$, there exists $\varepsilon'>0$ small such that for each representative of $(\F,u)$ of each $[\F,u]\in \mathcal{M}_{J'}(\dot F,\widehat{W};\bs\gamma_+;\bs\gamma_-)$ and $\gamma\in\bs\gamma_+$ there exists a value
$$s'_{\gamma,+}(\F,u):= \max\{ {s'\in [1,+\infty)}~|~ I_\gamma((\F,u),s')=\mathcal{A}_\alpha(\gamma)-\varepsilon'\}.$$
The same holds for $(\F',u')$ which is $C^1$-close to some representative $(\F,u)$ of $K$.
\end{lemma}

\begin{proof}
We only need to ensure that there are no holomorphic cylinders of $(\widehat{W},J')$ that agree with trivial half-cylinders at the positive end of $\widehat{W}$. This is done by taking an end-generic perturbation $J'$ of $J$ given in Lemma \ref{end generic}.
\end{proof}

\coblu We analogously define $s'_{\gamma_,-}(\F,u)$ for $\gamma\in \bs\gamma_-$ by replacing $\max$ by $\min$ and $\mathcal{A}_\alpha(\gamma)-\varepsilon'$ by $\mathcal{A}_\alpha(\gamma)+\varepsilon'$. \cb 

From now on we assume that in the $\op{(Cob)}$ case $J$ is sufficiently generic so that the conclusion of Lemma~\ref{tenren2} holds.   {\coblu We write $s'_+$ and $s'_-$ instead of $s'_{\gamma,+}$ and $s'_{\gamma,-}$} if $\gamma$ is understood.

\nom[q]{${\bf q}$}{Auxiliary punctures when $\chi(\dot F)=0$ or $1$, chosen systematically so that $\chi(\ddot F)<0$ where $\ddot F= \dot F-{\bf q}$}

We now define the unordered set ${\bf q}:=\cup_{\gamma\in \bs\gamma_+\cup\bs\gamma_-}{\bf q}_\gamma$ of punctures for $(\F,u)$ as follows, {\coblu provided $s'_{\gamma,\pm}(\F,u)$ exist:}
\begin{enumerate} [label = (Q)]
\item \label{Q} Identify $u^{-1}(\{s={\coblu s'_\pm(\F,u)}\})\simeq \R/\mathcal{A}_\alpha(\gamma)\Z$ by projecting to the $t$-coordinate, where we are still using simple coordinates for $\gamma$. Pick the point $t=0\in \R/\mathcal{A}_\alpha(\gamma)\Z$ corresponding to the asymptotic marker and then define
    $${\bf q}_\gamma=\{(s,t)=({\coblu s'_\pm(\F,u)},\tfrac{k}{2m(\gamma)}\mathcal A_\alpha(\gamma))~|~k=1,\dots,2m(\gamma)\}.$$
\end{enumerate}
Let us write $\ddot F=\dot F-{\bf q}$; since the cardinality of $\bf q_\gamma$ is $2m(\gamma)$, we have $\chi(\ddot F)<0$.  (Later on, when we specifically treat contact homology, there is only one $\gamma$.)

\nom[Fdd]{$\ddot F$}{The punctured Riemann surface $F-{\bf p}-{\bf q}$}

\nom[Fu5]{$(\mathcal F,u)$}{Holomorphic map $u$ with domain $\dot F=F-{\bf p}$, where $\mathcal{F}=(F,j,{\bf p},{\bf q},{\bf r})$, $(F,j)$ is a closed Riemann surface, ${\bf p}$ is the set of punctures, ${\bf q}$ are the auxiliary punctures, and ${\bf r}$ are the asymptotic markers}

\nom[ModFpq]{$\op{Mod}(F,{\bf p},{\bf q},{\bf r})$}{Mapping class group of $F$ fixing ${\bf p}$ and ${\bf r}$ pointwise and ${\bf q}$ setwise}

\begin{changeofnotation} \label{changeofnotation}
From now on, we write
$$(\F,u)=((F,j,{\bf p},{\bf q},{\bf r}),u),$$
with the understanding that ${\bf q}=\varnothing$ when $\chi(\dot F)< 0$, ${\bf q}$ is an unordered set, ${\bf p}$ and ${\bf q}$ are disjoint, and ${\bf q}$ are removable punctures.  We modify Definitions~\ref{defn: equiv relation sim zero} and \ref{defn: equiv relation sim} so that the equivalence relations
$$\F=(F,j,{\bf p},{\bf q},{\bf r})\sim_0 \F'=(F,j',{\bf p},{\bf q}',{\bf r}')$$
and $(\F,u)\sim (\F',u')$ require $\phi({\bf q})={\bf q}'$ as   unordered \cb sets in addition. \coblu Also we modify the definitions of $\op{Teich}(F,{\bf p},{\bf r})$, $\op{Mod}(F,{\bf p},{\bf r})$, and $\op{Aut}(F,j,{\bf p},{\bf r})$ to $\op{Teich}(F,{\bf p},{\bf q},{\bf r})$, $\op{Mod}(F,{\bf p},{\bf q},{\bf r})$, and $\op{Aut}(F,j,{\bf p},{\bf q},{\bf r})$ so that diffeomorphisms of $F$ take ${\bf q}$ to ${\bf q}$ setwise in addition.  

Puncture reorderings and marker rotations are assumed to take ${\bf q}$ to ${\bf q}$ setwise and the automorphism group of $W_\delta^{l+1,p}(\dot F, \widehat W; \bs\gamma_+;\bs\gamma_-)$ generated by puncture reorderings and marker rotations is still denoted $G(\dot F,\widehat W; \bs\gamma_+;\bs\gamma_-)$. \cb
\end{changeofnotation}

The above uniform choice of ${\bf q}$ helps stabilize the automorphism group of $(\F,u)$ {\coblu and the following analog of Lemma~\ref{lemma: stabilizer} holds:}

\begin{lemma} \label{lemma: stabilizer2}
\coblu The stabilizer $\op{Mod}(F,{\bf p}, {\bf q},{\bf r})_{[j]}$ is trivial. \cb
\end{lemma}

Fix $(F,j,{\bf p})$.  (This is equivalent to choosing a Teichm\"uller slice of $\op{Teich}(\dot F)$, since $\chi(\dot F)\geq 0$.) Let $K_0$ be a compact subset of $\mathcal{M}_J(\dot F, \widehat W; \bs\gamma_+;\bs\gamma_-)$. The following is a slight variation of the definition that was given for $\chi(\dot F)<0$:

\begin{align*}
\mathcal{B}& =\mathcal{B}_{K_0} =  \mathcal{B}_{K_0}(\dot F, \widehat{W};\bs\gamma_+;\bs\gamma_-)\\
=&\left\{ (\F,u)=((F,j,\mathbf{p},{\bf q}, \mathbf{r}),u) \left|
\begin{array}{cc}
u\in W_\delta ^{k+1,p}( \dot F, \widehat{W};\bs\gamma_+;\bs\gamma_-), \\
u \mbox{ is $\varepsilon''$-close to $u'$ where $[u']\in K_0$},\\
u(r_{\pm, i})=x_{\gamma_{\pm,i}^s}, {\bf q} \mbox{ satisfies \ref{Q}}
\end{array}
\right.\right\},
\end{align*} 
where $\varepsilon''>0$ is a sufficiently small constant that depends on $\varepsilon'$ and $K_0$, and $\varepsilon''$-closeness is measured with respect to the norm on $W_\delta^{k+1,p}(\dot F, (u')^*D_\epsilon)$ (we choose $k,p$ as in Remark~\ref{rmk: values of k and p} to ensure $u$ is $C^1$-close to $u'$ so we can apply Lemma~\ref{tenren} or ~\ref{tenren2}). {\em As before, when the Teichm\"uller slice, i.e., the choice of $(F,j,{\bf p})$, and the compact subset $K_0$ are understood, they will often be omitted from the notation, e.g., $\overline{\partial}_J$, $\mathcal{B}$, $\mathcal{E}$.}\cb

The fibers $\mathcal{E}_{(\F,u)}$ of the bundle $\pi: \mathcal{E}\to \mathcal{B}$ are defined by
$$\mathcal{E}_{(\F,u)}=W^{k,p}_\delta(\dot F, \wedge^{0,1}u^* T\widehat{W})$$
as before and $\overline\bdry_J$ is the same as before.  The linearization $L_{(\F,u)}$ is the same as that of Equation~\eqref{eqn: linearization} with the term $T\widetilde{\mathcal{U}}(j)$ removed (or viewed as the zero vector space).

\subsection{Description of the linearized $\overline\bdry$-operator $D_u$}
By McDuff-Salamon~\cite[Proposition~3.1.1]{MS},
$$D_u: W^{k,p}_{\delta} (\dot F, u^* T\widehat W)\to W^{k-1,p}_{\delta} (\dot F, \wedge^{0,1} u^* T\widehat W)$$
is given by
\begin{equation} \label{Du}
D_u \xi=\frac{1}{2}(\nabla \xi +J(u)\nabla \xi \circ j) -\frac{1}{2}J(u)(\nabla_\xi J)(u)\partial_J(u).
\end{equation}
Here, by abuse of notation, we do not distinguish between sections of $u^*T\widehat W$ and sections of $T\widehat W$ along $u$.

Suppose $(s,t)\in \R^\pm\times \R/\mathcal{A}_\alpha(\gamma)\Z $ is a holomorphic coordinate around a puncture $p\in {\bf p}_\pm$ of $\dot F$, $p$ corresponds to the Reeb orbit $\gamma$, and $u(s,t)=(s,t,\eta (s,t))$, written in simple coordinates for $\gamma$, is $J$-holomorphic. Then, near this puncture,
\begin{equation} \label{Dulocal}
D_u \xi=\frac{1}{2} (\nabla_s \xi +J(u)\nabla_t \xi + \widetilde S^\gamma \xi)\otimes (ds -i dt),
\end{equation}
where
$$J(u)=\left(\begin{array}{ccc}
0 & -1 & 0\\
1 & 0 & 0\\
X\eta & -j_{0}X\eta & j_{0}
\end{array}\right) \quad \text{and  }\quad \widetilde S^\gamma=\left(\begin{array}{ccc}
0 & 0 & 0\\
0 & 0 & 0\\
0 & 0 & -j_0X
\end{array}\right).$$

Note that if $u$ is $J$-holomorphic, then $D_u u_s=D_u u_t=0$.  If, instead of using the basis
$$\{\partial_s, \partial_t, \partial_{x_1},\dots,\partial_{x_n},\partial_{y_1},\dots,\partial_{y_n}\},$$
we use the basis
$$\{ u_s, u_t, \partial_{x_1},\dots,\partial_{x_n},\partial_{y_1},\dots,\partial_{y_n} \},$$
then Equation \eqref{Dulocal} becomes
\begin{equation} \label{Dusimple}
D_u \xi=\frac{1}{2}(\partial_s \xi +J_0 \partial_t \xi+\widetilde S^\gamma\xi)\otimes (ds-idt),
\end{equation}
where
$$J_0=\left(\begin{array}{ccc}
0 & -1 & 0\\
1 & 0 & 0\\
0 & 0 & j_{0}
\end{array}\right)$$
is the standard complex structure on $\R^{2n+2}$.

Note that $\partial_{x_i}, \partial_{y_i}$ is not a continuous vector field on $[0,\mathcal A_\alpha (\gamma) ]\times D_\delta^{2\overline m} \times D_\delta^{2\underline m}/_{\sim}$ due to the identification $\sim$, but it does not affect the matrix representation of $\widetilde S^\gamma$ and $J_0.$
The vector field $\xi = (\xi^\circ, \overline \xi, \underline \xi) \in \R^2\times \R^{2\overline m} \times \R^{2 \underline m}$ in Equation~\eqref{Dusimple} needs to satisfy the $\xi^\circ(s,\mathcal A_\alpha (\gamma)) = \xi^\circ(s,0), \overline \xi(s,\mathcal A_\alpha (\gamma)) = \overline \xi(s,0)$, and $\underline \xi(s, \mathcal A_\alpha (\gamma)) = -\underline \xi(s,0)$.
\cb

If $u$ is immersed, we denote by $N$ the normal bundle of $u$ such that
$$N=\R \langle\partial_{x_1},\dots,\partial_{x_{2n}},\partial_{y_1},\dots,\partial_{y_{2n}}\rangle$$
near the ends.  We define
$$D^N_u:W_{\delta}^{k,p}(\dot F, u^*N)\to W_{\delta}^{k-1,p}(\dot F, \wedge^{0,1}u^*N)$$
by projecting $D_u$ along the splitting $u^*T\widehat W=u^*N\oplus T\dot F$.

\subsection{The asymptotic operator} \label{subsection: asymptotic operator}
 
Let $\gamma$ be a Reeb orbit of $\mathbf{R}_\alpha$.
We denote by $\widetilde{\mathfrak C}$ the space of functions $f = (f^\circ, \overline f, \underline f):[0, \mathcal A_\alpha (\gamma)] \to \R^2 \times \R^{2\overline m}\times \R^{2\underline m}$ that satisfies $f^\circ(\mathcal A_\alpha (\gamma)) = f^\circ(0)$, $\overline f(\mathcal A_\alpha (\gamma)) = \overline f(0)$ and $\underline f(\mathcal A_\alpha (\gamma)) = -\underline f(0).$
Let $W^{1,2}(\widetilde{\mathfrak C})$ and $L^2(\widetilde{\mathfrak C})$ be the completion of $\widetilde{\mathfrak C}$ under $W^{1,2}$ and $L^2$-norms respectively.
Let
$$\widetilde A^\gamma=-J_0\frac{\partial}{\partial t}-\widetilde S^\gamma:W^{1,2}(\widetilde{\mathfrak C})\to L^{2}(\widetilde{\mathfrak C})$$
be the asymptotic operator for $\gamma$. Then $\widetilde A^{\gamma}$ is self-adjoint. 
\cb
Let
$$\dots \leq \lambda_{-2}^\gamma < \lambda_{-1}^\gamma=0=\lambda_1^\gamma < \lambda_2^\gamma\leq \dots$$ be eigenvalues of $\widetilde A^{\gamma}$ and
$$\dots,f_{-2}^\gamma,f_{-1}^\gamma,f_1^\gamma,f_2^\gamma,\dots$$
be the associated complete set of orthonormal eigenfunctions. 
\nom[fi]{$\{f_i^\gamma\}_i$}{Orthonormal eigenfunctions of $\widetilde A^\gamma$}
Here the self-adjoint\-ness and orthonormality are defined with respect to the standard inner product in $\R^{2n+2}$; 
\coblu
the two zero eigenvalues $\lambda_{-1}^\gamma$ and $\lambda_{1}^\gamma$ correspond to the eigenfunctions \coblue $(1,0,0, \dots, 0)$ and $(0, 1, 0, \dots, 0)$; \cb the fact that all other eigenfunctions are non-zero follows from the non-degenerency of $\gamma$ (see \cite{BH} for a calculation of eigenfunctions in dimension three).
\cb

We will write $\widetilde A$ and $\widetilde S$ for $\widetilde A^\gamma$ and $\widetilde S^\gamma$ if $\gamma$ is understood; similarly we may suppress the superscript $\gamma$ in $\lambda_i^\gamma, f_i^\gamma,\dots$ below, when they are understood.

\begin{rmk}
$\widetilde A$ is an ``enlargement'' of the operator $A: W^{1,2}(\mathfrak C)\to L^{2}(\mathfrak C)$ in the proof of Lemma \ref{end generic} which takes into account only the transverse directions (i.e., the $\bdry_{x_i}$- and $\bdry_{y_i}$-directions).
\end{rmk}
\cb 
Summarizing the above discussion, we have:

\begin{lemma} \label{Dukernel}
Suppose $(s,t)\in \R^\pm\times \R/\mathcal{A}_\alpha(\gamma)\Z $ is a holomorphic coordinate around a puncture $p\in {\bf p}_\pm$ of $\dot F$, $p$ corresponds to the Reeb orbit $\gamma$, and near $p$, $u$ admits a lift to $\R\times \R/\mathcal{A}_\alpha(\gamma)\Z\times D$ which can be written as $(s,t)\mapsto (s,t,\eta (s,t))$ with respect to simple coordinates for $\gamma$ and is $J$-holomorphic. Then, near $p$, 
$$D_u \xi =\frac{1}{2}(\partial_s \xi -\widetilde A \xi )\otimes_{\C} (ds-idt),$$
with respect to the basis
$$\mathfrak{B}:=\{ u_s, u_t, \partial_{x_1},\dots,\partial_{x_n},\partial_{y_1},\dots,\partial_{y_n} \}.$$

Hence, for any $\xi \in \ker D_u$, near $p\in {\bf p}_+$ we have
$$\xi(s,t)=\sum_{\lambda_i<0} c_{+,i} e^{\lambda_i s} f_i(t),$$
and near $p\in {\bf p}_-$
we have
$$\xi(s,t)=\sum_{\lambda_i>0} c_{-,i} e^{\lambda_i s} f_i(t),$$
where the $c_{\pm,i}$'s are constant.
\end{lemma}

\subsection{The adjoint operator}

Let
$$D_u^*: W^{k,p}_\delta (\dot F,\wedge ^{0,1} u^{*}T\widehat{W})\to W^{k-1,p}_\delta (\dot F,u^{*}T\widehat{W})$$
be the adjoint operator of $D_u$ defined by
\begin{equation} \label{Du*}
\int_{\dot F} \langle \zeta, D_u \xi \rangle_1 dvol_{\dot F} =\int_{\dot F} \langle D_u^* \zeta, \xi \rangle_0 dvol_{\dot F} ,
\end{equation}
where $\langle, \rangle_1$ is an inner product on $W_\delta^{k,p}(\dot F,\wedge^{0,1}u^{*}T\widehat{W})$, $\langle, \rangle_0$ is an inner product on $ W^{k+1,p}_\delta (\dot F,u^{*}T\widehat{W})$, and $dvol_{\dot F}$ is defined using a metric $g_{\dot F}$ on $\dot F$. We additionally assume that, near $p\in {\bf p}_\pm$:
\begin{itemize}
\item[(i)] $g_{\dot F}$ restricts to $ds^2+dt^2$;
\item[(ii)]$\langle, \rangle_0$ is defined by requiring the basis ${\mathfrak B}$ to be an orthonormal basis; and
\item[(iii)] $\langle, \rangle_1$ is induced by $\langle, \rangle_0$ and $g_{\dot F}$.
\end{itemize}

\begin{rmk}
Recall that the adjoint operator $D_u^*$ involve some choices.  We are not using the Riemannian metric $g_0$ in the definition of $D_u^*$.
\end{rmk}

The following lemma is almost evident:

\begin{lemma} \label{Dukernel-star}
With the assumptions of Lemma~\ref{Dukernel}, if $\zeta\in W^{k,p}_\delta(\dot F,\wedge^{0,1}u^*T\widehat{W})$ and $\zeta=\zeta^0\otimes (ds-idt)$ near $p\in {\bf p}_\pm$, then we have
$$D_u^* (\zeta^0\otimes (ds-idt)) =\frac{1}{2}(-\partial_s \zeta^0 -\widetilde A \zeta^0),$$
with respect to the basis ${\mathfrak B}$. Hence, for any $\zeta\in \ker D_u^*$, near $p\in{\bf p}_+$ we have
$$\zeta^0(s,t)=\sum_{\lambda_i>0} c_{+,i} e^{-\lambda_i s} f_i(t),$$
whereas near $p\in{\bf p}_-$ we have
$$\zeta^0(s,t)=\sum_{\lambda_i<0} c_{-,i} e^{-\lambda_i s} f_i(t),$$
where the $c_{\pm,i}$'s are constant.
\end{lemma}

\section{Interior semi-global Kuranishi charts} \label{section: semi-global Kuranishi charts}

Consider the moduli space
$$\mathcal{M}=\mathcal{M}_J^{\op{ind}=k}(\dot F,\widehat W; \bs\gamma_+;\bs\gamma_-)$$
\coblu and the finite group $G=G(\dot F, \widehat W; \bs\gamma_+;\bs\gamma_-)$. \cb
As discussed in Section~\ref{subsection: symp vs cob}, there are two slightly different cases to consider: (Symp) and $\op{(Cob)}$.

The goal of this section is to construct a tuple
$$(\K,\pi_{\V}: \E\to \V,\overline\bdry_J, \psi,\coblu G),$$
called an {\em interior semi-global Kuranishi chart}. In the (Symp) case this consists of:
\begin{enumerate}
 
\item[(I1)] \label{chart 1} a compact subset \coblu $\K\subset \mathcal{M}/\R$; \cb
\item[(I2)] \label{chart 2}  a finite rank bundle \coblu $\pi_{\V}: \E\to \V$ over a finite-dimensional manifold, called an {\em obstruction bundle}; \cb
\item[(I3)] \label{chart 3} a section $\overline\bdry_J:\V\to \E$;
\item[(I4)] \label{chart 4} a homeomorphism $\psi: \overline\bdry_J^{-1}(0) \to \coblu \mathcal M/\R \cb$  onto an open subset of \coblu $\mathcal M/\R$ \cb and $\K \subset \op{Im}(\psi)$;
\item[(I5)] $\dim \V - \op{rk} \E = \op{vdim} \mathcal M/\R$; and
\item[(I6)] \coblu $\K$ is invariant under the finite group $G$, the action of $G$ extends to an action on $\pi_{\V}: \E\to \V$ by bundle automorphisms, and the quotient by $G$ gives a global quotient orbibundle (i.e., can be described by a single orbibundle chart).
\end{enumerate}
In the $\op{(Cob)}$ case, $\mathcal{M}/\R$ is replaced by $\mathcal{M}$.

In practice, we take $\K$ to be ``large" in the sense that the only curves that are excluded are very close to breaking.  This will be made more precise later, when we explain the notions of close to breaking and neck length.   In particular, if $\mathcal{M}/\R$ is compact, then we take $\K=\mathcal{M}/\R$. \cb 

We will also construct a multisection $\mathfrak s$ of the bundle $\E\to \V$, called the {\em obstruction multisection}, such that $\mathfrak s \pitchfork \overline \bdry_J$.
\cb

\subsection{Construction of $\overline{\bdry}_J$-transverse subbundles} \label{subsection: construction of transverse subbundles}

{\em   Suppressing the Teichm\"uller slice $\widetilde {\mathcal U}$ or the compact subset $K_0$ from the notation,}
\cb consider the Banach bundle $\pi: \mathcal{E}\to \mathcal{B}$ and the section $\overline{\bdry}_J: \mathcal{B}\to \mathcal{E}$ corresponding to the quadruple $(\dot F,\widehat W,\bs\gamma_+;\bs\gamma_-)$, as defined in
  
Section~\ref{subsubsection: negative Euler char} for the case $\chi(\dot F) < 0$ and Section~\ref{subsubsection: nonnegative Euler char} for the case $\chi(\dot F) \geq 0$. Let $Z=\overline\bdry_J^{-1}(0)$.
\cb

\begin{defn}
Let $\mathcal{V}$ be a subset of $\mathcal{B}$.
A subbundle $E\to \mathcal{V}$ of $\mathcal{E}|_{\mathcal{V}}$ is {\em $\overline{\bdry}_J$-transverse} if $E$ has finite rank and, for each $(\F,u)\in \mathcal{V}$,
\begin{equation} \label{eqn: transverse}
\op{Im} L_{(\F,u)} + E_{(\F,u)} =\mathcal{E}_{(\F,u)},
\end{equation}
where $E_{(\F,u)}$ and $\mathcal{E}_{(\F,u)}$ refer to the fibers of $E$ and $\mathcal{E}$ over $(\F,u)$. Here $+$ in Equation~\eqref{eqn: transverse} means span and {\em does not mean direct sum.}
\end{defn}

The goal of this subsection is to prove the following:

\begin{thm} \label{thm: construction of L-transverse subbundle}
  Given a compact subset $K$ of $Z$, there exist a positive integer $\ell_0$ and a sufficiently small open neighborhood $\mathcal{N}(K)\subset \mathcal{B}$ of $K$ such that for all $\ell \geq \ell_0$, there exists a  $\overline{\bdry}_J$-transverse rank $\ell(l_++l_-)$ subbundle
$$E\to \mathcal{N}(K)$$
\cb
of $\mathcal{E}|_{\mathcal{N}(K)}$ satisfying:
\begin{enumerate}
\item if $(\F,u), (\F',u')\in \mathcal{N}(K)$ and there exists a $C^\infty$-diffeomorphism $\phi: (\F,u)\xrightarrow{\sim} (\F',u')$, then there is a canonical identification $\phi_*E_{(\F,u)}= E_{(\F',u\circ \phi^{-1})}$ which is induced from
$$\phi_*:\wedge^{0,1} T\dot F\otimes u^* T\widehat{W} \xrightarrow\sim  \wedge ^{0,1} T\dot F\otimes (u\circ\phi^{-1})^{*}T\widehat{W};$$
\item  if $(\widehat{W}=\R\times M,J)\in \op{(Symp)}$ and $(\F,u), (\F,u_T)\in \mathcal{N}(K)$, where $u_T$ is $u$ translated up by $s=T$, i.e., $u_T=\Phi_T\circ u$ where
$$\Phi_T: \R\times M\xrightarrow\sim \R\times M, \quad (s,x)\mapsto (s+T,x)$$
is the $s=T$ translation map, then $(\Phi_T)_*E_{(\F,u)}= E_{(\F,u_T)}$.
\end{enumerate}
\end{thm}

\begin{rmk}
When $\op{dim}(M)=3$, it is possible to obtain an effective bound on $\ell_0$ using the positivity of intersections and winding numbers (see \cite{HLS,HWZ2,We}).
\end{rmk}

\begin{rmk} \label{rmk: dependence on ell and varepsilon}
  In addition to $\ell$, the bundle $E\to \mathcal{N}(K)$ also depends on $\varepsilon=\varepsilon(K)>0$ small satisfying (i) and (ii) in Step 1A when $\chi(\dot F)<0$ and (i) and (ii) with $\dot F$ replaced by $\ddot F$ when $\chi(\dot F)\geq 0$. \cb  The dependence on $\epsilon$ and $\ell$ will be discussed in Section~\ref{subsection: stabilization}.
\end{rmk}

\begin{proof}
The proof consists of three steps. Let $g_{(\dot F,j)}$ (resp.\ $g_{(\ddot F,j)}$) be the complete finite volume hyperbolic metric on $\dot F$ (resp.\ $\ddot F=\dot F-{\bf q}$) compatible with $j$, when $\chi(\dot F)<0$ (resp.\ $\chi(\dot F)\geq 0$). 

\nom[g]{$g_{(\dot F,j)}$ or $g_{(\ddot F,j)}$}{Complete hyperbolic metric on $\dot F$ or $\ddot F$ compatible with $j$}

\s\n
{\em Step 1.} (Definition of $s_{\pm,i}$) Let $K$ be a compact subset of $Z$.  
\cb
The goal of this step is to define continuous maps $s_{\pm,i}:K \to \R$
with differentiable extensions $s_{\pm,i}: \mathcal{N}(K)\to \R$, where the differentiability depends on the choice of $W^{k+1,p}$. Here each $s_{\pm,i}$ corresponds to $\gamma_{\pm,i}$   and $\mathcal{N}(K)$ is a small open neighborhood of $K$ which we may shrink as we go along. \cb

The collection $\{s_{\pm,i}\}$ will be invariant under group actions as follows:
\begin{itemize}
\item[(a)] If $(\F,u)=((F,j,{\bf p},{\bf q},{\bf r}),u)\sim (\F',u')=((F,j',{\bf p},{\bf q}',{\bf r}'),u')\in \mathcal{N}(K)$ via the automorphism $\phi:\dot F\xrightarrow{\sim}\dot F$, then  $\{s_{\pm,i}(\F,u)\}=\{s_{\pm,i}(\F',u')\}$.
\item[(b)] If $(\F,u), (\F,u_T)\in \mathcal{N}(K)$, then $\{s_{\pm,i}(\F,u_T)\}=\{s_{\pm,i}(\F,u)+T\}$.
\end{itemize}

Let $\kappa>0$ be a predetermined constant which is smaller than $1/100$ the radii of all the simple coordinate charts of all the orbits used in $(\F,u)$ (and is independent of $K$).

\s\n
{\em Step 1A.} Suppose $\chi(\dot F)<0$.   There exists $\varepsilon=\varepsilon(K)>0$ small such that, for all $(\F,u)\in K$,
\nom[Thin]{$\op{Thin}_\varepsilon(\dot F,g_{(\dot F,j)})$ or $\op{Thin}_\varepsilon(\ddot F,g_{(\ddot F,j)})$}{The $\varepsilon$-thin part of $(\dot F,g_{(\dot F,j)})$ or $(\ddot F,g_{(\ddot F,j)})$}
\be
\item[(i)] the $\varepsilon$-thin part $\op{Thin}_\varepsilon(\dot F,g_{(\dot F,j)})$ of $(\dot F,g_{(\dot F,j)})$, is the disjoint union of cusps 
\begin{equation}\label{equation: cusps}
C_{+,1},\dots,C_{+,l_+},C_{-,1},\dots,C_{-,l_-}
\end{equation}
and possibly some annuli (of moduli $\geq 0$), where each $\bdry C_{\pm,i}$ is a circle which corresponds to $\gamma_{\pm,i}$ and  
\item[(ii)] $\op{Im}(u|_{C_{\pm,i}})$ is $\kappa$-close  to $[s_{+,i},\infty)\times \gamma_{+,i}^s$ or $(-\infty,s_{-,i}]\times \gamma_{-,i}^s$ for some $s_{\pm,i}$.
\ee
Here the distance is measured with respect to the Hausdorff metric derived from the Riemannian metric $g_0$ on $\widehat{W}$ which is $s$-invariant at the ends. The existence of $\varepsilon>0$ satisfying (i) for all $(\F,u)\in K$ is standard hyperbolic geometry and the existence of $\varepsilon>0$ satisfying (ii) for all $(\F,u)\in K$ relies on the exponential decay estimates of the ends of $u$ from \cite{HWZ1}. Observe that (i) and (ii) also hold for all $(\F,u)\in \mathcal{N}(K)$, provided $\mathcal{N}(K)$ is a sufficiently small neighborhood of $K$. \cb

We then choose a specific $s_{\pm,i}=s_{\pm,i}(\F,u)$ as follows: First consider the $m(\gamma_{\pm,i})$-fold covering map
$\R/\mathcal{A}_\alpha(\gamma_{\pm,i})\Z\to \R/\mathcal{A}_\alpha(\gamma_{\pm,i}^s)\Z$.   Using canonical cylindrical coordinates, we write $u$ on $C_{\pm,i}$ as \cb
$$u(s,t)=(s,t,\eta(s,t)),$$
after lifting to the $m(\gamma_{\pm,i})$-fold cover $\R\times (\R/\mathcal{A}_\alpha(\gamma_{\pm,i})\Z)\times D$, subject to the condition that $t=0$ corresponds to the asymptotic marker $r_{\pm,i}$. 

\nom[s]{$s_{\pm,i}(\F,u)$}{The $s$-coordinate indicating the location of the perturbations $\widetilde f_j^{\gamma_{\pm,i}}$, where $\gamma_{\pm,i}$ refers to the $i$th orbit at the positive/negative end}

\begin{defn} [$s_{\pm,i}$ and $c_{\pm,i}$]\label{defn: somewhat informal}
The   average of the $s$-coordinates of the points on $u|_{\bdry C_{\pm,i}}$ with $t= \mathcal{A}_\alpha(\gamma_{\pm,i}^s), 2\mathcal{A}_\alpha(\gamma_{\pm,i}^s),\dots, \mathcal{A}_\alpha(\gamma_{\pm,i})$ \cb will be denoted by $s_{\pm,i}$.  Also
the point on $\bdry C_{\pm,i}$ corresponding to $r_{\pm,i}$ will be denoted by $c_{\pm,i}$.
\end{defn}


\nom[E0]{$\varepsilon>0$}{Small positive constant used to determine the thin part of $(\dot F, g_{(\dot F,j)})$ or $(\ddot F, g_{(\ddot F,j)})$}

\s\n
{\em Step 1B.} Suppose $\chi(\dot F)=0$ or $1$.  In view of the discussion from Section~\ref{subsubsection: nonnegative Euler char}   (with $\varepsilon'>0$ satisfying Lemmas~\ref{tenren} and \ref{tenren2}), \cb  we take $\ddot F= \dot F-{\bf q}$ with $\chi(\ddot F)<0$. Let $g=g_{(\ddot F,j)}$ and apply the considerations from Step 1A with $\dot F$ replaced by $\ddot F$ to obtain $s_{\pm,i}=s_{\pm,i}(\F,u)$.

(a) and (b) are immediate from the construction, both in Steps 1A and 1B.

\begin{rmk}[Trivial cylinders]  
\label{rmk: trivial cylinder}
We do not construct Kuranishi charts for trivial cylinders when $(\widehat W,J)$ is in the (Symp) case. On the other hand, when a symplectization is viewed in $\op{(Cob)}$, our general prescription is to slightly perturb the almost complex structure using Lemma~\ref{end generic} so that the perturbed trivial cylinders no longer have trivial ends. We note that trivial (possibly multiply-covered) cylinders are automatically transverse. \cb
\end{rmk}

\s\n
{\em Step 2.}   (Definition of $E$) \cb
Let $\varepsilon=\varepsilon(K)$ be the constant from Step 1.

Let $\beta:\R\to [0,1]$ be a nondecreasing smooth cutoff function such that $\beta(s)=0$ for $s\leq 0$ and $\beta(s)=1$ for $s\geq 1$ and let $\beta^+_{s_0}(s)=\beta(s-s_0)$ and $\beta^-_{s_0}(s)=-\beta(s_0-s)$, where $s_0>0$ is a constant.

Given $(\F,u)\in \mathcal{N}(K)$ and $s_{\pm,i}=s_{\pm,i}(\F,u)$ from Step 1   (and still using canonical cylindrical coordinates), \cb let
 
\begin{equation} \label{f tilde}
\widetilde{f}^{\gamma_{\pm,i}}_j(s,t):=\frac{\partial \beta^\pm_{s_{\pm,i}}}{\partial s}(s)f^{\gamma_{\pm,i}}_j(t)\otimes \pi_j (ds-idt)\in \Gamma (\dot F ,\wedge^{0,1} u^{*}T\widehat{W}).
\end{equation}
\cb
Here $\pi_j:W^{k,p}(\dot F, \wedge ^1_{\C} T^{*} \dot F)\to W^{k,p}(\dot F, \wedge^{0,1}T^{*} \dot F)$ is the projection with respect to the domain complex structure and $\widetilde{f}^{\gamma_{\pm,i}}_j(s,t)$ has support on a single end of $\dot F$.
We then define
\nom[L]{$\ell$}{Positive integer appearing in the definition of $E^{\ell,\varepsilon}_{(\mathcal{F}, u)}$, which counts the number of asymptotic eigenfunctions at each end}
\begin{equation} \label{equation: def of each component of E}
E_{(\F,u),{\gamma_{\pm,i}}}^{\ell,\varepsilon}:=\R \langle \widetilde{f}^{\gamma_{\pm,i}}_1,\dots,\widetilde{f}^{\gamma_{\pm,i}}_\ell \rangle
\end{equation}
for ${\gamma_{\pm,i}}\in \bs\gamma_\pm$ and define
\nom[E]{$E^{\ell,\varepsilon}_{(\mathcal{F}, u)}$}{Fiber of obstruction bundle over $(\mathcal{F},u)$}
\begin{equation}
E_{(\F,u)}^{\ell,\varepsilon}:=\bigoplus_{{\gamma_{\pm,i}}\in\bs{\gamma}_+\cup\bs\gamma_-} E_{(\F,u),{\gamma_{\pm,i}}}^{\ell,\varepsilon}.
\end{equation}

This gives us a bundle   $E: = E^{\ell,\varepsilon} \to \mathcal{N}(K)$ with fiber $E_{(\F,u)}^{\ell,\varepsilon}$ of dimension $\ell(l_++l_-)$. The following will be proven in the Appendix: \cb

\begin{lemma} \label{lemma: in Wkp and C1}
$E_{(\F,u)}\subset \mathcal{E}_{(\F,u)}$ for all $(\F,u)\in \mathcal{N}(K)$ and $E \to \mathcal{N}(K)$ is of class $C^1$ if $k\geq 3$ and $p>2$.
\end{lemma}


\s\n
{\em Step 3.} ($\overline{\bdry}_J$-transversality)  We first prove the following transversality lemma, which is applicable for a single $(\F,u)\in K$.

\begin{lemma}[Transversality lemma] \label{lemma: transversality lemma}
For each $(\F, u)\in K$, there exist $\ell_0 \in \N$ and $\varepsilon_0>0$ such that
\begin{equation} \label{eqn: trans lemma}
E_{(\F,u)}^{{\ell,\varepsilon}} + \op{Im} D_u=W^{k,p}_\delta (\wedge ^{0,1}\dot F \otimes u^{*}T\widehat W)
\end{equation}
for all $\ell\geq \ell_0$ and $0<\varepsilon\leq \varepsilon_0$.
\end{lemma}

\begin{rmk}
Observe that we are actually proving a slightly stronger statement than
$$E_{(\F,u)}^{{\ell,\varepsilon}} + \op{Im} L_{(\F,u)}=W^{k,p}_\delta (\wedge ^{0,1}\dot F \otimes u^{*}T\widehat W).$$
\end{rmk}

\begin{proof}
We first claim that, if Equation~\eqref{eqn: trans lemma} does not hold, then there exists a nonzero $\zeta\in \ker D_u^*$ which is $L^2$-orthogonal to $E_{(\F,u)}^{\ell,\varepsilon}$.   To see this, consider the projection map $\pi:W^{k,p}_\delta (\wedge ^{0,1}\dot F \otimes u^{*}T\widehat W)=\ker D_u^*\oplus \op{Im} D_u\to \ker D_u^*$. If Equation~\eqref{eqn: trans lemma} does not hold, then $\pi(E_{(\F,u)}^{\ell,\varepsilon})$ cannot be all of $\ker D_u^*$. Taking $\zeta\in \ker D_u^*$ to be \cb nonzero and orthogonal to $\pi(E_{(\F,u)}^{\ell,\varepsilon})$, we have $\zeta\perp E_{(\F,u)}^{\ell,\varepsilon}$ since $\zeta\perp \op{Im} D_u$.

Let $\varepsilon_0>0$ be sufficiently small so that on the $\varepsilon_0$-thin neighborhood of a puncture of $\dot F$ corresponding to the Reeb orbit $\gamma\in \bs\gamma_+$ ($\gamma\in \bs\gamma_-$ is similar), $\zeta\in \ker D_u^*$ can be written with respect to the basis $\mathfrak{B}$ from Lemmas~\ref{Dukernel}   and ~\ref{Dukernel-star} \cb as
$$\zeta(s,t)=\sum_{\lambda_i>0} c_i e^{-\lambda_i^\gamma s} f_i^\gamma(t) \otimes (ds-idt),$$
and let $0<\varepsilon\leq \varepsilon_0$. The key point is that, when $\varepsilon_0$ is sufficiently small, $\mathfrak{B}$ is arbitrarily close to $\{\bdry_s,\bdry_t,\bdry_{x_1},\dots,\bdry_{x_n},\bdry_{y_1},\dots,\bdry_{y_n}\}$ and, without loss of generality, we may assume $\zeta(s,t)$ is written with respect to the latter.  There exists a basis $\{\zeta_1,\dots, \zeta_k\}$ for $\ker D_u^*$ for which the coefficients $c_{ji}$ (this means $c_i$ for $\zeta_j$) form a matrix in row echelon form; in particular, there are no rows with all $c_{ji}=0$ since the corresponding $\zeta_j$   would be identically zero \cb by unique continuation, which is a contradiction.

Let $\ell_0>0$ be larger than the smallest $i$ for which $c_{ki}\not=0$.  Now observe that if $\zeta\in \ker D_u^*$ is $L^2$-orthogonal to $E_{(\F,u)}^{\ell_0,\varepsilon}$, then its coefficients $c_i=0$ for all $i\leq \ell_0$, since the $L^2$-inner product of $\zeta$ and $\widetilde f_i^\gamma$ is nonzero if $c_i\not=0$.  This implies that $\zeta=0$.  The lemma then follows.
\end{proof}

Lemma~\ref{lemma: transversality lemma} can be improved as follows:

\begin{lemma}[Family transversality lemma] \label{lemma: family transversality lemma}
There exist $\ell_0 \in \N$, $\varepsilon_0>0$,   and a sufficiently small open neighborhood $\mathcal{N}(K)\subset \mathcal{B}$ of $K$, \cb such that
$$E_{(\F,u)}^{{\ell,\varepsilon}} + \op{Im} D_u=W^{k,p}_\delta (\wedge ^{0,1}\dot F \otimes u^{*}T\widehat W).$$
for all $(\F, u)\in \mathcal{N}(K)$, $\ell\geq \ell_0$, and $0<\varepsilon\leq \varepsilon_0$
\end{lemma}

\begin{proof}
This follows from Lemma~\ref{lemma: transversality lemma}, the compactness of $K$, and the following well-known property of Fredholm maps: Let $L,L':V\to W$ be Fredholm maps between Banach spaces. If $L$ and $L'$ are close in the operator norm and $\op{Im}(L)+W'=W$ for a finite-dimensional subspace $W'$, then $\op{Im}(L')+W'=W$.


For $(\F_0,u_0), (\F_1,u_1)\in K$ that are close, we identify the Banach spaces involved in $L_{(\F_i,u_i)}$, $i=0,1$, by defining an isomorphism $\mathcal P: u_0^*T\widehat W \xrightarrow{\sim} u_1^*T\widehat W$ as follows:
\begin{itemize}
\item[(i)] on a large compact region of $\dot F$, $\mathcal P$ is given by parallel translation along the shortest geodesics between $u_0(z)$ and $u_1(z)$;
\item[(ii)] on the ends of $\dot F$ (assumed to be sufficiently small), $\mathcal P$ is given by identifying the bases
$$\{(u_0)_s,(u_0)_t, \partial_{x_1},\dots,\partial_{x_n},\partial_{y_1},\dots,\partial_{y_n}\}$$
$$\{(u_1)_s, (u_1)_t, \partial_{x_1},\dots,\partial_{x_n},\partial_{y_1},\dots,\partial_{y_n} \};$$
\item[(iii)] $\mathcal{P}$ interpolates between (i) and (ii) in the intermediate region.
\end{itemize}
Then $L_{(\F_0,u_0)}$ and $L_{(\F_1,u_1)}$ are close in the operator norm; the details are left to the reader.
\end{proof}

  This concludes the proof of Theorem~\ref{thm: construction of L-transverse subbundle}.
\end{proof}
\cb

\subsection{The obstruction bundles $\pi_V:E|_V\to V$} \label{sec: the bundles}

Let $K$ be a compact subset of $\mathcal{M}/\R$ in the $\op{(Symp)}$ case and $\mathcal{M}$ in the $\op{(Cob)}$ case.

\subsubsection{Definition of $\pi_V:E|_V\to V$} \label{subsubsection: defn of pi sub V}

With the bundle $E=E^{\ell,\varepsilon}\to \mathcal{N}(K)$ constructed in Theorem~\ref{thm: construction of L-transverse subbundle}, we define
\begin{equation} \label{eqn: V and V prime}
V:= \overline\bdry_J^{-1}(E)\subset \mathcal{N}(K). 
\end{equation}
Since $E$ is $\overline\bdry_J$-transverse and of class $C^1$, $V\supset K$ is a finite-dimensional submanifold of $\mathcal{N}(K)$ of class $C^1$.  Once we have extracted $V$, since the elements of $V$ are smooth by elliptic regularity, $V$ can be ``upgraded to" a smooth finite-dimensional manifold of the expected dimension.

Observe that if $(\F,u)\in V$, then $u:\dot F\to \widehat{W}$ is a smooth map whose ends are $J$-holomorphic and asymptotic to $\bs\gamma_+\cup \bs\gamma_-$. Let
\begin{equation}
\pi_V:E|_V\to V
\end{equation}
be the restriction of $E\to \mathcal{N}(K)$ to $V$ and let $\overline\bdry_J:  V\to E|_V$ be the restriction of $\overline\bdry_J$ to $V$. The vector bundle $\pi_V$ is a smooth vector bundle over $V$.

\subsubsection{Patching together the bundles $\pi_V: E|_V\to V$} \label{subsubsection: patching together}

We would like to ``patch together" the obstruction bundles $\pi_V: E|_V\to V$ \coblu so that (I1)--(I6) are satisfied. \cb  To this end we prove Lemmas~\ref{lemma: canonical identification}, \ref{lemma: canonical identification 2}, and \ref{lemma: canonical identification 3} below, which can roughly be stated as:
\begin{enumerate}
\item independence of the choice of Teichm\"uller slice;
\item \coblu equivariance under puncture reorderings and marker rotations; and
\item invariance under $\R$-translations in the $\op{(Symp)}$ case.
\end{enumerate}
\coblu Recall that $\op{Aut}(F,j,{\bf p},{\bf q},{\bf r})=1$ by Lemmas~\ref{lemma: stabilizer} and \ref{lemma: stabilizer2}, so there are no nontrivial domain automorphisms. \cb

We first consider (1). Suppose we are in the $\op{(Cob)}$ case; an analogous result holds in the (Symp) case, but will not be stated explicitly.  \coblu In the $\chi(\dot F)<0$ case let $\widetilde{\mathcal{U}}_0$ and $\widetilde{\mathcal{U}}_1$ be two Teichm\"uller slices of $\mathcal{U}\subset \coblu \op{Teich}(F,{\bf p},{\bf r})$ and let $\phi_x: F\xrightarrow{\sim} F$, $x\in \mathcal{U}$, \cb be the smooth family of diffeomorphisms taking $j_x\in \widetilde{\mathcal{U}}_0$ to $j'_x\in \widetilde{\mathcal{U}}_1$ from Section~\ref{subsubsection: topology}.  We take bundles $\pi_{\widetilde{\mathcal{U}}_i}: \mathcal{E}_{\widetilde{\mathcal{U}}_i}\to \mathcal{B}_{\widetilde{\mathcal{U}}_i}$ and compact subsets $K_i\subset \mathcal{B}_{\widetilde{\mathcal{U}}_i}$ which are \coblu identified with $K$.  In the $\chi(\dot F)\geq 0$ case taking a Teichm\"uller slice of $\op{Teich}(\dot F)$ is equivalent to fixing $(F,j, {\bf p})$. Writing $\widetilde{\mathcal{U}_i}=\{(F,j_i,{\bf p})\}$ for $i=0,1$, we have a diffeomorphism $\phi$ taking $j_0$ to $j_1$, and we take bundles $\pi_{\widetilde{\mathcal{U}_i},K_i}: \mathcal{E}_{\widetilde{\mathcal{U}_i},K_i} \to \mathcal{B}_{\widetilde{\mathcal{U}_i}, K_i}$ where the $K_i$ are identified with $K$. \cb

Let $\pi_{V_i}:E_i|_{V_i}\to V_i$, $i=0,1$, be the bundle $\pi_V:E|_V\to V$ extracted from the Banach bundles $\pi_{\widetilde{\mathcal{U}}_i}$ or \coblu $\pi_{\widetilde{\mathcal{U}_i},K_i}$ \cb as appropriate, using Theorem~\ref{thm: construction of L-transverse subbundle}.

\begin{lemma}[First patching lemma] \label{lemma: canonical identification}
Suppose we are in the $\op{(Cob)}$ case. After possibly shrinking $V_i$ subject to the condition $V_i\supset K_i$, $V_0$ and $V_1$ are diffeomorphic and the bundles $\pi_{V_i}:E_i|_{V_i}\to V_i$ are isomorphic and the identifications only depend on the {\coblu family of diffeomorphisms $\{\phi_x\}_{x\in \mathcal{U}}$ in the $\chi(\dot F)<0$ case or $\phi$ in the $\chi(\dot F)\geq 0$ case.}
\end{lemma}

\begin{proof}
Let $(\F,u)=((F,j,\mathbf p, \mathbf q, \mathbf r), u) \in V_0$, $(\F',u')=((F, j',\mathbf p, \mathbf q', \mathbf r'), u') \in \mathcal{N}(K_1)$, and $\phi$ be the diffeomorphism satisfying
$$\phi_*((F,j,{\bf p},{\bf q},{\bf r}),u)= ((F,j',{\bf p},{\bf q}',{\bf r}'),u').$$
Then $\phi$ induces a natural identification $\phi_* E_{(\F,u)}\simeq E_{(\F',u')}$ provided
$$E_i=E_i^{\ell,\varepsilon}\to \mathcal{N}(K_i), \quad i=0,1$$ 
are constructed as in the proof of Theorem~\ref{thm: construction of L-transverse subbundle} with fixed $\ell$ and $\varepsilon$.   This is due to the fact that, in addition to $\ell$ and $\varepsilon$, $E_{(\F,u)}$ and $\{s_{\pm,i}(\F,u)\}$ (and likewise $E_{(\F',u')}$ and $\{s_{\pm,i}(\F',u')\}$) only depend on the image of $u$ and the hyperbolic geometry (i.e., the thick-thin decomposition) of $g_{(\dot F,j)}$. In particular, \cb $(\F',u')\in V_1$, provided $(\F,u)\in V_0$ is sufficiently close to $K_0$.

The family {\coblu of diffeomorphisms $\{\phi_x\}, x\in \mathcal{U}$, in the $\chi(\dot F)<0$ case or the diffeomorphism $\phi$ in the $\chi(\dot F)\geq 0$ case} gives a smooth map $\Phi$ from $E_0|_{V_0}\to V_0$ to $E_1|_{V_1}\to V_1$, after possibly shrinking $V_0$.  Since $\Phi$ is a diffeomorphism onto its image and the image of $V_0$ contains $K_1$, by the invariance of domain, $\Phi$ is a diffeomorphism $E_0|_{V_0}\to V_0$ to $E_1|_{V_1}\to V_1$, after possibly shrinking $V_1$.
\end{proof}

\begin{rmk}
We pass to $\pi_{V_i}:E_i|_{V_i}\to V_i$ at an early stage to avoid the well-known ``loss of derivative" issue when dealing with Sobolev spaces.  In particular, we avoid identifying $E_0\to \mathcal{N}(K_0)$ and $E_1\to \mathcal{N}(K_1)$.
\end{rmk}

(2) is stated more precisely as follows: Suppose we are in the $\op{(Cob)}$ case.  \coblu Let $\mathcal{U}=\mathcal{U}_{[j]}$ be a neighborhood of $[j]$ and $\widetilde {\mathcal{U}}_0$ be a Teichm\"uller slice of $\mathcal{U}$ containing $j$. Then $\widetilde {\mathcal{U}}_1=g(\widetilde{\mathcal{U}}_0)$ is a Teichm\"uller slice of \coblu $g(\mathcal{U})$, where $g$ is a composition of puncture reorderings and marker rotations. \cb Let $\pi_{V_i}:E_i|_{V_i}\to V_i$, $i=0,1$, be bundles constructed using $\widetilde{\mathcal{U}}_i$ and neighborhoods of $K_i$, where \coblu $K_1=g(K_0)$. \cb

\begin{lemma}[Second patching lemma]\label{lemma: canonical identification 2}
Suppose we are in the $\op{(Cob)}$ case. After possibly shrinking $V_i$ subject to the condition $V_i\supset K_i$, $V_0$ and $V_1$ are diffeomorphic via the isomorphism $g$ and the bundles $\pi_{V_i}: E_i|_{V_i}\to V_i$ are isomorphic.
\end{lemma}

\begin{proof}
Similar to that of Lemma~\ref{lemma: canonical identification} and \coblu relies on the fact that $E_{(\F,u)}$ and $E_{g(\F,u)}$ are naturally isomorphic.
\end{proof}

(3) is stated more precisely as follows:

\begin{lemma}[Third patching lemma] \label{lemma: canonical identification 3}
Suppose we are in the $\op{(Symp)}$ case and $K_i$, $i=0,1$, are   compact subsets of $\mathcal{B}_{\widetilde{\mathcal{U}}_i}$ or $\mathcal{B}_{\widetilde{\mathcal{U}}_i,K_i}$ and \cb $(\F,u)\in \mathcal{N}(K_0)$ if and only if $(\F,u_T)\in \mathcal{N}(K_1)$, where $u_T$ is $u$ translated up by $s=T$. Then, after possibly shrinking $V_i$ subject to the condition $V_i\supset K_i$, the bundles $\pi_{V_i}: E_i|_{V_i}\to V_i$ are canonically isomorphic.
\end{lemma}

\begin{proof}
Similar to that of Lemma~\ref{lemma: canonical identification}.
\end{proof}

\subsection{Interior semi-global Kuranishi chart} \label{subsection: semi-global Kuranishi chart}

In this subsection we construct an interior semi-global Kuranishi chart 
$$(\K,\pi_{\V}: \E\to \V,\overline\bdry_J,\psi,{\coblu G})$$
and an obstruction multisection ${\mathfrak s}$, where \coblu $\K$ is a compact subset of $\mathcal{M}/\R$ in the $\op{(Symp)}$ case and $\mathcal{M}$ in the $\op{(Cob)}$ case, and $G=G(\dot F,\R\times M;\bs\gamma_+;\bs\gamma_-)$.

\subsubsection{Construction of the chart}

We will treat the $\op{(Symp)}$ case; \coblu the $\op{(Cob)}$ case analogous but simpler since there is no target $\R$-translation. 

Let $\K$ be a $G$-invariant compact subset of $\mathcal{M}/\R$ and let $K_0$ be the lift of $\K$ to a $G$-invariant slice of $\mathcal{M}$ under the $\R$-action. In what follows we will not distinguish between $K$ and $K_0$.  

\begin{prop}
There exists a tuple $(\K,\pi_{\V}: \E\to \V,\overline\bdry_J,\psi,G)$ such that the rank of $\pi_{\V}$ depends on $\ell_0\gg 0$ and (I1)--(I6) hold. 
\end{prop}

\begin{proof}
When $\chi(\dot F)<0$, let $\mathcal{E}_{\widetilde{\mathcal U}}\to \mathcal B_{\widetilde{\mathcal U}}$ be the Banach bundle for some Teichm\"uller slice $\widetilde{\mathcal U}$ and $\overline\bdry_J: \mathcal B_{\widetilde{\mathcal U}} \to \mathcal{E}_{\widetilde{\mathcal U}}$ be the $\overline\bdry_J$-section.  When $\chi(\dot F)\geq 0$, the Banach bundle $\mathcal{E}_{K_0}\to \mathcal B_{K_0}$ is defined using $K_0$.

In either case Theorem~\ref{thm: construction of L-transverse subbundle} gives a $\overline\bdry_J$-transverse subbundle $E_0\to \mathcal{N}(K_0)$. If we set $V_0=\overline\bdry_J^{-1}(E_0)$, then $\pi_{V_0}: E_0|_{V_0}\to V_0$ is a finite rank vector bundle over a finite-dimensional manifold. \coblu  
Note that $\pi_{V_0}: E_0|_{V_0}\to V_0$ is independent of the choice of Teichm\"uller slice by Lemma~\ref{lemma: canonical identification}.  Next we shrink $V_0$ subject to $V_0\supset K_0$ so that $\pi_{V_0}: E_0|_{V_0}\to V_0$ is invariant under $G$.  Let $V$ be the quotient of $V_0$ by $\R$-translations of the target and let $E$ be the quotient of $E_0$ by the induced $\R$-action, i.e., $((\F,u),\zeta)\sim' ((\F,u_T),\zeta_T)$, where $u_T$ is $u$ translated up by $s=T$ and $\zeta_T$ is $\zeta$ translated up by $s=T$. 
The quotient exists by Lemma \ref{lemma: canonical identification 3}.  (I1), (I2), (I3), and (I5) follow immediately.


(I4) The map $\psi: \overline\bdry _J^{-1} (0)\to \mathcal{M}/\R$ induced by $(\mathcal{F},u)\mapsto (\mathcal{F},u)$ is a homeomorphism onto the image in view of the discussion of the topology of $\mathcal{M}$ from Section~\ref{subsubsection: topology}.

(I6) By construction $\K$ is invariant under the finite group $G$.  The action of $G$ extends to an action on $\pi_{\V}: \E\to \V$ by bundle automorphisms by Lemma~\ref{lemma: canonical identification 2}.  The quotient by $G$ gives a global quotient orbibundle after verifying a couple of facts.

(a) If $\ell_0$ is sufficiently large, then the action of $G$ on $V$ is effective: Let $g\not= e\in G$ and suppose $g(\F,u)=(\F,u)$. 
If $\ell_0$ is sufficiently large, then there exists $(\F,u')\in V$ that is close to $(\F,u)$ and such that $\overline\bdry_J(\F,u')=\zeta$ and $\zeta$ is not fixed under $g$: Start with $\ell_1$ sufficiently large so that $E^{\ell_1}_{(\F,u)}$ is $D_u$-transverse. Then there exists $\zeta_0\in E^{\ell_1+1}_{(\F,u)}$ with nontrivial $\widetilde{f}^{\gamma_{\pm,i}}_{\ell_1+1}$ component for which there exists $\xi_0$ satisfying $D_u \xi_0=\zeta_0$. When $\widetilde{f}^{\gamma_{\pm,i}}_{\ell_1+1}$ is not invariant under $g$ (it is easy to find such an $\ell_1$ and $\widetilde{f}^{\gamma_{\pm,i}}_{\ell_1+1}$), $\xi_0$ is also not invariant under $g$.  Passing from infinitesimal to local, there exists $(\F,u')$ near $(\F,u)$ that is not invariant under $g$ when $\ell_0\geq \ell_1+1$.

(a) in turn implies that the action of $G$ on $E$ is effective.  

(b) $V$ is Hausdorff: This is a consequence of the standard fact that $V= V_0/G$, $V_0$ is Hausdorff, and $G$ is a finite group.
\end{proof}

\subsubsection{Construction of ${\mathfrak s}$} \label{subsubsection: obstruction multisection} 

Let $s$ be a generic smooth section of $E\to V$ which is transverse to $\overline\bdry_J$ and is a perturbation of the zero section.  Taking the collection $(g(s))_{g\in G}$ where we assign the weight $1/|G|$ to each branch, we obtain a smooth {\coblu \em lifted} multisection ${\mathfrak s}$ of the global quotient orbifold for $\pi_{\V}: \E\to \V$ which is a perturbation of the zero (multi-)section and such that ${\mathfrak s}\pitchfork \overline\bdry_J$.  
\cb

\subsection{Dependence on $\varepsilon$ and $\ell$} \label{subsection: stabilization}

In this subsection we briefly discuss the effect of increasing $\ell>0$ and shrinking $\varepsilon>0$.

\subsubsection{Stabilizations} \label{subsubsection: stabilizations}

\begin{defn}
Let $E\to \mathcal{N}(K)$ be a $\overline\bdry_J$-transverse subbundle of $\mathcal{E}|_{\mathcal{N}(K)}$. A subbundle $E'\to \mathcal{N}(K)$ of $\mathcal{E}|_{\mathcal{N}(K)}$ such that $E\subset E'$ is a   {\em rank $a$  stabilization of $E$} if for some $a\in \Z^+$ there is a vector bundle isomorphism $E\oplus \R^a\xrightarrow\sim E'$ that takes $(x,0)\mapsto x$. \cb
\end{defn}

An example of a stabilization of $E^{\ell,\varepsilon}\to \mathcal{N}(K)$ is $E^{\ell+1,\varepsilon}\to \mathcal{N}(K)$.

\begin{defn} \label{defn: stab pair}  
The pair $(\pi_{V'}:E'\to V',\overline\bdry_J')$ consisting of a finite rank vector bundle $\pi_{V'}$ over a finite-dimensional manifold and a section $\overline\bdry_J'$ of $\pi_{V'}$ is a {\em rank $a$ stabilization of $(\pi_V:E\to V,\overline\bdry_J)$} if 
\begin{enumerate}
\item there is a splitting $E'\simeq E''\oplus \R^a $ and a bundle map 
\begin{equation} \label{equation: diagram}
\begin{diagram}
E & \rInto^{\tilde \jmath=(j,0)} & E'\simeq E''\oplus \R^a\\
\dTo^{\pi_{V}} &  & \dTo_{\pi_{V'}} \\
V & \rInto^{i} & V'
\end{diagram}
\end{equation}
where $i:V\hookrightarrow V'$ is an inclusion and $j: E\to E''|_{i(V)}$ is an isomorphism of vector bundles; and
\item $i:V\hookrightarrow V'$ extends to a diffeomorphism
\begin{equation}\label{eqn: 2222022} i: V\times {\coblu B^a}\xrightarrow\sim V'\end{equation}
so that $\overline{\bdry}'_J(i(v,x))=   (j(\overline{\bdry}_J(v)),x) \cb$, where \coblu $x$ is viewed both as an element of an open Euclidean ball $B^a=\{x\in \R^a~|~ |x|<1\}$ and as an element of $\R^a$. \cb
\end{enumerate}
\end{defn}

Given a stabilization $E'$ of $E\to\mathcal{N}(K)$, by setting $V=\overline\bdry_J^{-1}(E)$ and $V'=\overline\bdry_J^{-1}(E')$, we obtain a stabilization $(\pi_{V'}:E'|_{V'} \to V', \overline\bdry_J)$ of $(\pi_{V}:E|_{V} \to V, \overline\bdry_J)$. Note that Definition~\ref{defn: stab pair}(2) holds after shrinking $V'$ subject to $V'\supset K$ if necessary.

We now consider triples $(\pi_{\V}: \E\to \V, \overline\bdry_J,G)$ consisting of a global quotient orbibundle $(\pi_{\V}: \E\to \V,G)$ and a section $\overline\bdry_J'$ of the orbibundle. 

\begin{defn}
\coblu The triple $(\pi_{\V'}: \E'\to \V', \overline\bdry_J',G)$ is a {\em rank $a$ stabilization of $(\pi_{\V}: \E\to \V,\overline\bdry_J,G)$}
if 
\be
\item $(\pi_{V'}:E'\to V',\overline\bdry_J')$ is a rank $a$ stabilization of  $(\pi_V:E\to V,\overline\bdry_J)$;
\item all the maps in \eqref{equation: diagram} and \eqref{eqn: 2222022} are $G$-equivariant;
\item the splitting $E'\simeq E''\oplus \R^a $ from Definition~\ref{defn: stab pair}(1) is not necessarily a direct sum decomposition of $G$-bundles; and
\item given $v\in \V$ and $g\in G$, $g$ induces maps
\begin{gather*}
\{v\}\times B^a\to \{g(v)\}\times B^a, \quad (v,x)\mapsto (g(v),L^1_{g,v}(x)),\\
E''_v\oplus \R^a\to E''_{g(v)}\oplus \R^a, \quad (y,x)\mapsto (L^0_{g,v}(y), L^1_{g,v}(x)),
\end{gather*} 
where $L^0_{g,v}$ is a linear map and $L^1_{g,v}$ is the same linear map on $B^a$ and $\R^a$.
Here $E''_v$ is the fiber of $E''$ over $v$ and $E''_v\oplus \R^a$ is the fiber of $E''\oplus \R^a$ over $(v,pt)\in V\times B^a$.
\ee
\end{defn}

Given a stabilization $(\pi_{V'}:E'\to V',\overline\bdry_J',G)$ of $(\pi_V:E\to V,\overline\bdry_J,G)$ and a section $s: V\to E$, we construct the {\em stabilization $s': V'\to E'$ of $s$} as follows: 
{\coblu Let $h:[0,\infty)\to [0,1]$ be a smooth function such that $h(y)=1$ on $0\leq y\leq \frac{1}{3}$ and $h(y)=0$ on $y\geq \frac{2}{3}$.}
We then define 
\begin{equation} \label{eqn: stabilization}
{\coblu s'(i(v,x))= (j(h(|x|)s(v)),0).}
\end{equation}

\begin{defn}
The pair $(E',s')$ consisting of a stabilization $E'$ of $E$ and a stabilization $s'$ of $s$ is called a {\em stabilization of $(E,s)$.}
\end{defn}

\begin{lemma} \label{lemma: peaches}
If $(E',s')$ is a stabilization of $(E,s)$, then $i(\overline\bdry^{-1}_J(s))=(\overline\bdry'_J)^{-1}(s')$.
\end{lemma}

\begin{proof}
This is immediate from the definition of $(E',s')$:   If $i(v,x)\in (\overline\bdry'_J)^{-1}(s')$, then
\coblu 
$$\overline\bdry'_J (i(v,x)) = s' (i(v,x)),$$
which is equivalent to 
$$(j(\overline\bdry_J(v)),x)=(j(h(|x|)s(v)),0).$$ 
Hence $x=0$, $h(|x|)=1$, and $\overline\bdry_J(v)=s(v)$. \cb 
\end{proof}
 
Analogously, if ${\mathfrak s}$ is a $\overline\bdry_J$-transverse {\coblu lifted} multisection of {\coblu $(\pi_{\V}: \E\to \V,\overline\bdry_J,G)$,} then we can construct the {\em stabilization ${\mathfrak s}'$ of ${\mathfrak s}$} which is a {\coblu lifted} multisection of the stabilization {\coblu $(\pi_{\V'}: \E'\to \V',\overline\bdry_J',G)$} such that ${\mathfrak s}'$ is $\overline\bdry_J'$-transverse and $i(\overline\bdry^{-1}_J({\mathfrak s}))=(\overline\bdry_J')^{-1}({\mathfrak s}')$ as weighted branched manifolds. \coblu This is straightforward since the orbibundles are global quotient orbibundles. \cb

\subsubsection{Changing $\varepsilon$}

The effect of changing $\varepsilon>0$ can be summarized as follows:

\begin{rmk}[Changing $\varepsilon>0$]\label{rmk: shrinking varepsilon}
There exist $\varepsilon_0>0$ and $\ell_0>0$ such that Theorem~\ref{thm: construction of L-transverse subbundle} holds for all $\varepsilon>0$ satisfying $\varepsilon_0\geq\varepsilon$ and $\ell$ satisfying $\ell\geq \ell_0$.  If we fix $\ell$, then for any two $0<\varepsilon_1,\varepsilon_2<\varepsilon_0$, there exists a $1$-parameter family of $\overline\bdry_J$-transverse bundles $E^{\ell,\varepsilon_t}\to \mathcal{N}(K)$, $t\in[1,2]$.  In the process of extracting weighted branched manifolds in later sections, it follows that the weighted branched manifolds corresponding to $\varepsilon_1$ and $\varepsilon_2$ are cobordant.
\end{rmk}

\section{Gluing} \label{section: gluing}

The goal of this section is to define the terms that appear in the gluing theorems and also state the gluing theorems that we need.

Suppose we are considering $\op{(Symp)}$ case; in this case we need to quotient by the $\R$-translations in the target. A key observation is that $E_{(\F,u)}$ is equivariant under $\R$-translations by the construction in Theorem~\ref{thm: construction of L-transverse subbundle}.  The necessary modifications can be made for the $\op{(Cob)}$ case.

\subsection{The setup}

Let
$$\mathcal{M}=\mathcal{M}^{\op{ind}=k}_J(\dot F,\R\times M;\bs\gamma_+;\bs\gamma_-).$$ 

\nom[Mbar]{$\overline{\mathcal{M}/\R}$}{Not quite the compactification of the moduli space $\mathcal{M}/\R$; see Convention~\ref{convention for boundary}}

\begin{convention} \label{convention for boundary}
We use the notation $\overline{\mathcal{M}/\R}$ to mean the space of SFT buildings in $\R\times M$ of total Fredholm index $\op{ind}=k$ which limit to $\bs\gamma_+$ at the positive end and $\bs\gamma_-$ at the negative end, and whose domain, when glued up, has topological type $\dot F$.  We also write
$$\bdry (\mathcal{M}/\R)= \overline{\mathcal{M}/\R}- \mathcal{M}/\R.$$
Note that $\bdry (\mathcal{M}/\R)$ is not necessarily the set-theoretic boundary of $\mathcal{M}/\R$: for example it is possible that $\mathcal{M}/\R=\varnothing$ and $\bdry (\mathcal{M}/\R)\not=\varnothing$.
\end{convention}

We will explain   a slightly simplified situation; the general case is only notationally more involved.

\begin{assumption} \label{assumption: simple}
There exist
$$\mathcal{M}_1=\mathcal{M}^{\op{ind}=k_1}_J(\dot F_1,\R\times M;\bs\gamma;\bs\gamma_-),$$
$$\mathcal{M}_2=\mathcal{M}^{\op{ind}=k_2}_J(\dot F_2,\R\times M;\bs\gamma_+;\bs\gamma),$$
with $k_1+k_2=k$ such that $\dot F_1$ and $\dot F_2$ are connected and  
$$\bdry (\mathcal{M}/\R)  \supset \mathcal{M}_1/\R\times \mathcal{M}_2/\R,$$
where $\mathcal{M}_2/\R$ has higher SFT level (i.e., has larger $s$ coordinate) than $\mathcal{M}_1/\R$.
Also we choose compact subsets \coblu $K_1\subset \mathcal{M}_1/\R$ and $K_2\subset \mathcal{M}_2/\R$. \cb
\end{assumption}

For $i=1,2$, $(F_i,{\bf p}_i,{\bf q}_i)$ and hence $\ddot F_i$ are fixed. Let $\varepsilon_i'=\varepsilon'(K_i)>0$ be the constant that appears in the definition of $s'_+(\F_i,u_i)$ in Section~\ref{subsubsection: nonnegative Euler char}, in case $\chi(\dot F)=0$ or $1$. We take a Teichm\"uller slice $\widetilde{\mathcal{U}}_i$ such that each $[\F_i,u_i]\in K_i$ admits a representative $(\F_i,u_i)=((F_i,j_i,{\bf p}_i,{\bf q}_i,{\bf r}_i),u_i)$ with $j_i\in \widetilde{\mathcal{U}}_i$. Define
$$\overline\bdry^{\widetilde{\mathcal{U}}_i}_J: \mathcal{E}_{\widetilde{\mathcal{U}}_i}\to \mathcal{B}_{\widetilde{\mathcal{U}_i}}$$
as before.  Let $Z_i=(\overline\bdry_J^{\widetilde{\mathcal{U}}_i})^{-1}(0)$, \coblu view $K_i$ as a compact subset of $Z_i$, \cb and let $\mathcal{N}(K_i)$ be a sufficiently small neighborhood of $K_i$ in $\mathcal{B}_{\widetilde{\mathcal{U}_i}}$.

We are also assuming that
\begin{equation} \label{eqn: puncture-preserving}
    {\bf p}=({\bf p}_1\cup {\bf p}_2)-({\bf p}_{+,1}\cup{\bf p}_{-,2}),\quad {\bf r}=({\bf r}_1\cup {\bf r}_2)-({\bf r}_{+,1}\cup{\bf r}_{-,2}),
\end{equation}
and $F=F_1\# F_2$, where $\#$ is a multiple connected sum and the gluing occurs along subsets ${\bf p}_{+,1}\subset {\bf p}_1$, ${\bf p}_{-,2}\subset {\bf p}_2$ that correspond to ${\bf r}_{+,1}\subset {\bf r}_1$, ${\bf r}_{-,2}\subset {\bf r}_2$.

We apply Theorem~\ref{thm: construction of L-transverse subbundle} and Section~\ref{subsubsection: defn of pi sub V} to $K_i\subset Z_i$ to obtain
$$E_i=E_i^{\ell,\varepsilon}\to V_i/\R,$$
where $V_i=(\overline\bdry^{\widetilde{\mathcal{U}}_i}_J)^{-1}(E_i)\subset \mathcal{N}(K_i)$ as in Equation~\eqref{eqn: V and V prime}   and $V_i/\R$ is the quotient by the target $\R$-translation. By slight abuse of notation the total space of the obstruction bundle over $V_i/\R$ will also be denoted by $E_i$. We are also assuming that the same constants $\ell, \varepsilon$ apply to both $E_1$ and $E_2$. 

We fix a slice $V_i/\R\to V_i$ and write $(\F_i,u_i)\in V_i/\R$ to indicate the representative of its equivalence class in $V_i/\R$ in that slice. We also assume that
$$\min_j\{s_{-,j}(\F_2,u_2)\}\gg 0, \quad \max_j \{s_{+,j}(\F_1,u_1)\}\ll 0$$
for all $(\F_i,u_i)\in V_i/\R$.

\subsection{Close to breaking} \label{subsection: close to breaking}

Let $\delta>0$ be a small constant.\footnote{Unrelated to the weight $\delta$ used in the Sobolev spaces earlier.} 
Given a constant $T\in \R$, let $u_{i,T}$ (resp.\ $u_T$) be $u_i$ (resp. $u$) translated up by $s=T$ units.

We define the notion of {\em close to breaking} in the spirit of \cite[Definition 7.1]{HT2} as follows:  

\begin{defn}[Close to breaking] \label{defn: close to breaking}
A map $(\F,u)=((F,j,{\bf p},{\bf q},{\bf r}),u)$ is {\em $\delta$-close to breaking into the building}
$$(\F_1,u_1)\cup (\F_2,u_2),\quad (\F_i,u_i)=((F_i,j_i,{\bf p}_i,{\bf q}_i,{\bf r}_i),u_i)\in V_i/\R,\quad i=1,2,$$
if, after possibly translating $u$, there exist $T_2> T_2'>0$ and $T_1> T_1'>0$ such that the following hold:  
\begin{itemize}
    \item The domain $\dot F_2'$ of $u_{2,T_2}|_{\{s\geq T_2'\}}$ is obtained from $\dot F_2= F_2-{\bf p}_2$ by removing ends $C_{2,\gamma}$ for all $\gamma\in\bs\gamma$. 
    There is a biholomorphism
    $$(C_{2,\gamma},j_2)\simeq ((-\infty,T_2']\times \R/\mathcal{A}_\alpha(\gamma)\Z,j_{\op{std}})$$ 
    with coordinates $(s,t)$ and standard complex structure $j_{\op{std}}:\bdry_s\mapsto \bdry_t$, such that the asymptotic marker corresponds to $t=0$, $u_{2,T_2}|_{C_{2,\gamma}}$ is $\delta$-close in the $C^1$-norm to the map
    $$C_{2,\gamma}\to \R\times \R/  \mathcal{A}_\alpha(\gamma^s) \Z\times D\subset \R\times M,\quad (s,t)\mapsto (s,t,0),$$
    where we are using simple coordinates for $\gamma$, and the postcomposition of $u_{2,T_2}|_{C_{2,\gamma}}$ by the projection onto $\R\times \R/  \mathcal{A}_\alpha(\gamma^s) \Z$ is $(j_2,j_{\op{std}})$-holomorphic.
    \item The domain $\dot F_1'$ of $u_{1,-T_1}|_{\{s\leq -T_1'\}}$ satisfies an analogous condition with ends $C_{1,\gamma}$.
    \item {\coblu $T_2'+T_1'\geq 1/\delta$ and} the domain $\dot F'$ of $u|_{\dot F- \{-T_1'< s< T_2'\}}$ is obtained from $\dot F= F-{\bf p}$ by removing annuli $A_\gamma$ for all $\gamma\in\bs\gamma$. 
    There is a biholomorphism 
    $$(A_\gamma,j)\simeq ([-T_1',T_2'] \times \R/\mathcal{A}_\alpha(\gamma)\Z, j_{\op{std}})$$ 
    such that $u|_{A_\gamma}$ is $\delta$-close in the $C^1$-norm to the map 
    $$A_\gamma\to \R\times \R/  \mathcal{A}_\alpha(\gamma^s)\Z \times D\subset \R\times M,\quad (s,t)\mapsto (s,t,0),$$
    where we are using simple coordinates for $\gamma$, and the postcomposition of $u|_{A_\gamma}$ by the projection onto $\R\times \R/  \mathcal{A}_\alpha(\gamma^s) \Z$ is $(j,j_{\op{std}})$-holomorphic.
    \item ${\bf q}_1,{\bf q}_2$, and ${\bf q}$ are disjoint from $C_{1,\gamma},C_{2,\gamma}$, and $A_\gamma$, respectively.
    \item There is a diffeomorphism
    $$\phi:\dot F = F-{\bf p} \xrightarrow{\sim}\dot F'_1\cup(\sqcup_\gamma A_\gamma)\cup \dot F'_2$$ 
    which takes $\dot F'$ to $\dot F'_1\sqcup \dot F_2'$, takes ${\bf q}'_1$ to ${\bf q}_1$ and ${\bf q}'_2$ to ${\bf q}_2$ where ${\bf q}= {\bf q}'_1\sqcup {\bf q}'_2$, is the identity on $A_\gamma$, and identifies the corresponding ends. Here the right-hand side is obtained by gluing the corresponding boundary components of $A_\gamma$ and $C_{2,\gamma}$ (also $A_\gamma$ and $C_{1,\gamma}$) via their identifications with $\R/\mathcal{A}_\alpha(\gamma)\Z$.  Then $\phi_* j|_{\ddot F}$ is $\delta$-close in the Teichm\"uller metric to $j_1|_{\ddot F'_1}\cup j_{\op{std}}|_{\cup_\gamma A_\gamma} \cup j_2|_{\ddot F'_2}$ where $\ddot F'_i=\dot F'_i-{\bf q}_i$.
    \item $u_{2,T_2}|_{\dot F'_2}$ is $\delta$-close in the $C^1$-norm to $u|_{\dot F'_2}$ and $u_{1,-T_1}|_{\dot F'_1}$ is $\delta$-close in the $C^1$-norm to $u|_{\dot F'_1}$ via the identification $\phi$.
\end{itemize}
We also say that ${\bf q}$ is {\em $\delta$-close to ${\bf q}_1\cup {\bf q}_2$.}
\end{defn}

\nom[G1]{$\mathcal{G}_\delta(V_1/\R,\dots, V_m/\R)$}{Holomorphic curves $\delta$-close to breaking into curves in $V_1,\dots, V_m$}
\nom[G2]{$\widetilde{\mathcal{G}}_\delta(V_1/\R,\dots, V_m/\R)$}{Slice of $\mathcal{G}_\delta(V_1/\R,\dots, V_m/\R)$}
Let $\mathcal{G}_\delta(V_1/\R,V_2/\R)$ be the set of equivalence classes of maps
$$(\F,u)=((F,j,{\bf p},{\bf q}, {\bf r}),u)$$
modulo $\R$-translations which have representatives that are $\delta$-close to breaking into some building $(\F_1,u_1)\cup (\F_2,u_2)$ with $(\F_i,u_i)\in V_i/\R$ and let $\widetilde{\mathcal{G}}_\delta(V_1/\R,V_2/\R)$ be a slice of $\mathcal{G}_\delta(V_1/\R,V_2/\R)$ whose elements are maps which are $\delta$-close to breaking.

\subsection{Removable punctures} \label{subsection: removable punctures} 

Let us write $\F^\varnothing=(F,j,{\bf p},\varnothing,{\bf r})$ when $\F=(F,j,{\bf p},{\bf q},{\bf r})$.
\nom[Fcal]{$\F^\varnothing$}{$\F^\varnothing=(F,j,{\bf p},\varnothing,{\bf r})$ when $\F=(F,j,{\bf p},{\bf q},{\bf r})$}
Given $(\F_i,u_i)=((F_i,j_i,{\bf p}_i, {\bf q}_i, {\bf r}_i),u_i)$, $i=1,2$, and $(\F^\varnothing,u)$ which is $\delta$-close to $(\F_1^\varnothing,u_1) \cup (\F_2^\varnothing,u_2)$, we explain how to select removable punctures ${\bf q}={\bf q}'_1\sqcup {\bf q}'_2$ for $F$ following Section~\ref{subsubsection: nonnegative Euler char}, so that $(\F,u)$ is $\delta$-close to $(\F_1,u_1) \cup (\F_2,u_2)$ for possibly larger $\delta$: \cb
\coblu Let $\gamma\in\bs\gamma$ and $\gamma_\pm\in \bs\gamma_\pm$. Denote the ends and the neck regions of $u$ by $u(s,t)=(s,t,\eta(s,t))$ using simple coordinates such that $s\in [T_2'',\infty)$, $s\in (-\infty, -T_1'']$, and $s\in[-T_1',T_2']$. If $\chi(\dot F_i)<0$, then ${\bf q}_i'=\varnothing$, where $i=1,2$.  If $\chi(\dot F_1)\geq 0$, \coblu then as in Lemma~\ref{tenren} we define
\begin{align} \label{eqn: s prime 1}
	s'_{\gamma,+}(\F,u) & :=  \max \{ s\in [-T_1',T_2']~|~ I_\gamma((\F,u),s')=\mathcal{A}_\alpha(\gamma)-\varepsilon'\},\\
\nonumber	s'_{\gamma_-,-}(\F,u) & :=  \min \{ s\in (-\infty,-T_1'']~|~ I_{\gamma_-}((\F,u),s')=\mathcal{A}_\alpha(\gamma_-)+\varepsilon'\}.
\end{align}
If $\chi(\dot F_2)\geq 0$, then
\begin{align} \label{eqn: s prime 2}
    s'_{\gamma_+,+}(\F,u) & :=  \max \{ s\in[T_2'',\infty) ~|~ I_{\gamma_+}((\F,u),s')=\mathcal{A}_\alpha(\gamma_+)-\varepsilon'\},\\
\nonumber 	s'_{\gamma,-}(\F,u) & :=  \min \{ s\in [-T_1',T_2']~|~ I_\gamma((\F,u),s')=\mathcal{A}_\alpha(\gamma)+\varepsilon'\}.
\end{align}
Then ${\bf q}'_2$ (resp.\  ${\bf q}'_1$) is a list of punctures defined using $s'_{\gamma_+,+}(\F,u)$ and $s'_{\gamma,-}(\F,u)$ (resp.\ $s'_{\gamma,+}(\F,u)$ and $s'_{\gamma_-,-}(\F,u)$)) and \ref{Q} in Section~\ref{subsubsection: map I gamma}. \cb
Hence for $(\F,u)$ close to breaking into $(\F_1,u_1)\cup (\F_2,u_2)$,   the data $(\F_1,u_1)\cup (\F_2,u_2)$ and $(\F^\varnothing,u)$ are sufficient to recover ${\bf q}$, i.e., there exists a {\em puncturing map} 
\begin{equation} \label{eqn: puncturing map}
   \mathfrak{Punc}: (\F^\varnothing,u) \mapsto (\F,u)
\end{equation}
such that ${\bf q}$ satisfies Equations~\eqref{eqn: s prime 1} and \eqref{eqn: s prime 2}.
\cb

\subsection{The bundle $E'$} \label{subsection: the bundle E'}

\nom[E prime]{$E'_{(\F,u)}$}{Fiber of obstruction bundle over $\widetilde{\mathcal{G}}_\delta(V_1/\R,\dots, V_m/\R)$}

We define the bundle
\begin{equation}\label{eqn: bundle E prime two level}
E'\to \widetilde{\mathcal{G}}_\delta(V_1/\R,V_2/\R)
\end{equation}
for $\delta>0$ small as follows:  Given $(\F,u)=((F,j,{\bf p},{\bf q},{\bf r}),u)\in \widetilde{\mathcal{G}}_\delta(V_1/\R,V_2/\R)$, consider   the $\varepsilon$-thin part $\op{Thin}_\varepsilon(\ddot F,g)$ for $g=g_{(\ddot F,j)}$ as in Step 1 of the proof of Theorem~\ref{thm: construction of L-transverse subbundle}: it is a union of cusps (as in Equation~\eqref{equation: cusps}) and annuli.

\nom[s10]{$s_C(\F,u)$}{The $s$-coordinate indicating the location of the perturbations for $(\F,u)\in \widetilde{\mathcal{G}}_\delta(V_1/\R,\dots, V_m/\R)$, where $C$ is a boundary component of $\op{Thin}_\varepsilon(\ddot F,g)$}
Let $C$ be a boundary component of $\op{Thin}_\varepsilon(\ddot F,g)$.  
We then apply Definition~\ref{defn: somewhat informal} to obtain the average $s_C(\F,u)$ of the $s$-coordinates of $u|_C$ over 
$$t=\mathcal{A}_\alpha(\gamma^s), 2\mathcal{A}_\alpha(\gamma^s),\dots, m(\gamma)\mathcal{A}_\alpha(\gamma^s),$$ 
where $\gamma$ is the orbit corresponding to $C$, $\gamma^s$ is the simple orbit underlying $\gamma$, and $m(\gamma)$ is the multiplicity of $\gamma$ over $\gamma^s$. Note that $s_C(\F,u)$ agrees with $s_{\pm,i}(\F,u)$ from the proof of Theorem~\ref{thm: construction of L-transverse subbundle} when $C$ is the boundary of the cusp corresponding to the $i$th $\pm$ end of $(\F,u)$.  We then define $E^{\ell,\varepsilon}_{(\F,u),C}$ as in Equation~\eqref{equation: def of each component of E}, using $s_C(\F,u)$.
We then set \cb
\begin{equation} \label{equation: defn of E prime}
E'_{(\F,u)}:=\bigoplus_C E^{\ell,\varepsilon}_{(\F,u),C},
\end{equation}
where $C$ ranges over all boundary components $C$ of $\op{Thin}_\varepsilon(\ddot F,g)$ corresponding to cusps or glued edges. \cb

\subsection{Statements of gluing theorems} \label{subsection: statements of gluing theorems}

In this subsection we collect the gluing theorems that we need to construct the semi-global Kuranishi structures.  They are standard gluing results and will be proven in the next section.

In general, the $2$-component building condition in Assumption~\ref{assumption: simple} is replaced by an analogous $m$-component condition:   We have moduli spaces $\mathcal{M}$ and $\mathcal{M}_i$, $i=1,\dots, m$,  such that
$$\bdry (\mathcal{M}/\R) \supset \mathcal{M}_1/\R \times \dots \times \mathcal{M}_m/\R,$$
and compact subsets $K_i$ of $\mathcal{M}_i/\R$. Let $Z_i=(\overline\bdry_J^{\widetilde{\mathcal{U}}_i})^{-1}(0)$, \coblu view $K_i$ as a compact subset of $Z_i$, \cb and let $E_i|_{V_i}\to V_i/\R$ be the bundle from Section~\ref{subsubsection: defn of pi sub V} with $V_i$ a small neighborhood of $K_i$. We are gluing $m$ non-trivial (= not trivial cylinder) components, one from each $V_i$, with a prescribed identification of a subset of ends. The notion of a map $(\F,u)$ that is $\delta$-close to breaking into the building
$$(\F_1,u_1)\cup \dots \cup (\F_m,u_m)$$ 
with $(\F_i,u_i)\in V_i/\R$ is defined as in the $2$-component case. Hence we can define $\mathcal{G}_\delta(V_1/\R,\dots, V_m/\R)$ as the set of equivalence classes of maps with representatives that are $\delta$-close to $(\F_1,u_1)\cup \dots \cup (\F_m,u_m)$ and $\widetilde{\mathcal{G}}_\delta(V_1/\R,\dots, V_m/\R)$ as a slice of $\mathcal{G}_\delta(V_1/\R,\dots, V_m/\R)$.  \cb Then we can define
\begin{equation}\label{eqn: bundle E prime multiple level}
E'\to \widetilde{\mathcal{G}}_\delta(V_1/\R,\dots, V_m/\R),
\end{equation}
using Equation~\eqref{equation: defn of E prime} (the definitions of $s_C(\F,u)$ and $E^{\ell,\varepsilon}_{(\F,u),C}$ carry over with no change) and
\begin{equation} \label{eqn: bundle E prime multiple level two}
    \widetilde{\mathcal{G}}^{E'}_\delta(V_1/\R,\dots,V_m/\R):= \overline\bdry_J^{-1}(E')\subset \widetilde{\mathcal{G}}_\delta(V_1/\R,\dots,V_m/\R).
\end{equation}
We also introduce the notation $V_i^{\text{sh}}$ (resp.\ $V_i^{\text{en}}$) to denote a slight $\R$-invariant shrinking (resp.\ enlargement) of $V_i$ which contains $K_i$ and is still open.

\begin{thm}[Gluing] \label{thm: gluing} 
$ $

\begin{enumerate}
    \item [{(A)}] For a sufficiently large   gluing parameter bound $R=R(K_1,\dots, K_m)>0$, \cb there exists a gluing map
\begin{equation}
G_{(1,\dots,m)}:(V_1/\R) \times\dots \times (V_m/\R) \times [R,\infty)^{m-1}\to \widetilde{\mathcal{G}}^{E'}_\delta(V_1/\R,\dots, V_m/\R)
\end{equation}
which is a $C^1$-diffeomorphism onto its image and whose image contains
$$\widetilde{\mathcal{G}}^{E'}_{\delta'}(V_1^{\text{\em sh}}/\R,\dots,V_m^{\text{\em sh}}/\R)$$
for sufficiently small $\delta'>0$.
\coblue
\item[{(B)}] The gluing map $G_{(1,\dots,m)}$ does not depend on the choices of $K_1,\dots,K_m$.
More precisely, for another choice of compact subsets $K_i'\supset K_i$ and neighborhoods $V'_i \supset K'_i$ such that $V'_i \supset V_i$ for $i\in\{1,\dots, m\}$, there exist a gluing parameter bound $R'=R(K'_1,\dots, K'_m) > R = R(K_1,\dots, K_m)>0$ and a gluing map 
$$G'_{(1,\dots,m)}:(V'_1/\R) \times\dots \times (V'_m/\R) \times [R',\infty)^{m-1}\to \widetilde{\mathcal{G}}^{E'}_\delta(V'_1/\R,\dots, V'_m/\R).$$
If we use the same $\ell,\varepsilon$, then $G'_{(1,\dots,m)}= G_{(1,\dots,m)}$ on their common domain 
$$(V_1/\R) \times\dots \times (V_m/\R) \times [R',\infty)^{m-1}.$$
\end{enumerate}
\end{thm}

\begin{rmk} \label{rmk: need to use same ell and varepsilon}
\coblue In Theorem~\ref{thm: gluing}(B), it is important that we use the same $\ell,\varepsilon$ to define $V_i$ and $V_i'$. The use of the orthogonal projection (Equation~\eqref{eqn: orthogonal projection}) in Equation~\eqref{def of I}, which depends on $\ell,\varepsilon$, is the source of slight discrepancies when applying the contraction mapping theorem.
\end{rmk}

Assuming that for $1\leq i < j \leq m$, $V_i,\dots, V_j$ can be glued into a one-component curve and writing $V_{(i,\dots,j)}/\R=\widetilde{\mathcal{G}}^{E'_{(i,\dots,j)}}_{\delta'}(V_i^{\text{sh}}/\R,\dots,V_j^{\text{sh}}/\R)$, where $E'_{(i,\dots,j)}$ is $E'$ for $V_i/\R,\dots$, $V_j/\R$, we also have the following:

\begin{thm}[Iterated gluing] \label{thm: iterated gluing}
For a sufficiently large   gluing parameter bound $R=R(K_1,\dots, K_m)>0$, \cb there exists a gluing map
\begin{align*}
&G_{(1,\dots, (i,\dots,j),\dots m)}:  (V_1/\R)\times\dots\times (V_{i-1}/\R)\times (V_{(i,\dots,j)}/\R) \\
& \times (V_{j+1}/\R)\times  \dots\times (V_m/\R) \times [R,\infty)^{m+i-j-1} \to \widetilde{\mathcal{G}}^{E'_{(1,\dots,m)}}_\delta(V_1/\R,\dots, V_m/\R),
\end{align*}
such that $G_{(1,\dots, (i,\dots,j),\dots m)}\circ (\op{id},\dots, G_{(i,\dots,j)},\dots, \op{id})$ is $C^1$-close to the gluing map $G_{(1,\dots,m)}$   
and the error goes to $0$ as $\min_{i=1}^{m-1}(T_i)\to \infty$, where $(T_1,\dots, T_{m-1})\in [R,\infty)^{m-1}$.  \coblu An analogous statement also holds for higher-order iterated gluing maps.
 
\end{thm}

\section{Details of gluing}\label{section: details of gluing}

We prove Theorem~\ref{thm: gluing} with $m=2$, following the general outline of \cite{HT2} and \cite{BH}.  The case for general $m$ only differs in notation.  The proof of Theorem~\ref{thm: iterated gluing} also uses the same type of estimates as in Theorem~\ref{thm: gluing} and we only provide a sketch.

We are assuming that $\varepsilon>0$ is small, that $(\alpha,J)$ is an $L$-simple pair (see Definition~\ref{defn: L-supersimple 1}), and that $J$ is smooth.

We write $E_+\to V_+$ for $E_2\to V_2$ and $E_-\to V_-$ for $E_1\to V_1$.  Let $(\F_\pm,u_\pm)_{\R} \in V_\pm/\R$, where $\F_\pm=(F_\pm,j_\pm,{\bf p}_\pm,{\bf r}_\pm)$. For simplicity we assume that $(\F_+,u_+)$ is a curve from $\gamma_+$ to $\bs\gamma=(\gamma_1,\dots,\gamma_k)$ and $(\F_-,u_-)$ is a curve from $\gamma_1$ to $\bs\gamma_-=(\gamma_{-,1},\dots,\gamma_{-,l})$, i.e., we are gluing the first negative end of $(\F_+,u_+)$ to the positive end of $(\F_-,u_-)$.  We fix representatives $(\F_\pm,u_\pm)$ of $(\F_\pm,u_\pm)_{\R}$ such that $s_{\mp,1}(\F_\pm,u_\pm)=0$. Here $s_{\pm,i}(\F,u)$ is as defined in the proof of Theorem~\ref{thm: construction of L-transverse subbundle}.  We write ``$(\F_\pm,u_\pm)\in V_\pm/\R$'' to mean $(\F_\pm,u_\pm)$ is the chosen representative in its equivalence class.  Note that the elements of $(E_\pm)_{(\F_\pm,u_\pm)}$ are supported on $-1\leq s\leq 0$ for $(\F_+,u_+)$ and on $0\leq s\leq 1$ for $(\F_-,u_-)$.

\begin{notation}
In general, if we decorate $(\mathcal{F},u)$ with subscripts and superscripts as in $(\mathcal{F}^*_\star,u^*_\star)$, then $\mathcal{F}^*_\star=(F^*_\star, j^*_\star, {\bf p}^*_\star,{\bf r}^*_\star)$ and $\dot F^*_\star=F^*_\star-{\bf p}^*_\star$.
\end{notation}

In the next several subsections we retrace the steps of Sections 8.1--8.4 of \cite{BH}, adapted to the current setting.

\subsection{Pregluing}\label{subsection:pregluing}

Fix a constant $T_0\gg 0$.  Also let $T\gg T_0$, which is allowed to vary.

\begin{notation}
If $\tau_T: \R\times M \xrightarrow{\sim} \R\times M$ is  given by the translation $(s,x)\mapsto (s+T,x)$, then let $u_{\pm,T}=\tau_{\pm (T+T_0)}\circ u_\pm$.  In general a subscript or superscript $+,T$ (resp.\ $-,T$) indicates the result of translating up (resp.\ down) by $s=T+T_0$.
\end{notation}

Recall the simple coordinates $(s,t,x=(x_1,\dots,x_n),y=(y_1,\dots,y_n))$ on a sufficiently small neighborhood $\R\times\R/\mathcal{A}_\alpha(\gamma_1)\Z \times D$ of $\R\times \gamma_1$; here we are slightly abusing notation and passing to a finite cover if $\gamma_1$ is not simple. We assume that $\varepsilon>0$ which appears in the definition of $E^{\ell,\varepsilon}$ is sufficiently small so that $u_-|_{s\geq -T_0}$ (resp.\ $u_+|_{s\leq T_0}$) can be written as a graph $(s,t,\eta_-(s,t))$ (resp.\ $(s,t,\eta_+(s,t))$) \cb over $[0,\infty)\times \gamma_1$ (resp.\ $(-\infty,0]\times\gamma_1$).

Fix constants $0<h<1$ and $r\gg h^{-1}$.  We take $T_0> 5r$.
Choose a cutoff function $\beta: \mathbb{R} \to [0,1]$ such that $\beta(s)=0$ for $s\leq 0$ and $\beta(s)=1$ for $s\geq 1$.  Let $\beta_{-,T}(s)=\beta(\frac{T-s}{ hr})$ and $\beta_{+,T}(s)=\beta(\frac{T+s}{hr})$.

For $T'\leq T+T_0$, let $\mathcal{E}_{+, T'}$ be the component of
$$u_{+,T}^{-1}(\{s< T'\})\simeq u_+^{-1}(\{s< T'-(T+T_0)\})$$
corresponding to $\gamma_1$, let $\mathcal{E}_{-,T'}$ be
$$u_{-,T}^{-1}(\{s> -T'\})\simeq u_-^{-1}(\{s> -T'+(T+T_0)\}),$$
and let $A_{[T_1,T_2]}$ be the annulus $[T_1,T_2]\times \R/\mathcal{A}_\alpha(\gamma_1)\Z$ with coordinates $(s,t)$ and the standard complex structure. 

\begin{figure}[ht]
\begin{center}
\psfragscanon
\psfrag{s}{\tiny $s$}
\psfrag{A}{\tiny $-T$}
\psfrag{B}{\tiny $-T+hr$}
\psfrag{C}{\tiny $T-hr$}
\psfrag{D}{\tiny $T$}
\psfrag{E}{\tiny $\beta_{+,T}$}
\psfrag{F}{\tiny $\beta_{-,T}$}
\psfrag{G}{\tiny $-T-T_0$}
\psfrag{H}{\tiny $T+T_0$}
\includegraphics[width=13cm]{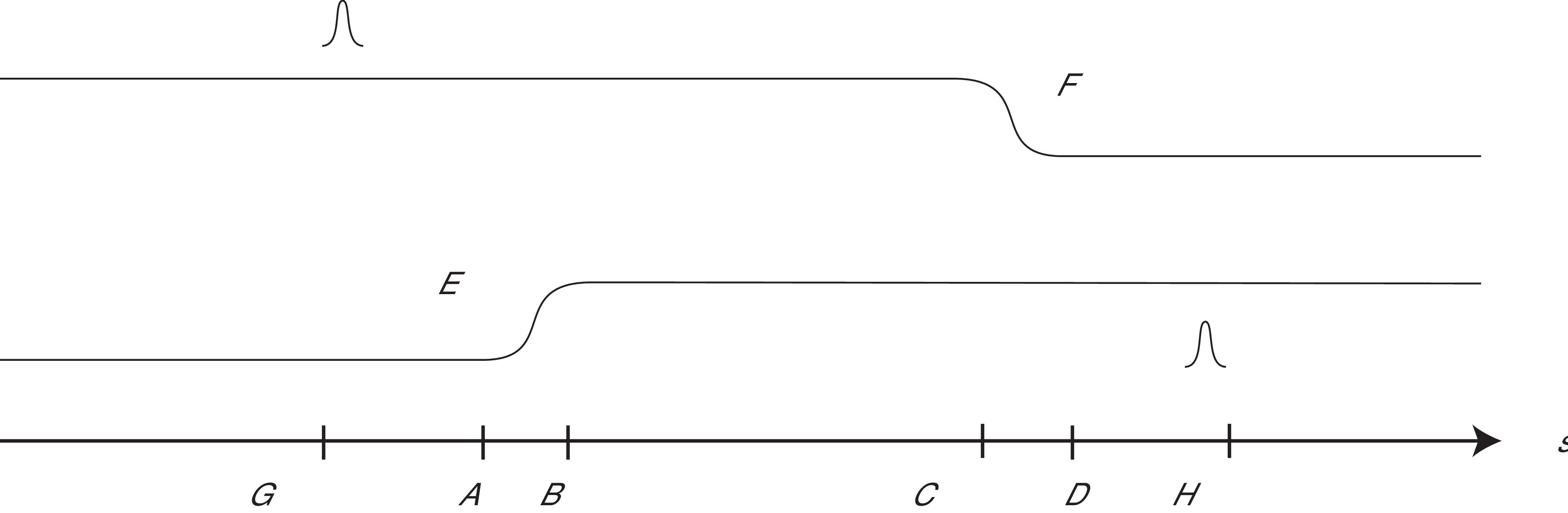}
\end{center}
\caption{The cutoff functions $\beta_{\pm,T}$.  The little bump functions indicate the support of $(E_\pm)_{(\F_\pm,u_{\pm,T})}$ } \label{fig: cutoff functions}
\end{figure}

\begin{defn}[Pregluing] \label{defn: pregluing} $\mbox{}$
\be
\item The {\em pregluing $u_*=u_*((\F_+,u_+)_{\R}, (\F_-,u_-)_{\R},T)$ of $(\F_+,u_+)_{\R} \in V_+/\R$ and $(\F_-,u_-)_{\R} \in V_-/\R$ with gluing parameter $T$} is defined as follows:\footnote{In the case of gluing $(\F_+,u_+)_{\R}$ and $(\F_-,u_-)_{\R}$, we have made a specific choice of $u^*$.  In general the pregluing is unique only up to $\R$-translation.} Let $F^\circ_+:=\dot F_+ -\mathcal{E}_{+,T}$ and $F^\circ_-:=\dot F_- -\mathcal{E}_{-,T}$.  The domain $(\dot F =F-{\bf p},j)$ is then obtained by gluing $(F^\circ_+,j_+|_{F^\circ_+})$, $(F^\circ_-,j_-|_{F^\circ_-})$, and $A_{[-T,T]}$ in the expected way. The map $u^*$ agrees with $u_{+,T}$ on $F^\circ_+$ and with $u_{-,T}$ on $F^\circ_-$, and is given by
\begin{equation}\label{eqn: pregluing}
u_*(s,t)= (s,t, \beta_{+,T}(s)\eta_{+,T}(s,t)+\beta_{-,T}(s)\eta_{-,T}(s,t) )
\end{equation}
on $A_{[-T,T]}$.

\item Given $(\F_i,u_i)\in V_i/\R$, $i=1,\dots,m$, with a prescribed identification of the ends and gluing parameters $T_1,\dots,T_{m-1}$, we can analogously define a pregluing $u^*((\F_1,u_1)_{\R},\dots,  (\F_m,u_m)_{\R},T_1,\dots,T_{m-1})$ which is unique up to $\R$-translation.
\ee
\end{defn}

\cb

\subsection{Gluing of domain complex structures} \label{subsection: domain complex structures}

We consider the following domain gluing, which is close to (but not quite the same as) the gluing of the domains in Section~\ref{subsection:pregluing}.  The reason for the slight discrepancy is that each of $F^\circ_+$ and $F^\circ_-$ is obtained from $\dot F_+$ and $\dot F_-$ by removing a cusp that is close to but not necessarily the same as a $\widetilde \varepsilon_\pm$-thin cusp for some $\widetilde \varepsilon_\pm$.

Fix $\dot F_-$, $\dot F_+$, and $\varepsilon''>0$ small. We are gluing $\dot F_-$ and $\dot F_+$ along one ``positive end'' of $\dot F_-$ and one ``negative end'' of $\dot F_+$. Let $\widetilde{\mathcal{U}}_\pm$ be a Teichm\"uller slice for a small open set $\mathcal{U}_\pm\subset \op{Teich}(\dot F_\pm)$ and let $g(j_\pm)$ be the complete finite-volume hyperbolic metric compatible with $j_\pm\in \widetilde{\mathcal{U}}_\pm$. We also fix $j^0_\pm\in \widetilde{\mathcal{U}}$. Let $\mathcal{E}_\pm^{\varepsilon''}(j_\pm)$ be the (interior of the) $\varepsilon''$-thin component of $g(j_\pm)$ corresponding to the end that is glued and let $F_\pm^\bullet(j_\pm):= \dot F_\pm - \mathcal{E}_\pm^{\varepsilon''}(j_\pm)$.  Also choose $r_\pm(j_\pm)\in \bdry F_\pm^\bullet(j_\pm)$ smoothly with respect to $j_\pm$.

Given $(T,\tau)\in [R,\infty)\times \R$, we glue
$$(F_-^\bullet(j_-),j_-) \cup ([0,T]\times \R/\Z,j_{std}) \cup (F_+^\bullet(j_+),j_+)$$
so that $r_-(j_-)$ is identified with $(0,0)\in [0,T]\times \R/\Z$ and $r_+(j_+)$ is identified with $(T,\tau)\in [0,T]\times \R/\Z$.
We denote the resulting Riemann surface by $(\dot F(j_-,j_+,T,\tau),j(j_-,j_+,T,\tau))$. We also write $\dot F=\dot F(j_-^0,j_+^0,R,0)$; this is our reference surface.  (Note that when $(T,\tau_1), (T,\tau_2)\in [R,\infty)\times \R$ and $\tau_2-\tau_1\in \Z$, the complex structures are diffeomorphic, but differ by a Dehn twist about a separating surface and correspond to distinct points in Teichm\"uller space.)

This defines a map
$$\kappa:\widetilde{\mathcal{U}}_-\times \widetilde{\mathcal{U}}_+\times[R,\infty)\times \R\to \op{Teich}(\dot F),$$
$$(j_-,j_+,T,\tau)\mapsto (\dot F(j_-,j_+,T,\tau),j(j_-,j_+,T,\tau)).$$

\begin{thm} \label{thm: Kazdan-Warner}
For $R\gg 0$, $\kappa$ is a $C^1$-embedding.
\end{thm}

\begin{proof}[Idea of proof.]
This follows from Kazdan-Warner~\cite{KW1,KW2}. For $R\gg 0$, we can down a metric with negative curvature on $\dot F(j_-,j_+,T,\tau)$ which is conformally equivalent to $j(j_-,j_+,T,\tau)$, agrees with $g(j_\pm)$ on $F_\pm^\bullet(j_\pm)$, and whose curvature is very close to $-1$ on $[0,T]\times \R/\Z$. In order to find a hyperbolic metric in the same conformal class we solve an equation of type $\Delta u = 1-he^u$, where $u:\dot F\to \R$, $\Delta$ is a Laplacian, and $h:\dot F\to \R$ is close to $1$ (and approaches $1$ as $R\to \infty$); see \cite[Section 2]{KW1}.  When $h$ is close to $1$, this can be solved using the inverse function theorem; see \cite{KW2}.  We remark that negative curvature implies that the linearized operator is invertible and also that solving the equation can be done in a differentiable family.

The above discussion implies that the Fenchel-Nielsen coordinates for $j_\pm\in \widetilde{\mathcal{U}}_\pm$ are very close to the corresponding Fenchel-Nielsen coordinates for $\kappa(j_-,j_+,T,\tau)\in \op{Teich}(\dot F)$ when $R\gg 0$, implying the theorem.
\end{proof}

\subsection{Banach spaces} \label{subsection: banach spaces}

The function spaces that we use are {\em Morrey spaces}, following \cite[Section 5.5]{HT2}.\footnote{This is rather nonstandard and we chose to adopt it to avoid redoing some work in \cite{HT2} for $W^{k,p}$-spaces.} Let $u:\dot F\to \R\times M$ be a finite energy holomorphic curve. On $\dot F$ we choose a Riemannian metric so that the ends are cylindrical and  we use the $\R$-invariant Riemannian metric on $\R\times M$.

The {\em Morrey space} $\mathcal{H}_0(\dot F, \wedge^{0,1} T^*\dot F\otimes  u^*T(\R\times M))$ is the Banach space which is the completion of the compactly supported sections of $\wedge^{0,1} T^*\dot F\otimes u^*T(\R\times M)$ with respect to the norm
$$\| \xi\|= \left( \int_{\dot F} |\xi|^2  \right)^{1/2} + \left( \sup_{x\in \dot F} \sup_{\rho\in(0,1]} \rho^{-1/2} \int_{B_\rho(x)} |\xi|^2  \right)^{1/2},$$
where $B_\rho(x)\subset \dot F$ is the ball of radius $\rho$ about $x$.  Similarly, $\mathcal{H}_1(\dot F, u^*T(\R\times M))$ is the completion of the compactly supported sections of $u^*T(\R\times M)$ with respect to $$\| \xi \|_*= \|\nabla \xi\| + \|\xi\|.$$

The analog of the usual Sobolev embedding theorem is the following:

\begin{lemma} \label{lemma: sobolev embedding}
There is a bounded linear map
$$\mathcal{H}_1(\dot F, u^*T(\R\times M))\to C^{0,1/4}(\dot F, u^*T(\R\times M))\cap L^\infty(\dot F, u^*T(\R\times M)), \quad \xi\mapsto \xi,$$
where $C^{0,1/4}$ denotes the space of H\"older continuous functions with exponent $\frac{1}{4}$.
\end{lemma}

We also define the weighted Morrey spaces $\mathcal{H}_{i,\delta}(*)=\mathcal{H}_{0,\delta}(\dot F_\pm, \wedge^{0,1} T^*\dot F_\pm\otimes  (u_\pm)^*T(\R\times M))$ and $\mathcal{H}_{1,\delta}(\dot F_\pm, (u_\pm)^*T(\R\times M))$ with weight $\delta>0$ as the space of sections $\psi$ such that $\psi\cdot f_\delta^\pm\in \mathcal{H}_i (*)$. Here $f_\delta^\pm: \R\times M\to \R^+$ is the weight function for $u_\pm$ which agrees with $e^{\delta |s|}$ at the ends and is normalized so that $f_\delta^\pm(s,x)= e^{\delta|s|}$ for $\mp s\geq 0$. Their norms are denoted by $\|\cdot \|_\delta$ and $\|\cdot\|_{*,\delta}$.  We will also write $\|\cdot\|_\delta^{\pm,T}$ and $\|\cdot\|_{*,\delta}^{\pm,T}$ with respect to the translates $f_\delta^{\pm,T}$.

\subsection{The setup}\label{subsection: gluing}

Let $\psi_{\pm}$ be sections of $u_{\pm}^*T(\R\times M)$ and let $\psi_{\pm,T}=\tau_{\pm (T+T_0)}\circ \psi_\pm$ be sections of $u^*_{\pm,T} T(\R\times M)$.  The exponential maps $\exp_{u_{\pm,T}}$ can be chosen such that
$$\exp_{u_{\pm,T}}\psi_{\pm,T}= (s,t,\eta_{\pm,T})+\psi_{\pm,T}$$
on $\mathcal{E}_{\pm,T+T_0}$, by taking the Riemannian metric $g_0$ on $\R\times M$ to be $\R$-invariant and equal to the standard flat metric on $\R\times\R/\mathcal{A}_\alpha(\gamma_1)\Z \times D$.

We then deform $u_*$ to
\begin{equation}\label{equation: deform u*}
u=\exp_{u_*} (\beta_{+,T}\psi_{+,T}+\beta_{-,T}\psi_{-,T}),
\end{equation}
where $\exp_{u_*}$ agrees with $\exp_{u_{\pm,T}}$ on $\dot F_\pm -\mathcal{E}_{\pm,T+T_0}$ and with
$$\exp_{u_*} (\beta_{+,T}\psi_{+,T}+\beta_{-,T}\psi_{-,T})=(s,t,\eta_*)+\beta_{+,T}\psi_{+,T}+\beta_{-,T}\psi_{-,T}$$
on $A_{[-(T+T_0), T+T_0]}$.
Here $\eta_*=\beta_{+,T}\eta_{+,T} +\beta_{-,T}\eta_{-,T}$.  We want to solve for $\psi_+$ and $\psi_-$ in the equation
\begin{equation}
\overline{\partial}u\in E'_{(\F,u)},
\end{equation}
with $(\F,u)=(F,j,{\bf p}, u)$.  Here $J$ is understood and $E'$ is as given in Equation~(6.3.1).


We identify
$$\Phi=\Phi_T:\wedge^{0,1}u^*  T(\mathbb R\times M)\xrightarrow\sim\wedge^{0,1}(u_*)^*T(\R\times M)$$
as follows: away from the neck region $A_{[-(T+2T_0), T+2T_0]}$ we parallel transport using a complex linear connection induced by the Levi-Civita connection of an adapted Riemannian metric (see \cite[p.39]{MS}); on $A_{[-(T+T_0),T+T_0]}$ we map $\eta\otimes (ds-idt)\mapsto \eta\otimes (ds-idt)$; and on $A_{[-(T+2T_0), -(T+T_0)]}$ and $A_{[T+T_0,T+2T_0]}$ we interpolate between the two identifications.
Similarly, we identify
$$\Phi_\pm=\Phi_{\pm}^{\psi_{\pm,T}}: \wedge^{0,1}(\exp_{u_{\pm,T}}\psi_{\pm,T})^{*}T(\R\times M)\xrightarrow\sim\wedge^{0,1}u_{\pm,T}^{*}T(\R\times M).$$

The equation $\overline{\partial} u\in E'_{(\F,u)}$, or more precisely the equation
\begin{equation}\label{d-bar2}
{\Phi} \overline{\partial} u\in {\Phi } E'_{(\F,u)},
\end{equation}
can then be written as
\begin{equation} \label{eqn: d-bar rewritten}
\beta_{-,T} \left(D_{-,T} \psi_{-,T} + e_{-,T} +\mathcal{L}_- +\mathcal{R}_-\right)+\beta_{+,T}\left(D_{+,T}\psi_{+,T} + e_{+,T}+\mathcal{L}_+ +\mathcal{R}_+\right)\in {\Phi} E_{(\mathcal F, u)}',
\end{equation}
where $D_\pm$ (resp.\ $D_{\pm,T}$) is the linearization of $\overline{\partial}$ for $u_{\pm}$ (resp.\ $u_{\pm,T}$), $e_\pm:=\overline{\partial} u_\pm$ (resp.\ $e_{\pm,T}:=\overline{\partial} u_{\pm,T}$), and $\mathcal{L}_\pm$ and $\mathcal{R}_\pm$ are described below. (The descriptions of $\mathcal{L}_\pm$ and $\mathcal{R}_\pm$ on $A_{[-T,T]}$ are obtained by expanding $\tfrac{\bdry u}{\bdry s} +J(u) \tfrac{\bdry u}{\bdry t}$, where we write
$$u(s,t)=\begin{pmatrix} s \\ t \\ \beta_{+,T}\eta_{+,T} +\beta_{-,T}\eta_{-,T}\end{pmatrix}+\beta_{+,T}\psi_{+,T}+\beta_{-,T}\psi_{-,T}$$
and
$$J(u)=
\begin{pmatrix}
0 & -1 & 0\\
1 & 0 & 0\\
X\eta' & -j_0 X\eta'  & j_0
\end{pmatrix},$$
where:
\begin{itemize}
\item we are using simple coordinates $(s,t,(x,y))$ written as column vectors;
\item $\eta'= \beta_{+,T}(\eta_{+,T}+\psi_{+,T}^N) +\beta_{-,T}(\eta_{-,T}+\psi_{-,T}^N)$;
\item $X$ is the constant matrix defined in the paragraph after Definition 3.1.1;
\item $j_0$ is the standard complex structure on $\mathbb R^{2n}$;
\item $\psi^N_{\pm,T}$ is the $(x,y)$-component of $\psi_{\pm,T}$.
\end{itemize}
The details are left to the reader.)


\s\n
1. $\mathcal L_\pm=\mathcal L_\pm(\psi_{-,T},\psi_{+,T})$ has support on $A_{[-T,T]}$ and is given by
\begin{align} \label{equation: mathcal L}
2\mathcal L_\pm= &
\left(
\begin{array}{ccc}
0\\0\\ \frac{\partial \beta_{\mp,T}}{\partial s}\eta_{\mp,T}
\end{array}
\right)+\frac{\partial \beta_{\mp,T}}{\partial s}\psi_{\mp,T}\\
\nonumber &
+\widetilde\beta_{\pm,T}\beta_{\mp,T}\left(
\begin{array}{ccc}
0 & 0 & 0\\
0 & 0 & 0 \\
 X \eta_{\pm,T} &  -j_0 X \eta_{\pm,T} & 0
\end{array}
\right)\frac{\partial \psi_{\mp,T}}{\partial t}
\\
\nonumber &+\widetilde\beta_{\pm,T}\beta_{\mp,T}\left(
\begin{array}{ccc}
0 & 0 & 0\\
0 & 0 & 0 \\
 X \psi^N_{\pm,T} &  -j_0 X \psi^N_{\pm,T} & 0
\end{array}
\right)\frac{\partial \psi_{\mp,T}}{\partial t}\\
\nonumber &+
\widetilde\beta_{\pm,T}(\beta_{\pm,T}-1)\left(
\begin{array}{ccc}
0 & 0 & 0\\
0 & 0 & 0 \\
X\eta_{\pm,T} & -j_0X\eta_{\pm,T} & 0
\end{array}
\right)
\frac{\partial \psi_{\pm,T}}{\partial t}.
\end{align}
Here:
\begin{itemize}
\item the factor $\otimes (ds-idt)$ is omitted from each term;
\item $\widetilde\beta_{-,T}(s)=\beta(\frac{T+hr-s}{hr})$ and $\widetilde\beta_{+,T}(s)=\beta(\frac{T+hr+s}{hr})$;
\end{itemize}

\begin{terminology} $\mbox{}$
\begin{enumerate}
\item We say that a function $F(\psi)$ is ``linear''  if there exists a constant $c>0$ such that $|F(\psi)(x)|<c|\psi(x)|$ at every point $x$ of the domain of $\psi$.  We will use the shorthand ${\mathfrak l}(\psi)$ to denote an unspecified ``linear" map.
\item Following \cite[Definition~5.1]{HT2} we say that $F(\psi)$ is {\em type $1$ quadratic} if it can be written as
$$F(\psi)=P(\psi) + Q(\psi)\cdot \nabla\psi,$$
where there exists a constant $c>0$ such that $|P(\psi)(x)|\leq c|\psi(x)|^2$ and $|Q(\psi)(x)|\leq c|\psi(x)|$ at every point $x$ of the domain of $\psi$.
\end{enumerate}
\end{terminology}

\s\n
2. The remaining terms of ${\Phi} \overline{\partial}u$ are grouped into $\beta_-\mathcal{R}_-$ and $\beta_+\mathcal{R}_+$.
One can see that $\mathcal{R}_{\pm}= \mathcal{R}_{\pm} (\psi_{\pm,T})$ is type $1$ quadratic. On $A_{[-T,T]}$, we can explicitly write
\begin{equation} \label{equation: mathcal R}
\mathcal R_\pm=\frac{1}{2} \beta_{\pm,T}
\left(
\begin{array}{ccc}
0 & 0 & 0\\
0 & 0 & 0\\
 X\psi_{\pm,T}^N & - j_0X\psi_{\pm,T}^N & 0
\end{array}
\right)
\frac{\partial \psi_{\pm,T}}{\partial t}.
\end{equation}

\s

We also write $E'_{(\F,u)}=E'_{\psi_\pm,T}$ and decompose
$$E_{\psi_\pm,T}'=(E'_{\psi_\pm,T})_+\oplus (E'_{\psi_\pm,T})_-,$$
where $(E'_{\psi_\pm,T})_\pm$ corresponds to $u_{\pm,T}$. In order to solve Equation \eqref{eqn: d-bar rewritten} we solve for $(\psi_-,\psi_+)$ in the pair of equations
\begin{equation} \label{theta+-}
D_{\pm,T} \psi_{\pm,T} + e_{\pm,T} + \mathcal{L}_{\pm}(\psi_{-,T},\psi_{+,T}) + \mathcal{R}_{\pm}(\psi_{\pm,T})\in \Phi_{\pm}^{\psi_{\pm,T}} (E'_{\psi_\pm,T})_\pm.
\end{equation}
Let
\begin{equation}\label{eqn: orthogonal projection}
\Pi_\pm^{\psi_\pm,T}:\mathcal{H}_{0,\delta}(\dot F, \wedge^{0,1}u_{\pm,T}^{*}T(\R\times M))\to  (\Phi_\pm^{\psi_{\pm,T}} (E'_{\psi_\pm,T})_\pm)^\perp
\end{equation}
be the projection onto the $L^2$-orthogonal complement of $\Phi_\pm^{\psi_{\pm,T}} (E'_{\psi_\pm,T})_\pm$.
Then Equation~\eqref{theta+-} is equivalent to
\begin{equation} \label{Theta pm}
\Theta_\pm(\psi_-,\psi_+) :=\Pi_\pm^{\psi_\pm,T}(D_{\pm,T} \psi_{\pm,T} + e_{\pm,T} + \mathcal L_\pm(\psi_{-,T},\psi_{+,T}) + \mathcal{R}_\pm(\psi_{\pm,T}))=0.
\end{equation}

Since $D_{\pm,T}$ is transverse to $E_\pm$ and $\Phi_{\pm}^{\psi_{\pm,T}} (E'_{\psi_\pm,T})_\pm$ is close to $E_\pm$ when $T\gg 0$ and $\|\psi_\pm\|_{*,\delta}$ is small, it follows that $D_{\pm,T}$ is transverse to $\Phi_{\pm}^{\psi_{\pm,T}} (E'_{\psi_\pm,T})_\pm$, i.e.,
$$\Pi_\pm^{\psi_\pm,T} D_{\pm,T}: \mathcal{H}_{1,\delta}(\dot F, u_{\pm,T}^{*}T(\R\times M))\to (\Phi_\pm^{\psi_{\pm,T}} (E'_{\psi_\pm,T})_\pm)^\perp$$
is surjective.

\subsection{Definition of gluing map} \label{subsection: defn of gluing map}

In this subsection we make some estimates and define the gluing map at the end.  We use the convention that constants such as $C$ may change from line to line when making estimates.

Let $\lambda= \min \{\lambda_1, |\lambda_{-1}|\}$. Suppose that $\delta$ satisfies $0<100\cdot\delta<\lambda$.  Let $\mathcal{H}_\pm =\mathcal{H}_{1,\delta}(\dot F,u_\pm^{*}T(\R\times M))$, let $\mathcal{H}_\pm^{\perp,\psi_{\pm},T}$ be the $L^2$-orthogonal complement of
$$\ker \left( \Pi_\pm^{\psi_\pm,T}\circ D_{\pm,T}\circ \tau_{\pm(T+T_0)}\right)$$
in $\mathcal{H}_{\pm}$, where $\tau_{\pm(T+T_0)}$ is the map induced by $\tau_{\pm(T+T_0)}$, and let $\mathcal{B}_{\pm}$ be the closed ball of radius $\widetilde\varepsilon$ in $\mathcal{H}_{\pm}$ centered at $0$.


The following lemma is modeled on \cite[Proposition 5.6]{HT2}, but the estimates are slightly different.

\begin{lemma} \label{azalea}
There exist $r\gg 0$ and $C,\widetilde\varepsilon>0$ such that for $T\gg 0$ and $\|e_+\|_{*,\delta}$ small the following holds:
\begin{enumerate}
\item There is a map $P_+:\mathcal{B}_-\to \mathcal{H}_+$ such that $\Theta_+(\psi_-,P_+(\psi_-))=0$ and $\psi_+=P_+(\psi_-)\in \mathcal{H}_+^{\perp,\psi_+,T}$.
\item $\| P_+(\psi_-)\|_{*,\delta}\leq C (\|e_{+}\|_{*,\delta}+ r^{-1/2}e^{-2(\lambda-\delta) T}+(e^{-2\delta T}+r^{-1})\|\psi_-\|_{*,\delta})$.
\item $\| P_+(\psi_-) - P_+(\psi'_-)\|_{*,\delta} \leq C \| \psi_- - \psi_-'\|_{*,\delta} (e^{-2\delta T} + r^{-1}  +\widetilde\varepsilon)$.
\end{enumerate}
\end{lemma}

Let $\mathcal{I}_+: \mathcal{B}_-\times \mathcal{B}_+\to \mathcal{H}_+$ be the map given by:
\begin{align} \label{def of I}
\mathcal{I}_+(\psi_-,\psi_+) =&  -\tau_{+(T+T_0)}^{-1}(\Pi_+^{\psi_+,T}D_{+,T})^{-1}\Pi_+^{\psi_+,T}(e_{+,T}  + \mathcal L_+(\psi_{-,T},\psi_{+,T}) + \mathcal{R}_+(\psi_{+,T})).
\end{align}
Here $(\Pi_+^{\psi_+,T}D_{+,T})^{-1}$ is the bounded inverse of $\Pi_+^{\psi_+,T}D_{+,T}$.

\begin{proof}
(1) We are trying to solve for $\psi_+$ in
$$\Pi_+^{\psi_+,T} (D_{+,T} \psi_{+,T}+e_{+,T}+\mathcal L_+(\psi_{-,T},\psi_{+,T})+\mathcal R_+(\psi_{+,T}))=0.$$
Using the function $\mathcal{I}_+$ we define $P(\psi_-)$ as the unique fixed point, i.e., $\psi_+$ satisfying
\begin{equation} \label{fixed point}
\mathcal{I}_+(\psi_-,\psi_+)=\psi_+.
\end{equation}

The unique fixed point is guaranteed by the contraction mapping theorem, which follows from Claims~\ref{claim 1} and \ref{claim 2}.

\begin{claim} \label{claim 1}
\begin{align*}
\|\mathcal{I}_+(\psi_-,\psi_+)\|_{*,\delta}  \leq & (\widetilde C_1(\|\psi_+\|_{*,\delta})+\widetilde C_2(T))\|e_{+}\|_{*,\delta} \\
\nonumber & + \widetilde C(\|\psi_+\|_{*,\delta}) ( r^{-1/2}e^{-2(\lambda-\delta) T}+(e^{-2\delta T}+r^{-1})\|\psi_-\|_{*,\delta}\\
\nonumber &\quad  +\|\psi_-\|_{*,\delta}\|\psi_+\|_{*,\delta}  +e^{-2\lambda T}\|\psi_+\|_{*,\delta}+\|\psi_+\|^2_{*,\delta}),
\end{align*}
where $\widetilde C(\sharp)$, $\widetilde C_1(\sharp)$, $\widetilde C_2(\sharp)$ are continuous functions of $\sharp$;  $\widetilde C_1(\sharp)\to 0$ as $\sharp\to 0$; and $\widetilde C_2(\sharp)\to 0$ as $\sharp\to\infty$.
\end{claim}

\begin{proof}[Proof of Claim~\ref{claim 1}]
We will be using Lemma~\ref{lemma: sobolev embedding} (more precisely $|\zeta|_{C^0}\leq C \|\zeta\|_{*,\delta}$) several times without explicit mention.

We first obtain
\begin{align} \label{one}
\left\| \mathcal L_+(\psi_{-,T},\psi_{+,T}) \right\|_\delta^{+,T}\leq &  C(r^{-1/2}e^{-2(\lambda-\delta) T}+r^{-1}\|\psi_-\|_{*,\delta}+e^{-2\delta T}\|\psi_-\|_{*,\delta}\\
\nonumber & \quad +\|\psi_-\|_{*,\delta}\|\psi_+\|_{*,\delta} +e^{-2\lambda T}\|\psi_+\|_{*,\delta}),
\end{align}
where each term on the right-hand side corresponds to the terms on the right-hand side in Equation~\eqref{equation: mathcal L}.  The first term has bound
\begin{align*}
\| \tfrac{\bdry \beta_{-,T}}{\bdry s}& \eta_{-,T}\|_\delta^{+,T}  = \| \tfrac{\bdry \beta_{-,T}}{\bdry s} \eta_{-,T} f_\delta^{+,T}\|
=\| \tfrac{\bdry \beta_{-,T}}{\bdry s} \eta_{-,T} e^{-\delta(s-(T+T_0))}\|  \\
& \leq C(r^{-1}\cdot e^{-\lambda\cdot 2T}\cdot e^{\delta\cdot 2T_0}\cdot \sqrt{hr}\cdot  + r^{-1}\cdot e^{-\lambda \cdot 2T}\cdot e^{\delta\cdot 2T_0})\\
& \leq C(r^{-1/2}e^{-2(\lambda-\delta) T}),
\end{align*}
where we are using the bound $|\frac{\bdry\beta_-}{\bdry s}|< cr^{-1}$ for some $c>0$ and exponential decay bounds for $\eta_-$. The second term has bound
\begin{align*}
\|\tfrac{\bdry\beta_{-,T}}{\bdry s}\psi_{-,T} \cdot f_\delta^{+,T}\|& = \|\tfrac{\bdry\beta_{-,T}}{\bdry s}\psi_{-,T} \cdot e^{-\delta (s-(T+T_0))} \| \leq\|\tfrac{\bdry\beta_{-,T}}{\bdry s}\psi_{-,T} \cdot e^{2\delta T}\| \\
& \leq Cr^{-1} \|\psi_-\|_\delta\leq Cr^{-1}\|\psi_-\|_{*,\delta}.
\end{align*}
The third term has bound
\begin{align*}
C\|\widetilde\beta_{+,T}\beta_{-,T} & e^{\lambda (s-(T+T_0))}\tfrac{\bdry \psi_{-,T}}{\bdry t} e^{-\delta (s-(T+T_0))}\| \\
& = C \|\widetilde\beta_{+,T}\beta_{-,T} e^{ (\lambda-2\delta)s-\lambda(T+T_0)}  \tfrac{\bdry \psi_{-,T}}{\bdry t}e^{\delta( s+ (T+T_0))}  \| \\
& \leq  C  e^{-2\delta(T+T_0)}\|\psi_-\|_{*,\delta} \leq C e^{-2\delta T}\|\psi_-\|_{*,\delta}.
\end{align*}
The fourth term has bound
\begin{align*}
C\|\widetilde\beta_{+,T}\beta_{-,T} & \psi_{+,T} \psi_{-,T} e^{-\delta(s-(T+T_0))}\| \\
& \leq  C |\psi_{+,T} e^{-\delta(s-(T+T_0))}|_{C^0}\cdot \|\widetilde\beta_{+,T}\beta_{-,T} \psi_{-,T}\|\\
& \leq C \|\psi_+\|_{*,\delta} \|\psi_-\|_{\delta}\leq C \|\psi_+\|_{*,\delta} \|\psi_-\|_{*,\delta}.
\end{align*}
The last term has bound
\begin{align*}
C\| (\widetilde \beta_{+,T})(\beta_{+,T}-1)\eta_{+,T} \tfrac{\bdry\psi_{+,T}}{\bdry t} e^{-\delta(s-(T+T_0))}\| & \leq C e^{-2\lambda T} \|\psi_+\|_{*,\delta}.
\end{align*}

Also we obtain
\begin{align}\label{two}
\| \mathcal{R}_+(\psi_{+,T})\|^{+,T}_\delta \leq C\|\psi_+\|^2_{*,\delta},
\end{align}
since $\mathcal{R}_+(\psi_{+,T})$ is type 1 quadratic.

Next we consider the $L^2$-projection $\Pi_+$ to $(\Phi_+ E'_+)^\perp$, where we are suppressing ${\psi_\pm,T}$ from the notation. Let $e_1(\psi_+,T),\dots,e_k(\psi_+,T)$ be an orthonormal basis for $\Phi_+ E'_+$; it is not hard to see that $e_i(\psi_+,T)$ can be taken to be continuous with respect to $\psi_+\in \mathcal{B}_+$. 
Then
\begin{align*}
\Pi_+(\zeta)=\zeta - \sum_i \langle e_i(\psi_+,T),\zeta\rangle e_i(\psi_+,T),
\end{align*}
where $\langle\cdot,\cdot\rangle$ is the $L^2$-inner product.
This allows us to estimate
\begin{align} \label{three}
\|\Pi_+(\zeta)\|_\delta \leq   \|\zeta\|_\delta +\sum_i\|\zeta\|_\delta\cdot \|e_i(\psi_+,T)\|_\delta\leq \widetilde C(\|\psi_+\|_{*,\delta}) \|\zeta\|_\delta.
\end{align}
The first inequality follows from $|\langle e_i(\psi_+,T),\zeta\rangle|\leq \|e_i(\psi_+)\|_{L^2} \cdot\|\zeta\|_{L^2} \leq \|\zeta\|_\delta$, where $\| \cdot \|_{L^2}$ is the $L^2$-norm, and the second inequality follows from the continuity of $e_i(\psi_+,T)$ with respect to $\psi_+$.

\begin{claim} \label{claim: continuous}
$(\Pi_+ D_+)^{-1}$ depends continuously on $\psi_+\in\mathcal{B}_+$.
\end{claim}

Observe that the domain of $(\Pi_+ D_+)^{-1}$ uses the norm $\|\cdot \|_\delta$ and the range of $(\Pi_+ D_+)^{-1}$ uses $\|\cdot \|_{*,\delta}$.

\begin{proof}[Sketch of proof of Claim~\ref{claim: continuous}]
$\Pi_+$ depends continuously on $\psi_+\in\mathcal{B}_+$ and hence so does $\Pi_+ D_+$.  The inverse hence also depends continuously on $\psi_+$.
\end{proof}

We also obtain
\begin{align} \label{four}
\|\tau_{+(T+T_0)}^{-1}(\Pi_+^{\psi_+,T}D_{+,T})^{-1}\Pi_+^{\psi_+,T}(e_{+,T})\|_\delta \leq (\widetilde C_1(\|\psi_+\|_{*,\delta})+\widetilde C_2(T))\|e_{+}\|_{*,\delta}.
\end{align}
Claim~\ref{claim 1} then follows from combining Estimates~\eqref{one}, \eqref{two}, \eqref{three}, \eqref{four}, and Claim~\ref{claim: continuous}.
\end{proof}

If $r,T\gg 0$, $\widetilde\varepsilon>0$ is small, and $\|e_+\|_{*,\delta}$ is small, then $\|\mathcal I (\psi_-,\psi_+)\|_{*,\delta}<\widetilde\varepsilon$ whenever $\|\psi_-\|_{*,\delta}, \|\psi_+\|_{*,\delta}<\widetilde\varepsilon$. Hence $\mathcal{I}_+(\psi_-,\cdot)$ maps a radius $\widetilde\varepsilon$ ball into itself.

\begin{claim}\label{claim 2}
\begin{align*}
\|\mathcal{I}_+(&\psi_-,\psi_+) -\mathcal{I}_+(\psi_-, \psi_+')\|_{*,\delta}\\
\leq &  C ( \|\psi_-\|_{*,\delta}+e^{-2\lambda T}+\|\psi_+\|_{*,\delta} +\|\psi_+'\|_{*,\delta})\cdot \| \psi_+-\psi'_+\|_{*,\delta}\\
& + C\|\psi_+-\psi_+'\|_{*,\delta}\cdot (\|e_{+}\|_{*,\delta} + r^{-1/2}e^{-2(\lambda-\delta) T} +(e^{-2\delta T}+r^{-1})\|\psi_-\|_{*,\delta}\\
& \qquad\qquad \qquad +\|\psi_-\|_{*,\delta}\|\psi_+\|_{*,\delta}  +e^{-2\lambda T}\|\psi_+\|_{*,\delta}+\|\psi_+\|^2_{*,\delta}).
\end{align*}
\end{claim}

\begin{proof}[Proof of Claim~\ref{claim 2}]
This follows from the estimates
\begin{align*}
\| \mathcal{L}_+(\psi_{-,T},\psi_{+,T})-\mathcal{L}_+(\psi_{-,T},\psi_{+,T}')\|_\delta^{+,T} & \leq C (\|\psi_-\|_{*,\delta}+e^{-2\lambda T})\|\psi_+-\psi_+'\|_{*,\delta},\\
\|\mathcal{R}_+(\psi_{+,T})-\mathcal{R}_+(\psi_{+,T}')\|_\delta^{+,T} & \leq C( \|\psi_+\|_{*,\delta} +\|\psi_+'\|_{*,\delta}) \|\psi_+ - \psi_+'\|_{*,\delta},
\end{align*}
as well as the calculations from Claim~\ref{claim 1}.
\end{proof}

Hence if $r,T\gg 0$, $\widetilde\varepsilon>0$ is small, and $\|e_+\|_{*,\delta}$ is small, then $\mathcal{I}_+(\psi_-,\cdot)$ gives a contraction mapping, provided $\|\psi_+\|_{*,\delta}, \|\psi_+'\|_{*,\delta}\leq \widetilde\varepsilon$. This proves (1).  We also prove the following, which is used in (3):

\begin{claim}\label{claim 3}
\begin{align*}
\|\mathcal{I}_+(\psi_-,\psi_+') &-\mathcal{I}_+(\psi_-',\psi_+')\|_{*,\delta} \\
& \leq   \widetilde C(\|\psi_+'\|_{*,\delta} ) (e^{-2\delta T} + r^{-1} + \|\psi_+'\|_{*,\delta})\|\psi_- - \psi_-'\|_{*,\delta}.
\end{align*}
\end{claim}

\begin{proof}[Proof of Claim~\ref{claim 3}]
This follows from the estimate
\begin{align*}
\| \mathcal{L}_+(\psi_{-,T},\psi_{+,T}')-\mathcal{L}_+(\psi_{-,T}',\psi_{+,T}')\|^{+,T}_\delta & \leq C (e^{-2\delta T}+ r^{-1} + \|\psi_+'\|_{*,\delta})\|\psi_- - \psi_-'\|_{*,\delta}.
\end{align*}
as well as the calculations from Claim~\ref{claim 1}.
\end{proof}

\s\n
Now we continue the proof of Lemma~\ref{azalea}.

(2)
Since $\psi_+=P_+(\psi_-)$ satisfies $\mathcal{I}_+(\psi_-,\psi_+)=\psi_+$, Claim~\ref{claim 1} gives:
\begin{align*}
\|\psi_+\|_{*,\delta}  \leq & \widetilde C(\| \psi_+\|_{*,\delta}) (\|e_{+}\|_{*,\delta} + r^{-1/2}e^{-2(\lambda-\delta) T} +(e^{-2\delta T}+r^{-1})\|\psi_-\|_{*,\delta}\\
&+\|\psi_-\|_{*,\delta}\|\psi_+\|_{*,\delta} +e^{-2\lambda T}\|\psi_+\|_{*,\delta}+\|\psi_+\|^2_{*,\delta}).
\end{align*}
By moving the last three terms on the right-hand side to the left, we obtain (2) since $\widetilde C(\| \psi_+\|_{*,\delta})(\|\psi_-\|_{*,\delta}+e^{-2\lambda T}+ \|\psi_+\|_{*,\delta})<1 $ for $\widetilde\varepsilon>0$ small and $r>0$ large.

\s
(3) Letting $\psi_+=P_+(\psi_-)$ and $\psi_+'= P_+(\psi_-')$,
\begin{align*}
\| P_+(\psi_-)- & P_+(\psi_-')\|_{*,\delta} =  \| \I (\psi_-,\psi_+)- \I (\psi_-',\psi_+')\|_{*,\delta}\\
\leq & \| \I (\psi_-,\psi_+)- \I (\psi_-,\psi'_+)\|_{*,\delta} + \| \I (\psi_-,\psi'_+) - \I (\psi_-',\psi_+')  \|_{*,\delta},
\end{align*}
and the two terms on the right-hand side are bounded using Claims~\ref{claim 2} and \ref{claim 3}.  The terms from the right-hand side of the inequality in Claim~\ref{claim 2} can be moved to the left.  $\|\psi_-\|_{*,\delta}\leq \widetilde \varepsilon$ by assumption. If $r,T\gg 0$, $\widetilde\varepsilon>0$ is small, and $\|e_{+}\|_{*,\delta}$ is small, then $\|\psi_+\|_{*,\delta}=\|P_+(\psi_-)\|_{*,\delta}\leq \widetilde \varepsilon$ by (2). (3) then follows.
\end{proof}

Analogously, we have the following for $P_-$:

\begin{lemma} \label{Tsutsuji}
There exist $r\gg 0$ and $C, \widetilde\varepsilon>0$ such that for $T\gg 0$ and $\|e_-\|_{*,\delta}$ small the following holds:
\begin{enumerate}
\item There is a map $P_-:\mathcal{B}_+\to \mathcal{H}_-$ such that $\Theta_-(P_-(\psi_+),\psi_+))=0$ and $\psi_-=P_-(\psi_+)\in \mathcal{H}_-^{\perp,\psi_-,T}$.
\item $\| P_-(\psi_+)\|_{*,\delta}\leq C (\|e_{-}\|_{*,\delta}+ r^{-1/2}e^{-2(\lambda-\delta) T}+(e^{-2\delta T}+r^{-1})\|\psi_+\|_{*,\delta})$.
\item $\| P_-(\psi_+) - P_-(\psi'_+)\|_{*,\delta} \leq C \| \psi_+ - \psi_+'\|_{*,\delta} (e^{-2\delta T} + r^{-1} +\widetilde\varepsilon)$.
\end{enumerate}
\end{lemma}

\begin{lemma}\label{estimate for psi and tau}
There exist $r\gg 0$, $\widetilde\varepsilon>0$, and $\widetilde\varepsilon_0>0$ such that for $T\gg 0$ there is a unique solution $(\psi_-,\psi_+)\in \mathcal{B}_-\times \mathcal{B}_+$ to the equations $\Theta_+(\psi_-,\psi_+)=0$ and $\Theta_-(\psi_-,\psi_+)=0$ subject to the constraints $\psi_\pm\in \mathcal{H}_\pm^{\perp,\psi_\pm,T}$.
\end{lemma}

Alternatively, we can view the desired $(\psi_-,\psi_+)$ as the unique fixed point of the map
$$\mathcal{I}= (\mathcal{I}_-,\mathcal{I}_+): \mathcal{B}_-\times\mathcal{B}_+ \to \mathcal{H}_-\times\mathcal{H}_+,$$
$$(\psi_-,\psi_+)\mapsto (\mathcal{I}_-(\psi_-,\psi_+),\mathcal{I}_+(\psi_-,\psi_+)),$$
subject to the constraints $\psi_\pm\in \mathcal{H}_\pm^{\perp,\psi_\pm,T}$.

\begin{proof}
We are looking for the unique solution $\psi_-$ to $P_-\circ P_+(\psi_-)=\psi_-$ subject to the constraints.  The existence follows from the contraction mapping principle and Lemmas~\ref{azalea} and \ref{Tsutsuji}. Then $(\psi_-,\psi_+=P_+(\psi_-))$ is the unique solution to $\Theta_\pm(\psi_-,\psi_+)=0$.
\end{proof}

\begin{rmk} \label{rmk: iteration}
In more practical terms, the fixed point can be obtained by starting with
$$(\psi_-^{(0)},\psi_+^{(0)})=(0,0),$$
applying the iteration
$$(\psi_-^{(i+1)},\psi_+^{(i+1)})=\mathcal{I}(\psi_-^{(i)},\psi_+^{(i)}),$$
and taking the limit $i\to\infty$. By Claim~\ref{claim 1} and its analog for $\mathcal{I}_-$, $\|\psi_\pm^{(1)}\|_{*,\delta}\to 0$ as $T\to\infty$ and similarly $\|\psi_\pm^{(i)}\|_{*,\delta}\to 0$ as $T\to\infty$.
\end{rmk}

The above remark implies the following lemma:

\begin{lemma} \label{lemma: T to infty}
If $(\psi_-,\psi_+)\in \mathcal{B}_-\times \mathcal{B}_+$ is the unique solution to $\Theta_+(\psi_-,\psi_+)=0$ and $\Theta_-(\psi_-,\psi_+)=0$ subject to the constraints $\psi_\pm\in \mathcal{H}_\pm^{\perp,\psi_\pm,T}$, then $\|\psi_\pm\|_{*,\delta}\to 0$ as $T\to \infty$.
\end{lemma}

We can finally define the gluing map  as
$$G:V_-/\mathbb R \times V_+/\mathbb R \times [T_0,\infty)\to \widetilde{\mathcal{G}}^{E'}_\delta(V_-/\R, V_+/\R)$$
$$(u_-,u_+,T)\mapsto u= \exp_{u_*}(\beta_{+,T}\psi_{+,T} +\beta_{-,T}\psi_{-,T}),$$
where $(\psi_-,\psi_+)$ is the solution from Lemma~\ref{estimate for psi and tau}.

\subsection{$C^1$-smoothness and injectivity of the gluing map} \label{subsection: C1 smoothness}

The goal of this subsection is to prove:

\begin{thm} \label{thm: gluing map smooth}
The gluing map $G$ is a $C^1$-smooth embedding.
\end{thm}

We start by observing that:

\begin{lemma} \label{lemma: psi pm smooth}
The solution $(\psi_-,\psi_+)$ from Lemma~\ref{estimate for psi and tau} is $C^\infty$-smooth.
\end{lemma}

\begin{proof}[Sketch of proof]
The solution $(\psi_-,\psi_+)$ is in Morrey class $\mathcal{H}_{1,\delta}$ and hence is in Sobolev class $W^{1,2}_\delta$ and in $C^0$.  We apply the usual elliptic bootstrapping technique, where the necessary estimates for Morrey spaces are given in \cite{Mo}: First observe that $(\psi_-,\psi_+)$ satisfies a nonlinear Cauchy-Riemann type operator. We can then view $(\psi_-,\psi_+)$ as satisfying a {\em linear} Cauchy-Riemann type operator whose coefficients depend on $(\psi_-,\psi_+)$. Differentiating Equation~\eqref{theta+-} with respect to $s$, we also see that $(\frac{\bdry\psi_-}{\bdry s},\frac{\bdry \psi_+}{\bdry s})$ satisfies a linear Cauchy-Riemann type operator such that the coefficients of the first order part are in $\mathcal{H}_{1,\delta}$ and the coefficients of the zeroth order part are in $\mathcal{H}_{0,\delta}$; the same holds for $\frac{\bdry \psi_\pm}{\bdry t}$.  With the above conditions on the coefficients, the relevant elliptic estimates are given by the last two inequalities on \cite[p.~145]{Mo} to improve $\nabla \psi_\pm$ from Sobolev class $W^{0,2}_\delta$ to $W^{1,2}_\delta$ and \cite[Theorem 5.4.1]{Mo} to improve $\nabla\psi_\pm$ from $W^{1,2}_\delta$ to $\mathcal{H}_{1,\delta}$.
\end{proof}

Let $\mathcal{S}_\pm\subset V_\pm/\R$ be small neighborhoods of $u_\pm$ in $V_\pm/\R$. Each point of $\mathcal{S}_\pm$ is given by the pair $(\exp_{u_\pm } \phi_\pm ,j_\pm')$, where $\phi_\pm=(\phi^\circ_\pm,\mathfrak{a}_\pm)\in \mathcal{H}_\pm\oplus \R^{2l^\pm}$, $j_\pm'\in \widetilde{\mathcal{U}}_\pm$, $l^\pm$ is the total number of punctures of $\dot F_\pm$ (see Section 4.2.1), and $\widetilde{\mathcal{U}}_\pm$ is a Teichm\"uller slice containing $j_\pm$ (the complex structure for $u_\pm$). We will usually write $(\phi_\pm,j_\pm')\in \mathcal{S}_\pm$.

{\em By abuse of notation we write $\|\phi_\pm^\circ+{\mathfrak a}_\pm\|_{*,\delta}$ or $\| \phi_\pm \|_{*,\delta}$ for the sum of $\|\phi_\pm^\circ\|_{*,\delta}$ on $\mathcal{H}_\pm$ and the standard norm $|{\mathfrak a}_\pm|$ on $\R^{2l^\pm}$.}

We assume that:
\begin{itemize}
\item[(T$_1$)] the slice $\widetilde{\mathcal{U}}_\pm$ is smooth; and
\item[(T$_2$)] there exists a small disk $D_\pm\subset \dot F_\pm- \mathcal{E}_{\pm,T+T_0}$ such that all $j'_\pm\in \widetilde{\mathcal{U}}_\pm$ agree on $\dot F_\pm -D_\pm$.
\end{itemize}
We may additionally assume that, for any ${\mathfrak a}$ in the summand $\R^2$ of $\R^{2l^+}$ (resp.\ $\R^{2l^-}$) corresponding to the end that is glued, ${\mathfrak a}$ is a linear combination of $\bdry_s$ and $\bdry_t$ on $s\leq 0$ for $u_+$ (resp.\ on $s\geq 0$ for $u_-$).

We say that $\phi_\pm=(\phi_\pm^\circ,{\mathfrak a}_\pm)\in \mathcal{H}_\pm\oplus \R^{2l^\pm}$ or $\phi_\pm^\circ\in \mathcal{H}_\pm$ {\em satisfies} ($\#_k$) if the following holds:
\begin{enumerate}
\item[($\#_k$)] $\phi_\pm^\circ$ is in class $C^k$ and $\|\nabla^i \phi_\pm^\circ|_{-2(T+T_0)-1<\pm s<1}\|_\delta$, $\forall i\leq k$, is bounded above by $C\|\phi_\pm^\circ\|_{*,\delta}$.
\end{enumerate}
Here $C$ is a constant that does not depend on $\phi_\pm^\circ$. Note that ($\#_k$) for $\phi_\pm^\circ$ implies ($\#_k$) for $\phi_\pm^\circ+{\mathfrak a}_\pm$, since ${\mathfrak a}_\pm$ lives in the finite-dimensional vector space $\R^{2l^\pm}$. In what follows we assume that $k\gg 0$.

Our gluing setup is slightly more complicated than one initially expects: this is to get around the loss of one derivative as explained in Lemma~\ref{lemma: sufficiently differentiable}.  Let
$$\mathcal{J}=(\mathcal{J}_-,\mathcal{J}_+):\mathcal{S}_-\times \mathcal{S}_+ \times [T_0,\infty) \to\mathcal{B}_-\times\mathcal{B}_+$$
be a $C^1$-smooth function such that $\op{Im}(\mathcal{J})$ is bounded and each $(\psi_-,\psi_+)\in \op{Im}(\mathcal{J})$ satisfies ($\#_k$).  We then consider the function
$$\mathcal{I}^{\mathcal{J}}=(\mathcal{I}^{\mathcal{J}}_-,\mathcal{I}^{\mathcal{J}}_+): \mathcal{S}_-\times \mathcal{S}_+ \times [T_0,\infty) \to \mathcal{H}_-\times \mathcal{H}_+,$$
\begin{align*}
\mathcal{I}^{\mathcal{J}}_+&((\phi_-,j_-'),(\phi_+,j_+'),T)\\
=& -\tau^{-1}_{\pm(T+T_0)}(\Pi_+^{\psi_+ + \phi_+,j_+', T}D_{+,T}^{j_+'})^{-1}\Pi_+^{\psi_+ + \phi_+,j_+',T} [D_{+,T}^{j_+'}(\psi_{+,T} + \phi_{+,T}) +  e_{+,T} \\
& \qquad \qquad +\mathcal L_+(\psi_{-,T} +\phi_{-,T},\psi_{+,T}+\phi_{+,T}) +\mathcal{R}_+(\psi_{+,T}+\phi_{+,T})],
\end{align*}
where $\psi_\pm=\mathcal{J}_\pm((\phi_-,j_-'),(\phi_+,j_+'),T)$, $D_{\pm,T}^{j_\pm'}$ is $D_{\pm,T}$ with respect to $j_\pm'$, $j'_+$ in $\Pi_+^{\psi_+ + \phi_+,j_+',T}$ indicates the dependence on $j_+'$, and $e_{+,T}=\overline\bdry u_{\pm,T}$.
$\mathcal{I}^{\mathcal{J}}_-$ is defined analogously.  We will later apply an iteration scheme similar to the one described in Remark~\ref{rmk: iteration}; see the proof of Theorem~\ref{thm: gluing map smooth} at the end of this subsection.

\begin{rmk}
Observe that $\mathcal{I}^{\mathcal{J}}_+$ has the term $D_{+,T}^{j_\pm'}(\psi_{+,T} + \phi_{+,T})$ unlike the expression for $\mathcal{I}_+$ which does not have $D_{+,T}\psi_{+,T}$.
\end{rmk}

The expression inside the brackets $[\cdot]$ can be written as
\begin{align*}
D_{+,T}^{j_\pm'}\psi_{+,T} + e_{+,T}(\phi_+) +   \mathcal L_+(\psi_{-,T} +\phi_{-,T},\psi_{+,T}+\phi_{+,T}) +\mathcal{R}'_+(\psi_{+,T}+\phi_{+,T}),
\end{align*}
where $\mathcal{R}'_+(\psi_{+,T}+\phi_{+,T})= \mathcal{R}_+(\psi_{+,T}+\phi_{+,T})- \mathcal{R}_+(\phi_{+,T})$
and $e_{+,T}(\phi_+):= \overline\bdry \exp_{u_{+,T}}(\phi_{+,T})$.

\begin{rmk}
In order to avoid the (already cumbersome) notation, in the rest of the subsection we omit $j'_\pm$ from the notation and also will not explicitly treat variations of $j_\pm$ in the proof of Theorem~\ref{thm: gluing map smooth}.  The derivatives of the gluing map with respect to $j'_\pm$ are straightforward to control in view of (T$_1$) and (T$_2$).
\end{rmk}

\begin{lemma} \label{lemma: sharp}
Assuming $\mathcal{S}_\pm$ are sufficiently small, all $\phi_\pm\in \mathcal{S}_\pm$ satisfy $(\#_k)$.
\end{lemma}

\begin{proof}
This follows from the elliptic bootstrapping estimates used in the proof of Lemma~\ref{lemma: psi pm smooth}.
\end{proof}

Next let us write $B(\psi_++\phi_+,T) = (\Pi_+^{\psi_+ + \phi_+,T} D_+)^{-1}\Pi_+^{\psi_++\phi_+,T}.$

\begin{lemma} \label{lemma: sufficiently differentiable}
With the above conditions on $\mathcal{J}$, $B(\psi_++\phi_+,T)$ is sufficiently differentiable with respect to $\phi_+\in \mathcal{S}_+$ and $T\in [T_0,\infty)$.
\end{lemma}

\begin{proof}
This uses the arguments used in Claim~\ref{claim: continuous} as well as Lemma~\ref{lemma: sharp}. The reason for requiring ($\#_k$) is that when we differentiate $e_i(\psi_+ + \phi_+,T)$ with respect to $\phi_+$ or $T$, we lose one derivative, which must be recovered using ($\#_k$).
\end{proof}

\begin{lemma}\label{lemma: I smoothness}
If $\mathcal{S}_\pm$ and $\widetilde\varepsilon>0$ are sufficiently small, $T,r\gg 0$, $\mathcal{J}$ is $C^1$-smooth, and $\op{Im}(\mathcal{J})$ is bounded, then $\mathcal{I}^{\mathcal{J}}$ is $C^1$-smooth. Moreover, if $\op{Im}(\mathcal{J})$ lies in a sufficiently small ball about the origin and the derivative $D\mathcal{J}$ is small, then the derivative $D\mathcal{I}^{\mathcal{J}}$ is small.
\end{lemma}

\begin{proof}
We write $\psi_\pm=\mathcal{J}_\pm(\phi_-,\phi_+,T)$.   Suppose that $\mathcal{J}_\pm$ is a constant function; we later explain how to modify the proof in the general case.
%

We first consider the partial derivative $D^\flat_1\mathcal{I}_+^{\mathcal{J}}$, where the superscript $\flat$ indicates that we are assuming that $\mathcal{J}_\pm$ is constant. Let $\zeta_\pm +{\mathfrak a}_\pm \in T_{u_\pm}\mathcal{S}_\pm$, where $\zeta_\pm\in\mathcal{H}_\pm$ and ${\mathfrak a}_\pm\in \R^{2l^\pm}$. We compute
\begin{align} \label{I}
D^\flat_1 & \mathcal{I}_+^{\mathcal{J}} (\phi_-,\phi_+,T)(\zeta_-+{\mathfrak a}_-) \\
\nonumber & = \tfrac{d}{d\tau}|_{\tau=0}B(\psi_+ +\phi_+,T) (\mathcal{L}_+(\psi_{-,T}+\phi_{-,T}+\tau(\zeta_{-,T}+{\mathfrak a}_{-,T}),\psi_{+,T}+\phi_{+,T}))\\
\nonumber &= B(\psi_+ +\phi_+,T)  (\tfrac{\bdry\beta_{-,T}}{\bdry s}(\zeta_{-,T}+{\mathfrak a}_{-,T}) + {\mathfrak l}(\eta_{+,T})\tfrac{\bdry\zeta_{-,T}}{\bdry t} + {\mathfrak l}(\psi_{+,T}+\phi_{+,T})\tfrac{\bdry\zeta_{-,T}}{\bdry t}),
\end{align}
using the assumption that ${\mathfrak a}_-$ is constant for $s\geq 0$.
Estimates similar to those of Claim~\ref{claim 1} imply that
\begin{align} \label{result of I}
\|D^\flat_1 \mathcal{I}_+^{\mathcal{J}} (\phi_-,\phi_+,T)(\zeta_-+{\mathfrak a}_-) \|_{*,\delta} \leq c(r^{-1} + \widetilde{C}(T) + \|\psi_+\|_{*,\delta} +\|\phi_+\|_{*,\delta})\cdot \|\zeta_{-}+{\mathfrak a}_-\|_{*,\delta},
\end{align}
where $\widetilde{C}(T)\to 0$ as $T\to\infty$. Hence the partial derivative $D^\flat_1$ exists. We are assuming that $\|\phi_+\|_{*,\delta}$ is sufficiently small and $T,r\gg 0$. If we assume that $\|\psi_+\|_{*,\delta}$ in addition, then there exists $0<C\ll 1$ such that
$$\|D^\flat_1 \mathcal{I}_+^{\mathcal{J}} (\phi_-,\phi_+,T)(\zeta_-+{\mathfrak a}_-) \|_{*,\delta} \leq C\cdot \|\zeta_{-}+{\mathfrak a}_-\|_{*,\delta}.$$

Similarly,
\begin{align*}
(D^\flat_1 &\mathcal{I}_+^{\mathcal{J}}(\phi_-,\phi_+,T)- D^\flat_1 \mathcal{I}_+^{\mathcal{J}}(\phi_-',\phi_+',T')) (\zeta_-+{\mathfrak a}_-) \\
= & B(\psi_+ +\phi_+,T) (\tfrac{\bdry\beta_{-,T}}{\bdry s}(\zeta_{-,T}+{\mathfrak a}_{-,T}) + {\mathfrak l}(\eta_{+,T})\tfrac{\bdry\zeta_{-,T}}{\bdry t} + {\mathfrak l}(\psi_{+,T}+\phi_{+,T})\tfrac{\bdry\zeta_{-,T}}{\bdry t}) \\
& - B(\psi_+ +\phi_+', T') (\tfrac{\bdry\beta_{-,T'}}{\bdry s}(\zeta_{-,T'}+{\mathfrak a}_{-,T'}) + {\mathfrak l}(\eta_{+,T'})\tfrac{\bdry\zeta_{-,T'}}{\bdry t} + {\mathfrak l}(\psi_{+,T'}+\phi'_{+,T'})\tfrac{\bdry\zeta_{-,T'}}{\bdry t}),
\end{align*}
and in view of Lemma~\ref{lemma: sufficiently differentiable} and the ($\#_k$)-condition applied to $\zeta_-$ there exists a constant $C>0$ such that:
\begin{align*}
\|((D^\flat_1 &\mathcal{I}_+^{\mathcal{J}}(\phi_-,\phi_+,T)- D^\flat_1 \mathcal{I}_+^{\mathcal{J}}(\phi_-',\phi_+',T')) (\zeta_-+{\mathfrak a}_-)\|_{*,\delta} \\
& \leq C (\|\phi_+-\phi_+'\|_{*}^c + |T-T'|) \cdot \|\zeta_{-}+{\mathfrak a}_-\|_{*,\delta}.
\end{align*}
(Moreover, $0<C\ll 1$.) Here the superscript $c$ means restriction of the function to
$$-(\max(T,T')+T_0) \leq s\leq (\max(T,T') + T_0).$$
This proves that $D^\flat_1\mathcal{I}_+^{\mathcal{J}}$ is in $C^0$.

Next we compute
\begin{align} \label{main}
D^\flat_2  \mathcal{I}_+^{\mathcal{J}} (\phi_-,&\phi_+,T)(\zeta_++{\mathfrak a}_+) =   \tfrac{dB(\psi_++\phi_++\tau(\zeta_++{\mathfrak a}_+),T)}{d\tau}|_{\tau=0}(v (\psi_-+\phi_-,\psi_++\phi_+,T)) \\
\nonumber &  +B(\psi_+ +\phi_+,T) (D^\flat_{+,T}(\zeta_{+,T}+{\mathfrak a}_{+,T}) \\
\nonumber & \quad + D^\flat_2 \mathcal{L}_+(\psi_{-,T}+\phi_{-,T},\psi_{+,T}+\phi_{+,T})(\zeta_{+,T}+{\mathfrak a}_{+,T}) \\
\nonumber & \quad +\tfrac{d}{d\tau}|_{\tau=0} e_{+,T}(\phi_+ +\tau(\zeta_++{\mathfrak a}_+)) + D^\flat\mathcal{R}'_+(\psi_{+,T}+\phi_{+,T}) (\zeta_{+,T}+{\mathfrak a}_{+,T})),
\end{align}
where we are writing
\begin{align*}
v(\psi_-+\phi_-,\psi_++\phi_+,T)& =D^\flat_{+,T}\psi_{+,T} +e_{+,T}(\phi_+) + \mathcal{L}_+(\psi_{-,T}+\phi_{-,T},\psi_{+,T}+\phi_{+,T}) \\
& \qquad\qquad+ \mathcal{R}'_+(\psi_{+,T}+\phi_{+,T}).
\end{align*}
We compute
\begin{align} \label{L}
D^\flat_2 \mathcal{L}_+(&\psi_{-,T}+\phi_{-,T}, \psi_{+,T}+\phi_{+,T})(\zeta_{+,T}+{\mathfrak a}_{+,T}) \\
\nonumber & = {\mathfrak l}(\tfrac{\bdry(\psi_{-,T}+\phi_{-,T})}{\bdry t}) (\zeta_{+,T}+{\mathfrak a}_{+,T})  + {\mathfrak l}(\eta_{+,T})\tfrac{\bdry\zeta_{+,T}}{\bdry t},
\end{align}
\begin{align}\label{R}
D^\flat\mathcal{R}'_+  (\psi_{+,T}&+\phi_{+,T})  (\zeta_{+,T}+{\mathfrak a}_{+,T}) \\
\nonumber = & {\mathfrak l}(\psi_{+,T}+\phi_{+,T}) (\zeta_{+,T}+{\mathfrak a}_{+,T}) + {\mathfrak l}(\psi_{+,T}+\phi_{+,T})(\nabla \zeta_{+,T}) \\
\nonumber &  + {\mathfrak l}(\nabla (\psi_{+,T}+\phi_{+,T}))(\zeta_{+,T}+{\mathfrak a}_{+,T}).
\end{align}
Combining Equations~\eqref{main}, \eqref{L}, and \eqref{R} we obtain
\begin{align}\label{II}
D^\flat_2 & \mathcal{I}_+^{\mathcal{J}} (\phi_-,\phi_+,T)(\zeta_++{\mathfrak a}_+) \\
\nonumber =&  \tfrac{dB(\psi_++\phi_++\tau(\zeta_++{\mathfrak a}_+),T)}{d\tau}|_{\tau=0}(v(\psi_- +\phi_-,\psi_+ + \phi_+,T)) \\
\nonumber & + B(\psi_+ +\phi_+,T)  [\tfrac{d}{d\tau}|_{\tau=0} e_{+,T}(\phi_+ +\tau(\zeta_++{\mathfrak a}_+)) \\
\nonumber & \quad +{\mathfrak l}(\tfrac{\bdry(\psi_{-,T}+\phi_{-,T})}{\bdry t}) (\zeta_{+,T}+{\mathfrak a}_{+,T})  + {\mathfrak l}(\eta_{+,T})\tfrac{\bdry\zeta_{+,T}}{\bdry t} + {\mathfrak l}(\psi_{+,T}+\phi_{+,T})(\zeta_{+,T}+{\mathfrak a}_{+,T}) \\
\nonumber & \quad + {\mathfrak l}(\psi_{+,T}+\phi_{+,T})(\nabla \zeta_{+,T}) + {\mathfrak l}(\nabla (\psi_{+,T}+\phi_{+,T}))(\zeta_{+,T}+{\mathfrak a}_{+,T})].
\end{align}

Let $[\cdot]$ denote the expression inside the brackets. We first bound
$$\| \tfrac{d}{d\tau}|_{\tau=0} e_{+,T}(\phi_+ +\tau(\zeta_++{\mathfrak a}_+))\|_\delta^{+,T}\leq c \|e_+\|_{*,\delta}\cdot \|\zeta_++{\mathfrak a}_+\|_{*,\delta},$$
where $c>0$.  The remainder of the terms can be estimated as in Claim~\ref{claim 1} and:
\begin{align}\label{result of II}
\|[\cdot]\|_{\delta}^{+,T} & \leq   c(\|e_+\|_{*,\delta}+\widetilde C(T)+\|\psi_+\|_{*,\delta}+\|\psi_-\|_{*,\delta}\\
\nonumber & \qquad\qquad +\|\phi_+\|_{*,\delta} + \|\phi_-\|_{*,\delta} )\cdot \|\zeta_++{\mathfrak a}_+\|_{*,\delta},
\end{align}
where $\widetilde{C}(T)\to 0$ as $T\to \infty$.

We also compute
\begin{align*}
\tfrac{d}{d\tau}|_{\tau=0} &\Pi_+^{\psi_+ + \phi_+ + \tau (\zeta_++{\mathfrak a}_+),T}(v) \\
= &  \sum_i\langle \tfrac{d}{d\tau}|_{\tau=0}e_i(\psi_++\phi_+ + \tau (\zeta_++{\mathfrak a}_+),T),v\rangle e_i(\psi_++\phi_+,T)  \\
& + \sum_i \langle e_i(\psi_++\phi_+,T),v \rangle \tfrac{d}{d\tau}|_{\tau=0} e_i(\psi_++\phi_+ + \tau (\zeta_++{\mathfrak a}_+),T),
\end{align*}
where $\tfrac{d}{d\tau}|_{\tau=0}e_i(\psi_+ +\phi_+ + \tau (\zeta_++{\mathfrak a}_+),T)$ has terms of the form
$$\langle \nabla e_i(\psi_++\phi_+,T)\cdot  (\zeta_++{\mathfrak a}_+), e_j(\psi_+ +\phi_+,T)\rangle.$$ 
We then bound
\begin{align} \label{result of II prime}
\|\tfrac{d}{d\tau}|_{\tau=0}& \Pi_+^{\psi_+ + \phi_+ + \tau (\zeta_++{\mathfrak a}_+),T}(v)\|_\delta^T \leq c \|v\|_\delta^T \cdot \|\zeta_++{\mathfrak a}_+\|_{*,\delta}\\
\nonumber &\leq c(\|\psi_+\|_{*,\delta}+\|\psi_-\|_{*,\delta} + \|\phi_+\|_{*,\delta} + \|\phi_-\|_{*,\delta})\|\zeta_++{\mathfrak a}_+\|_{*,\delta}.
\end{align}
Combining Equations~\eqref{result of II} and \eqref{result of II prime} we obtain
\begin{align} \label{result of II main}
\| D^\flat_2\mathcal{I}_+^{\mathcal{J}} (\phi_-,\phi_+,T) (\zeta_++{\mathfrak a}_+)\|_{*,\delta} &\leq c(\|e_+\|_{*,\delta}+\widetilde C(T)+\|\psi_+\|_{*,\delta}+\|\psi_-\|_{*,\delta}\\
\nonumber &\qquad +\|\phi_+\|_{*,\delta} + \|\phi_-\|_{*,\delta} )\cdot\|\zeta_++{\mathfrak a}_+\|_{*,\delta}.
\end{align}
Hence the partial derivative $D^\flat_2\mathcal{I}_+$ exists.  Assuming that $\|\psi_\pm\|_{*,\delta}$, $\|\phi_\pm\|_{*,\delta}$, and $\|e_+\|_{*,\delta}$ are sufficiently small and $T\gg 0$, it follows that there exists $0<C\ll 1$ such that
\begin{align*}
\| D^\flat_2\mathcal{I}_+^{\mathcal{J}} (\phi_-,\phi_+,T) (\zeta_++{\mathfrak a}_+)\|_{*,\delta} & \leq C \|\zeta_++{\mathfrak a}_+\|_{*,\delta}.
\end{align*}

Similarly,
\begin{align*}
(D^\flat_2  &\mathcal{I}_+^{\mathcal{J}} (\phi_-,\phi_+,T)  -  D^\flat_2  \mathcal{I}_+^{\mathcal{J}} (\phi'_-,\phi'_+,T')) (\zeta_++{\mathfrak a}_+)  \\
=&\tfrac{dB(\psi_++\phi_++\tau (\zeta_++{\mathfrak a}_+),T)}{d\tau}|_{\tau=0}(v(\psi_-+\phi_-,\psi_++\phi_+,T))\\
&- \tfrac{dB(\psi_+ +\phi_+' + \tau (\zeta_++{\mathfrak a}_+),T')}{d\tau}|_{\tau=0}(v(\psi_-+\phi_-',\psi_++\phi_+',T'))\\
&+  B(\psi_+ +\phi_+,T)  ({\mathfrak l}(\tfrac{\bdry(\psi_{-,T}+\phi_{-,T})}{\bdry t}) (\zeta_{+,T}+{\mathfrak a}_{+,T}) + {\mathfrak l}(\eta_{+,T})\tfrac{\bdry\zeta_{+,T}}{\bdry t} \\
& \quad + {\mathfrak l}(\psi_{+,T}+\phi_{+,T}) (\zeta_{+,T}+{\mathfrak a}_{+,T})  + {\mathfrak l}(\psi_{+,T}+\phi_{+,T})(\nabla \zeta_{+,T}) \\
& \quad + {\mathfrak l}(\nabla (\psi_{+,T}+\phi_{+,T})) (\zeta_{+,T}+{\mathfrak a}_{+,T}) )\\
&- B(\psi_+ +\phi'_+,T')  ({\mathfrak l}(\tfrac{\bdry(\psi_{-,T'} +\phi'_{-,T'})}{\bdry t}) (\zeta_{+,T'}+{\mathfrak a}_{+,T'}) + {\mathfrak l}(\eta_{+,T'})\tfrac{\bdry\zeta_{+,T'}}{\bdry t}\\
&\quad + {\mathfrak l}(\psi_{+,T'}+\phi'_{+,T'}) (\zeta_{+,T'}+{\mathfrak a}_{+,T'})  + {\mathfrak l}(\psi_{+,T'}+\phi_{+,T'}')(\nabla \zeta_{+,T'})\\
&\quad + {\mathfrak l}(\nabla (\psi_{+,T'}+\phi'_{+,T'})) (\zeta_{+,T'}+{\mathfrak a}_{+,T'}) ).
\end{align*}
We have bounds
\begin{align*}
\|{\mathfrak l}(\tfrac{\bdry(\psi_{-,T}+\phi_{-,T})}{\bdry t})& (\zeta_{+,T}+{\mathfrak a}_{+,T})  - {\mathfrak l}(\tfrac{\bdry(\psi_{-,T'} +\phi'_{-,T'})}{\bdry t}) (\zeta_{+,T'}+{\mathfrak a}_{+,T'}) \|_\delta \\
& \leq C (\| \phi_- - \phi_-'\|_{*,\delta}^c +|T-T'|) \cdot \|\zeta_++{\mathfrak a}_+\|_{*,\delta},
\end{align*}
\begin{align*}
\| {\mathfrak l}(\eta_{+,T})\tfrac{\bdry\zeta_{+,T}}{\bdry t} -  {\mathfrak l}(\eta_{+,T'})\tfrac{\bdry\zeta_{+,T'}}{\bdry t}\|_\delta & \leq  C|T-T'| \cdot \|\zeta_+\|_{*,\delta},
\end{align*}
\begin{align*}
\|{\mathfrak l}(\psi_{+,T}+\phi_{+,T})  & (\zeta_{+,T}+{\mathfrak a}_{+,T})  -{\mathfrak l}(\psi_{+,T'}+\phi'_{+,T'})  (\zeta_{+,T'}+{\mathfrak a}_{+,T'})  \|_\delta \\
& \leq C (\| \phi_+ - \phi_+'\|_{*,\delta}^c +|T-T'|) \cdot \|\zeta_++{\mathfrak a}_+\|_{*,\delta},
\end{align*}
\begin{align*}
\|{\mathfrak l}(\psi_{+,T}+\phi_{+,T})& (\nabla \zeta_{+,T})  -{\mathfrak l}(\psi_{+,T'}+\phi_{+,T'}')(\nabla \zeta_{+,T'})\|_\delta \\
& \leq C (\| \phi_+ - \phi_+'\|_{*,\delta}^c +|T-T'|) \cdot \|\zeta_+\|_{*,\delta},
\end{align*}
\begin{align*}
\| {\mathfrak l}(\nabla (\psi_{+,T}+\phi_{+,T}))& (\zeta_{+,T}+{\mathfrak a}_{+,T}) - {\mathfrak l}(\nabla (\psi_{+,T'}+\phi'_{+,T'})) (\zeta_{+,T'}+{\mathfrak a}_{+,T'}) )\|_\delta  \\
& \leq C (\| \phi_+ - \phi_+'\|_{*,\delta}^c +|T-T'|) \cdot  \|\zeta_++{\mathfrak a}_+\|_{*,\delta}.
\end{align*}
Together with the bounds for
$$B(\psi_++\phi_+,T)-B(\psi_++\phi_+,T') ~ \mbox{ and } ~ \tfrac{dB(\psi_++\phi_++\tau(\zeta_++{\mathfrak a}_+),T)}{d\tau}|_{\tau=0},$$
we obtain
\begin{align*}
\|(D^\flat_2  &\mathcal{I}_+^{\mathcal{J}} (\phi_-,\phi_+,T)  -  D^\flat_2  \mathcal{I}_+^{\mathcal{J}} (\phi'_-,\phi'_+,T'))(\zeta_++{\mathfrak a}_+)\|_{*,\delta} \\
&\leq C(\| \phi_+ - \phi_+'\|_{*,\delta}^c   + \| \phi_- - \phi_-'\|_{*,\delta}^c + |T-T'|)\cdot  \|\zeta_++{\mathfrak a}_+\|_{*,\delta},
\end{align*}
where $C>0$. Hence $D^\flat_2\mathcal{I}_+^{\mathcal{J}}$ is in $C^0$.

Next consider
\begin{align} \label{III}
D^\flat_3\mathcal{I}_+^{\mathcal{J}}(\phi_-,\phi_+,T)(\mathfrak t)= & \tfrac{d}{d\tau}|_{\tau=0}B(\psi_+ + \phi_+, T+\tau{\mathfrak t})(v(\psi_- + \phi_-, \psi_+ +\phi_+,T))\\
\nonumber &+ B(\psi_+ +\phi_+,T) \tfrac{d}{d\tau}|_{\tau=0} v(\psi_- + \phi_-, \psi_+ +\phi_+,T+\tau {\mathfrak t}).
\end{align}
The bound
\begin{align} \label{result of III}
\|D^\flat_3\mathcal{I}_+^{\mathcal{J}}(\phi_-,\phi_+,T)(\mathfrak t)\|_{*,\delta} \leq c (\|\psi_+\|_{*,\delta}+\|\psi_-\|_{*,\delta} +\|\phi_+\|_{*,\delta} + \|\phi_-\|_{*,\delta} )\cdot |{\mathfrak t}|,
\end{align}
follows from applying Lemma~\ref{lemma: sufficiently differentiable} to the first line and ($\#_k$) to the second.  Its differentiability also follows in a similar manner and is left to the reader.

Consider the general case where $\mathcal{J}_\pm$. One can then verify that
\begin{align} \label{difference}
D_1&\mathcal{I}_+^{\mathcal{J}}(\phi_-,\phi_+,T)(\zeta_- +{\mathfrak a}_-) - D_1^\flat\mathcal{I}_+^{\mathcal{J}}(\phi_-,\phi_+,T)(\zeta_- +{\mathfrak a}_-)\\
\nonumber &= D_1^\flat\mathcal{I}_+^{\mathcal{J}}(\phi_-,\phi_+,T)\circ D\mathcal{J}_-(\phi_-,\phi_+,T)(\zeta_-+{\mathfrak a}_-).
\end{align}
The second row of Equation~\eqref{difference} can be bounded above by $C\cdot\|\zeta_- +{\mathfrak a}_-\|_{*,\delta}$ with $0<C\ll 1$ if  $\|\psi_\pm\|_{*,\delta}$ and $D\mathcal{J}$ are sufficiently small.  The situation for $D_2\mathcal{I}_+^{\mathcal{J}}$ and $D_3\mathcal{I}_+^{\mathcal{J}}$ are similar.

This completes the proof that $\mathcal{I}_+^{\mathcal{J}}$ and hence $\mathcal{I}^{\mathcal{J}}$ are in $C^1$.  Moreover, if $\op{Im}(\mathcal{J})$ lies in a sufficiently small ball about the origin and the derivative $D\mathcal{J}$ is small, then $D\mathcal{I}^{\mathcal{J}}$ is small.
\end{proof}

We are now in a position to prove Theorem~\ref{thm: gluing map smooth}.

\begin{proof}[Proof of Theorem~\ref{thm: gluing map smooth}]$\mbox{}$

\s\n
{\em $C^1$-smoothness.} Recall that we are assuming that $\mathcal{S}_\pm$ are sufficiently small neighborhoods of $u_\pm$. By Lemmas~\ref{azalea}--\ref{estimate for psi and tau} we see that the unique solution $(\psi_-,\psi_+)$ to $\Theta_\pm(\psi_-,\psi_+)=0$ for $(\phi_-,\phi_+,T)\in \mathcal{S}_-\times\mathcal{S}_+\times[T_0,\infty)$ is obtained as follows:  Let $\mathcal{J}_0=0$, $\mathcal{J}_1=\mathcal{I}^{\mathcal{J}_0}$, and $\mathcal{J}_i=\mathcal{I}^{\sum_{j=1}^{i-1} \mathcal{J}_j}$.  Also let us write $(\mathcal{J}_i)_\pm$ as the $\mathcal{H}_\pm$ component of $\mathcal{J}_i$. Then
\begin{equation}\label{eqn: psi pm}
(\psi_-,\psi_+)=\sum_{j=1}^\infty \mathcal{J}_j(\phi_-,\phi_+,T),
\end{equation}
and
\begin{equation}\label{eqn: defn of G}
G(\psi_-,\psi_+,T):= \exp_{u_*}(\beta_{-,T}(\phi_{-,T}+\psi_{-,T}) +\beta_{+,T}(\phi_{+,T}+\psi_{+,T})).
\end{equation}
Each $\mathcal{J}_j(\phi_-,\phi_+,T)$ is in $C^1$ by Lemma~\ref{lemma: I smoothness} and it remains to show a bound of type $\|D\mathcal{J}_i\|\leq C^i$, where $0<C\ll 1$.

We will check the easiest case $\mathcal{J}_2=\mathcal{I}^{\mathcal{J}_1}$, leaving the higher $\mathcal{J}_i$ to the reader. Let $\psi^1_\pm=(\mathcal{J}_1)_\pm(\phi_-,\phi_+,T)$. Then
\begin{align} \label{uno}
\Pi_+^{\phi_{+},T}D_{+,T}(\psi_{+,T}^1) =-\Pi_+^{\phi_{+},T}(&D_{+,T}(\phi_{+,T}) +  e_{+,T}   \\
\nonumber & +\mathcal L_+(\phi_{-,T},\phi_{+,T}) +\mathcal{R}_+(\phi_{+,T})).
\end{align}
Using Equation~\eqref{uno} we can write:
\begin{align} \label{eqn: J 2}
(\mathcal{J}_2)_\pm & (\phi_-,\phi_+,T) \\ \nonumber=&(\Pi_+^{\psi_+^1+\phi_+,T}D_{+,T})^{-1}(\Pi_+^{\psi_+^1+\phi_+,T}-\Pi_+^{\phi_+,T})(D_{+,T}(\psi_{+,T}^1+\phi_{+,T})+e_{+,T}\\
\nonumber &\qquad \qquad +\mathcal L_+(\psi_{-,T}^1 +\phi_{-,T},\psi_{+,T}^1+\phi_{+,T}) + \mathcal{R}_+(\psi_{+,T}^1+\phi_{+,T}) )\\
\nonumber & +(\Pi_+^{\psi_+^1+\phi_+,T}D_{+,T})^{-1}\Pi_+^{\phi_+,T} ( \mathcal L_+(\psi_{-,T}^1 +\phi_{-,T},\psi_{+,T}^1+\phi_{+,T})- \mathcal L_+(\phi_{-,T},\phi_{+,T})\\
\nonumber & \qquad \qquad +\mathcal{R}_+(\psi_{+,T}^1+\phi_{+,T}) -\mathcal{R}_+(\phi_{+,T})).
\end{align}

Calculations similar to those of Lemma~\ref{lemma: I smoothness} imply that:
\begin{align*}
\| D \mathcal{J}_2&(\phi_-,\phi_+,T) (\zeta_-+{\mathfrak a}_-,\zeta_++{\mathfrak a}_+,{\mathfrak t})\|_{*,\delta}\\
& \leq c (\|\psi_-^1\|_{*,\delta}+\|\psi_+^1\|_{*,\delta})(\|\zeta_-+{\mathfrak a}_-\|_{*,\delta}+ \|\zeta_++{\mathfrak a}_+\|_{*,\delta}+|{\mathfrak t}|).
\end{align*}
[Sketch of calculation: The derivative of Equation~\eqref{eqn: J 2} consists of six terms since Equation~\eqref{eqn: J 2} is the sum of two terms, each of which is a product of three terms. The six terms can all be estimated in the same manner and we calculate some representative terms. We estimate
\begin{align}
\label{temp1} \tfrac{d}{d\tau}|_{\tau=0}( \mathcal L_+(\psi_{-,T}^1 +\phi_{-,T}+&\tau \zeta_{-,T},\psi_{+,T}^1+\phi_{+,T}+\tau\zeta_{+,T})\\
\nonumber &- \mathcal L_+(\phi_{-,T}+\tau\zeta_{-,T},\phi_{+,T}+\tau\zeta_{+,T}))
\end{align}
as follows.  Observe that $\mathcal L_+$ is bilinear in each variable.  Hence we have an expression of the form
\begin{align*}
(\psi_- + \phi_-+\tau\zeta_-)(\psi_+ + \phi_+ + \tau\zeta_+)-(\phi_-+\tau\zeta_-)(\phi_+ +\tau\zeta_+)\\
= (\psi_-+\phi_-)(\psi_+ + \phi_+) -\phi_-\phi_+ + \tau(\zeta_-\psi_+ + \zeta_+\psi_-) +\mbox{h.o.}
\end{align*}
Hence Equation~\eqref{temp1} is bounded above by
$$ c(\|\psi_-^1\|_{*,\delta}+\|\psi_+^1\|_{*,\delta})(\|\zeta_-\|_{*,\delta}+\|\zeta_+\|_{*,\delta}).$$
Similarly,
\begin{align}
\mathcal L_+(\psi_{-,T}^1 +\phi_{-,T},\psi_{+,T}^1+\phi_{+,T})- \mathcal L_+(\phi_{-,T},\phi_{+,T})
\end{align}
is bounded above by
$$c(\|\psi_-^1\|_{*,\delta}+\|\psi_+^1\|_{*,\delta})(\|\phi_-\|_{*,\delta}+\|\phi_+\|_{*,\delta}).$$
This concludes the sketch.]
We also calculate:
\begin{align*}
\|(D \mathcal{J}_2&(\phi_-,\phi_+,T)-  D \mathcal{J}_2(\phi_-',\phi_+',T'))(\zeta_-+{\mathfrak a}_-,\zeta_++{\mathfrak a}_+,{\mathfrak t})\|_{*,\delta}\\
& \leq c (\|\psi_-^1\|_{*,\delta}+\|\psi_+^1\|_{*,\delta})(\|\phi_+-\phi'_+\|^c_* + \|\phi_--\phi_-'\|^c_*+|T-T'|)\\
& \qquad \cdot (\|\zeta_-+{\mathfrak a}_-\|_{*,\delta} + \|\zeta_++{\mathfrak a}_+\|_{*,\delta}+|{\mathfrak t}|).
\end{align*}
This completes the proof of Theorem~\ref{thm: gluing map smooth}.

\s\n
{\em Local embedding.} It suffices to show that $G$ is a local $C^1$-embedding, i.e., $DG$ is an isomorphism at any $(\phi_-,\phi_+,T)$, provided $T_0\gg 0$. (It is not hard to see that, for any $d>0$, there exists $T_0\gg 0$ such that if $(\phi_-,\phi_-,T)$ and $(\phi_-',\phi_+',T')$ are a distance $d$ apart, then $G(\phi_-,\phi_+,T)$ and $G(\phi_-',\phi_+',T')$ cannot be equal.)

First consider the map
$$H:(\phi_-,\phi_+,T)\mapsto (\phi_-,\phi_+,T)+\sum_{j=1}^\infty(\mathcal{J}_j(\phi_-,\phi_+,T),0).$$
Since its leading term is the identity map and the subsequent terms $\mathcal{J}_j$  have derivatives bounded by a constant $C^j$ with $0<C\ll 1$, it follows that $DH$ is invertible at any $(\phi_-,\phi_+,T)$, provided $\widetilde\varepsilon>0$ is sufficiently small and $T_0>0$ is sufficiently large.

Next, instead of $G$ we consider the map
$$\widetilde G: (\phi_-,\phi_+,T)\mapsto \beta_{-,T}(\phi_{-,T}+\psi_{-,T})+\beta_{+,T}(\phi_{+,T}+\psi_{+,T})$$
without the $\exp$.  If $T_0\gg 0$, then from the considerations of the previous paragraph $D\widetilde G(\phi_-,\phi_+,T)|_{T_{\phi_-}\mathcal{S}_-\times T_T[T_0,\infty)}$ is very close to
$$(\zeta_-+{\mathfrak a}_-,{\mathfrak t})\mapsto \zeta_{-,T}+{\mathfrak a}_{-,T}-{\mathfrak t}\tfrac{\bdry }{\bdry s}$$
when the right-hand side is restricted to $s\leq -(T+T_0)$.  Since $T_{\phi_-}\mathcal{S}_-$ and $\tfrac{\bdry}{\bdry s}$ are independent, we see that $D\widetilde G$ is injective, which in turn implies that $G$ is a local $C^1$-embedding.
\end{proof}

\subsection{Surjectivity of the gluing map} \label{subsection: surjectivity}

\begin{thm} \label{thm: surjectivity}
Given compact subsets $\mathcal{K}_\pm \subset \mathcal{S}_\pm$, there exist $T_0\gg 0$ and $\widetilde \delta>0$ such that, for all curves $u_0$ that are $\widetilde\delta$-close to breaking into $(u'_-,u'_+,T')\in \mathcal{K}_-\times \mathcal{K}_+\times[2T_0,\infty)$, there exists a triple $(\phi_-,\phi_+,T)\in \mathcal{S}_-\times\mathcal{S}_+\times[T_0,\infty)$ such that $G(\phi_-,\phi_+,T)=u_0$.
\end{thm}

\begin{proof}
Let $T_0>0$ be sufficiently large and $\widetilde\delta>0$ be sufficiently small. Let $u_0$ be $\widetilde\delta$-close to breaking into $(u'_-,u'_+,T')$. After possibly translating $u_0$ in the $\R$-direction and slightly modifying $T'$, there exists a decomposition of the domain $(\dot F_0,j_0)$ of $u_0$ as:
$$\dot F_0 = (F_0)^\circ_+ \cup (F_0)^\circ_- \cup A_{[-T',T']},$$
where
\begin{itemize}
\item $\mathcal{A}_\varepsilon$ is the $\varepsilon$-thin annular part of $\dot F_0$ (with respect to the complete finite-volume hyperbolic metric compatible with $j_0$) corresponding to the neck that is being stretched;
\item $A_{[-T',T']}= u_0^{-1}(\{-T'\leq s\leq T'\})\cap \mathcal{A}_\varepsilon$, $\op{int} ((F_0)^\circ_+)\cup \op{int}( (F_0)^\circ_-)=\dot F_0- A_{[-T',T']}$, and  $(F_0)^\circ_+$, $(F_0)^\circ_-$ correspond to the top and bottom levels;
\item writing $\bdry \mathcal{A}_\varepsilon= C_+\sqcup C_-$ corresponding to the top and bottom levels, $s_{C_+}(u_0)$ and $s_{C_-}(u_0)$ are as defined in Section~\ref{subsection: the bundle E'} and we assume that $s_{C\pm}(u_0)= \pm (T'+T_0)$.
\end{itemize}

Next consider
$$(\dot F'_+,j'_+):= ((F_0)^\circ_+,j_0)\cup (A_{(-\infty,T']},j_{std}),$$
$$(\dot F'_-,j'_-):= ((F_0)^\circ_-,j_0)\cup (A_{[-T',\infty)},j_{std}),$$
such that $\dot F'_\pm$ extends $(F_0)^\circ_\pm\cup A_{[-T',T']}$.
For $\widetilde \delta>0$ small and $T_0\gg 0$, $j'_\pm$ is close to the domain complex structure of $u'_\pm$.  We may replace $u'_\pm$ by $u''_\pm\in \mathcal{S}_\pm$ with domain $(\dot F'_\pm, j'_\pm)$ which is close to $u'_\pm$ and such that $u_0$ is $\widetilde\delta$-close to breaking into $(u''_-,u''_+,T')$, after possibly slightly enlarging $\widetilde \delta$.

Let $u^{(1)}_*$ be the pregluing of $u^{(1)}_-=u''_-$ and $u^{(1)}_+=u''_+$ with gluing parameter $T^{(1)}=T'$, as defined in Section~\ref{subsection:pregluing}, and let
$$\dot F^{(1)}=(F^{(1)})^\circ _+\cup (F^{(1)})^\circ_-\cup A_{[-T^{(1)},T^{(1)}]}$$
be the domain of $u^{(1)}_*$, defined analogously.

We first solve for $\psi_\pm^{(1)}$ in
\begin{align} \label{eqn: first step}
\Phi \overline\bdry \exp_{u^{(1)}_*}(\beta_{-,T^{(1)}}\psi_{-,T^{(1)}}^{(1)} +\beta_{+,T^{(1)}}\psi_{+,T^{(1)}}^{(1)})\in \Phi E'_{u_0},
\end{align}
as in Section~\ref{subsection: defn of gluing map}, where $\Phi$ is the parallel transport to $\wedge^{0,1}(u^{(1)}_*)^* T(\R\times M)$.
Let us write $u^{(1)}= \exp_{u^{(1)}_*}(\beta_{-,T^{(1)}}\psi^{(1)}_{-,T^{(1)}} +\beta_{+,T^{(1)}}\psi^{(1)}_{+,T^{(1)}})$.  Here the superscript $(1)$ indicates that we are in the first round of an iterative scheme. Note that $\|\psi_\pm^{(1)}\|_{*,\delta}\to 0$ as $T_0\to\infty$ by Lemma~\ref{lemma: T to infty}.

Next suppose $\phi^{(1)}$ satisfies
$$u_0=\exp_{u^{(1)}_*}(\beta_{-,T^{(1)}}\psi_{-,T^{(1)}}^{(1)} +\beta_{+,T^{(1)}}\psi_{+,T^{(1)}}^{(1)} +\phi^{(1)}).$$
We decompose $\phi^{(1)}=\phi_{-,T^{(1)}}^{(1)} + \phi_{+,T^{(1)}}^{(1)}$, such that
\begin{itemize}
\item $\phi_{+,T^{(1)}}^{(1)}=\phi^{(1)}$ on $(F')^\circ_+$, $\phi_{+,T^{(1)}}^{(1)}=0$ on $(F')^\circ_-$, and $$\phi_{+,T^{(1)}}^{(1)}(s,t)=\phi^{(1)}(s,t)\beta(\tfrac{s+hr}{2hr})$$ on $A_{[-T^{(1)},T^{(1)}]}$;
\item $\phi_{-,T^{(1)}}^{(1)}=0$ for $(F')^\circ_+$, $\phi_{-,T^{(1)}}^{(1)}=\phi$ on $(F')^\circ_-$, and $$\phi_{-,T^{(1)}}^{(1)}(s,t)=\phi^{(1)}(s,t)(1-\beta(\tfrac{s+hr}{2hr}))$$ on $ A_{[-T^{(1)},T^{(1)}]}$.
\end{itemize}
Here $\phi_\pm^{(1)}\in \mathcal{H}_\pm\oplus \R^{2l^\pm}$ and we are writing $\|\phi_\pm^{(1)}\|_{*,\delta}$ for the sum of $\|\cdot\|_{*,\delta}$ on $\mathcal{H}_\pm$ and the standard norm on $\R^{2l^\pm}$ as before.

We then solve for $(\phi_\pm')^{(1)}$ in
$$\overline\bdry\exp_{u_\pm^{(1)}}(\phi_\pm^{(1)} +(\phi_\pm')^{(1)})\in\Phi_\pm^{\phi_\pm^{(1)}+(\phi_\pm')^{(1)}}E_{\phi_\pm^{(1)}+(\phi'_\pm)^{(1)}},$$
where $(\phi_\pm')^{(1)}\in \mathcal{H}_\pm$, $\Phi_\pm^{\phi_\pm^{(1)}+(\phi_\pm')^{(1)}}$ is the parallel transport
$$\wedge^{0,1}(\exp_{u_\pm^{(1)}}(\phi_\pm^{(1)} +(\phi_\pm')^{(1)}))^* T(\R\times M)\xrightarrow\sim \wedge^{0,1}(u_\pm^{(1)})^* T(\R\times M)$$
and $E_{\phi_\pm^{(1)} +(\phi'_\pm)^{(1)}}$ is the obstruction bundle for $\exp_{u_\pm^{(1)}}(\phi_\pm^{(1)} +(\phi_\pm')^{(1)})$.

\begin{claim} \label{first}
$\|(\phi_\pm')^{(1)}\|_{*,\delta}\leq C \|\phi_\pm^{(1)}\|_{*,\delta}$, where $0<C\ll 1$.
\end{claim}

\begin{proof}[Proof of Claim~\ref{first}.]
This is proved using the contraction mapping theorem as in Section~\ref{subsection: defn of gluing map} and relies on the fact that $\overline\bdry\exp_{u_\pm^{(1)}}(\phi_\pm^{(1)})$ is close to an element in $\Phi_\pm^{\phi_\pm^{(1)}}E_{\phi_\pm^{(1)}}$.  We will indicate a proof of this fact under the simplified assumption that $E_{\phi_\pm^{(1)}}=0$; the general case is only more complicated in notation.

First observe that $\overline\bdry \exp_{u_+^{(1)}}(\phi_+^{(1)})|_{s\geq -T_0}$ is the $s=-(T^{(1)}+T_0)$ translate of $\overline\bdry \exp_{u_*^{(1)}}(\phi_{+,T^{(1)}}^{(1)})|_{s\geq T^{(1)}}$ and that
\begin{equation} \label{eqn: dbar}
\overline\bdry \exp_{u_*^{(1)}}(\phi_{+,T^{(1)}}^{(1)}+\psi_{+,T^{(1)}}^{(1)})=0,\quad \overline\bdry \exp_{u_*^{(1)}}(\psi_{+,T^{(1)}}^{(1)})=0,  \quad \overline\bdry \exp_{u_*^{(1)}}(0)=0.
\end{equation}
We claim that
\begin{equation} \label{eqn: bandaid}
\|\overline\bdry \exp_{u_*^{(1)}}(\phi_{+,T^{(1)}}^{(1)})|_{s\geq -T_0}\|_\delta^{+,T^{(1)}}\leq c\|\psi_{+,T^{(1)}}^{(1)}\|_{*,\delta}^{+,T^{(1)}} \cdot \|\phi_{+,T^{(1)}}^{(1)}\|_{*,\delta}^{+,T^{(1)}}
\end{equation}
for some constant $c>0$. We write $u,\psi,\phi$ for $u_*^{(1)},\psi_{+,T^{(1)}}^{(1)},  \phi_{+,T^{(1)}}^{(1)}$ and use local coordinates $(\sigma,\tau)$ for the domain.  Assume that
\begin{enumerate}
\item[($\dagger$)] $\exp_u(\phi)=u+\phi$, $\exp_u(\psi)=u+\psi$, $\exp_u(\phi+\psi)=u+\phi+\psi$,
\end{enumerate}
e.g., the target metric is flat.
Using Equation~\eqref{eqn: dbar} and writing $v=u+\psi$ we obtain:
\begin{align}
\label{mouse1} \overline\bdry (u+\phi) &= \tfrac{\bdry\phi}{\bdry \sigma} +J(u)\tfrac{\bdry \phi}{\bdry \tau} + \nabla J(u)(\phi)\tfrac{\bdry(u+\phi)}{\bdry \tau} + Q(\phi)\tfrac{\bdry(u+\phi)}{\bdry \tau},\\
\label{mouse2} 0=\overline\bdry (v+\phi) &= \tfrac{\bdry\phi}{\bdry \sigma} +J(v)\tfrac{\bdry \phi}{\bdry \tau} + \nabla J(v)(\phi)\tfrac{\bdry(v+\phi)}{\bdry \tau} + Q(\phi)\tfrac{\bdry(v+\phi)}{\bdry \tau},
\end{align}
where we are writing $J(u+\phi)=J(u) + \nabla J(u)(\phi) + Q(\phi)$ and $Q(\phi)$ is a quadratic term. Taking the difference between the two, we can locally bound
$\|\overline\bdry (u+\phi)\|$ by terms of the form $c \|\psi\|_* \|\phi\|_*$. In general, when ($\dagger$) does not hold, we may take $\exp_u(\phi+\psi)=u+\phi+\psi + Q(\phi,\psi)$, where $Q(\phi,\psi)$ is a pointwise function of $\phi$ and $\psi$ times a pointwise bilinear function of $\phi$ and $\psi$, and we obtain the local bound $\|\overline\bdry (u+\phi)\|\leq c \|\psi\|_* \|\phi\|_*$.  The standard exponential decay estimates then yield Equation~\eqref{eqn: bandaid}.

Next we bound $\overline\bdry \exp_{u_+^{(1)}}(\phi_+^{(1)})$ on $-(T^{(1)}+T_0)-hr\leq s\leq -T_0$, which is the $s=-(T^{(1)}+T_0)$ translate of $\overline\bdry \exp_{u_*^{(1)}}(\phi_{+,T^{(1)}}^{(1)})|_{-hr\leq s\leq T^{(1)}}$.  Applying the same procedure as above\footnote{Note that we are dealing with a flat metric here, so ($\dagger$) holds.} with
$$u=u_*^{(1)}, \quad v= u_*^{(1)}+ \beta_{-,T^{(1)}}\psi_{-,T^{(1)}}^{(1)} +\beta_{+,T^{(1)}}\psi_{+,T^{(1)}}^{(1)},$$
$$\psi=\beta_{-,T^{(1)}}\psi_{-,T^{(1)}}^{(1)} +\beta_{+,T^{(1)}}\psi_{+,T^{(1)}}^{(1)}, \quad \phi=\phi_{+,T^{(1)}}^{(1)}$$
as well as exponential bounds on annuli from \cite[Lemma 2.3]{HT2} we obtain
\begin{equation} \label{eqn: bandaid2}
\|\overline\bdry \exp_{u_+^{(1)}}(\phi_+^{(1)})|_{-(T^{(1)}+T_0)-hr\leq s\leq -T_0}\|_\delta^{+,T^{(1)}}\leq C(T^{(1)})\cdot \|\phi_{+,T^{(1)}}^{(1)}\|_{*,\delta}^{+,T^{(1)}},
\end{equation}
where $C(T^{(1)})\to 0$ as $T^{(1)}\to\infty$.  Note that $\phi_+^{(1)}|_{s\leq -(T^{(1)}+T_0) -hr}=0$ by definition.

We then invert the error term $\overline\bdry \exp_{u_+^{(1)}}(\phi_+^{(1)})$ using Estimates~\eqref{eqn: bandaid} and \eqref{eqn: bandaid2} and the contraction mapping theorem.  This proves the claim.
\end{proof}

Also observe that $\|\phi_\pm^{(1)}\|_{*,\delta}$ is bounded above by a fixed constant times $\widetilde \delta$; this follows from elliptic bootstrapping as in Lemma~\ref{lemma: psi pm smooth}.

We iterate the above procedure with $(u^{(1)}_-,u^{(1)}_+,T^{(1)})$ replaced by
$$(u^{(2)}_-,u^{(2)}_+,T^{(2)})=(\exp_{u^{(1)}_-}^\circ(\phi_-^{(1)}+(\phi_-')^{(1)}), \exp_{u^{(1)}_+}^\circ(\phi_+^{(1)} +(\phi_+')^{(1)}),T^{(1)}+T_-+T_+).$$
Here $\exp_{u^{(1)}_\pm}^\circ(\phi_\pm^{(1)}+(\phi_\pm')^{(1)})$ is $\exp_{u^{(1)}_\pm}(\phi_\pm^{(1)}+(\phi_\pm')^{(1)})$ normalized by shifting $T_\pm$ units in the $\pm s$-direction; the amount that we shift is determined by the choice of slice (i.e., the choice of representative $\exp_{u^{(1)}_\pm}^\circ(\phi_\pm^{(1)}+(\phi_\pm')^{(1)})$ of $\V_\pm/\R$). Let $u_*^{(2)}$ be the resulting pregluing.

We can similarly verify that, if we replace $u_*^{(1)}$ and $\psi_\pm^{(1)}$ by $u_*^{(2)}$ and $\psi_\pm^{(1)}+\psi_\pm^{(2)}$ in Equation~\eqref{eqn: first step}, then:
\begin{equation} \label{third}
\|\psi_\pm^{(2)}\|_{*,\delta}\leq C \|\phi_\pm^{(1)}\|_{*,\delta},
\end{equation}
where $0<C\ll 1$.
The bounds from Claim~\ref{first} and Estimate \eqref{third} imply that the error between $u_0$ and $u^{(i)}$ is of order $C^i$ after the $i$th iteration.  The proof then follows.
\end{proof}

This concludes the proof of Theorem~\ref{thm: gluing}(A).

\subsection{Sketch of proof of Theorem~\ref{thm: gluing}(B)}\label{subsection: proof of theorem gluing part b}

 Theorem~\ref{thm: gluing}(B) is a consequence of the following observations:  By Section~\ref{subsection:pregluing}, the pregluing only depends on the constants $h,r,T_0$, the cutoff function $\beta$, and the values of $s_{\pm,i}(\F_j,u_j)$, where $\{(\F_j,u_j)~|~j=1,\dots,m\}$, is the collection of curves we are trying to glue. In the gluing setup from Section~\ref{subsection: gluing}, $\Phi$ only depends on the Riemannian metric on $\R\times M$; $e_{\pm,\tau}$ is simply $\overline \bdry u_{\pm T}$; and $\mathcal{L}_\pm$ and $\mathcal{R}_\pm $ are given by formulas Equation~\eqref{equation: mathcal L} and \eqref{equation: mathcal R} and hence do not depend on the choice of compact subset $K_j'$.  Finally, in Section~\ref{subsection: defn of gluing map}, the definitions of $\mathcal{H}_\pm^{\perp, \psi_\pm, T}$, $\Pi_\pm^{\psi_\pm,T}$, $(\Pi_\pm^{\psi_\pm,T}D_{\pm,T})^{-1}$, and hence of $\mathcal{I}_\pm$ also do not depend on the choice of $K_j'$.

\subsection{Sketch of proof of Theorem~\ref{thm: iterated gluing}} \label{subsection: multiple gluing}

In this subsection we sketch the proof of Theorem~\ref{thm: iterated gluing}.  Without loss of generality assume that $m=3$.  We are comparing $G_{(1,2,3)}$ and $G_{((1,2),3)}\circ (G_{(1,2)},\op{id})$. We are gluing $u_1\cup u_2\cup u_3$, where $(\mathcal{F}_i,u_i)$, $i=1,2,3$, is a representative of $V_i/\R$. For ease of notation we assume that $u_2$ and $u_3$ only have one negative end.

\subsubsection{Description of $G_{(1,2,3)}$} \label{subsubsection: description of G123}

Let $u_*^{(123)}$ be the pregluing of $u_1,u_2,u_3$ with gluing parameters $T_{12}, T_{23}$ (see Definition~\ref{defn: pregluing}). \cb We solve for $\psi_1^{(123)},\psi_2^{(123)},\psi_3^{(123)}$ which are analogs of $\psi_-,\psi_+$ in the equations
$$\Theta_1^{(123)}(\psi_1^{(123)},\psi_2^{(123)})=0,$$
$$\Theta_2^{(123)}(\psi_1^{(123)},\psi_2^{(123)},\psi_3^{(123)})=0,$$
$$\Theta_3^{(123)}(\psi_2^{(123)},\psi_3^{(123)})=0,$$
which are analogs of $\Theta_\pm(\psi_-,\psi_+)=0$. This yields
$$u_{(123)}=G_{(1,2,3)}(u_1,u_2,u_3,T_{12},T_{23}).$$

We also write $\psi_i^{(123),\dagger}$, $i=1,2,3$, for the first approximation of $\psi_i^{(123)}$ when applying the contraction mapping principle; this corresponds to $\mathcal{J}_1$ in the proof of Theorem~\ref{thm: gluing map smooth} and is obtained by using the linearized $\overline\bdry$-operator to invert the errors that arise from the pregluing of $u_1,u_2,u_3$. Estimates similar to those of Sections~\ref{subsection: defn of gluing map} imply the existence of functions $\widetilde C(T)$ and $\widetilde D(T)$, $T:=\min(T_{12},T_{23})$, such that
\begin{align} \label{alpha}
\|\psi_i^{(123)}-\psi_i^{(123),\dagger}\|_{*,\delta}& \leq \widetilde C(T) \|\psi_i^{(123),\dagger}\|_{*,\delta},\\
\nonumber \|\psi_i^{(123),\dagger}\|_{*,\delta} & \leq \widetilde D(T),
\end{align}
and $\widetilde C(T),\widetilde D(T)\to 0$ as $T\to\infty$.

\subsubsection{Description of $G_{((1,2),3)}\circ (G_{(1,2)},\op{id})$}

Let $u_*^{(12)}$ be the pregluing of $u_1,u_2$ with gluing parameter $T_{12}$.  We solve for $\psi_1^{(12)},\psi_2^{(12)}$ in
$$\Theta_1^{(12)}(\psi_1^{(12)},\psi_2^{(12)})=0, \quad \Theta_2^{(12)}(\psi_1^{(12)},\psi_2^{(12)})=0,$$
which yields $u_{(12)}=G_{(1,2)}(u_1,u_2,T_{12})$. Similarly we define $\psi_i^{(12),\dagger}$, $i=1,2$.  We have
\begin{align}\label{beta}
\|\psi_i^{(12)}-\psi_i^{(12),\dagger}\|_{*,\delta}& \leq \widetilde C(T_{12}) \|\psi_i^{(12),\dagger}\|_{*,\delta},\\
\nonumber \|\psi_i^{(12),\dagger}\|_{*,\delta} &\leq \widetilde D(T_{12}),
\end{align}
where $\widetilde C(T_{12}), \widetilde D(T_{12})\to 0$ as $T_{12}\to\infty$.

Next let $u_*^{((12)3)}$ be the pregluing of $u_{(12)},u_3$ with gluing parameter $T_{23}$. We solve for $\psi_{(12)}^{((12)3)}, \psi_3^{((12)3)}$ in
$$\Theta_{(12)}^{((12)3)}(\psi_{(12)}^{((12)3)},\psi_3^{((12)3)})=0,\quad \Theta_{3}^{((12)3)}(\psi_{(12)}^{((12)3)},\psi_3^{((12)3)})=0,$$
which yields
$$u_{((12)3)}=G_{((1,2),3)}(u_{(12)},u_3,T_{23}).$$
Similarly we define $\psi_{(12)}^{((12)3),\dagger}, \psi_3^{((12)3),\dagger}$. For $i=(12)$ and $i=3$ we have
\begin{align}
\label{delta} \|\psi_i^{((12)3)}-\psi_i^{((12)3),\dagger}\|_{*,\delta}& \leq \widetilde C(T) \|\psi_i^{((12)3),\dagger}\|_{*,\delta}, \\
\nonumber \|\psi_i^{((12)3),\dagger}\|_{*,\delta} &\leq \widetilde D(T),
\end{align}
where $\widetilde C(T), \widetilde D(T) \to 0$ as $T\to \infty$.

\subsubsection{Conclusion}

The $C^0$-closeness follows from Estimates~\eqref{alpha}, \eqref{beta}, and \eqref{delta}. Roughly speaking, the errors from the pregluing go to zero as $T_{12},T_{23}\to \infty$.  The $C^1$-closeness is left to the reader and follows from Estimates~\eqref{result of I}, \eqref{result of II main}, and \eqref{result of III} from Section~\ref{subsection: C1 smoothness}. In words, the derivatives of the errors from the pregluing go to zero as $T_{12},T_{23}\to \infty$ as well.

\section{Construction of semi-global Kuranishi structures} \label{section: construction of semi-global}

{\em Starting from this section we specialize to contact homology.} 

\nom[Mi]{$\mathcal{M}_i$}{Moduli spaces $\mathcal{M}_{J}^{\op{ind}=k_i}(\dot F_i,\R\times M;\gamma_{i,+};\bs\gamma_{i,-})$ indexed using the action $\vartheta$-sorting}
Let $\mathcal{M}_1,\mathcal{M}_2,\dots,\mathcal{M}_\rho$ be a sequence of \coblu distinct \cb moduli spaces
$$\mathcal{M}_i=\mathcal{M}_{J}^{\op{ind}=k_i}(\dot F_i,\R\times M;\gamma_{i,+};\bs\gamma_{i,-}),$$
such that $\dot F_i$ is a planar surface and each component of each level of a building in $\bdry\mathcal{M}_i$ is either a trivial cylinder or in $\mathcal{M}_j$ with $j<i$.  (A specific choice will be given in Section~\ref{subsubsection: complexity}.)   The ends of the moduli spaces $\mathcal{M}_i$ are $\vartheta$-sorted using the action. Also let $\mathcal{S}=\{\mathcal{M}_1,\mathcal{M}_2,\dots,\mathcal{M}_\rho\}$ and \coblu $G_i=G(\dot F_i,\R\times M;\gamma_{i,+};\bs\gamma_{i,-})$. \cb
\nom[Gi]{$G_i$}{Automorphisms group $G(\dot F_i,\R\times M;\gamma_{i,+};\bs\gamma_{i,-})$ generated by puncture reorderings and marker rotations}

\begin{convention}\label{convention for moduli spaces}
In view of Convention~\ref{convention for boundary}, we include moduli spaces $\mathcal{M}_i$ in the list even when they are empty, provided $\bdry\mathcal{M}_i\not=\varnothing$.
\end{convention}

  This section is organized as follows: In Section~\ref{subsection: CH trees} we introduce combinatorial objects called {\em contact homology (or CH) trees of $\op{(Symp)}$ type} that encode the {\coblu gluing} data for the various strata of the SFT-compactified moduli spaces that appear in the definition of $\bdry^2=0$ and in Section~\ref{subsection: symmetries of moduli spaces} we describe automorphisms of the moduli spaces and their effect on the various strata.  Section~\ref{subsection: definition of semi-global Kuranishi structures} is devoted to the definition of semi-global Kuranishi structures and multisections of semi-global Kuranishi structures, and culminates in Proposition~\ref{prop: tomatoes}, which describes the weighted branched manifold extracted from a semi-global Kuranishi structure and a multisection. Sections~\ref{subsection: trimming} and \ref{subsection: universal choice of gluing parameter} prove some preparatory lemmas and the semi-global Kuranishi structures and multisections for $\bdry^2=0$ are constructed in Sections~\ref{subsection: overview of construction}--\ref{subsection: verification}.  In Sections~\ref{subsection: cob case} and \ref{subsection: chain homotopy case} we explain the modifications needed for chain maps and chain homotopy.

\subsection{CH trees} \label{subsection: CH trees}

\subsubsection{Definition of CH trees}  \label{subsubsection: defn of CH trees}


\nom[T]{$T=(V(T),E(T)=G(T)\sqcup F(T),O(T))$}{Contact homology (CH) tree; $V(T)$ are the vertices, $E(T)$ are the edges, $G(T)$ are the glued edges, $F(T)$ are the free edges, and $O(T)$ is the reordering function}
\nom[VT]{$V(T)$}{Vertices of tree $T$}
\nom[ET]{$E(T)$}{Edges of tree $T$}
\nom[GT]{$G(T)$}{Glued edges of tree $T$}
\nom[FT]{$F(T)$}{Free edges of tree $T$}

\begin{defn} [Pre-CH tree of $\op{(Symp)}$ type]\label{defn: pre-CH trees} 
A {\em pre-CH tree of $\op{(Symp)}$ type} (or simply a {\em pre-CH tree}) is a directed tree\footnote{Trees are assumed to be connected. A collection of trees will be called a {\em forest}.} \footnote{A directed tree $T$ has a vertex, called the {\em root}, and all the edges of $T$ point away from the root. The root corresponds to the topmost level of an SFT building, so the edges point down in the SFT sense.} 
$$T=(V(T), E(T)=G(T)\sqcup F(T)),$$
where the sets of vertices $V(T)$ and edges $E(T)$ are finite and the following hold:
\begin{enumerate}
\item $V(T)$ is labeled using the {\em vertex labeling function}
$$l_{V(T)}:V(T)\to\{1,\dots,\rho\}.$$
\NoIndent{\noindent The vertex labeling function assigns a moduli space $\mathcal{M}_{l_{V(T)}(v)}$ to each vertex $v\in V(T)$ and is not necessarily injective or surjective.}
\item $E(T)$ consists of the {\em glued edges} $G(T)$ and the {\em free edges} $F(T)$. 
The edges of $G(T)$ have initial and terminal points but the edges of $F(T)$ only have an initial point.
\NoIndent{\noindent Given an edge $e\in E(T)$, we write $i(e)$ for the initial point of $e$ and $t(e)$ for the terminal point of $e$ (if it exists).  Given a vertex $v\in V(T)$, we write $E(T)_{v}$ for the set of all edges $e$ with $i(e)=v$, $G(T)_{v}=E(T)_v\cap G(T)$, and $F(T)_v=E(T)_v\cap F(T)$.}
\item There exists an {\em orbit assignment map} $\op{Orb}_T: E(T) \to \mathcal P_\alpha$ 
such that:

\be
\item for each $v\in V(T)$, $\op{Orb}_T({E(T)_v})$, counted with multiplicity and modulo the ordering, agrees with $\bs \gamma_-$ where $\mathcal M_{l_{V(T)}(v)} = \mathcal M(\gamma_+; \bs\gamma_-)$;
\item for each $e\in G(T)$, $\op{Orb}_T(e)$ agrees with the positive end $\gamma_+$ of $\mathcal M_{l_{V(T)}(t(e))}=\mathcal{M}(\gamma_+;\bs\gamma_-)$.
\ee
\NoIndent{\noindent In other words, each edge $e\in G(T)$ corresponds to a single gluing of an interior semi-global Kuranishi chart $\V_{l_{V(T)}(i(e))}$ on the upper level with $\V_{l_{V(T)}(t(e))}$ on the lower level.  

We also define the {\em root orbit} $\mathfrak{r}(T)$ as the positive end $\gamma_+$ of $\mathcal M_{l_{V(T)}(v)}=\mathcal{M}(\gamma_+;\bs\gamma_-)$ where $v$ is the root vertex.}
\end{enumerate}
\end{defn}

\nom[root orbit]{$\mathfrak{r}(T)$}{the positive end $\gamma_+$ of $\mathcal M_{l_{V(T)}(v)}=\mathcal{M}(\gamma_+;\bs\gamma_-)$ where $v$ is the root vertex.}

\begin{defn} [CH tree of $\op{(Symp)}$ type]\label{defn: CH trees} 
A {\em CH tree of $\op{(Symp)}$ type} (or simply a {\em CH tree}) is a pre-CH tree $T$ together with a {\em reordering function}, i.e., a bijective function 
$$O(T):F(T)\to\{1,\dots,|F(T)|\}$$  
that {\em respects the ordering $\vartheta$,} i.e., for any $e_1,e_2\in F(T)$, if $\vartheta(\op{Orb}_T(e_1))< \vartheta(\op{Orb}_T(e_2))$, then $O(T)(e_1)< O(T)(e_2)$.
\end{defn}

\begin{rmk} \label{rmk: unique reordering function}  If $\op{Orb}_T|_{F(T)}$ is injective, then the reordering function $O(T)$ is completely determined by $\vartheta$.
\end{rmk}

{\coblu A CH tree $T$ gives a prescription for gluing elements of $\mathcal{M}_{l_{V(T)}(v)}$ to give an element of the moduli space $\mathcal{M}_{T\git T}$, where $T\git T$ is the contraction given in Section~\ref{subsubsection: contraction of CH trees}.}

\subsubsection{Isomorphisms of CH trees}

\begin{defn} \label{def: isomorphism of CH trees}
An {\em isomorphism $\theta:T\to T'$ of pre-CH trees} is an isomorphism of directed trees which
\begin{enumerate}[label = (\arabic*)]
\item \label{CH1} sends $V(T), G(T), F(T)$ to $V(T'),G(T'), F(T')$; and
\item \label{CH2} preserves the labels of vertices, i.e., $l_{V(T)}=l_{V(T')}\circ \theta|_{V(T)}$.
\end{enumerate}
An {\em isomorphism $\theta:T\to T'$ of CH trees} is an isomorphism of pre-CH trees which
\begin{enumerate}
\item[(3)] \label{CH3} maps $O(T)$ to $O(T')$, i.e., $O(T)=O(T')\circ\theta|_{F(T)}.$
\end{enumerate}
We denote the isomorphism class of $T$ of CH trees by $[T]$.
\end{defn}

\begin{rmk} \label{rmk: loquats} $\mbox{}$
\be
\item[(i)] Given an isomorphism $\theta: T\to T'$ of CH trees, we have
$\op{Orb}_T= \op{Orb}_{T'} \circ \theta|_{E(T)}$ by Definition~\ref{def: isomorphism of CH trees}; Definition~\ref{defn: pre-CH trees}(3)(b) for glued edges; and the $\vartheta$-respecting property in Definition~\ref{defn: CH trees} and Definition~\ref{defn: pre-CH trees}(3)(a) for free edges.
\item[(ii)] There is a unique isomorphism class of $1$-vertex CH trees whose vertex is labeled $i$.
\ee
\end{rmk}
 
\subsubsection{Canonical representatives} \label{subsubsection: canonical rep}

In this subsection we explain how to obtain a canonical planar representative (modulo isotopies of the plane) of an isomorphism class $[T]$ of CH trees.\footnote{Hopefully this puts to rest some doubts that may arise from using isomorphism classes $[T]$ in the definition of semi-global Kuranishi structures (Definition~\ref{defn: semi-global Kuranishi structures Symp}) below.}

We first discuss pre-CH trees. 
The representatives will be planar, with the root at the top, say along $y=N$.  We say that a planar tree $T$ is {\em $\vartheta$-sorted} if the edges $e_{v,1},\dots, e_{v,k(v)}$ leaving each vertex $v$, go from some $y=N'$ to $y=N'-1$, where $N'$ depends only on $v$, and are ordered from left to right, and $(\op{Orb}_T(e_{v,1}), \dots, \op{Orb}_T(e_{v,k(v)}))$ is $\vartheta$-sorted.  

Let $T$ be a $\vartheta$-sorted planar pre-CH tree and the edges leaving the root go to $y=N-1$ with $\vartheta$-sorted orbit assignments of the form 
$$(\underbrace{\gamma_{1},\dots, \gamma_{1}}_{i_{1} \text{ copies }},\dots, \underbrace{\gamma_{k},\dots, \gamma_{k}}_{i_{k} \text{ copies }}).$$ 

We will inductively construct:
\be
\item canonical $\vartheta$-sorted planar representatives for all pre-CH trees $T$; and
\item a total ordering $\prec$ of $\mathcal{P}_\alpha \sqcup \{\mbox{isomorphism classes of pre-CH trees}\}$ that extends the ordering on $\mathcal{P}_\alpha$ from $\vartheta$.
\ee

If we restrict to $1$-vertex pre-CH trees $T$, then (1) is immediate and (2) follows from  lexicographically comparing 
$$(\mathfrak{r}(T), \gamma_1,\dots,\gamma_1,\dots,\gamma_k,\dots, \gamma_k).$$
If $\mathfrak{r}(T)=\gamma$, then $\gamma\prec [T]$; this includes the case where $T$ has no edges. 

Suppose we have inductively constructed (1) and (2) when restricted to all pre-CH trees $T'$ with $\mathfrak{r}(T')\prec \mathfrak{r}(T)$.  Then for each grouping $(\gamma_j,\dots,\gamma_j)$, $j=1,\dots, k$, 
\be
\item[(a)] if the edge is not free, we replace the edge labeled $\gamma_j$ by the pre-CH tree below it in the planar representation to obtain $(\gamma_{j1},\dots, \gamma_{ji_j})$, where $\gamma_{j\ell}$ is either $\gamma_j$ or a pre-CH tree; and
\item[(b)] we permute $(\gamma_{j1},\dots, \gamma_{ji_j})$ (and the edges that go to $\gamma_{j\ell}$ and the subtrees below if they exist) so that it is in lexicographic order; in particular, the free edges come first.
\ee
Hence (1) and (2) for $T$ (more precisely, the extension of the ordering $\prec$ to include $T$) can be constructed using 
$$(\mathfrak{r}(T), \gamma_{11},\dots,\gamma_{1i_1},\dots,\gamma_{k1},\dots, \gamma_{ki_k}).$$

Next we discuss CH trees, i.e., add the reordering function $O(T)$ to the data.
For a $1$-vertex CH tree $T$, we can normalize the reordering function $O(T)$ by Remark~\ref{rmk: loquats}(ii) so the labels are $1,\dots, i_1+\dots+ i_k$ from left to right, since we can permute all the free edges that go to $\gamma_j$, $j=1,\dots, k$. In general, if $\gamma_{jj'}=\gamma_{jj'+1}$ are orbits, then we permute the corresponding free edges $e',e''$ if necessary so that $O(T)(e')< O(T)(e'')$. If $\gamma_{jj'}=\gamma_{jj'+1}$ are trees, then we switch their positions if $\op{min}O(T)(F(\gamma_{jj'}))> \op{min}O(T)(F(\gamma_{jj'+1}))$.  This yields a canonical $\vartheta$-sorted planar representative of $T$ together with canonical normalized $O(T)$; see Figure~\ref{fig: canonical} for example.

\begin{figure}[ht]
	\begin{overpic}[scale=1.5]{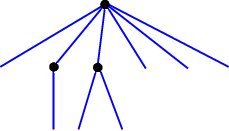}
		\put(2,34){\small{$\gamma_1$}} \put(21,34){\small{$\gamma_1$}}
		\put(36.3,34){\small {$\gamma_1$}}
		\put(59,24){\small $\gamma_2$}
		\put(77,24){\small $\gamma_2$}
		\put(96,24){\small $\gamma_2$}
		\put(23,-5){\small $T_1$}
		\put(32,-5){\small $\prec$}
		\put(40,-5){\small $T_2$}
	\end{overpic}
	\caption{Example of a canonical representative of a CH tree. The black dots represent vertices. The first $\gamma_1$ edge and all the $\gamma_2$ edges are free.  The subtrees $T_1$ and $T_2$ lie below the second and third $\gamma_1$ edges and are assumed to satisfy $T_1\prec T_2$. We may take the reordering function $O(T)$ to label the free edges by $4,1,2,3,5,6,7$ from left to right.}
	\label{fig: canonical}
\end{figure}

\subsubsection{Contraction of CH trees} \label{subsubsection: contraction of CH trees}

\begin{defn}[Good subtree/subforest]  
A {\em good subtree} $S$ of a CH tree $T$ is a subtree with no free edges and at least one glued edge. A possibly empty disjoint union $S=\sqcup_i S_i$ of good subtrees of $T$ is a {\em good subforest} of $T$. 
\end{defn}


\begin{defn}
If $S, S'$ are good subforests of CH trees $T, T'$, respectively, then the pair $(T,S)$ is {\em isomorphic to} $(T',S')$ if there exists a CH tree isomorphism $\theta: T \to T'$ that maps $S$ to $S'$. We denote the isomorphism class of $(T,S)$ by $[(T,S)]$ .
\end{defn}

\begin{defn}[Contraction] \label{defn: contraction}
Let $T$ be a CH tree and $S=\sqcup_i S_i$ be a good subforest of $T$. 
\nom[TS]{$T\git S$}{Contraction of CH tree $T$ by good subforest $S$}
Then the {\em contraction $T'=T\git S$ of $T$ along $S$} is the CH tree given as follows:
\begin{enumerate}
\item Each subtree $S_i$ is replaced by a vertex labeled by $\tau(S_i)\in\{1,\dots, \rho\}$, such that, if we glue elements of $\overline{\mathcal{M}_j/\R}$, $j\in V(S_i)$, according to the gluing prescription given by $S_i$, we obtain elements of $\overline{\mathcal{M}_{\tau(S_i)}/\R}$. For the vertices not in $S_i$, the vertex labeling function $l_{V(T)}$ remains the same.
\item The glued edges of $T$ not in $S$ remain glued edges of $T'$ and the free edges of $T$ remain free edges of $T$. Hence there exists a natural bijection 
$$c:E(T)-E(S)\xrightarrow{\sim} E(T').$$ 
\item $\op{Orb}_{T'}$ and $O(T')$ are defined by 
$$\op{Orb}_{T'}\circ c=\op{Orb}_T|_{E(T)-E(S)}\quad \mbox{and} \quad O(T')\circ c|_{F(T)}=O(T)|_{F(T)}.$$
\end{enumerate}
\end{defn}

{\em We write $T\git T$ for $T\git S$ when $S$ is the largest good subforest of T.}

\nom[TT]{$T'<T$}{Partial order on CH trees; holds if $T'=T\git S$}

\begin{defn}
We define a (strict) partial order on CH trees by $T' < T$ if $T' = T\git S$ for some non-empty good subforest $S$. 
The partial order also descends to isomorphism classes of CH trees: $[T']<[T]$ if there exist $T'\in [T']$ and $T\in[T]$ such that $T'<T$.
\end{defn}

\subsubsection{Degeneration of CH trees}

Let $S$ be a good subforest of a CH tree $T$.  A {\em degeneration of $T$ along $S$} is roughly a forest $T_1\sqcup\dots\sqcup T_m$ of CH trees with some data indicating how to glue them. It is obtained from $T$ as follows:  First replace each edge $e\in S$ by a free edge $e'$ such that $i(e') = i(e)$ and the orbit assignments of $e'$ and $e$ remain the same.  This yields $m = |E(S)|+1$ connected components $T_1,\dots, T_m$, where each $T_i$ has an induced CH tree structure with the possible exception of the reordering function. We then choose a reordering function $O(T_i)$ for each $T_i$ that is consistent with the reordering function induced from $O(T)$ (i.e., we decide the ordering of the newly added free edges relative to the existing free edges). {\em Observe that, in general, the reordering function $O(T_i)$ is not uniquely determined by the above requirement.}  

The above degeneration procedure is {\coblu formally given} as follows:
\nom[Tarrow]{$\arr{T}=\arr{T}(S)$}{Degeneration of $T$ along $S$}

\begin{defn}\label{def: degeneration}
{\em A degeneration $\arr{T}(S)$ (or simply $\arr{T}$ if $S$ is understood) of $T$ along $S$} is a tuple $(T_1 \sqcup \dots \sqcup T_m, \iota^{E(T)}, \iota^{V(T)}, \theta)$ such that:
\begin{enumerate} [label = (D\arabic*)]
    \item \label{D1} $T_i$ is a CH tree for $i\in \{1,\dots , m\}.$
    \item \label{D2} $\iota^{E(T)}: E(S) \to F(T_1 \sqcup \dots \sqcup T_m)$ is an injective map.
    \item \label{D3} $\iota^{V(T)}: E(S) \to V(T_1 \sqcup \dots \sqcup T_m)$ is an injective map.
    \item \label{D4} $T_1 \sqcup \dots \sqcup T_m$ can be glued into a pre-CH tree $\widehat T$ by replacing the free edge $\iota^{E(T)}(e)$ with a glued edge that goes from $i(\iota^{E(T)}(e))$ to $\iota^{V(T)}(e)$ for all $e\in E(S)$. 
    \item \label{D5} $\theta: \widehat T \to T$ an isomorphism of pre-CH trees that identifies the glued edge corresponding to $\iota^{E(T)}(e)$ with $e$ for all $e\in E(S)$ and the bijection 
    $$\theta: F(T_1\sqcup \dots\sqcup T_m) - \iota^{E(T)}(E(S))\to F(T)$$ 
    is compatible with the orderings $O(T_i)$ and $O(T)$ for each $i \in \{1,\dots , m\}$, i.e., for $e_1,e_2\in F(T_i) - \iota^{E(T)}(E(S))$, $O(T_i)(e_1) > O(T_i)(e_2)$ if and only if $O(T)(\theta(e_1))> O(T)(\theta(e_2))$.
\end{enumerate}
\end{defn}

If the pair $(T,S)$ is isomorphic to the pair $(T',S')$ and $\arr T$ is a degeneration of $T$ along $S$, then it is clear that $\arr T$ is also a degeneration of $T'$ along $S'$. Hence, it makes sense to refer to a degeneration of $[(T,S)].$

\begin{defn}
Let $\arr T$ and $(\arr T)'$ be degenerations of $T$ along $S$. We say that {\em $\arr T$ is isomorphic to $(\arr T)'$} if there exist CH tree isomorphisms $\theta_i: T_i \to T_i'$ for $i \in\{1, \dots, m\}$ such that they commute with \ref{D2}--\ref{D5} in the obvious sense.
We denote the isomorphism class by $[\arr T(S)]$ or simply $[\arr T].$
\end{defn}

\begin{example} \label{example: degeneration}
Let $T$ be a CH tree with two vertices $v_1, v_2$ and $S$ be the subforest that contains the only glued edge $e$ from $v_2$ to $v_1$. Suppose that $\mathcal M_{l_{V(T)}(v_2)} = \mathcal M(\gamma_+; \bs\gamma_-)$ and $\op{Orb}_T(e) = \gamma_{-,1}$, where 
$$\bs\gamma_- = (\underbrace{\gamma_{-,1},\dots, \gamma_{-,1}}_{i_{1} \text{ copies }},\dots, \underbrace{\gamma_{-,k},\dots, \gamma_{-,k}}_{i_{k} \text{ copies }})$$
is $\vartheta$-sorted. Let $\arr T = (T_1 \sqcup T_2, \iota^{E(T)}, \iota^{V(T)}, \theta)$ be a degeneration of $T$ along $S$. Then $T_1$ and $T_2$ are one-vertex trees; we take $V(T_i)=\{v_i\}$, $i=1,2$.  Let us denote $E(T_2) = \{ e_2^1, \dots, e_2^n \}$ with $n= i_1 + \dots + i_k$, where for convenience we are taking the superscript of $e_2^j$ to agree with the reordering function $O(T_2)$.  Then $\iota^{E(T)}(e) = e_2^q$ for some $q\in \{1,\dots , i_1\}$. We denote this degeneration by 
$$\arr T_q = (T_{1,q} \sqcup T_{2,q}, \iota_q^{E(T)}, \iota_q^{V(T)}, \theta_q).$$

We claim that the $\arr T_q$ are non-isomorphic for different $q$. Indeed, suppose $\arr T_q$ is isomorphic to $\arr T_{q'}$ and $\theta_i: T_{i,q} \to T_{i,q'}$ is the CH tree isomorphism for $i \in \{1,2\}.$
Since $\iota^{V(T)}_{q'} (e) = \theta_1 (\iota^{V(T)}_q(e))$ and $\theta_1$ maps $O(T_{1,q})$ to $O(T_{1,q'})$, it follows that $q = q'$. On the other hand, as long as $q = q'$ one can always find $\theta_1$ and $\theta_2.$  Hence there are $i_1$ non-isomorphic degenerations of $T$ along $S$.
\end{example}

\begin{rmk} \label{rmk: special degeneration}
When $\op{Orb}_T|_{E(S)}$ is injective, we can define a special degeneration $\arr T^s$ by choosing a special reordering function $O(T_i)$ as follows: 
\be
\item if $F(T_i) \cap \op{Im}\iota^{E(T)} = \varnothing$, then $O(T_i)$ is uniquely determined by $O(T)$; 
\item if there exists $e \in F(T_i) \cap \op{Im}\iota^{E(T)}$, then we choose $O(T_i)$ such that $O(T_i) (e) \geq O(T_i)(e')$ for all $e'\in F(T_i)$ and $\op{Orb}_{T_i}(e') = \op{Orb}_{T_i}(e)$. 
\ee
In Example~\ref{example: degeneration}, this means choosing $q = i_1$.
\end{rmk}

\subsection{Symmetries of moduli spaces} \label{subsection: symmetries of moduli spaces}

\coblu The group $G_i=G(\dot F_i, \R\times M;\gamma_{i,+};\bs\gamma_{i,-})$ generated by marker rotations and puncture reorderings can be viewed as a group of automorphisms acting on $\overline{\mathcal{M}_i/\R}$. The boundary strata of $\overline{\mathcal{M}_i/\R}$ are described by CH trees $T$ such that $T\git T=i$. 

Given a CH tree $T$ such that $T\git T=i$ and $\sigma\in G_i$, we can define a CH tree $T_\sigma$ such that:
\begin{enumerate}[label = ($\sigma$\arabic*)]
\item $V(T_\sigma) = V(T)$,\footnote{Note that we have an equality of sets, not just a bijection.}
\item $E(T_\sigma) = E(T)$, 
\item $l_{V(T_\sigma)} = l_{V(T)}$, 
\item $\op{Orb}_{T_\sigma} = \op{Orb}_T$, and
\item $O(T_\sigma) = \sigma \circ O(T)$,
\end{enumerate}
where $\sigma$ is sometimes viewed as a permutation of the negative ends (e.g., in $\sigma\circ O(T)$). 
There is an isomorphism $\sigma_*: T \to T_\sigma$ of pre-CH trees (but not necessarily an isomorphism of CH trees). Moreover $G_i$ permutes the set of CH trees with underlying pre-CH tree $T$.

\begin{defn}[Subgroups of $G_i$] \label{defn: subgroups of Gi} $\mbox{}$

\be
\item Let $G_i(T)\subset G_i$ be the stabilizer of the CH tree $T$ such that $T\git T=i$; by this we mean $G_i(T)$ contains all the marker rotations and the puncture reorderings $\sigma$ such that $T_\sigma=T$. 

\item The superscript $-$ (resp.\ $+$) as in $G_i^-$ indicates the subgroup generated by marker rotations and puncture reorderings of the negative ends (resp.\ marker rotations of the positive ends) and the subscript $P$ (resp.\ $M$) as in $G_{i,P}$ indicates the subgroup generated by puncture reorderings (resp.\ marker rotations).  
\item We write $G_{i,M,k}^-$ to indicate the marker rotations of the $k$th negative end.
\ee
\end{defn}
   
\nom[Gi2]{$G_i(T), G_i^\pm, G_{i,P},G_{i,M}, G_{i,M,k}^-$}{Subgroups of $G_i$}

By the constructions in the later subsections, elements of $G_i$ will give isomorphisms of semi-global Kuranishi charts, which are labelled by isomorphism classes of CH trees. A marker rotation does not change the isomorphism class of the tree and hence maps a chart to itself, whilst a puncture reordering can map a chart to another one. \cb

Let $(T,S)$ be a pair consisting of a CH tree $T$ and a good subtree $S$, $\arr{T}=(T_1 \sqcup \dots \sqcup T_m, \iota^{E(T)}, \iota^{V(T)}, \theta)$ a degeneration of $T$ along $S$, and {\coblu $\sigma \in G_i$.} Then $S_\sigma: = \sigma_* (S)$ is a good subforest of $T_\sigma$ and we choose a specific degeneration 
$$\arr{T_\sigma} = (T_{1,\sigma} \sqcup \dots \sqcup T_{m,\sigma}, \iota^{E(T_\sigma)}, \iota^{V(T_\sigma)}, \theta_\sigma)$$ 
of $T_\sigma$ along $S_\sigma$ as follows: Each $T_{i,\sigma}$ is given by the tree $T_i$ with a (possibly) different reordering function $O(T_{i,\sigma})$.
The reordering function $O(T_{\sigma})$ orders $F(T_\sigma)$, which contains all the free edges of $T_i$ except those that come from $S$, which are denoted by $F_{S,i}$. 
The reordering function $O(T_{i,\sigma})$ is defined by 
\be
\item for $e \in F_{S,i}$, $O(T_{i,\sigma})(e) = O(T_i)(e)$,
\item for $e,e' \in F(T_\sigma)$, $O(T_{i,\sigma})(e) < O(T_{i,\sigma})(e')$ if $\sigma (O(T_i) (e))< \sigma (O(T_i) (e')).$
\ee 
It is easy to see that maps $\iota^{E(T_\sigma)}, \iota^{V(T_\sigma)}$ and $\theta_\sigma$ are also determined canonically.
One can verify that $\arr{T_\sigma}$ is a degeneration of $T_\sigma$ along $S_\sigma$ and that {\coblu $[\arr{T_\sigma}]$ is independent of the choice of representative of $[(T,S)]$.}

Since the free edges of $T_i$ and $T_{i,\sigma}$ are canonically identified, comparing the reordering functions $O(T_{i})$ and $O(T_{i,\sigma})$ gives us a permutation {\coblu $\sigma_i \in G^-_{T_i\git T_i,P}$} such that $(\sigma_i)_*: T_i \to T_{i,\sigma}$. The map $\sigma \mapsto (\sigma_1, \dots, \sigma_m)$ induces an injective group homomorphism denoted by {\coblu $\delta_{\arr{T}}: G^-_{T\git T,P} \to \prod_{i=1}^m G^-_{T_i\git T_i,P}$.}

\subsection{Definition of semi-global Kuranishi structures} \label{subsection: definition of semi-global Kuranishi structures}

{\em To simplify notation, we write $T$ instead of $[T]$ for the isomorphism class unless otherwise stated.}

\nom[Kur]{$\Kur=\Kur^L(\alpha,J)$}{Semi-global Kuranishi structure in $\op{(Symp)}$ case; $L$ is the action threshold, $\alpha$ is the contact form, and $J$ is the almost complex structure}
\nom[KMi]{$\Kur(\mathcal{M}_i)$}{The full subcategory of $\Kur$ with objects $\mathcal{C}_T$ such that $T\git T=i$}
\coblu We will work in the $\op{(Symp)}$ case.  

\begin{defn} [Semi-global Kuranishi structures $\op{(Symp)}$] \label{defn: semi-global Kuranishi structures Symp}
A {\em semi-global Kuranishi structure $\Kur$ for $(\overline{\mathcal M_1/\R}, \dots, \overline{\mathcal M_\rho/\R})$} is the disjoint union of the categories  $\Kur (\mathcal{M}_i)$ over all $i=1,\dots,\rho$, together with the {\em strata compatibility conditions}, given by Definition~\ref{defn: strata compatibility}.
\end{defn}

\begin{defn}[Definition of $\Kur(\mathcal{M}_i)$] $\Kur(\mathcal{M}_i)$ is a category with data (K1)--(K5) and the action of $G_i$ on it given by (K6)--(K8):
\begin{enumerate}[label = (K\arabic*)]
\item \label{K1} The objects are semi-global Kuranishi charts 
$$\mathcal C_T = (\pi_T: \E_T \to \V_T, \overline{\partial}_T, \psi_T,{\coblu G_{i}(T)})$$ 
indexed by isomorphism classes $T$ of CH trees {\coblu $T$ such that $T\git T=i$}, together with their {\em slight enlargements}
$$\mathcal C_T^\en = (\pi_T^\en: \E_T^\en \to \V_T^\en, \overline{\partial}_T^\en, \psi_T^\en, {\coblu G_{i}(T)} ),$$ 
such that:
\nom[CT]{$\mathcal C_T=(\pi_T: \E_T \to \V_T, \overline{\partial}_T, \psi_T,{\coblu G_{i}(T)} )$}{Semi-global Kuranishi chart corresponding to the CH tree $T$}
\nom[CTen]{$\mathcal C_T^\en=(\pi_T^\en: \E_T^\en \to \V_T^\en, \overline{\partial}_T^\en, \psi_T^\en,{\coblu G_{i}(T)} )$}{Slight enlargement of the semi-global Kuranishi chart $\mathcal C_T=(\pi_T: \E_T \to \V_T, \overline{\partial}_T, \psi_T )$}
\begin{itemize}
    \item[(a)] {\coblu $(\pi_T: \E_T \to \V_T,{\coblu G_{i}(T)})$ and $(\pi_T^{\en}:\E_T^{\en}\to \V_T^{\en},{\coblu G_{i}(T)})$ are global orbibundle charts,}  $\overline{\partial}_T^{\en}: \V_T^{\en} \to \E_T^{\en}$ is a section of the global quotient orbibundle, and $\psi_T^{\en}: (\overline{\partial}_T^{\en})^{-1}(0) \to \mathcal M_{i}/\R$ is a homeomorphism onto an open subset such that
    $$\dim \V_T^{\en} - \op{rank} \E_T^{\en} = \op{vdim} (\mathcal M_{i}/\R).$$
    \item[(b)] $\V_T^{\en}$ is a {\coblu manifold} with a compatible metric $d_T$, $\V_T$ is an open subset of $\V_T^{\en}$ whose closure $\overline\V_T$ in $\V_T^{\en}$ is complete with respect to $d_T$, $\pi_T=\pi_T^\en|_{\V_T}$ (by this we mean the restriction of $\pi_T^\en$ to $(\pi_T^\en)^{-1}(\V_T)$), $\overline{\partial}_T= \overline{\partial}_T^{\en}|_{\V_T}$, and $\psi_T=\psi_T^{\en}|_{\overline{\partial}^{-1}_T(0)}$.  
\end{itemize}

\NoIndent{The slight enlargements $\V_T^{\en}$ of $\V_T$ exist for technical reasons which can be summarized in Example~\ref{example: topologies not the same} below; they can be viewed as replacements for {\coblu manifolds with corners}. We denote $\bdry \V_T= \overline \V_T- \V_T$. When $T = (i)$, we also write $T = i$.}
\begin{itemize}
    \item[(c)]   $\mathcal C_i = (\pi_i: \E_i \to \V_i, \overline{\partial}_i, \psi_i, {\coblu G_i})$ and its slight enlargement $\mathcal C_i^\en = (\pi_i^\en: \E_i^\en \to \V_i^\en, \overline{\partial}_i^\en, \psi_i^\en, {\coblu G_i})$ are interior semi-global Kuranishi charts for a compact subset $\K_i$ of $\mathcal M_i/\R$ and $\V_i$ is bounded with respect to the metric $d_i$. 
    \end{itemize}
\item \label{K2} For each $T' < T$, there is a unique morphism, called the ``restriction-inclusion" morphism, consisting of
\be
\item[(i)] $\phi_{T',T}: \mathcal C_{T'} \to \mathcal C_T$ encoded by $(\V_{T',T}, \phi_{T',T}^\sharp, \phi_{T', T}^\flat)$ and
\item[(ii)] enlargement $\phi_{T',T}^\en: \mathcal C_{T'}^\en \to \mathcal C_T^\en$ encoded by $(\V_{T',T}^\en,(\phi_{T',T}^\sharp)^\en, (\phi_{T', T}^\flat)^\en)$
\ee
such that:
\be
\item
$\V_{T',T}^{\en}$ is an open subset of $\V_{T'}^{\en}$, $\E_{T',T}^{\en}:=\E_{T'}^{\en}\vert_{\V_{T',T}^{\en}}$, and there is a $G_{i}(T)\cap G_{i}(T')$-equivariant bundle embedding given by the commutative diagram
\begin{equation} \label{diagram: tikz diagram}
\begin{tikzcd}
\E_{T',T}^{\en}  \arrow[r,"(\phi_{T',T}^\sharp)^{\en}"] \arrow[d] & \E_T^{\en} \arrow[d]\\
\V_{T',T} ^{\en}  \arrow[r,"(\phi_{T',T}^\flat)^{\en}"] &  \V_T^{\en}.
\end{tikzcd}
\end{equation}
\item $\V_{T',T}$ is an open subset of $\V_{T',T}^\en\cap \V_{T'}$ whose closure $\overline{\V}_{T',T}$ is complete with respect to $d_{T'}$, $\E_{T',T}= \E_{T',T}^\en|_{\V_{T',T}}$, $\phi_{T',T}^\sharp:=(\phi_{T',T}^\sharp)^\en|_{\E_{T',T}}$ has image in $\E_T$, $\phi_{T', T}^\flat:=(\phi_{T', T}^\flat)^\en|_{\V_{T',T}}$ has image in $\V_T$, and the analog of Diagram~\eqref{diagram: tikz diagram} with $^{\en}$ removed is also $G_{i}(T)\cap G_{i}(T')$-equivariant.
\item[(c)] Writing $\V^\en_{T',T} - \V_{T',T}= \V^1_{T',T}\sqcup \V^2_{T',T}$, where $\V^1_{T',T}\subset \V^\en_{T'}-\V_{T'}$ and $\V^2_{T',T}\subset \V_{T'}$, we have: $$(\phi_{T',T}^\flat)^{\en}(\V^2_{T',T})\subset \V_T^\en - \V_T \quad \mbox{and} \quad (\phi_{T',T}^\flat)^{\en}(\V^1_{T',T})\subset \V_T.$$
\item[(d)] $(\phi_{T',T}^\sharp)^{\en}\circ\overline\bdry_{T',T}^{\en}=\overline\bdry_T^{\en}\circ (\phi_{T',T}^\flat)^{\en}$,
where $\overline\bdry_{T',T}^{\en}:=\overline\bdry_{T'}^{\en}|_{\V_{T',T}^{\en}}$.
\item[(e)] $\psi_{T}^\en\circ (\phi_{T',T}^\flat)^\en=\psi_{T'}^\en$ on $(\overline\bdry^\en_{T'})^{-1}(0)\cap \V_{T',T}^\en$.
\item[(f)] Over $(\phi_{T',T}^\flat)^{\en}(\V_{T',T}^{\en})$, $(\overline\bdry_T^{\en})_*:T\V_T^{\en}\to \E_T^{\en}$ descends to an isomorphism
$$T\V_T^{\en}/ (\phi_{T',T}^\flat)^{\en}_* (T\V_{T',T}^{\en})\xrightarrow\sim \E_T^{\en}/(\phi_{T',T}^\sharp)^{\en}_*\E_{T',T}^{\en}.$$
\ee
\NoIndent{(f) is the abstract version of the requirement on $\overline\bdry'_J$ and $\overline\bdry_J$ from Definition~\ref{defn: stab pair}(2).}

\item \label{K3} The identity morphism consisting of $\phi_{T,T}: \mathcal C_T\to \mathcal C_T$ and  $\phi_{T,T}^\en: \mathcal C_T^\en\to \mathcal C_T^\en$ is given by the data 
\begin{gather*}
    (\V_{T,T}, \phi_{T,T}^\sharp=\op{id}|_{\E_T}, \phi_{T, T}^\flat=\op{id}|_{\V_T})\\
    (\V_{T,T}^\en, (\phi_{T,T}^\en)^\sharp=\op{id}|_{\E_T^\en}, (\phi_{T, T}^\en)^\flat=\op{id}|_{\V_T^\en}).
\end{gather*}

\item \label{K4} The composition of 
\[
    (\phi_{T',T}: \mathcal C_{T'}\to \mathcal C_{T},\phi_{T',T}^\en: \mathcal C_{T'}^\en\to \mathcal C_{T}^\en)  
\]
and 
\[ 
(\phi_{T'',T'}: \mathcal C_{T''}\to \mathcal C_{T'}, \phi_{T'',T'}^\en: \mathcal C_{T''}^\en\to \mathcal C_{T'}^\en)
\]
given by 
\[
((\V_{T',T},\phi_{T',T}^\sharp,\phi_{T',T}^\flat),(\V_{T',T}^\en,(\phi_{T',T}^\sharp)^\en,(\phi_{T',T}^\flat)^\en))
\]
and
\[
((\V_{T'',T'},\phi_{T'',T'}^\sharp,\phi_{T'',T'}^\flat),(\V_{T'',T'}^\en,(\phi_{T'',T'}^\sharp)^\en,(\phi_{T'',T'}^\flat)^\en))
\]
is 
$$(\phi_{T'',T}: \mathcal C_{T''}\to \mathcal C_{T},\phi_{T'',T}^\en: \mathcal C_{T''}^\en\to \mathcal C_{T}^\en)$$ 
given by 
$$((\V_{T'',T},\phi_{T'',T}^\sharp,\phi_{T'',T}^\flat),(\V_{T'',T}^\en,(\phi_{T'',T}^\sharp)^\en,(\phi_{T'',T}^\flat)^\en))$$ 
such that 
\begin{gather*}
 \V_{T'',T}= (\phi_{T'',T'}^\flat)^{-1}(\V_{T',T}), \quad \V_{T'',T}^\en= ((\phi_{T'',T'}^\flat)^\en)^{-1}(\V_{T',T}^\en), \\
    \phi_{T'',T}^\flat= \phi_{T',T}^\flat \circ \phi_{T'',T'}^\flat, \quad \phi_{T'',T}^\sharp= \phi_{T',T}^\sharp\circ \phi_{T'',T'}^\sharp,\\
    (\phi_{T'',T}^\flat)^\en= (\phi_{T',T}^\flat)^\en \circ (\phi_{T'',T'}^\flat)^\en, \quad  (\phi_{T'',T}^\sharp)^\en= (\phi_{T',T}^\sharp)^\en \circ (\phi_{T'',T'}^\sharp)^\en.
\end{gather*}
Composition is clearly associative.

\item \label{K5}
$\mathcal M_i /\R = \cup_{T\git T = i} \psi_T (\overline{\bdry}_T^{-1}(0))$.

\NoIndent{Given CH trees $T,T'$ {\coblu with $T\git T=T'\git T=i$,} an {\em isomorphism of charts} is a pair
$$\phi=(\phi^\sharp,\phi^\flat):\mathcal{C}_{T'}\to \mathcal{C}_{T},  \quad \phi^\en=((\phi^\sharp)^\en,(\phi^\flat)^\en):\mathcal{C}^\en_{T'}\to \mathcal{C}^\en_{T},$$ 
where $(\E_{T'}^\en\to \V_{T'}^\en) \xrightarrow{((\phi^\sharp)^\en,(\phi^\flat)^\en)} (\E_T^\en\to \V_{T}^\en)$ is {\coblu a bundle} isomorphism such that $(\phi^\sharp)^\en\circ\overline\bdry_{T'}^\en=\overline\bdry_T^\en\circ (\phi^\flat)^\en$ and $\psi_{T}^\en\circ (\phi^\flat)^\en=\psi_{T'}^\en$, and which restricts to {\coblu a bundle} isomorphism  $(\E_{T'}\to \V_{T'})\xrightarrow{(\phi^\sharp, \phi^\flat)} (\E_T\to \V_T)$.  We will usually denote an isomorphism of charts by $\phi^\en=((\phi^\sharp)^\en,(\phi^\flat)^\en):\mathcal{C}^\en_{T'}\to \mathcal{C}^\en_{T}$.}

\item \label{K6} \coblu For any $T$ and $r\in G^+_{i,M}$ or $G^-_{i,M}$,\footnote{Recall the notation from Definition~\ref{defn: subgroups of Gi}.}  there exists an automorphism 
$$r^\en=((r^\sharp)^\en, (r^\flat)^\en):\mathcal C_T^\en\to \mathcal C_T^\en$$ such that $\psi_T^\en \circ (r^\flat)^\en \circ (\psi_T^\en)^{-1}$ is the same map as rotating the asymptotic markers of the domain Riemann surfaces by $r$.

\item \label{K7} For any $T$ and $\sigma \in G^-_{i,P}$, \cb there exists an an isomorphism 
$$\sigma^\en = ((\sigma^\sharp)^\en, (\sigma^\flat)^\en): \mathcal C_T^\en \to \mathcal C_{T_\sigma}^\en$$ 
such that $\psi_T^\en \circ (\sigma^\flat)^\en \circ (\psi_T^\en)^{-1}$ is given by
permuting the labels of the negative ends of the domain Riemann surfaces by $\sigma$.

\item \label{K8} \coblu The collection of restrictions of $\E_{T'}^\en\to \V_{T'}^\en$ to $\V_{T',T}^\en$ (resp.\ $\E_{T'}\to \V_{T'}$ to $\V_{T',T}$) and inclusions $\E_{T',T}^\en\to \V_{T',T}^\en$ to $\E_T^\en\to \V_T^\en$ (resp.\ $\E_{T',T}\to \V_{T',T}$ to $\E_T\to \V_T$) given by the commutative diagram from \ref{K2} commute with $G_i$.

\NoIndent{In other words, $G_i$ acts on $\{\mathcal C_T\}_{T\git T=i}$ with stabilizer $G_i(T)$ for $T$ and on $\{\mathcal{C}_{T'}\to \mathcal{C}_T\}_{T'<T}$ with stabilizer $G_i(T)\cap G_i(T')$ for the pair $T'<T$.}


\end{enumerate}
\end{defn}

{\coblu Note that $\Kur(\mathcal{M}_i)$ admits a fully faithful functor from the category $\mathcal{T}_i$ whose objects are CH trees $T$ with $T\git T=i$ and morphisms which are given by $T'<T$.}

Let $S$ be a good subforest of $T$ and $\arr{T}=\arr{T}(S) = (T_1 \sqcup \dots \sqcup T_m, \iota^{E(T)}, \iota^{V(T)}, \theta)$ be a degeneration with $m=m(S)$.  We write $\V_{\arr{T}}^{\en} = \V_{T_1}^{\en} \times \dots \times \V_{T_m}^{\en}$  and  $\E_{\arr{T}}^{\en} = \op{pr}_{T_1}^* \E_{T_1}^{\en}\oplus \dots \oplus\op{pr}_{T_m}^* \E_{T_m}^{\en}$, where $\op{pr}_{T_k}: \V_{\arr{T}}^{\en}\times (R-\epsilon_0, \infty)^{m-1} \to \V_{T_k}^{\en}$ is the $k$th projection map.

For each $e\in E(S)$, suppose $t(e)$ is the root of $T_{a(e)}$ and $e$ becomes the $k(e)$th free edge of $T_{b(e)}$.
There is a canonical isomorphism \coblu
$$\theta_e: G^+_{T_{a(e)}\git T_{a(e)},M}\stackrel\sim\longrightarrow G^-_{T_{b(e)}\git T_{b(e)},M,k(e)}.$$ 
We then set $G_{\arr{T},M} = \textstyle (\prod_{j=1}^m G_{T_j\git T_j,M}) /G^{\op{diag}}_{\arr{T},M},$ where
\begin{gather*}
G^{\op{diag}}_{\arr{T},M} =\textstyle \prod_{e \in E(S)} \{ (r, \theta_e(r)) ~|~ r\in G^+_{T_{a(e)}\git T_{a(e)},M}\}.
\end{gather*}
There is a natural inclusion $\tau_{\arr{T}}: G_{T\git T,M} \to G_{\arr{T},M}$; also recall $\delta_{\arr{T}}: G^-_{T\git T,P}\to \prod_{i=1}^m G^-_{T_i\git T_i,P}$ from Section~\ref{subsection: symmetries of moduli spaces}: \cb

\begin{defn}[Stata compatibility conditions] \label{defn: strata compatibility}
The {\em strata compatibility conditions for $\Kur$} are given as follows:
\be[label = (SC\arabic*)]
\item \label{SC1} 
There exist global constants $R>0$ and $\epsilon_0>0$ small, ``intermediate'' {\coblu $G_{T\git T}(T)$-invariant} open subsets $\V_T^{\en/2}$ for all $T$ satisfying 
$$\V_T\subset \V_T^{\en/2}\subset \V_T^\en,  \quad \overline\V_T\subset \V_T^{\en/2}, \quad \overline\V_T^{\en/2}\subset \V_T^{\en},$$ 
and analogously defined {\coblu bundles} $\E_{\arr{T}}^{\en/2}\to \V_{\arr{T}}^{\en/2}$ such that, for any CH tree $T$ with $|V(T)|\geq 2$ and a good subforest $S$ of $T$, there exists a $C^1$-bundle map $(\widetilde{\mathfrak G}_{\arr{T}}^{\en}, \mathfrak G_{\arr{T}}^{\en})$ {\coblu between bundles of the same rank}
$$
\begin{tikzcd}
\E_{\arr{T}}^{\en/2} \arrow[r,"\widetilde{\mathfrak{G}}_{\arr{T}}^{\en}"] \arrow[d] & \E_T^{\en} \arrow[d]\\
\V_{\arr{T}}^{\en/2} \times (R-\epsilon_0, \infty)^{m-1}   \arrow[r, "\mathfrak{G}_{\arr{T}}^{\en}"] &  \V_T^{\en},
\end{tikzcd}
$$
called the {\em gluing map} and which satisfies the following:
\be
\item $\mathfrak{G}_{\arr{T}}^{\en}$ is a $C^1$-map which is a homeomorphism onto its image.
\item $\widetilde{\mathfrak G}_{\arr{T}}^{\en} \circ (\overline\partial_{T_1}^{\en}, \dots,\overline \partial_{T_m}^{\en})|_{\V_{\arr{T}}^{\en/2} \times (R-\epsilon_0, \infty)^{m-1}}$ and $\overline \partial_T^{\en} \circ \mathfrak G_{\arr{T}}^{\en}$
are $C^1$-close and their difference goes to zero as $\op{min}_{j=1}^{m-1} I_j \to \infty$, where $I_j$ is the coordinate for the $j$th $(R-\epsilon_0, \infty)$ factor corresponding to an edge $e_j\in E(S)$. 
\item Let $S_j$ be a good subforest of $T_j$ which we view as a subforest of $T$. 
The simultaneous and iterated gluing maps
\begin{gather*}
    \mathfrak G_{\arr{T} (S\cup S_j)}^{\en}(v_1,\dots,v_{j-1},\mathbf v_{\arr{T_j}(S_j)},v_{j+1}, \dots, v_m,\mathbf I_{\arr{T_j}(S_j)}, \mathbf I_{\arr{T}(S)}),\\
    \mathfrak G_{\arr{T}(S)}^{\en} (v_1,\dots,v_{j-1}, \mathfrak G_{\arr{T_j}(S_j)}^{\en}(\mathbf{v}_{\arr{T_j}(S_j) }, \mathbf{I}_{\arr{T_j} (S_j)}), v_{j+1},\dots , v_m, \mathbf{I}_{\arr{T}(S)} )
\end{gather*}
where $v_i \in \V_{T_i}^{\en/2}$, $\mathbf{v}_{\arr{T_j}(S_j) }\in \V_{\arr{T_j}(S_j)}^{\en/2}$, $\mathbf{I}_{\arr{T_j} (S_j)}\in (R-\epsilon_0,\infty)^{m(S_j)-1}$, and $\mathbf I_{\arr{T}(S)}\in (R-\epsilon_0, \infty)^{m(S)-1}$, are $C^1$-close with error $\to 0$ as the minimum of all the components of $\mathbf I_{\arr{T} (S\cup S_j)} \to \infty$.
\ee

\NoIndent{Let us denote 
$$\mathfrak G^{R_0, T} :=  \mathfrak G_{\arr{T}(T_{\op{max}})}^\en(\V_{\arr{T}(T_{\op{max}})} ^{\en/2}\times ((R-\epsilon_0, \infty)^{|G(T)|}-(R-\epsilon_0, R_0)^{|G(T)|})),$$
where $T_{\op{max}}$ is the maximal good subforest of $T$ consisting of all of $V(T)$ and $G(T)$. }

\be
\item[(d)] $\mathfrak G^{R-\epsilon_0, T}\supset \overline\V_T$ and $\overline \V_T -  \mathfrak G^{R_0, T}$ is compact for any $R_0 \geq R-\epsilon_0$.
\ee

\NoIndent{In view of (c), the sets $\mathfrak G^{R_0, T}$ and 
$$(\mathfrak G^{R_0, T})' := \bigcup_{S'\subset T} \mathfrak G_{\arr{T}(S')}^\en(\V_{\arr{T}(S')}^{\en/2} \times (R_0, \infty)^{m(S')-1}),$$ 
where the union is taken over all good subforests $S'$, can be used interchangeably, i.e., (d) holds with $(\mathfrak G^{R_0, T})'$ instead of $\mathfrak G^{R_0, T}$.}

\item \label{SC2} {\coblu The maps $\mathfrak G_{\arr{T}}^\en$ and $\widetilde{\mathfrak G}_{\arr{T}}^\en$ are $G^{\op{diag}}_{\arr{T},M}$-invariant and for any $\sigma \in  G_{T\git T}$ the following diagrams commute:}
\[
\begin{tikzcd}
\V_{\arr{T}}^{\en/2} \times (R-\epsilon_0, \infty)^{m-1} \arrow[r,"\mathfrak{G}^\en_{\arr{T}}"] \arrow[d, "(\delta_{\arr{T}}(\sigma)^\sharp)^\en\times \op{id}"] & \V_T ^\en \arrow[d, "(\sigma^\sharp)^\en" ]\\
\V_{\arr{T_\sigma}}^{\en/2} \times (R-\epsilon_0, \infty)^{m-1}   \arrow[r, "\mathfrak{G}^\en_{\arr{T_\sigma}}"] &  \V_{T_\sigma}^\en
\end{tikzcd}
\hspace{1cm}
\begin{tikzcd}
\E^{\en/2}_{\arr{T}} \arrow[r,"\widetilde{\mathfrak{G}}^\en_{\arr{T}}"] \arrow[d, "(\delta_{\arr{T}}(\sigma)^\flat)^\en"] & \E_T^\en \arrow[d, "(\sigma^\flat)^\en" ]\\
\E^{\en/2}_{\arr{T_\sigma}} \arrow[r, "\widetilde{\mathfrak{G}}^\en_{\arr{T_\sigma}}"] &  \E_{T_\sigma}^\en
\end{tikzcd}
\]
\ee
\end{defn}

{\em For the rest of the paper we will usually omit ${\en}$ from our notation, namely when we refer to bundle maps and multisections, we assume they admit slight enlargements with superscripts $^{\en}$.}

Let $\sim_{\Kur}$ be the equivalence relation on $\coprod_{T\git T=i} \V_T^\en$ that is induced by (K2)(a), i.e., 
$$(\phi_{T',T}^\flat)^\en(x)\sim_{\Kur} x, \quad \forall T'< T ~\mbox{ and }~ x\in \V_{T',T}^\en.$$
We define the topology on 
$$\mathcal{V}_i:= \left(\textstyle\coprod_{T\git T=i} \V_T \right)/\sim_{\Kur}$$
as follows: First take the quotient topology on $(\coprod_{T\git T=i} \overline{\V}_T)/\sim_{\Kur}$ and then take the subset topology on $\mathcal{V}_i$. This slightly roundabout procedure is justified by the following example which describes the pathology of taking the quotient topology with respect to an open relation. 

\begin{example}[Issues with taking the quotient topology] \label{example: topologies not the same}
  Consider the subsets $V_1=\{x<0\}$ and $V_2=\{y=0\}$ of standard $\R^2_{(x,y)}$ with the subset topologies.  Then the quotient topology $\mathcal{T}_1$ on $(V_1\sqcup V_2)/\sim$, where $\sim$ identifies $(x,0)\in V_1$ with $(x,0)\in V_2$, is not the same as the subset topology $\mathcal{T}_2$ on $V=\{x<0 \mbox{ or } y=0\}$ induced from $\R^2$. This is because for example 
$$\{(x,y)~|~ (x+1)^2 + y^2< 1  \}  \cup \{ (x,y) ~|~ x\in (-1,1), y=0\}$$ 
is in $\mathcal{T}_1$ but not in $\mathcal{T}_2$. To circumvent the issue, as explained in \cite[Section 2]{FO32}, we first take the quotient topology on $(\overline V_1\sqcup \overline V_2)/\sim$, where $\overline V_1=\{x\leq 0\}$ and $\overline V_2=V_2$, and then take the subset topology on $(V_1\sqcup V_2)/\sim$.
\end{example}

\begin{lemma} \label{lemma: Hausdorff}
$\mathcal{V}_i$ is Hausdorff.
\end{lemma}

\begin{proof}  
Each $\overline{\V}_T$ is Hausdorff.  We then use the fact that if, for $i=1,2$, $X_i$ is a Hausdorff topological space and $A_i$ a closed subset of $X_i$ and $\phi:A_1\xrightarrow{\sim} A_2$ is a homeomorphism, then $(X_1\coprod X_2)/ (x\sim \phi(x))$, $\forall x\in A_1$, with the quotient topology is also Hausdorff. Hence $(\coprod_{T\git T=i} \overline{\V}_T)/\sim_{\Kur}$ is Hausdorff and the restriction to $\mathcal{V}_i$ is also Hausdorff.
\end{proof}

Since we will use $\overline{\bdry}^{-1}(\mathfrak s)$ as a replacement of $\overline{\bdry}^{-1}(0)$, we define: 

\begin{defn}[Multisection of a semi-global Kuranishi structure] \label{def: multisection of Kur}
A {\em multisection of $\Kur$} is a collection $\mathfrak{S} = \{{\mathfrak s}_T^\en: \V_T^\en\to \E_T^\en\}_T$ of slight enlargements of multisections such that:
\begin{enumerate}
\item for any $T' < T$, $(\phi_{T',T}^\sharp)^\en \circ \mathfrak{s}_{T',T}^\en= \mathfrak{s}_T^\en \circ (\phi_{T',T}^\flat)^\en$, where $\mathfrak{s}_{T',T}^\en := \mathfrak{s}_{T}^\en|_{\V_{T',T}^\en}$;
\item there exists $R_0\gg 0$ such that for each good subforest $S$ of $T$ with $S\not=\varnothing$, if we write $\arr{T}=\arr{T}(S)=(T_1\sqcup \dots \sqcup T_{m(S)}, \iota^{E(T)}, \iota^{V(T)}, \theta)$, then on $ \mathfrak G_{\arr{T}}^\en(\V_{\arr{T}}^{\en/2} \times (R_0, \infty)^{m(S)-1})$  
\nom[Sigma]{$\mathfrak S=\{\mathfrak s_T^\en\}_T$}{Multisection of a semi-global Kuranishi structure $\Kur$}
$$\widetilde{\mathfrak G}_{\arr{T}}^\en \circ (\mathfrak s_{T_1}^\en, \dots, \mathfrak s_{T_m}^\en) \circ (\mathfrak G^\en_{\arr{T}})^{-1}\quad \mbox{and} \quad \mathfrak s_T^\en$$ 
are $C^1$-close and their difference goes to zero as $\min_{j=1}^{m(S)-1} I_j \to \infty$, where $(I_1,\dots, I_{m(S)-1})$ are the $(R_0,\infty)^{m(S)-1}$-coordinates; and 
\item {\coblu $\mathfrak S_i:=\{{\mathfrak s}_T^\en~|~ T\git T=i\}\subset \mathfrak S$ is preserved under $G_i$.}
\end{enumerate} 
A multisection $\mathfrak{S}=\{\mathfrak s_T^\en\}_T$ is {\em $C^1$-close to the $0$-section} if $\mathfrak{s}_T^\en$ is $C^1$-close to the $0$-section for each $T$.  We often write $\mathfrak{S}=\{\mathfrak s_T\}_T$ when the slight enlargements are understood.
\end{defn}

Definition~\ref{def: multisection of Kur}(1) implies that $(\overline{\bdry}_T^\en)^{-1}({\mathfrak s}_{T}^\en)$ and $(\overline{\bdry}_{T'}^\en)^{-1}(\mathfrak{s}_{T'}^\en)$ agree on the ``overlap" of semi-global Kuranishi charts $\V_T$ and $\V_{T'}$. Hence $\overline{\bdry}^{-1} (\mathfrak s)$ can be obtained by patching $\{\overline{\bdry}^{-1}_T (\mathfrak s_T)\}_T$ together.   In view of \ref{SC1}(d) and Definition~\ref{def: multisection of Kur}(2), we say that $\mathfrak{s}^\en_T$ is {\em product-like} outside a compact subset of $\overline{\V}_T$.   

Given a multisection $\mathfrak{S}_i=\{\mathfrak{s}_T\}_{T\git T=i}$ of $\Kur (\mathcal{M}_i)$, we define
\begin{equation} \label{eqn: defn of Z sub i}
\mathcal Z_i = \mathcal{Z}(\Kur(\mathcal{M}_i),\mathfrak{S}_i)=\left(\coprod_{T\git T=i} \overline\bdry^{-1}_T ({\mathfrak s}_T)\right)/\sim_{\Kur},
\end{equation}
where $\sim_{\Kur}$ is   as above and the topology is the subset topology of $\mathcal{V}_i$. \cb 
\nom[Z]{$\mathcal Z_i = \mathcal{Z}(\Kur(\mathcal{M}_i),\mathfrak{S})$}{Weighted branched manifold which is the intersection of the multisection $\mathfrak S$ and $\overline\bdry$, corresponding to the moduli space $\mathcal{M}_i$}
\nom[Z2]{$\mathcal{Z}(\Kur^L(\alpha, J;\gamma_+; \bs \gamma_-;A),\mathfrak{S})$}{$\mathcal{Z}(\Kur(\mathcal{M}_i),\mathfrak{S})$, where $
\Kur(\mathcal{M}_i)=\Kur^L(\alpha, J;\gamma_+; \bs \gamma_-;A)$}
\nom[K2]{$\Kur^L(\alpha, J;\gamma_+; \bs \gamma_-;A)$}{$\Kur(\mathcal{M}_i)$, where for $\mathcal{M}_i$ we have specified the orbits $\gamma_+, \bs\gamma_-$ at the positive and negative ends, the homology class $A\in H_2(M;\Z)$, and the fact that $\mathcal{A}_\alpha(\gamma_+)<L$}
When we want to specify the orbit $\gamma_+$ at the positive end of the topmost level, the orbits $\bs\gamma_-$ at the negative end of the bottommost level, and possibly the homology class $A\in H_2(M;\Z)$ (see Section~\ref{subsubsection: grading} for an explanation of how to assign a homology class $A\in H_2(M;\Z)$ to $[\F,u]$), and that fact that $\mathcal{A}_\alpha(\gamma_+)<L$, we write  $\mathcal{Z}(\Kur^L(\alpha, J;\gamma_+; \bs \gamma_-),\mathfrak{S})$ or $\mathcal{Z}(\Kur^L(\alpha, J;\gamma_+; \bs \gamma_-;A),\mathfrak{S})$.

\begin{prop} \label{prop: tomatoes}
If the multisection $\mathfrak{S} = \{{\mathfrak s}_T\}_T$ of $\Kur$ is $C^1$-close to the $0$-section and $\overline{\bdry}_T$ is transverse to $\mathfrak s_T$ for all $T$, then:
\be
\item The space $\mathcal Z_i$ is a weighted branched manifold of dimension $\op{vdim} \mathcal M_i/\R$.
\item If $\op{vdim} \mathcal M_i/\R = 0$, then $\mathcal Z_i$ consists of finitely many points with weights.
\item   If $\op{vdim} \mathcal M_i/\R = 1$, then there exists $\mathcal L_0 \in \R$ such that for any generic $\mathcal L > \mathcal L_0$, the {\em $\mathcal L$-boundary} of $\mathcal Z_i$ defined by 
\nom[bdryL]{$\partial^{\mathcal L} \mathcal Z_i$}{The $\mathcal L$-boundary of the zero set $\mathcal Z_i$ of $\mathfrak S$}
$$\partial^{\mathcal L} \mathcal Z_i := \bdry \coprod_{T\git T = i} \left[ \overline{\bdry}^{-1}_T(\mathfrak s_T) - \mathfrak G^{\mathcal{L},T} \right] /\sim_{\Kur}$$
is in bijection with the union 
$$\coprod_{T\git T = i,~ |V(T)| = 2} \left( \mathcal Z_{j(T)} \times \mathcal Z_{k(T)}\right)/{\coblu G^{\op{diag}}_{\arr T^s,M}}$$ 
of product weighted branched submanifolds, where $$\arr T^s = (j(T) \sqcup k(T), \iota^{E(T)}, \iota^{V(T)}, \theta)$$ is the special degeneration of $T$ as in Remark~\ref{rmk: special degeneration}.
\ee
\end{prop}

\begin{proof} 
For convenience, for any $T'<T$ we often omit the morphism $\phi_{T',T}$ from the notation and view $\V_{T',T}$ as a subspace of $\V_T$; we also often omit the map $\psi_T$ and view $\op{Im}(\psi_T) \subset \mathcal M_{T\git T}$ as a subspace of $\V_T$.  We write $\mathcal Z=\mathcal{Z}_i$.

(1) Lemma~\ref{lemma: Hausdorff} implies that $\mathcal{Z}$ is Hausdorff.  Then (1) follows from the transversality of $\mathfrak{S}$ and $\overline\bdry_J$ and {\coblu Condition \ref{K2}(f).}    

(2)  This is an immediate consequence of the compactness of $\mathcal{Z}$. To show the compactness  we start with the following:

\begin{claim} \label{claim: existence of Rzero}
There exists $R_0$ such that $\mathcal Z \cap \mathfrak G^{R', T} = \varnothing$ for all  $T$ such that $T\git T = i$ and $R' > R_0.$
\end{claim}

\begin{proof}[Proof of Claim~\ref{claim: existence of Rzero}]
Arguing by contradiction, suppose there exist:
\begin{itemize}
    \item[(i)] a CH tree $T$ such that $T\git T = i$ and $k=|V(T)| \geq 2$, 
    \item[(ii)] a good subtree $S$ and a degeneration $\arr{T}=\arr{T}(S) = T_{j_1} \sqcup \dots \sqcup T_{j_k}$ along $S$, and
    \item[(iii)] a sequence of curves $\{z_j\}_{j=j'}^\infty$ in $\mathcal Z\cap \mathfrak G_{\arr{T}}^\en(\V_{\arr{T}}^{\en/2} \times (j, \infty)^{m(S)-1})$, 
    with Fredholm index $\op{ind} z_j = 1$ and {\coblu $j'$ a large integer.}
\end{itemize}  
By (iii), the curve $z_j$ satisfies the equation $\overline{\partial}_T z_j = \mathfrak s_T(z_j)$. 
By Condition~\ref{SC1}(b) and Definition~\ref{def: multisection of Kur}(2), for $j\gg 0$ there exists $z_j' \in \V_T$ satisfying $\overline{\partial}_T' z_j' = \mathfrak s'_T(z_j')$, where 
\begin{gather*}
\overline{\partial}_T' = \widetilde{\mathfrak G}_{\arr{T}} ^\en\circ (\overline{\partial}_{T_{j_1}}^\en, \dots, \overline{\partial}_{T_{j_k}}^\en) \circ ({\mathfrak G}_{\arr{T}}^\en)^{-1},\\
\mathfrak s_T' = \widetilde{\mathfrak G}_{\arr{T}}^\en \circ (\mathfrak s_{T_{j_1}}^\en, \dots, \mathfrak s_{T_{j_k}}^\en) \circ ({\mathfrak G}_{\arr{T}}^\en)^{-1}.
\end{gather*}

Writing $(\mathfrak G_{\arr{T}}^\en)^{-1}(z_j')=(z_{j 1}, \dots, z_{jk}, I_1,\dots, I_{k-1})$,  we obtain $\overline{\partial}^\en_{T_{j_i}} z_{ji} = \mathfrak s_{T_{j_i}}^\en(z_{ji})$ for each $j,i$ with $j\gg 0$. Hence $\op{ind} z_{ji}\geq 1$, contradicting the fact that $\sum_{i=1}^k\op{ind} z_{ji}=\op{ind}z_j' =1$.
\end{proof}

For $R'\gg 0$, $((\overline{\bdry}^\en_T)^{-1}(0)\cap \overline\V_T) - \mathfrak G^{R', T}$ is compact by Condition \ref{SC1}(d) and hence is bounded since $\overline\V_T$ is complete. 
By Claim~\ref{claim: existence of Rzero}, $\overline{\bdry}^{-1}_T(\mathfrak s_T)  \cap \mathfrak G^{R', T} = \varnothing$. Let $U$  be a relatively compact open neighborhood of $((\overline{\bdry}^\en_T)^{-1}(0)\cap \overline\V_T) - \mathfrak G^{R', T}$ in $\overline\V_T$. Since $\mathfrak{S}=\{\mathfrak{s}_T\}_T$ is $C^1$-close to the $0$-section, $(\overline{\bdry}^\en_T)^{-1}(\mathfrak s_T^\en)\cap \overline\V_T \subset U$.  Hence $(\overline{\bdry}^\en_T)^{-1}(\mathfrak s_T^\en)\cap \overline\V_T$ is compact and so is $\mathcal Z$.

(3) The key point is the following claim which roughly states that $\mathcal Z$ has no ``interior boundary":

\begin{claim} \label{claim: caseAB}
If $\mathfrak s_T$ is $C^1$-close to the $0$-section for all $T$ such that $T\git T = i$, then the closure $\overline{\mathcal Z}$, obtained by gluing the closures of $\overline\bdry_T^{-1}(\mathfrak s_T)\subset \overline{\V}_T$ over all $T\git T=i$, is covered by $\{\V_T~|~T\git T=i\}$ and the ends of $\mathcal{Z}$ are close to a half-open interval times $\prod_{i=1}^2\overline{\partial}_{T_{j_i}}^{-1}(\mathfrak s_{T_{j_i}})$, where we are ranging over degenerations $T_{j_1}\sqcup T_{j_2}$ of $T$ along a single edge.
\end{claim}

\begin{proof}[Proof of Claim~\ref{claim: caseAB}]
Let $u \in (\overline{\bdry}_T^{\en})^{-1}(0) \cap \bdry \V_T$ such that $T\git T=i$. We will show that there exists a neighborhood $U$ of $u$ in $\V_T^\en$ such that if $\mathfrak s_T^{\en}$ is $C^1$-close to the $0$-section, then $(\overline{\bdry}_T^{\en})^{-1}(\mathfrak s_T^{\en}) \cap U \subset \V_{T''}$ for some $T''$ with $T'' \git T'' = i$. By (K1)(a), $\psi_T^{\en}(u)\in \mathcal{M}_i/\R$. Then by \ref{K5} there exists $T'$ such that $u \in \V_{T'}$.  We consider two cases: (A) $T'<T$ or (B) $T<T'$.

(A) By Conditions~\ref{K2}(a) (the embedding property and in particular the equivalence of the topology of $\V^\en_{T',T}$ and the induced topology), (c), (d), and (f) there exist a neighborhood $U$ of $u$ in $\V_T^\en$ such that $U \cap \V_{T',T} ^\en\subset \V_{T'}$ and coordinates $(v_{T',T}, x)$ of $U$, where $v_{T',T} \in \V_{T',T}^\en$ and $x$ is the coordinate in the normal direction, such that
\begin{gather*}\overline{\bdry}_T^{\en}(v_{T',T}, x) = (\overline{\bdry}_{T'}^{\en}(v_{T',T}), x + \op{h.o.}),\\
\mathfrak s_T^{\en}(v_{T',T}, x) = (\mathfrak s_{T',T}^{\en}(v_{T',T}), h(v_{T',T},x)),
\end{gather*}
where $h(v_{T',T},0)=0$ and has small $x$-derivative. If $U$ is small and hence $|x|$ is small, the equation 
\[\overline{\bdry}_T^{\en}(v_{T',T}, x) = \mathfrak s_T^{\en}(v_{T',T}, x)\]
implies $x = 0$. Hence 
$(\overline{\bdry}_T^{\en})^{-1}(\mathfrak s_T^{\en}) \cap U = (\overline{\bdry}_{T',T}^{\en})^{-1}(\mathfrak s_{T',T}^{\en}) \cap U \subset \V_{T'}.$

(B) This is straightforward: By Condition~\ref{K2}(c) there exists a neighborhood $U$ of $u$ in $\V_{T,T'}^\en$ such that $U \subset \V_{T'}$, so $(\overline{\bdry}^\en_T)^{-1}(\mathfrak s_T^\en) \cap U \subset \V_{T'}$.

This proves the existence of the neighborhood $U$ of $u$ and tree $T''$.

We now observe that $\mathfrak s_T$ is product-like outside of a compact subset $K_T$ of $\overline{\V}_T$.  Let $\mathscr{U}_T$ be a small neighborhood of the closure of  $\overline{\bdry}^{-1}_T(0)\cap K_T$ in $\overline{\V}_T$; by compactness we may assume that the closure of $\overline{\bdry}^{-1}_T(\mathfrak s_T)\cap K_T$ in $\overline{\V}_T$ (= $(\overline\bdry^\en)^{-1}(\mathfrak s_T^\en) \cap K_T$) is a subset of $\mathscr{U}_T$ by requiring $\mathfrak s_T$ to be sufficiently $C^1$-close to the $0$-section. 
By the existence of $U$ and $T''$ for each $u\in (\overline{\bdry}_T^{\en})^{-1}(0) \cap \bdry \V_T$ as above, 
$(\overline\bdry^\en)^{-1}(\mathfrak s_T^\en) \cap K_T\subset \cup_{T'\git T' = i} \V_{T'}$.

The elements of the ends of $\mathcal{Z}$ are close to elements $(z_1,\dots, z_k)$ of the product $\prod_{i=1}^k\overline{\partial}_{T_{j_i}}^{-1}(\mathfrak s_{T_{j_i}})$, together with some gluing parameters.  By the index count as in the proof of Claim~\ref{claim: existence of Rzero}, $\sum_{i=1}^k\op{ind} z_i=2$ and $\op{ind} z_i\geq 1$. Hence $k$ must be exactly two and the claim follows.
\end{proof}

In view of Claim~\ref{claim: caseAB}, the bijection in (3) is given by the gluing map $\mathfrak G_{\arr{T}}$ in \ref{SC1}, modulo the symmetry given by \ref{SC2}, and the special degeneration is chosen to avoid ``over-counting". 
\end{proof}

\subsection{Trimming} \label{subsection: trimming}

In Sections~\ref{subsection: trimming}--\ref{subsection: multisections on products} we assume the following: 
\begin{enumerate}
\item[(H)] \hypertarget{H}{} $\chi(\dot F_{i})<0$ for all the domains $\dot F_i$ of $\mathcal{M}_i$, where $\mathcal{M}_i\in \mathcal{S}$.
\end{enumerate}
Here (H) stands for ``hyperbolic". In particular, whenever there is a reference to ${\bf q}$, we assume ${\bf q}=\varnothing$. In Section~\ref{subsection: H does not hold} we make the necessary modifications when \hyperlink{H}{(H)} does not hold.  

\nom[delta]{$\delta>0$}{Small positive constant used in the definitions of $\delta$-close to breaking and $\mathcal G_\delta(V_1/\R,\dots, V_m/\R)$, etc.}
Choose $\delta>0$ small such that $(\alpha,J)$ is $L$-simple (see Definitions~\ref{Lsimple form} and \ref{defn: L-supersimple 1}) with parameter $\delta$.

We make the following definition:
\nom[nl]{$\nl(A)$}{Neck length of an annular component $A$ of $\op{Thin}_\varepsilon(\ddot F,g_{\ddot F})$}

\begin{defn}[Neck length]
Given $$(\F,u)=((F,j,{\bf p},{\bf q},{\bf r}),u)\in \widetilde{\mathcal{G}}_\delta(V_1/\R,\dots,V_m/\R),$$
recall that $g_{\ddot F}$ is the complete finite volume hyperbolic metric on $\ddot F=F-{\bf p}-{\bf q}$ that is compatible with $j$. If there is an annular component $A$ of $\op{Thin}_\varepsilon(\ddot F,g_{\ddot F})$ \cb such that $u|_A$ is close to a trivial cylinder over $\gamma$, then its {\em neck length} $\nl(A)$ with respect to $\gamma$ is the value $C$ such that $A$ is conformally equivalent to $(\R/\mathcal A(\gamma)\Z)\times [0,C]$ with the standard complex structure.
\end{defn}

Choose $\mathcal{L}\gg 0$ and  $\varepsilon''\gg \varepsilon''' >0$. 
Roughly speaking, $\mathcal{L}$ will determine the size in terms of neck lengths of the interior semi-global Kuranishi charts $\V_i$, $\varepsilon''$ will measure the widths in terms of neck lengths of the overlapping regions when including $\V_{T\git S}$ into $\V_T$, and $\varepsilon'''$ will measure the widths of the slight enlargements $\V_T^\en$ of $\V_T$.
\nom[E2]{$\varepsilon''>0$}{Small positive constant which measures the widths (in terms of neck lengths) of the overlapping regions when including $\V_{T\git S}$ into $\V_T$}
\nom[E3]{$\varepsilon'''>0$}{Small positive constant $0<\varepsilon'''\ll \varepsilon''$ which measures the widths (in terms of neck lengths) of the slight enlargements $\V_T^\en$ of $\V_T$}

Let $\K$ be a sufficiently large {\coblu $G$-invariant} compact subset of $\mathcal{M}/\R$ in the sense that (NL$_{\mathcal{L}+\varepsilon'''}$) holds, where for $\mathcal{L}'>0$ we define:  
\be
\item[(NL$_{\mathcal{L}'}$)] $[(\F,u)]\in \K$ whenever there is no annular component of $\op{Thin}_\varepsilon(\ddot F,g_{\ddot F})$ or all the annular components of $\op{Thin}_\varepsilon(\ddot F,g_{\ddot F})$ have neck length $\leq \mathcal{L}'$.
\ee
Consider an interior semi-global Kuranishi chart {\coblu $(\K,\pi_{\V}: \E\to \V,\overline\bdry_J,\psi,G)$} as constructed in Section~\ref{subsection: semi-global Kuranishi chart}.
{\em We will sometimes suppress $\overline\bdry_J$, $\psi$, {\coblu and $G$} from the notation.} 

By slightly shrinking $\V$ and $\K$ if necessary subject to (NL$_{\mathcal{L}+\varepsilon'''}$), and without changing notation, {\coblu we may assume that $\V$ and $\K$ are $G$-invariant.}

Next we explain how to further trim $\pi_{\V}: \E\to \V$ and $\K$.   Given $\mathcal{L}'$ such that $\mathcal{L}'\leq \mathcal{L}+\varepsilon'''$, let $\V_{\geq \mathcal{L}'}$ (resp.\ $\V_{> \mathcal{L}'}$, $\V_{= \mathcal{L}'}$) be the subset of $\V$ consisting of $[(\F,u)]$ for which there exists an annular component $A$ in $\op{Thin}_\varepsilon(\ddot F,g_{(\ddot F,j)})$ with $\nl(A)\geq \mathcal{L}'$ (resp.\ $\nl(A)>\mathcal{L}'$, $\nl(A)=\mathcal{L}'$) and let $\V_{<\mathcal{L}'}=\V-\V_{\geq \mathcal{L}'}$ and $\V_{\leq \mathcal{L}'}=\V-\V_{> \mathcal{L}'}$. 
Note that $\V_{< \mathcal{L}'}$ and $\K\cap \V_{\leq\mathcal{L}'-\varepsilon'''}$ are {\coblu still $G$-invariant.}

\begin{defn} [Interior charts] \label{defn: choice of interior chart}
The interior semi-global Kuranishi chart for $\mathcal{M}/\R$ that we use in \ref{K1}(c) is as follows:
\be 
\item[($\mathcal{I}$1)] the bundle is $\pi_{\V}|_{\V_{< \mathcal{L}}}:\E|_{\V_{< \mathcal{L}}}\to \V_{< \mathcal{L}}$, the compact set is $\K\cap \V_{\leq\mathcal{L}-\varepsilon'''}$, and the slight enlargement is $\pi_{\V}|_{\V_{< \mathcal{L}+\varepsilon'''}}:\E|_{\V_{< \mathcal{L}+\varepsilon'''}}\to \V_{< \mathcal{L}+\varepsilon'''}$;
\item[($\mathcal{I}$2)] the compact set $\K\cap \V_{\leq\mathcal{L}-\varepsilon'''}$ satisfies (NL$_{\mathcal{L}-\varepsilon'''}$); and
\item[($\mathcal{I}$3)] $\V_{< \mathcal{L}}$, $\V_{< \mathcal{L}+\varepsilon'''}$, and $\K\cap \V_{\leq\mathcal{L}-\varepsilon'''}$ {\coblu are all $G$-invariant;} and
\item[($\mathcal{I}$4)] $\overline\bdry$ (resp.\ $\overline\bdry^\en$) and $\psi$ (resp.\ $\psi^\en$) are restrictions of $\overline\bdry_J$ and $\psi$ to $\V_{<\mathcal{L}}$ (resp.\ $\V_{< \mathcal{L}+\varepsilon'''}$) and $\overline\bdry^{-1}(0)$ (resp.\ $(\overline\bdry^\en)^{-1}(0)$). 
\ee
{\em Resetting notation, we write $\pi:\E\to\V$, $\K$, and $\pi^\en: \E^\en\to \V^\en$ for the bundle, compact set, and slight enlargement satisfying ($\mathcal{I}$1)--($\mathcal{I}$3); the constant $\mathcal{L}$ is chosen in Section~\ref{subsubsection: interior charts and multisections}.}
\end{defn}

The boundary $\bdry \V$ can be decomposed into two parts, the {\em vertical boundary} $\bdry_v \V=\V_{=\mathcal{L}}$ and the {\em horizontal boundary} $\bdry_h \V$ which is the closure in $\bdry \V$ of $\bdry\V-\bdry_h\V$.

\cb

\subsection{Universal choices of $R,\varepsilon,\ell$} \label{subsection: universal choice of gluing parameter}

Recall $\mathcal{S}=\{\mathcal{M}_1,\dots,\mathcal{M}_\rho\}$. We explain how to choose the universal constants $R, \varepsilon$, and $\ell$.

\begin{rmk} \label{rmk: almost equivalence of gluing parameters}
Observe that, in view of the description of the pregluing from Section~\ref{subsection:pregluing}, the gluing parameter $T_e\in [R,\infty)$ corresponding to an edge $e$ of a CH tree is approximately equal to $C \nl(A_e) + D$, where $C>0$ and $D\in \R$ are fixed constants and $A_e$ is the $\varepsilon$-thin annulus of the glued domain corresponding to the edge $e$. Hence the neck length bound and the gluing parameter bound can be used interchangeably in Theorems~\ref{thm: gluing} and \ref{thm: iterated gluing}.  {\em In this subsection we take $R$ to be the neck length bound.}
\end{rmk}

\begin{lemma} \label{lemma: universal}
There exist a ``large'' {\coblu $G_i$-invariant} compact subset $\K_i$ of $\mathcal{M}_i/\R$ for each $\mathcal{M}_i\in \mathcal{S}$, a gluing parameter bound $R$, and constants $\varepsilon$ and $\ell$, such that (i$_{\mathcal{S}}$) and (ii$_{\mathcal{S}}$) hold for $R,\varepsilon,\ell$, where for any subset $\mathcal{S}'$ of $\mathcal{S}$ such that each component of each level of a building in $\bdry\mathcal{M}_i$ with $\mathcal{M}_i\in \mathcal{S}'$ is either a trivial cylinder or is in $\mathcal{S}'$ we define:
\be
\item[(i$_{\mathcal{S}'}$)] \coblue Theorems~\ref{thm: gluing}(A) \cb and \ref{thm: iterated gluing} hold for all possible gluings \coblu and iterated gluings of all orders \cb of {\coblue all the} $G_i$-invariant neighborhoods $V_i$ of $K_i$, possibly with multiplicities, such that $\mathcal{M}_i\in \mathcal{S}'$ with the {\em same} gluing parameter bound $R$.
Here 
all the $V_i$ depend on the {\em same} parameters $\varepsilon,\ell$.
\item[(ii$_{\mathcal{S}'}$)] For each $\mathcal{M}_i\in \mathcal{S}'$, the compact subset $\K_i$ contains all the elements of $\mathcal{M}_i/\R$ that are not \coblu in $\cup_T\mathcal{V}_T$, where $T$ ranges over all CH trees representing boundary strata of $\overline{\mathcal{M}_i/\R}$. Here if $T$  corresponds to the stratum $\times_{v\in V(T)} (\mathcal{M}_{l_{V(T)}(v)}/\R)$, then $\mathcal{V}_T$ is the intersection of all the images of the gluing maps and iterated gluing maps of all orders of Theorem~\ref{thm: gluing}(A) and \ref{thm: iterated gluing}, obtained by gluing the $V_{l_{V(T)}(v)}$ with gluing parameter bound $R$ (and parameters $\varepsilon,\ell$). 
\ee
\end{lemma}

\begin{proof}
If we want to indicate the dependence of $V_i$ on $\ell^i$ and $\varepsilon^i$, then we write $V_i=V_i^{\ell^i,\varepsilon^i}$.
The following is immediate:

\begin{claim}
If $(\ell^i)'\geq \ell^i$ and Theorems~\ref{thm: gluing} and \ref{thm: iterated gluing} hold for $V_i=V_i^{\ell^i,\varepsilon^i}\supset K_i$, $i=1,\dots,m$, and gluing parameter bound $R$, then Theorems~\ref{thm: gluing} and \ref{thm: iterated gluing} also hold for $V_i^{(\ell^i)',\varepsilon^i}$ and the same gluing parameter bound $R$, provided $V_i^{(\ell^i)',\varepsilon^i}$ is a sufficiently small neighborhood of $K_i$.
\end{claim}

The claim states that $\ell$ can be increased simply by shrinking the neighborhoods of $K_i$.  Hence $\ell$ can be adjusted at the very end without affecting $R$, $\varepsilon$, and the $K_i$.

\s\n
{\bf Step $0$.}  Let $\mathcal{S}_0$ be the set of all $\mathcal{M}_i$ such that $\mathcal{M}_i/\R=\overline{\mathcal{M}_i/\R}$.  If $\mathcal{M}_i\in \mathcal{S}_0$, then take $\K_i=\mathcal{M}_i/\R$. Choose $\varepsilon_0, \ell_0$ that works for all the $\K_i$ with $\mathcal{M}_i\in\mathcal{S}_0$ and let $R_0$ be a constant larger than the gluing parameter bounds in Theorems~\ref{thm: gluing}(A) and \ref{thm: iterated gluing} for all possible gluings of the $V_i$, possibly with multiplicities, such that $\mathcal{M}_i\in \mathcal{S}_0$.  Note that action considerations imply that there are only finitely many \coblu gluing types. \cb Hence (i$_{\mathcal{S}_{0}}$) and (ii$_{\mathcal{S}_{0}}$) hold with $R=R_0$, $\varepsilon=\varepsilon_0$, $\ell=\ell_0$.

\s\n
{\bf Step $j-1$.}  Suppose we have inductively chosen:
\begin{enumerate}
    \item $\mathcal{S}_0\subsetneq \dots \subsetneq \mathcal{S}_{j-1}\subset \{\mathcal{M}_1,\dots,\mathcal{M}_\rho\}$;
    \item a compact subset $\K_i$ of $\mathcal{M}_i/\R$ whenever $\mathcal{M}_i\in \mathcal{S}_{j-1}$; and
    \item a gluing parameter bound $R_{j-1}$ and constants $\varepsilon_{j-1}, \ell_{j-1}$; 
\end{enumerate} 
satisfying (i$_{\mathcal{S}_{j-1}}$) and (ii$_{\mathcal{S}_{j-1}}$) with $R=R_{j-1}$, $\varepsilon=\varepsilon_{j-1}$, and $\ell=\ell_{j-1}$.

We now come to the key point: If we increase $R_{j-1}$ to $R_{j-1}'>R_{j-1}$ with $\varepsilon_{j-1}$ and $\ell_{j-1}$ fixed, then we must enlarge $\K_i$ to some $\K_i'$ so that (ii$_{\mathcal{S}_{j-1}}$) still holds.  At first glance we then need to increase $R_{j-1}'$ since it depends on $\K_i'$, leading to a circular argument.  Upon closer inspection, however, we can enlarge $\K_i$ to $\K_i'$ without increasing $R_{j-1}$ so that (i$_{\mathcal{S}_{j-1}}$) and (ii$_{\mathcal{S}_{j-1}}$) hold for $\K_i'$, \coblue as we now explain: First (ii$_{\mathcal{S}_{j-1}}$) still holds with $\K_i$ replaced by $\K_i'\supset \K_i$, as making the compact set larger just makes (ii$_{\mathcal{S}_{j-1}}$) easier to attain. Let $V_i'$ be a small neighborhood of $\K_i'$. Next, by (ii$_{\mathcal{S}_{j-1}}$) for $\K_i$, all curves in $V_i'-V_i$ are obtained from a simultaneous boundary stratum gluing of $\overline{\mathcal{M}_i/\R}$ with gluing parameter bound $R_{j-1}$. Hence the gluing of curves $(u_{i_1},\dots, u_{i_m})\in V_{i_1}'\times\dots \times \V_{i_m}'$ from a possibly repeated collection of the $\K_i'$ corresponding to $\mathcal{M}_i\in \mathcal{S}_{j-1}$ can be performed with gluing parameter bound $R_{j-1}$, since the gluing can be viewed as the result of an iterated gluing where we first obtain the elements $u_{i_a}\in V_{i_a}' - V_{i_a}$ by a simultaneous boundary stratum gluing of $\overline{\mathcal{M}_{i_a}/\R}$ with gluing parameter bound $R_{j-1}$ and we then simultaneously glue the resulting curves $u_{i_a}\in V_{i_a}' - V_{i_a}$ and those $u_{i_a}$ that were in $V_{i_a}$. This is possible since Condition (i$_{\mathcal{S}_{j-1}}$) allows all possible gluings and iterated gluings. The consistency of the gluings defined above follows from Theorem~\ref{thm: gluing}(B).  
\cb   Hence for any $R_{j-1}'\geq R_{j-1}$ we can enlarge $\K_i$ to $\K_i'$ such that (i$_{\mathcal{S}_{j-1}}$) and (ii$_{\mathcal{S}_{j-1}}$) hold with $\K_i'$ and $R=R_{j-1}'$ instead.

On the other hand, in order to shrink $\varepsilon_{j-1}$ to $\varepsilon'_{j-1}$ subject to $0< \varepsilon'_{j-1}<\varepsilon_{j-1}$, we must first increase $R_{j-1}$. This also means we enlarge $\K_i$ for all $\mathcal{M}_i\in \mathcal{S}_{j-1}$ as in the previous paragraph. Then for $R_{j-1}$ and $\K_i$ sufficiently large, we may substitute $V_i^{\ell_{j-1},\varepsilon'_{j-1}}$ for $V_i^{\ell_{j-1},\varepsilon_{j-1}}$ since the gluing maps of Theorems~\ref{thm: gluing}(A) and \ref{thm: iterated gluing} can be made to hold simultaneously with the same $R_{j-1}$.

\s\n
{\bf Step $j$.} Let $\mathcal{M}_{i_j}$ be the moduli space with smallest subscript which is not in $\mathcal{S}_{j-1}$. Choose a compact subset $\K_{i_j}\subset \mathcal{M}_{i_j}/\R$ such that 
\begin{enumerate}
    \item[($\dagger_j$)] $\K_{i_j}$ contains all the curves that are not in the images of the gluing maps of Theorem~\ref{thm: gluing}(A) with gluing parameter bound $R_{j-1}$ for all the boundary strata of $\overline{\mathcal{M}_{i_j}/\R}$.
\end{enumerate} 
Let $\mathcal{S}_j=\mathcal{S}_{j-1}\cup \{\mathcal{M}_{i_j}\}$.  We will use $\varepsilon_j,\ell_j$ for $\K_{i_j}$ such that $0<\varepsilon_j\leq \varepsilon_{j-1}$ and $\varepsilon_{j-1},\ell_{j-1}$ for all $\K_i$ with $\mathcal{M}_i\in \mathcal{S}_{j-1}$.   

\s\n
{\bf Step $j1$.} Let $R_{j1}\geq R_{j-1}$ be a constant larger than the gluing parameter bounds in Theorems~\ref{thm: gluing}(A) and \ref{thm: iterated gluing} for all possible gluings of the $V_i$ (possibly with multiplicities) such that $\mathcal{M}_i\in \mathcal{S}_j$, subject to the condition that {\em $V_{i_j}$ be used at most once.}  If we increase $R_{j-1}$ so that it equals $R_{j1}$, then we must enlarge $\K_{i_j}$ to $\K_{i_j}'$ so that ($\dagger_j$) holds with gluing parameter bound $R_{j1}$. Using Step $j-1$, since gluing an element of $\K_{i_j}'-\K_{i_j}$ with elements of $\mathcal{S}_{j-1}$ can be done with gluing parameter bound $R_{j-1}$, we can enlarge $\K_{i_j}$ to $\K'_{i_j}$ without increasing $R_{j1}$, take $R_{j-1}=R_{j1}$, and then enlarge $\K_i$ to $\K_i'$ whenever $\mathcal{M}_i\in \mathcal{S}_{j-1}$. 

If we take $R_{j1}$ to be sufficiently large, then we may substitute $V_i^{\ell_{j-1},\varepsilon_j}$ for $V_i^{\ell_{j-1},\varepsilon_{j-1}}$ for all $\mathcal{M}_i\in \mathcal{S}_{j-1}$ as in Step $j-1$. We then reset notation and refer to the enlarged $\K'_i$ as $\K_i$.  

\s\n
{\bf Step $j2$.}
Next let $R_{j2}$ be a constant larger than the gluing parameter bounds in Theorems~\ref{thm: gluing}(A) and \ref{thm: iterated gluing} for all possible gluings of the $V_i$ (possibly with multiplicities) such that $\mathcal{M}_i\in \mathcal{S}_{j}$, subject to the condition that {\em $V_{i_j}$ be used at most twice}; we also require $R_{j2}\geq R_{j1}$. If we increase $R_{j-1}=R_{j1}$ so that it equals $R_{j2}$, then we must enlarge $\K_{i_j}$ to $\K_{i_j}'$ so that ($\dagger_j$) holds with gluing parameter bound $R_{j2}$. As before, since gluing an element of $\K_{i_j}'-\K_{i_j}$ with elements of $\mathcal{S}_j$, subject to the condition that $V_{i_j}$ be used only one additional time, can be done with gluing parameter bound $R_{j1}$, we can enlarge $\K_{i_j}$ to $\K'_{i_j}$ without increasing $R_{j2}$, take $R_{j1}=R_{j2}$, and then enlarge $\K_i$ to $\K_i'$ whenever $\mathcal{M}_i\in \mathcal{S}_{j-1}$. We then reset notation and refer to the enlarged $\K'_i$ as $\K_i$.  

\s\n
{\bf Step $jk$.}
We can repeat this process, constructing $R_{j3},R_{j4},\dots$ (the process terminates) and enlarging $\K_{i_j}$ and increasing the gluing parameter bounds such that $R_{j-1}=R_{j1}=R_{j2}=\dots = R_j$. 

\s\n
{\coblu The $G_i$-invariance of $V_i$ and $K_i$ can be additionally imposed during the above construction.}
\end{proof}

\cb

\subsection{Construction of the semi-global Kuranishi charts} \label{subsection: overview of construction}

\subsubsection{Complexity} \label{subsubsection: complexity}

\nom[CFu]{$c(\F,u)$}{Complexity which orders the moduli spaces}
The semi-global Kuranishi structure $\Kur^L(\alpha,J)$ is constructed by induction on a triple which we call the {\em complexity}
\begin{equation}\label{eqn: definition of complexity}
c(\F,u)=(c_1(\F,u),c_2(\F,u),c_3(\F,u))=(\mathcal{A}_\alpha(\gamma_+),E_\alpha(u),-\chi(\dot F)),
\end{equation}
where we are using lexicographic ordering. Here $u$ is a map from $\gamma_+$ to $\bs\gamma_-$. The complexity $c(T)$ of a CH tree $T$ is the complexity of $\overline{\mathcal{M}_{\tau(T)}/\R}$. 
\nom[cTZ]{$c(T)$}{The complexity of a CH tree $T$}

We choose $L\gg 0$ and only consider $\gamma_+$ satisfying $\mathcal{A}_\alpha(\gamma_+)< L$.  (Eventually we will take direct limits as $L\to \infty$.) 
We then choose a finite sequence $\mathcal{M}_1,\mathcal{M}_2,\dots,\mathcal{M}_\rho$ of {\coblu distinct} moduli spaces
$$\mathcal{M}_i=\mathcal{M}_{J}^{\op{ind}=k_i}(\dot F_i,\R\times M;\gamma_{i,+},\bs\gamma_{i,-}),$$
where the moduli spaces are ordered according to nondecreasing complexity. 

We make some observations regarding the complexity:  For a fixed $c_1$, the moduli space $\mathcal{M}_i$ with the smallest $c_2$ consist of branched covers of trivial cylinders, i.e., satisfy $E_\alpha(u)=0$. {\em We will not construct Kuranishi charts about trivial cylinders and hence trivial cylinders will not be included in the list.}  There are finitely many $\gamma_+$ and when $\gamma_+$ is fixed and the genus is zero, there is an upper bound on the number of punctures, which in turn gives an upper bound on the number of branch points. We list the $\mathcal{M}_i$ with the fewest number of branch points first.  Similarly, for $E_\alpha(u)>0$, we list the $\mathcal{M}_i$ with the smallest $E_\alpha(u)$ first and use the fact that there is an upper bound on the number of punctures. 

The following is immediate:

\begin{lemma}
Each component of each level of a building of $\bdry \mathcal{M}_i$ is either a trivial cylinder or in $\mathcal{M}_j$ with $j<i$.
\end{lemma}



\subsubsection{Interior charts and multisections} \label{subsubsection: interior charts and multisections}

By Lemma~\ref{lemma: universal} there exist a ``large" compact subset $\K_i$ of $\mathcal{M}_i/\R$ for all $i\in \{1,\dots,\rho\}$, a ``universal" gluing parameter bound $R>0$, and constants $\varepsilon$ and $\ell$, such that (i$_{\mathcal{S}}$) and (ii$_{\mathcal{S}}$) hold. In view of Remark~\ref{rmk: almost equivalence of gluing parameters}, we may take the large neck length bound $\mathcal{L}\gg 0$ from Section~\ref{subsection: trimming} to be the universal gluing parameter bound $R\gg 0$. 
\nom[Lgreaterthan0]{$\mathcal{L}\gg 0$}{Large neck length bound which serves as the universal gluing bound}

With this choice of $\mathcal{L}$, we take the interior semi-global Kuranishi charts $\mathcal{C}_i=(\pi_i: \E_i\to \V_i, \overline \bdry_i=\overline\bdry_J, \psi_i,{\coblu G_i)}$ for $\K_i\subset \mathcal{M}_i/\R$ and their slight enlargements $\mathcal{C}_i^\en=(\pi_i^\en: \E_i^\en\to \V_i^\en, \overline \bdry_i^\en=\overline\bdry_J, \psi_i^\en, {\coblu G_i)}$ as in Definition~\ref{defn: choice of interior chart}. 

Let $d_i$ be a metric on $\V_i^\en$ which is compatible with the topology on $\V_i^\en$; such a metric exists since manifolds are metrizable.  By construction, $\overline \V_i$ is compact and hence bounded with respect to $d_i$. 

We also choose an obstruction multisection ${\mathfrak s}_i$ and its enlargement $\mathfrak{s}^\en_i$, such that the collection $\{ \mathfrak s_i\}_i$ is invariant under marker rotations and puncture reorderings.\cb

\subsubsection{Boundary strata} \label{subsubsection: boundary strata}

We now explain how to construct charts corresponding to the boundary strata.

Recall Conventions~\ref{convention for boundary} and \ref{convention for moduli spaces}. Given $\mathcal{M}_{i_0}/\R$, its boundary $\bdry (\mathcal{M}_{i_0}/\R)$ is described as follows:
\begin{equation}
\bdry (\mathcal{M}_{i_0}/\R) = \left(\coprod_{T} \mathcal{M}_T'\right)/\sim, \quad \mathcal{M}_T': =\times_{v\in V(T)}(\mathcal{M}_{l_{V(T)}(v)}/\R),
\end{equation}
where the disjoint union is over all CH trees $T$ not equal to the single vertex tree $T_0=i_0$ and such that $T\git T=i_0$ and the equivalence relation $\sim$ is induced by isomorphisms $\theta: T\xrightarrow{\sim} T'$ of CH trees. (To simplify notation, we will often write ``$i\in V(T)$'' with $i\in \{1,\dots, \rho\}$ to mean that there is $v\in V(T)$ with $l_{V(T)}(v)=i$.)

  Let $T$ be a CH tree such that $T\git T=i_0$. We have already constructed 
$$\mathcal{C}_{i}=(\pi_{i}:\E_{i}\to \V_{i}, \overline\bdry_i=\overline\bdry_J, \psi_i, {\coblu G_i)} \quad \mbox{and}\quad \K_{i},  \quad \forall i\in V(T),$$
and the slight enlargement $\mathcal{C}_i^\en$.  {\em In general, a slight enlargement will be denoted by adding  a superscript $^\en$.} In this subsection we construct 
$$\mathcal{C}_T=(\pi_T: \E_T\to \V_T, \overline\bdry_T,\psi_T, {\coblu G_{T\git T}(T)}), \quad \K_T\subset  \mathcal{M}_{i_0}/\R,$$
and the slight enlargement, such that $\V_T \supset \K_T$ and $\{\K_T~|~T\git T=i_0\}$  covers all of $\mathcal{M}_{i_0}/\R$.

\begin{rmk} 
{\em Note that $\K_T$ is not necessarily compact,} although it admits a compactification inside $\overline{\mathcal{M}_{i_0}/\R}$.
\end{rmk}
\cb

\begin{warning}
Even if $\mathcal{M}_{i_0}/\R=\varnothing$ and the semi-global chart $\mathcal{C}_{i_0}$ is empty, we need to construct $\mathcal{C}_T$ if there exists a CH tree $T$ such that
\begin{itemize}
\item $T\git T=i_0$ and $\times_{i\in V(T)}(\mathcal{M}_{i}/\R)$ is nonempty.
\end{itemize}
There are examples that the authors learned from Michael Hutchings, where omitting $\mathcal{C}_T$ leads to some inconsistencies.
\end{warning}

\begin{rmk}
For our purposes it is not necessary to cover the actual boundary of the \cb SFT compactification $\overline{\mathcal{M}_{i_0}/\R}$ of $\mathcal{M}_{i_0}/\R$   with Kuranishi charts. 
\end{rmk}

\nom[GG]{$\mathcal{G}_\delta(\{\V_i^{\text{en}}\}_{i\in V(T)})$}{Set of equivalence classes of maps that are $\delta$-close to breaking into representatives of $\V_i^{\text{en}}$}
\nom[EE]{$\E'$}{Orbibundle over $\mathcal{G}_\delta(\{\V_i^{\text{en}}\}_{i\in V(T)})$}
\nom[GGtilde]{$\mathcal{G}_\delta^{\E'}(\{\V_i^{\text{en}}\}_{i\in V(T)})$}{Subset of $\mathcal{G}_\delta(\{\V_i^{\text{en}}\}_{i\in V(T)})$ such that $\overline\bdry_J u\in \E'$}
Let $\mathcal{G}_{\delta,T}=\mathcal{G}_\delta(\{\V_i^{\text{en}}\}_{i\in V(T)})$ be the set of equivalence classes of maps that are $\delta$-close to breaking into representatives of $\V_i^{\text{en}}$, where the ``gluing'' is prescribed by the set of glued edges $G(T)$.  Let $\E'$ be the bundle over $\mathcal{G}_{\delta,T}$, whose fiber $E'_{(\F,u)}$ over $(\F,u)$ is given by Equation~\eqref{equation: defn of E prime} {\coblu (note that the definition descends to the fiber $E'_{[\F,u]}$ over the equivalence class $[\F,u]$)} and let 
$$\mathcal{G}_{\delta,T}^{\E'}=\mathcal{G}_\delta^{\E'}(\{\V_i^{\text{en}}\}_{i\in V(T)})=\{[\F,u]\in  \mathcal{G}_{\delta,T}~|~ \overline\bdry_J u\in \E'\}.$$ 

\begin{defn}[Neck length functions]
We define the {\em neck length functions}
$$\nl_e([\F,u]):{\mathcal{G}}_{\delta,T}^{\E'} \to \R^+, \quad e\in G(T),$$
where $\nl_e([\F,u])$ is the neck length of the annular component $A_e$ 
of the $\varepsilon$-thin part of $g_{(\ddot F,j)}$ corresponding to $e$; if there is no such annular component $A_e$, we set $\nl_e([\F,u])=0$.  We also write $\nl_e(u)$ when $\F$ is implicit.
\end{defn}

\begin{defn}[$\V_T$, $\V_T^{\en/2}$, $\V_T^\en$, $\K_T$, $\pi_T:\E_T\to \V_T$, $\pi_T^\en:\E_T^\en\to \V_T^\en$, $d_T$] \label{defn: bundles ET over VT} $\mbox{}$
\be
\item $\V_T$ (resp.\ $\V_T^{\en/2}$, $\V_T^\en$) is the subset of ${\mathcal{G}}_{\delta,T}^{\E'}$ consisting of $[\F,u]$ such that $\nl_e(u)> \mathcal{L}-\varepsilon''$ (resp.\ $> \mathcal{L}-\varepsilon''-\tfrac{\varepsilon'''}{2}$, $> \mathcal{L}-\varepsilon''-\varepsilon'''$) for all $e\in G(T)$ and $\nl<\mathcal{L}$ (resp.\ $< \mathcal{L}+ \tfrac{\varepsilon'''}{2}$,  $<\mathcal{L}+\varepsilon'''$) for all other $\varepsilon$-thin annuli of $g_{(\ddot F,j)}$.

\item $\K_T$ is the subset of $(\mathcal{M}_{i_0}/\R)\cap {\mathcal{G}}_{\delta,T}^{\E'}$ consisting of $[\F,u]$ such that $\nl_e(u)\geq \mathcal{L}-\varepsilon'' + \tfrac{\varepsilon'''}{2}$ for all $e\in G(T)$ and $\nl\leq \mathcal{L}-\tfrac{\varepsilon'''}{2}$ for all other $\varepsilon$-thin annuli of $g_{(\ddot F,j)}$.

\item $\pi_T: \E_T\to \V_T$ and $\pi_T^\en: \E_T^\en\to \V_T^\en$ are restrictions of $\E'\to \mathcal{G}_{\delta,T}^{\E'}$ to $\V_T$ and $\V_T^\en$.

\item $d_T$ is a metric on $\V_T^\en$ satisfying
$$d_T([\F,u],[\F',u'])\geq \max_{e\in V(T)}|\nl_e(u)-\nl_e(u')|.$$
\ee
{\coblu We take $\V_T$, $\V_T^{\en/2}$, $\V_T^\en$, $\K_T$, $\pi_T:\E_T\to \V_T$, and $\pi_T^\en:\E_T^\en\to \V_T^\en$ to all be $G_{T\git T}(T)$-invariant.}
If there is an isomorphism $\theta:T\xrightarrow{\sim} T'$ of CH trees, then we identify $\pi_T: \E_T\to \V_T$ and $\pi_{T'}:\E_{T'}\to \V_{T'}$ via $\theta$.
\end{defn}

We briefly discuss the construction of the metric $d_T$ on $\V_T^\en$.  Since manifolds are metrizable, there exists a metric $d_T'$ on $\V_T^\en$.  We then take $d_T=d_T'+\phi$, where $\phi:\V_T^\en \times\V_T^\en \to \R^{\geq 0}$ is the continuous function 
$$\phi([\F,u],[\F',u'])= \max_{e\in V(T)}|\nl_e(u)-\nl_e(u')|.$$

\begin{lemma}\label{lemma: covers}
$\mathcal{M}_{i_0}/\R=\cup _{T\git T=i_0} \K_T$.
\end{lemma}

\begin{proof}
Let $[\F,u]\in \mathcal{M}_{i_0}/\R$. 
By Definition~\ref{defn: choice of interior chart}($\mathcal{I}$2), if there is no annular component of $\op{Thin}_\epsilon(\ddot F, g_{\ddot F})$ or all the annular components of $\op{Thin}_\epsilon(\ddot F, g_{\ddot F})$ have neck length $\leq \mathcal{L}-\tfrac{\varepsilon'''}{2}$, then $[\F,u]\in\K_{i_0}$. On the other hand, if there are annular components with neck length $> \mathcal{L}-\tfrac{\varepsilon'''}{2}$, then the set of all such determines a CH tree $T$ such that $T\git T=i_0$ and $[\F,u]\in \K_T$ by Definition~\ref{defn: bundles ET over VT}(2). 
\end{proof}

\subsubsection{Morphisms} \label{subsubsection: morphisms}

Let $T$ be a CH tree, $S=\sqcup_{j=1}^b S_j$ be a good subforest of $T$, and $T'=T\git S$.
We define the morphism
$$\phi_{T', T}: \mathcal{C}_{T'}\to \mathcal{C}_T, \quad \phi_{T', T}^\en: \mathcal{C}_{T'}^\en\to \mathcal{C}_T^\en$$
as follows ({\coblu we may take the morphism to be $G_{T\git T}(T)\cap G_{T'\git T'}(T')$-invariant by construction}): first restrict $\pi_{T'}: \E_{T'}\to \V_{T'}$ to 
$$\V_{T',T}:=\V_{T'}\cap\{\mathcal{L}>\nl_e > \mathcal{L}-\varepsilon''~|~e\in G(S)\},$$ 
then apply stabilization (see Section~\ref{subsubsection: stabilizations}), and finally include into $\pi_T:\E_T\to \V_T$. The composition of the second and third steps is clearly an embedding of vector bundles and is denoted by $(\phi^\sharp_{T', T}, \phi^\flat_{T', T})$. 

The slight enlargement $\V_{T',T}^\en$ is given by
$$\V_{T',T}^\en:=\V_{T'}^\en\cap\{\mathcal{L}+\varepsilon'''>\nl_e > \mathcal{L}-\varepsilon''-\varepsilon'''~|~e\in G(S)\},$$
and the embedding $((\phi^\sharp_{T', T})^\en, (\phi^\flat_{T', T})^\en)$ of the slightly enlarged bundles is defined in the same way as $(\phi^\sharp_{T', T}, \phi^\flat_{T', T})$.\cb

Given good subforests $S',S''\subset S$, one can easily verify the commutativity of the diagram:
\begin{equation} \label{eqn: four charts first version}
\begin{diagram}
\mathcal{C}_{T\git S} & \rTo & \mathcal{C}_{T\git S'} \\
\dTo & & \dTo\\
\mathcal{C}_{T\git S''} & \rTo & \mathcal{C}_{T},
\end{diagram}
\end{equation}
as well as the analogs for the slight enlargements.

\subsubsection{Examples} \label{subsubsection: examples} $\mbox{}$

\s\n
(A)  Suppose that $V(T)=\{i_1,i_2\}$ and $G(T)=\{e=(i_2,i_1)\}$.  Then there are two charts $\mathcal{C}_{T\git T}$ and $\mathcal{C}_{T}$ and a morphism $\mathcal{C}_{T\git T}\to \mathcal{C}_T$.

\s\n
(B) Suppose that $V(T)=\{i_1,i_2,i_3\}$ and $G(T)=\{e_1=(i_2,i_1),e_2=(i_3,i_2)\}$. Let
\begin{equation}\label{eqn: triple2}
S'=(\{i_1,i_2\},\{e_1\}), \quad S''=(\{i_2,i_3\},\{e_2\})
\end{equation}
be good subtrees of $T$.  There are four charts 
$$\mathcal{C}_{T\git T}, \quad \mathcal{C}_{T\git S'}, \quad \mathcal{C}_{T\git S''}, \quad \mathcal{C}_{T}$$
for $\Kur(\mathcal{M}_{T\git T})$, given schematically in Figure~\ref{fig: corner}. The morphisms satisfy \eqref{eqn: four charts first version} with $S=T$.

\begin{figure}[ht]
\begin{center}
\psfragscanon
\psfrag{A}{\tiny $\nl_{e_1}$}
\psfrag{B}{\tiny $\nl_{e_2}$}
\psfrag{E}{\tiny $\V_{T\git T}$}
\psfrag{F}{\tiny $\V_{T\git S'}$}
\psfrag{G}{\tiny $\V_{T\git S''}$}
\psfrag{H}{\tiny $\V_{T}$}
\psfrag{c}{\tiny $\mathcal{L}-\varepsilon''$}
\psfrag{d}{\tiny $\mathcal{L}$}
\psfrag{e}{\tiny $\mathcal{L}+\varepsilon''$}
\psfrag{f}{\tiny $\mathcal{L}-\varepsilon''$}
\psfrag{g}{\tiny $\mathcal{L}$}
\psfrag{h}{\tiny $\mathcal{L}+\varepsilon''$}
\psfrag{i}{\tiny $\infty$}
\includegraphics[width=5cm]{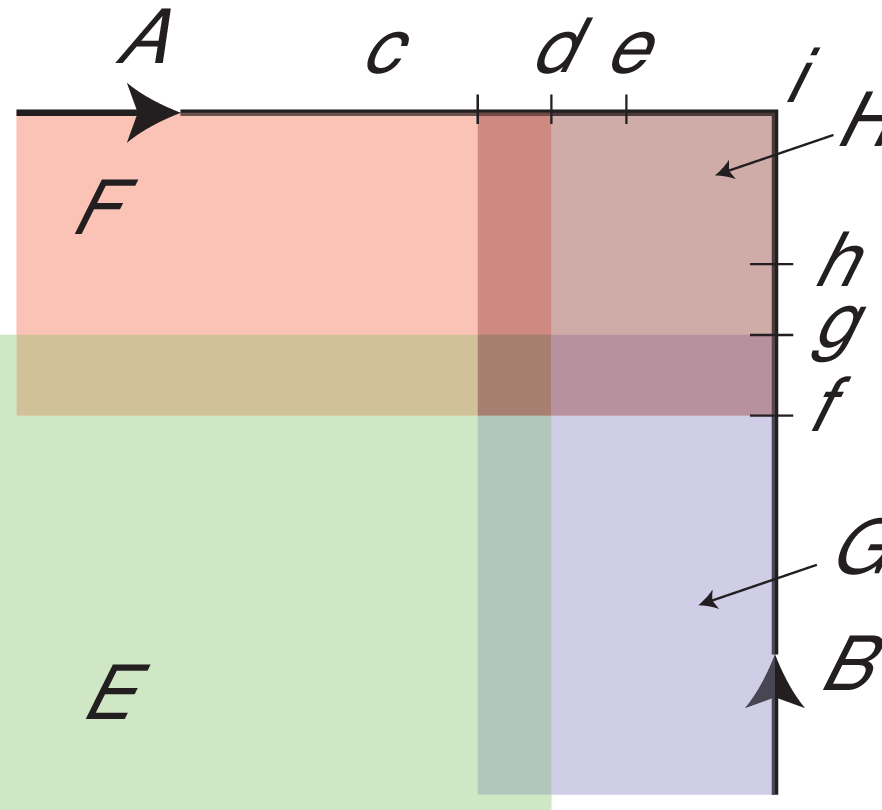}
\end{center}
\caption{Schematic diagram of $\V_{T\git T}$, $\V_{T\git S'}$, $\V_{T\git S''}$, and $\V_{T}$. The horizontal and vertical lines represent the $\nl_{e_1}$- and $\nl_{e_2}$-coordinate axes and we are considering the projections of the charts to the $(\nl_{e_1},\nl_{e_2})$-coordinate plane.} \label{fig: corner}
\end{figure}

\s\n
(C) Suppose that $V(T)=\{i_1,i_2,i_3\}$ and $G(T)=\{e_1=(i_3,i_1),e_2=(i_3,i_2)\}$.
Let
\begin{equation}\label{eqn: triple 4}
S'=(\{i_1,i_3\},\{e_1\}), \quad S''=(\{i_2,i_3\},\{e_2\})
\end{equation}
be good subtrees of $T$. Case (C) is similar to Case (B) except that $\mathcal{M}_{i_1}$ may equal $\mathcal{M}_{i_2}$, in which case $\op{Aut}(T)$ is nontrivial.

\subsubsection{Gluing maps} \label{subsubsection: gluing maps}

Let $S$ be a good subforest of $T$ and $\arr T= \arr T(S)= T_1\sqcup \dots \sqcup T_m$. We construct $\E_{\arr T}^{\en/2}\to \V_{\arr T}^{\en/2}\times (\mathcal{L}-\varepsilon''-\tfrac{\varepsilon'''}{2},\infty)^{m-1}$ from $\pi_{T_i}^\en: \E_{T_i}^\en\to \V_{T_i}^\en$, $i=1,\dots, m$, as {\coblu in Section~\ref{subsection: definition of semi-global Kuranishi structures} (two paragraphs above Definition~\ref{defn: strata compatibility}).}

We define the gluing map $(\widetilde{\mathfrak{G}}_{\arr T}^\en,\mathfrak{G}_{\arr T}^\en)$ as follows: The map
$$\mathfrak{G}_{\arr T}^\en:  \V_{\arr T}^{\en/2}\times (\mathcal{L}-\varepsilon''-\tfrac{\varepsilon'''}{2},\infty)^{m-1} \to {\mathcal{G}}_{\delta,T}^{\E'},$$
is obtained from Theorem~\ref{thm: gluing}. The map $\widetilde{\mathfrak{G}}_{\arr T}^\en: \E_{\arr T}^{\en/2}\to \E^\en_{T}$ is given by $\widetilde{\mathfrak{G}}_{\arr T}^\en(x,\xi)= (\mathfrak{G}_{\arr T}^\en(x), \phi_x (\xi))$, where $\phi_x$ canonically identifies the fibers over $x$ and $\mathfrak{G}_{\arr T}^\en(x)$.

By Theorem~\ref{thm: gluing} and our choice of gluing parameter bound $\mathcal{L}-\varepsilon''-\tfrac{\varepsilon'''}{2}$ (see Definition~\ref{defn: bundles ET over VT}):

\begin{thm} \label{thm: gluing, new version}
\ref{SC1}(a) holds for $\mathfrak{G}_{\arr T}^\en$ and 
$\op{Im}\mathfrak{G}_{\arr T}^\en\supset \V_T$.
\end{thm}

Theorem~\ref{thm: iterated gluing} becomes:

\begin{thm} \label{thm: iterated gluing, new version}
\ref{SC1}(c) holds for the collection of gluing maps $\{\mathfrak{G}_{\arr T}^\en\}_{\arr T}$.
\end{thm}

\subsection{Multisections on products} \label{subsection: multisections on products}

Given a CH tree $T$, the goal of this subsection is to inductively define a multisection ${\mathfrak s}_T$ of $\E_T\to \V_T$; the construction of its slight enlargement $\mathfrak s_T^\en$ will be omitted.

\subsubsection{Melding} \label{subsubsection: melding}

Let $X$ be a smooth manifold, 
$C$ be a submanifold of $\R^m$ of dimension $m$ with corners (e.g., $[0,1]^m$, $[0,1)^m$, $[c,\infty)^m$) and $\pi:X\to C$ be a smooth map. \cb  
Let $R_1,\dots,R_a$ be submanifolds of $C$ of dimension $m$ with corners such that $\cup R_i=C$. 
For $\varepsilon_0>0$ small, let $R_i'\subset C$ be an open set that contains the $\varepsilon_0$-neighborhood of $R_i$. 

\nom[Ri]{$R_i$}{Submanifold of $C$ of dimension $m$ with corners}

A {\em melding of a collection of functions $\{g_i:\pi^{-1}(R_i')\to \R^b\}_{i=1}^a$ with respect to $\pi:X\to C$, $\{R_1,\dots, R_a\}$, and $\{R_1',\dots, R_a'\}$} \cb is a function $g:X \to \R^b$ defined as follows:  For each $i$, choose a smooth function $\phi_i: C\to [0,1]$ which is equal to $1$ on $R_i$ and $0$ outside of $R_i'$ and such that $\sum_i \phi_i >0$. (In other words, we are taking a specific partition of unity.)  We then set
\begin{equation}
    g=\frac{ \sum_i \pi^*(\phi_i) g_i}{\sum_i \pi^*(\phi_i)}.
\end{equation}
{\coblu A melding of a collection of sections can be defined similarly.}

\begin{rmk}
\coblu When melding of a collection $\{\mathfrak{s}_i=(\mathfrak{s}_{i1},\dots,\mathfrak{s}_{in_i})\}_{i=1}^a$ of lifted multisections, we do not meld all possible $a$-tuples $(\mathfrak{s}_{1j_1},\dots, \mathfrak{s}_{aj_a})$ of branches; instead we specify which $a$-tuples of branches are to be melded, based upon which branches are close to each \cb other.
\end{rmk}

\subsubsection{Some more gluing maps} \label{subsubsection: some more gluing maps}

In this subsection we slightly reformulate the gluing maps from Section~\ref{subsubsection: gluing maps}. 

Given a good subtree $S$ of a CH tree $T$ (in particular $S\not=\varnothing$ by definition), we write $S^\star$ for the CH tree 
\nom[Sstar]{$S^\star$}{CH obtained from a good subtree $S$ of $T$ by adding all $e\in E(T)$ such that $i(e)\in V(S)$}
obtained from $S$ by adding all $e\in E(T)$ such that $i(e)\in V(S)$ and viewing them as free edges.  We similarly define $S^\star$ for a good subforest $S$ of $T$. 

Let $S=\sqcup_{j=1}^b S_j$ be a good subforest of $T$. Consider the bundle 
$$\E_{T,S}^\en:=(\times_j \E_{S_j^\star}^\en) \times ( \times_{i\in V(T)-V(S)}\E_i^\en) \to \V_{T,S}^\en:=(\times_j \V_{S_j^\star}^\en) \times ( \times_{i\in V(T)-V(S)}\V_i^{\text{en}}).$$
As in Section~\ref{subsubsection: boundary strata}, let $\mathcal{G}_{\delta,T,S}=\mathcal{G}_\delta(\{\V_{S_j^\star}^\en\}_j,\{\V_i^\en\}_{i\in V(T)-V(S)})$ be the set of equivalence classes of maps that are $\delta$-close to breaking into representatives of $\V_{S_j^\star}^\en$ and $\V_i^\en$, where the ``gluing" is prescribed by the set of glued edges $G(T)$.  Let $\E'$ be the bundle over $\mathcal{G}_{\delta,T,S}$, whose fiber $E'_{(\F,u)}$ over $(\F,u)$ {\coblu is given by Equation~\eqref{equation: defn of E prime} (again the definition descends to the fiber $E'_{[\F,u]}$ over the equivalence class $[\F,u]$),} and let 
\begin{equation} \label{eqn: G T S}
    \mathcal{G}_{\delta,T,S}^{\E'}=\mathcal{G}_\delta^{\E'}(\{\V_{S_j^\star}^\en\}_j,\{\V_i^\en\}_{i\in V(T)-V(S)})=\{[\F,u]\in  \mathcal{G}_{\delta,T,S}~|~ \overline\bdry_J u\in \E'\}.
\end{equation}
There exists a gluing map 
$$\mathfrak{G}_{T,S}^\en: \V^{\en/2}_{T,S}\times (\mathcal{L}-\varepsilon''-\tfrac{\varepsilon'''}{2},\infty)^{|G(T)|-|G(S)|}\to \mathcal{G}_{\delta,T,S}^{\E'},$$
obtained from Theorem~\ref{thm: gluing}, as well as $\widetilde{\mathfrak{G}}^\en_{T,S}: \E_{T,S}^{\en/2}\to \E_T^\en$ as in Section~\ref{subsubsection: gluing maps}. 


\begin{rmk} \label{rmk: relating two gluing maps}
The gluing map $\mathfrak G^\en_{\arr T(S)}$ from Section~\ref{subsubsection: gluing maps} agrees with $\mathfrak G^\en_{T,S^c}$, where $S^c$ is a good subforest of $T$ and $G(S^c)=G(T)-G(S)$.
\end{rmk}

Theorem~\ref{thm: iterated gluing} can be rephrased as follows:

\begin{thm} \label{thm: comparison of two gluing maps}
Let $S=\sqcup_{j=1}^a S_j$ and $S'=\sqcup_{j=1}^a S_j'$ be good subforests of $T$ such that $S_j'$ is a good subforest of $S_j^\star$. 
Then the gluing maps $\mathfrak G_{T,S'}^\en$ and $\mathfrak G_{T,S}^\en\circ (\times_{j=1}^a \mathfrak G_{S_j^\star,S'_j}^\en,\op{id})$
are $C^1$-close with error $\to 0$ as all the components of $(\mathcal{L}-\varepsilon''-\tfrac{\varepsilon'''}{2},\infty)^{|G(T)-G(S')|}$ go to $\infty$. (Here if $S_j=S_j'$, then we take $\mathfrak G_{S_j^\star,S'_j}^\en=\op{id}$.)
\end{thm}

\subsubsection{Definition of the multisection ${\mathfrak s}_T^\en$} \label{subsubsection: defn of sT}

We have already constructed {\coblu lifted} multisections $\mathfrak s_i^\en$ of $\E_i^\en\to \V_i^\en$ that are $C^1$-close to the $0$-section and invariant under {\coblu $G_i$.} Assuming that we have already constructed all the $\mathfrak s_{S^\star}^\en$, where $S$ ranges over all good subtrees of $T$ such that $S^\star\not=T$, {\coblu and all the $\mathfrak s_{T'}^\en$, where $T'<T$, and} we have invariance under marker rotations and puncture reorderings, we inductively construct the multisection $\mathfrak s_T^\en$ by defining multisections on subsets of $\V_T^\en$ and applying the melding procedure.  
By construction it will be clear that, starting with $\mathfrak s_i^\en$ that are $C^1$-close to the $0$-section, $\mathfrak s_T^\en$ can also be made $C^1$-close to the $0$-section.

We are still using $0<\varepsilon'''\ll \varepsilon''$ from Section~\ref{subsection: trimming}.
We first define ``rectangular" regions of $[\mathcal{L}-\varepsilon'',\infty)^{|G(T)|}$ of type $R_i$ (appearing in the definition of melding in Section~\ref{subsubsection: melding}): Let
$$R_{T,S,1}:=\{\mathcal{L}+\varepsilon'''\geq \nl_e\geq \mathcal{L}-\varepsilon''~|~ e\in G(S)\}\cap \{\nl_e\geq \mathcal{L}-\varepsilon''~|~e\in G(T)-G(S) \},$$
where $S\not=\varnothing$ is a good subforest of $T$,\footnote{Note that we are using $\mathcal{L}+\varepsilon'''$ instead of the more natural $\mathcal{L}$.  This is to give us some space so that Definition~\ref{def: multisection of Kur} holds.} and
$$R_{T,S,2}:=\{\mathcal{L}+\varepsilon''\geq\nl_e\geq \mathcal{L}+\varepsilon'''~|~ e\in G(S)\}\cap \{\nl_e\geq \mathcal{L}+\varepsilon''~|~e\in G(T)-G(S)\},$$
where $S$ is a good subforest of $T$.  Also let $R'_{T,S,j}$, $j=1,2$, be rectangular regions of type $R_i'$ obtained from $R_{T,S,j}$ by extending each side by $\varepsilon_0=\varepsilon'''$. 
See Figure~\ref{fig: newcorner} for the corner corresponding to the tree $T$ from Example (B) in Section~\ref{subsubsection: examples}.

\begin{figure}[ht]
\vskip.15in
	\begin{overpic}[scale=1]{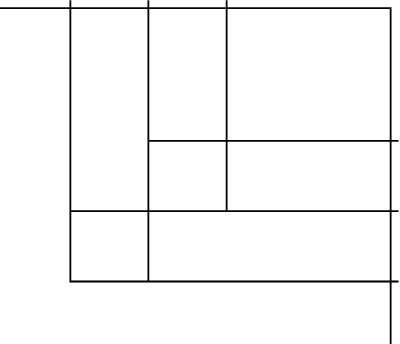}
		\put(-2,88){\small{$\nl_{e_1}$}} \put(10.5,88){\small{$\mathcal{L}-\varepsilon''$}}
		\put(30,88){\small $\mathcal{L}+\varepsilon'''$} \put(50,88){\small $\mathcal{L}+\varepsilon''$}
		\put(100,85){\small $\infty$}
		\put(100,0){\small $\nl_{e_2}$}
		\put(102,14){\small $\mathcal{L}-\varepsilon''$} \put(102,32){\small $\mathcal{L}+\varepsilon'''$}
		\put(102,49){\small $\mathcal{L}+\varepsilon''$}
		\put(70,63){\small $R_{T,\varnothing,2}$} \put(70,37){\small $R_{T,S'',2}$}
		\put(65,73){\tiny $\mathfrak s_{i_1}\times \mathfrak s_{i_2} \times \mathfrak s_{i_3}$}
		\put(70,44.5){\tiny $\mathfrak s_{i_1}\times \mathfrak s_{S''}$}
		\put(39,63){\small $R_{T,S',2}$}\put(39,37){\small $R_{T,T,2}$}
		\put(40.3,44.5){\tiny arbitrary}
		\put(39,73){\tiny $\mathfrak s_{S'}\times \mathfrak s_{i_3}$}
		\put(22,73){\tiny $\mathfrak{s}_{T\git S'}$}
		\put(19,53){\small $R_{T,S',1}$}\put(20,19){\small $R_{T,T,1}$} \put(60,19){\small $R_{T,S'',1}$}
		\put(22.6,27){\tiny $\mathfrak s_{T\git T}$}
		\put(63,27){\tiny $\mathfrak s_{T\git S''}$}
	\end{overpic}
	\caption{The rectangular regions and their multisections for the corner in Example (B) of Section~\ref{subsubsection: examples}, before melding. We are omitting $\widetilde {\mathfrak G}$ and $\op{stab}$ from the notation for the multisections.}
	\label{fig: newcorner}
\end{figure}

Next we extend $\mathfrak s_{T\git S}^\en$ to $\nl_T^{-1}(R'_{T,S,1})$ such that:
\begin{enumerate}
    \item if $S'\subset S$ is a good subforest, then $\op{stab}_{T\git S,T}({\mathfrak s}^\en_{T\git S})$ and $\op{stab}_{T\git S',T}({\mathfrak s}^\en_{T\git S'})$ agree on the overlap of $\nl_T^{-1}(R'_{T,S,1})$ and $\nl_T^{-1}(R'_{T,S',1})$;
    \item the collection $\{\mathfrak s_{T\git S}^\en\}_S$ is invariant under marker rotations and puncture reorderings.
\end{enumerate}
Here $\op{stab}_{T\git S,T}({\mathfrak s}^\en_{T\git S})$ is the restriction of $\mathfrak s_{T\git S}^\en$ to $\V_{T\git S, T}^\en$, followed by a stabilization given in Section~\ref{subsubsection: stabilizations} and inclusion of $\V_{T\git S,T}^\en$ into $\V_T^\en$. 
(1) is made possible by the commutativity of Diagram~\eqref{eqn: four charts first version}. 
\nom[stablization]{$\op{stab}_{T\git S,T}({\mathfrak s}^\en_{T\git S})$}{the restriction of $\mathfrak s_{T\git S}^\en$ to $\V_{T\git S, T}^\en$, followed by a stabilization given in Section~\ref{subsubsection: stabilizations} and inclusion of $\V_{T\git S,T}^\en$ into $\V_T^\en$}

\begin{defn}[Multisection $\mathfrak s_T^\en$]\label{defn: multisection sT}
The multisection ${\mathfrak s}_T^\en$ on $\E_T^\en\to \V_T^\en$ is the melding of the following multisections with respect to 
$$\nl_T=(\nl_e)_{\{e\in G(T)\}}: \V_T^\en\to (\mathcal L - \varepsilon''-\varepsilon''', \infty)^{|G(T)|}$$
and the collections $\{R_{T,S,j}\}$ and $\{R'_{T,S,j}\}$:
\be
\item[(R1)] on $\nl_T^{-1}(R'_{T,S,1})$, take the stabilization $\op{stab}_{T\git S,T}({\mathfrak s}_{T\git S}^\en)$ of ${\mathfrak s}_{T\git S}^\en$;
\item[(R2)] on $\nl_T^{-1}(R'_{T,S,2})$ such that $T\git S$ is a not a one-vertex tree, take the restriction of $(\widetilde {\mathfrak G}_{T,S}^\en)_* \mathfrak s^\en_{T,S}:=  \widetilde{\mathfrak G}_{T,S}^\en \circ\mathfrak s^\en_{T,S}\circ (\mathfrak G_{T,S}^\en)^{-1}$, where 
$$\mathfrak s^\en_{T,S}=(\mathfrak s_{S_j^\star}^\en)_j\times (\mathfrak s_{i}^\en)_{i\in V(T)-V(S)};$$
\item[(R3)] on $\nl_T^{-1}(R'_{T,S,2})$ such that $T\git S$ is a one-vertex tree, take an arbitrary multisection subject to invariance under marker rotations and puncture reorderings;
\ee
{\coblu where the collections of branches to be melded must satisfy the following {\em consistency condition}:
\be
\item[(CC)] the branches must agree on the overlaps as we approach the ``boundary".
\ee}
\end{defn}

Melding smoothes out the multisections defined on overlapping rectangular regions.  We assume that the smooth functions $\phi_{T,S,j}: R'_{T,S,j}\to [0,1]$ that appear in the definition of melding satisfy the following:
\begin{enumerate}
\item $\phi_{T,S,1}=0$ on $R'_{T,S,1}\cap \{\nl_e\geq \mathcal{L}+ \varepsilon'' +\varepsilon^{(4)} \mbox{ for some } e\in G(S)\}$, 
where $0<\varepsilon^{(4)}\ll \varepsilon'''$;
\item the collection $\{\phi_{T,S,j}\}$ is invariant under {\coblu marker rotations and puncture reorderings}.
\end{enumerate} 

{\coblu We say that two multisections $\mathfrak{s}$ and $\mathfrak{s}'$ on a bundle $E\to V$ {\em are $C^1$-close} if there exists a large integer $m$ such that both can be locally viewed as sections of $E^m\to V$ which are $C^1$-close modulo reordering.} 
The following lemma roughly states that whenever the regions of type (R1) and/or (R2) overlap, the relevant multisections are sufficiently close for ``long neck lengths'' and that melding does not modify things much for ``long neck lengths". 

\begin{lemma} \label{lemma: overlaps} $\mbox{}$
\be
\item If $S'\subset S$ is a good subforest, then $\op{stab}_{T\git S,T}({\mathfrak s}^\en_{T\git S})$ and $\op{stab}_{T\git S',T}({\mathfrak s}^\en_{T\git S'})$ agree on the overlap of $\nl_T^{-1}(R'_{T,S,1})$ and $\nl_T^{-1}(R'_{T,S',1})$.
\item If $S'\subsetneq S$ is a good subforest, then $(\widetilde{\mathfrak G}_{T,S}^\en)_*\mathfrak s^\en_{T,S}$ and $(\widetilde{\mathfrak G}_{T,S'}^\en)_*\mathfrak s^\en_{T,S'}$ are $C^1$-close on the overlap of $\nl_T^{-1}(R'_{T,S,2})$ and $\nl_T^{-1}(R'_{T,S',2})$ when all $\nl_e $, $e\in G(T)-G(S)$, are sufficiently large.
\item $(\widetilde {\mathfrak G}_{T,S}^\en)_*\mathfrak s^\en_{T,S}$ and $\op{stab}_{T\git S,T}({\mathfrak s}^\en_{T\git S})$ are $C^1$-close on the overlap of $\nl_T^{-1}(R'_{T,S,1})$ and $\nl_T^{-1}(R'_{T,S,2})$ when all $\nl_e $, $e\in G(T)-G(S)$, are sufficiently large.
\ee
\end{lemma}

\begin{proof}
(1) has already been arranged.

(2) is a consequence of Theorem~\ref{thm: comparison of two gluing maps} and the inductive step: Let $S=\sqcup_{j=1}^a S_j$ be a good subforest of $T$, let 
$S_j'$ be a good subforest of $S_j^\star$, and let $S'= \sqcup_j S_j'$.
By definition and Theorem~\ref{thm: comparison of two gluing maps},
\begin{align*}
(\widetilde{\mathfrak G}_{T,S}^\en)_*\mathfrak s^\en_{T,S}&=(\widetilde {\mathfrak G}_{T,S}^\en)_* ((\mathfrak s^\en_{S^\star_j})_{j=1}^a \times (\mathfrak s^\en_{i})_{i\in V(T)-V(S)} ),\\
    (\widetilde{\mathfrak G}_{T,S'}^\en)_*\mathfrak s^\en_{T,S'}& \approx (\widetilde {\mathfrak G}_{T,S}^\en)_*(((\widetilde{\mathfrak G}_{S_j^\star,S_j'}^\en)_*  \mathfrak s^\en_{S_j^\star,S_j'})_{j=1}^a \times (\mathfrak s^\en_{i})_{i\in V(T)-V(S)} ),
\end{align*}
where $\approx$ means $C^1$-close on their overlap. Since by induction we are assuming that $\mathfrak{s}^\en_{S_j^\star}$ and $(\widetilde{\mathfrak G}_{S_j^\star,S_j'}^\en)_*   \mathfrak s^\en_{S_j^\star,S_j'}$ are $C^1$-close on their relevant overlaps, with improved accuracy as $\nl_e \to \infty$ for all $e\in G(S_j^\star)-G(S_j')$, (2) follows.

(3) By induction we are assuming that 
$$(\widetilde {\mathfrak G}_{S_j^\star,S_j}^\en)_* \mathfrak s^\en_{S_j^\star,S_j}=  \mathfrak s^\en_{S_j^\star} = \op{stab}_{S_j^\star\git S_j^\star,S_j^\star}({\mathfrak s}^\en_{S_j^\star\git S_j^\star})$$
on the overlap $\nl_{S_j^\star}^{-1}(R'_{S_j^\star,S_j,2}) \cap \nl_{S_j^\star}^{-1}(R_{S_j^\star,S_j,1})$.
Then (3) is a consequence of the commutativity of
\begin{equation} \label{eqn: commut of stab and gluing}
\begin{diagram}
(\times_{j=1}^a\V_{{S_j^\star}\git S_j}^\en)\times (\times_{i\in V(T)-V(S)} \V_i^\en) & \rTo & (\times_{j=1}^a\V_{S_j^\star}^\en)\times  (\times_{i\in V(T)-V(S)} \V_i^\en) \\
\dTo & & \dTo\\
\V_{T\git S,T}^{\op{\tiny en}} & \rTo & \V_T^{\op{\tiny en}},
\end{diagram}
\end{equation}
where the horizontal arrows are inclusions and the vertical arrows are gluing maps. 
\end{proof}

\begin{lemma}\label{lemma: strata compatibility}
If $S=\sqcup_{j=1}^a S_j$ is a good subforest of $T$, then on all of 
$$R_{T,S}:=\{\nl_e> \mathcal{L}-\varepsilon''~|~e\in G(S)\}\cap \{\nl_e\gg 0~|~ e\in G(T)-G(S)\}$$
we have
\begin{equation} \label{eqn: approx}
\mathfrak{s}_T^\en \approx (\widetilde{\mathfrak G}_{T,S}^\en)_* ((\mathfrak s_{S_j^\star}^\en)_j \times (\mathfrak s_i^\en)_{i\in V(T)-V(S)}).
\end{equation}
\end{lemma}

\begin{proof}
The lemma is proved by induction.
\eqref{eqn: approx} holds on $R_{T,S,2}\cap R_{T,S}$ by (R2) of Definition~\ref{defn: multisection sT} and on $R_{T,S,1}\cap R_{T,S}$ by (R1) and 
$$\op{stab}_{T\git S,T}({\mathfrak s}_{T\git S}^\en)\approx  (\widetilde {\mathfrak G}_{T,S}^\en)_* ((\op{stab}_{S_j^\star\git S_j,S_j^\star}({\mathfrak s}_{S_j^\star\git S_j}^\en))_j, (\mathfrak s_i^{\en})_{i\in V(T)-V(S)}).$$
For any $S'=\sqcup_j S_j'$ as in the proof of Lemma~\ref{lemma: overlaps}(2), \eqref{eqn: approx} holds on $R_{T,S',2}\cap R_{T,S}$ since
\begin{align*}
    (\widetilde {\mathfrak G}_{T,S'}^\en)_*& ((\mathfrak s_{(S'_j)^\star}^\en)_j, (\mathfrak s_i^\en)_{i\in V(T)-V(S')})\\
    \approx & (\widetilde {\mathfrak G}_{T,S}^\en)_* (((\widetilde {\mathfrak G}^\en _{S_j^\star, S_j'})_*(\mathfrak s_{(S'_j)^\star}^\en, (\mathfrak s_i^\en)_{i\in V(S_j)-V(S'_j)}))_j,  (\mathfrak s_i^\en)_{i\in V(T)-V(S)}).
\end{align*}
by Theorem~\ref{thm: comparison of two gluing maps} and $\mathfrak{s}_{S_j^\star}\approx (\widetilde {\mathfrak G}^\en _{S_j^\star, S_j'})_*(\mathfrak s_{(S'_j)^\star}^\en, (\mathfrak s_i^\en)_{i\in V(S_j)-V(S'_j)})$ by induction on $R_{S_j^\star,S_j'}.$ The lemma follows by observing that the regions of the above three types cover $R_{T,S}$.
\end{proof}

\begin{lemma} \label{lemma: verify conditions of multisection}
The collection $\{\mathfrak s_T^\en\}_T$ satisfies the conditions of Definition~\ref{def: multisection of Kur}.
\end{lemma}

\begin{proof}
(1) is immediate from (R1). 
(2) follows from Remark~\ref{rmk: relating two gluing maps} and Lemma~\ref{lemma: strata compatibility}. (3) follows from the symmetry built into the construction.
\end{proof}
\cb

\subsection{Modifications when (H) does not hold} \label{subsection: H does not hold}
 
When \hyperlink{H}{(H)} is not satisfied, we generally need to add or remove additional punctures to the domain Riemann surface when we define morphisms $\phi_{T',T}: \mathcal{C}_{T'}\to \mathcal{C}_T$.
In this subsection we describe how to modify the definition of the interior semi-global Kuranishi charts to accommodate the additional punctures. 

Let $\delta>0$ be a small constant as before, used when referring to curves that are close to breaking. Recall that we are taking the moduli spaces $\mathcal{M}_1,\dots, \mathcal{M}_\rho$ to have domains of the form $\dot F_i= F_i -{\bf p}_i$, i.e., {\em we have not removed the punctures ${\bf q}_i$ yet.}

\s\n
{\em Initial step.} Suppose $\bdry(\mathcal{M}_{i_0}/\R)=\varnothing$. Then there exists a small positive constant $\varepsilon'=\varepsilon'(\mathcal{M}_{i_0})>0$ such that for each $[\F,u]=[(F,j,{\bf p}, {\bf r}),u]\in \mathcal{M}_{i_0}$ there exists a set of removable punctures ${\bf q}$ as in Section~\ref{subsubsection: nonnegative Euler char} when $\chi(\dot F)=0$ or $1$. Let $\K_{i_0}= \mathcal{M}_{i_0}/\R$. 
The only difference is that in the definition of $\V_{i_0}$
we use $s_{\pm,i}$ that depends on $((F,j,{\bf p}, {\bf q},{\bf r}),u)$ instead of $(\F,u)$.

\nom[FcalT]{$\F^T$}{$\F^T=(F,j,{\bf p},{\bf q}^T,{\bf r})$, where $(\F,u)=((F,j,{\bf p},{\bf q},{\bf r}),u)$ is close to breaking into a building described by a CH tree $T$ and ${\bf q}^T$ is $\delta$-close to the union of removable punctures of the levels}
\nom[Lprime]{$\mathcal{L}'\gg \mathcal{L}''\gg 0$}{Large parameters $\ll \mathcal{L}$ involved in the construction of the Kuranishi charts when (H) does not hold}
\s\n
{\em Puncturing map.} 
{\coblu Let $[\F,u]=[(F,j,{\bf p},{\bf r}),u]\in \mathcal{M}_{i_0}$ be such} that $(\F,u)$ is $\delta'$-close to breaking into the building $\times_{i\in V(T)}(\F_i,u_i)$, as prescribed by a CH tree $T$ such that $T\git T=i_0$, and that each $\F_i=(F_i,j_i,{\bf p}_i, {\bf r}_i)$ has been assigned a set ${\bf q}_i$ of removable punctures; we write $\F^i_i=(F_i,j_i,{\bf p}_i,{\bf q}_i, {\bf r}_i)$ by analogy with $\F^T$ below. {\coblu If $\delta'>0$ is sufficiently small, then there exists a set ${\bf q}^T= \sqcup_{i\in V(T)} {\bf q}_i'$ of removable punctures of $\F$ that is defined using the analogs of Equations~\eqref{eqn: s prime 1} and \eqref{eqn: s prime 2} and is close to $\sqcup_{i\in V(T)} {\bf q}_i$ and we can define the following {\em puncturing map}:}
$$\mathfrak{Punc}^T: (\F,u)\mapsto (\F^T=(F,j,{\bf p},{\bf q}^T,{\bf r}),u).$$
{\coblu We further shrink $\delta'>0$ if necessary so that $(\F^T,u)$ is $\delta$-close to $\times_{i\in V(T)}(\F_i^i,u)$.} 

\s
Choose $\mathcal{L}'',\mathcal{L}',\mathcal{L}$ such that $0\ll \mathcal{L}''\ll \mathcal{L}'\ll \mathcal{L}$.

\s\n
{\em Warm-up.}   We consider the simplest nontrivial case where 
$$V(T)=\{i_1,i_2\}, \quad G(T)=\{e=(i_2,i_1)\},$$
$$\bdry(\mathcal{M}_{i_0}/\R)= (\mathcal{M}_{i_1}/\R)\times (\mathcal{M}_{i_2}/\R).$$

Let $\mathcal{B}_{i_0}$ be the set of $[\F,u]=[(F,j,{\bf p}, {\bf r}),u]\in \mathcal{M}_{i_0}$ that are $\delta'$-close to breaking into
$$[\F_{i_1},u_{i_1}]\times [\F_{i_2},u_{i_2}]\in \mathcal{M}_{i_1}\times \mathcal{M}_{i_2}, ~~ \F_{i_j}=(F_{i_j},j_{i_j},{\bf p}_{i_j}, {\bf r}_{i_j}),~~ j=1,2,$$
where any of $\chi(\dot F), \chi(\dot F_{i_1}), \chi(\dot F_{i_2})$ may be nonnegative. Given $[\F,u]\in \mathcal{B}_{i_0}$, let $(\F^T,u)=\mathfrak{Punc}^T (\F,u )$ so that $(\F^T,u)$ is $\delta$-close to breaking into $(\F_{i_1}^{i_1},u_1)\cup (\F_{i_2}^{i_2}, u_2)$. We also choose a sufficiently large compact set $\K_{i_0}\subset \mathcal{M}_{i_0}$ such that $\K_{i_0}\cup\mathcal{B}_{i_0}=\mathcal{M}_{i_0}$ and a small constant $\varepsilon'(\K_{i_0})>0$ to define the removable punctures ${\bf q}$ for all $[\F,u]\in \K_{i_0}$. 

Consider the two hyperbolic metrics
$$g^{T\git T}:=g_{(\dot F -{\bf q} ,j)}, \quad g^T:=g_{(\dot F  - {\bf q}^T,j)}$$
where they are defined. 
Let $\nl_e^T(\F ,u )$ be the neck length function defined using $g^T$ and let $s_{\pm,j}^{T\git T}(\F,u)$ be $s_{\pm,j}(\F,u)$ defined using $g^{T\git T}$. 

For each $a<b$, choose a smooth nondecreasing function $\lambda_{a,b}:\R\to [0,1]$ such that $\lambda_{a,b}(a)=0$ and $\lambda_{a,b}(b)=1$ and then set $\lambda=  \lambda_{\mathcal{L}'',\mathcal{L}'}$. 

We now state the modifications (Mod$_1$) and (Mod$_2$) in the processes of trimming and constructing the obstruction bundle: 

\begin{enumerate}
    \item[(Mod$_1$)] In the definitions of $\K_{i_0}$ and $\V_{i_0}$ (see Definition~\ref{defn: choice of interior chart}), the neck length $\nl_e(\F,u)$ is taken to be $\nl_e^T(\F,u)$ if $(\F^T,u)$ is $\delta$-close to breaking and is taken to be zero otherwise. (In particular, $\K_{i_0}$ is chosen sufficiently large so that all the curves $(\F,u)$ with $\nl_e^T(\F,u)\leq \mathcal{L}-\varepsilon'''$ are contained.)
\end{enumerate}


\begin{enumerate}
    \item[(Mod$_2$)] $E_{(\F ,u )}$ is defined as in Step 2 of Theorem~\ref{thm: construction of L-transverse subbundle} but using $\widetilde{s}_{\pm,j}(\F ,u )$, which in turn is defined as:
\begin{enumerate}
\item $s_{\pm,j}^{T\git T}(\F ,u )$ on $\{\nl_e^T\leq \mathcal{L}''\}$ and  $s_{\pm,j}^{T}(\F ,u )$ on $\{\nl_e^T\geq \mathcal{L}'\}$; and
\item the interpolation $(1-\lambda(\nl_{e}^T))s_{\pm,j}^{T\git T}(\F ,u )+\lambda(\nl_{e}^T) s_{\pm,j}^T(\F ,u ),$ on $\{\mathcal{L}''\leq \nl_e^T\leq \mathcal{L}'\}$.
\end{enumerate}

\end{enumerate}

On $\V_{i_0}$ the bundle $\E_{i_0}$ is obtained from the bundle with fibers $E_{(\F ,u )}$; on $\V_{i_0,T}$, still defined using the parameter $\mathcal{L}$, the bundle $\E_T$ agrees with the stabilization $E'_{(\F ,u )}$ of $E_{(\F ,u )}$, constructed as in Equation~\eqref{equation: defn of E prime} using $g^T$.

\s\n
{\em General case.} We will describe a neighborhood of the corner of $\V_{i_0}$ given by the CH tree $T$ such that $i_0=T\git T$. Let $(\F,u)$ be as in the ``Puncturing map" paragraph and 
and let $(\F^T,u)=\mathfrak{Punc}^T (\F,u ).$  As before, if $\mathcal{B}_{i_0}$ is the set of curves that are $\delta'$-close to breaking, then we choose a sufficiently large compact set $\K_{i_0}\subset \mathcal{M}_{i_0}$ such that $\K_{i_0}\cup\mathcal{B}_{i_0}=\mathcal{M}_{i_0}$ and a small constant $\varepsilon'(\K_{i_0})>0$ to define the removable punctures ${\bf q}$ for all $[\F,u]\in \K_{i_0}$. 

Let $g^{T\git T}$ be the hyperbolic metric $g_{(\dot F-{\bf q},j)}$ and, for a good subforest $S$ of $T$ such that $T\git S\not=T\git T$, let $g^{T\git S}$ be the hyperbolic metric $g_{(\dot F-{\bf q}^{T\git S},j)}$, where {\coblu the set of removable punctures ${\bf q}^{T\git S}$ is as in the definition of $\mathfrak{Punc}^{T\git S}$.} Let $\nl^{T\git S}_e(\F,u)$ be the neck length function defined using $g^{T\git S}$ for $e\not\in E(S)$ and let $s_{\pm,j}^{T\git S}(\F,u)$ be $s_{\pm,j}(\F,u)$ defined using $g^{T\git S}$. Also let $e^C$ (``the complement of $e$") be the good subforest of $T$ consisting of glued edges $\not=e$.

\begin{enumerate}
    \item[(Mod$_1$)] In the definitions of $\K_{i_0}$ and $\V_{i_0}$  (see Definition~\ref{defn: choice of interior chart}), the neck length $\nl_e(\F,u)$ is taken to be $\nl_e^{T\git e^C}(\F,u)$ if $(\F^{T\git e^C},u)$ is $\delta$-close to breaking into a building with a long neck corresponding to $e$ and is taken to be zero otherwise.
\item[(Mod$_2$)] $E_{(\F,u)}$ is defined as in Step 2 of Theorem~\ref{thm: construction of L-transverse subbundle} but using $\widetilde{s}_{\pm,j}(\F,u)$, which is given as follows: 
for good subforests $S'\subset S$ of $T$, the interpolation 
$$\sum_{S'\subset S''\subset S} \left(\prod_{e\in E(S'')-E(S')} (1-\lambda(\nl^{T\git e^C}_{e})) \prod_{e\in E(S)-E(S'')} \lambda(\nl^{T\git e^C}_{e}) \right) s_{\pm,j}^{T\git S''}(\F,u),$$ 
where the summation is over all good subforests $S''$ such that $S'\subset S''\subset S$,
on
\begin{gather*}
    \{\mathcal{L}''\leq \nl_e^{T\git e^C}\leq \mathcal{L}'~|~e\in E(S)-E(S')\}\cap \{ \nl^{T\git e^C}_e> \mathcal{L}'~|~e\not\in E(S)\}\\
     \cap \{ \nl^{T\git e^C}_e\leq \mathcal{L}''~|~e\in E(S')\}.
\end{gather*}

\end{enumerate}

Near each corner/edge corresponding to $T$ with $T\git T=i_0$ on $\V_{i_0}$ the bundle $\E_{i_0}$ is obtained from the bundle with fibers $E_{(\F,u)}$ as above after quotienting by domain automorphisms.  The definitions of the edges and corners of $\E_{i_0}$ and $\V_{i_0}$ corresponding to the various $T$ with $T\git T=i_0$ agree by construction.  This gives the interior obstruction bundles $\pi_{i_0}:\E_{i_0}\to \V_{i_0}$ and hence the interior semi-global Kuranishi charts.  The slight enlargements $\pi_{i_0}^\en:\E_{i_0}^\en\to \V_{i_0}^\en$ can be defined in the same manner.

\s\n
{\em Example (B).} Figure~\ref{fig: corner2} corresponds to Example (B) in Section~\ref{subsubsection: examples}.
The numbers indicate the regions corresponding to the following: (1) $s^{T\git T}_{\pm,j}(\F,u)$, (2) $s^{T\git S'}_{\pm, j}(\F,u)$, (3) $s^{T\git S''}_{\pm,j}(\F,u)$, (4) $s^{T}_{\pm,j}(\F,u)$, and (5) and (6) are the interpolation regions. (6) is given by:
\begin{gather*}
    (1-\lambda(\nl^{T\git e_1^C}_{e_1}))(1-\lambda(\nl^{T\git e_2^C}_{e_2}))s^{T\git T} + \lambda(\nl^{T\git e_1^C}_{e_1})(1-\lambda(\nl^{T\git e_2^C}_{e_2}))s^{T\git S''} \\
    + (1-\lambda(\nl^{T\git e_1^C}_{e_1}))\lambda(\nl^{T\git e_2^C}_{e_2})s^{T\git S'} +
    \lambda(\nl^{T\git e_1^C}_{e_1})\lambda(\nl^{T\git e_2^C}_{e_2})s^{T}. 
\end{gather*}

\begin{figure}[ht]
\begin{center}
\psfragscanon
\psfrag{A}{\tiny $\nl_{e_1}^{T\git e_1^C}$}
\psfrag{B}{\tiny $\nl_{e_2}^{T\git e_2^C}$}
\psfrag{E}{\tiny (1)}
\psfrag{F}{\tiny (2)}
\psfrag{G}{\tiny (3)}
\psfrag{H}{\tiny (4)}
\psfrag{I}{\tiny (5)}
\psfrag{J}{\tiny (6)}
\psfrag{c}{\tiny $\quad\quad \mathcal{L}''$}
\psfrag{d}{\tiny $\mathcal{L}'$}
\psfrag{e}{\tiny $\mathcal{L}'+\varepsilon'''$}
\psfrag{f}{\tiny $\mathcal{L}''$}
\psfrag{g}{\tiny $\mathcal{L}'$}
\psfrag{h}{\tiny $\mathcal{L}'+\varepsilon'''$}
\psfrag{i}{\tiny $\infty$}
\psfrag{j}{\tiny $\mathcal{L}$}
\psfrag{k}{\tiny $\mathcal{L}$}
\includegraphics[width=5cm]{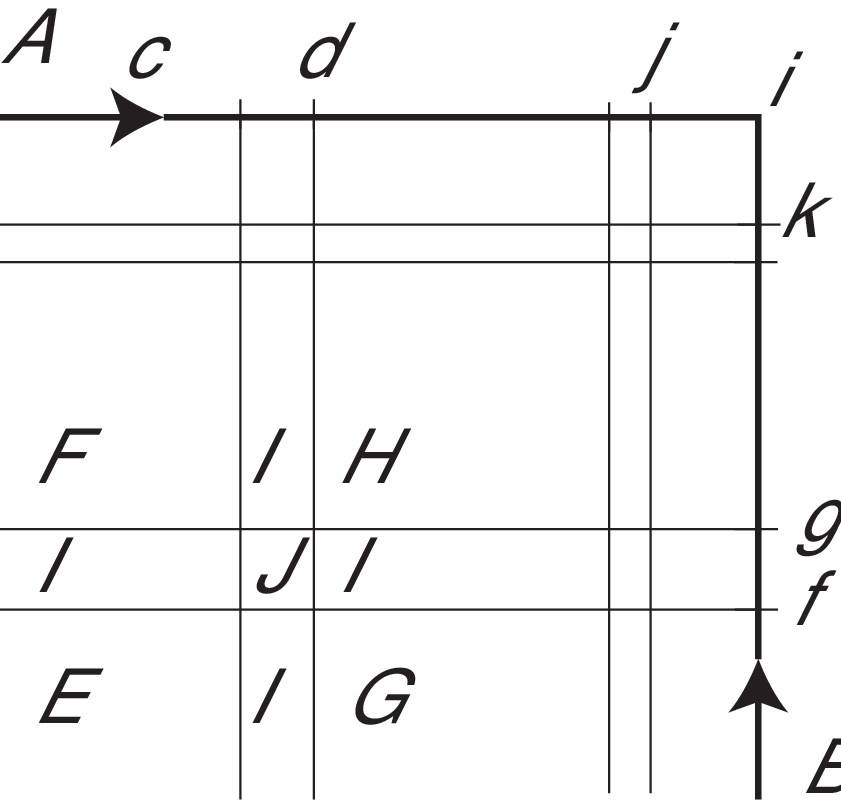}
\end{center}
\caption{The horizontal and vertical lines represent the $\nl_{e_1}^{T\git e_1^C}$- and $\nl_{e_2}^{T\git e_2^C}$-coordinate axes and we are considering the projections to the $(\nl_{e_1}^{T\git e_1^C},\nl_{e_2}^{T\git e_2^C})$-coordinate plane. } \label{fig: corner2}
\end{figure}

The bundles $\pi_T: \E_T\to\V_T$ and their enlargements are defined in the same way as in Section~\ref{subsubsection: boundary strata}, the morphisms $\phi_{T\git S, T}: \mathcal{C}_{T\git S}\to \mathcal{C}_T$ as in Section~\ref{subsubsection: morphisms}, the gluing maps as in Section~\ref{subsubsection: gluing maps}, and the multisections as in Section~\ref{subsection: multisections on products}, where all the occurrences of $\nl_e(\F,u)$ are replaced by $\nl_e^{T\git e^C}(\F,u)$. \cb

\subsection{Verification of \ref{K1}--\ref{K8} and \ref{SC1},\ref{SC2}} \label{subsection: verification}

We now verify that {\coblu \ref{K1}--\ref{K8} and \ref{SC1},\ref{SC2}} hold for our semi-global Kuranishi structure that we constructed in Sections~\ref{subsection: overview of construction}--\ref{subsection: H does not hold}. 

\ref{K1} The Kuranishi charts were constructed in Sections~\ref{subsubsection: interior charts and multisections} (defined in Definition~\ref{defn: choice of interior chart}) and \ref{subsubsection: boundary strata} (defined in Definition~\ref{defn: bundles ET over VT}). \ref{K1}(a) and (c) are immediate from the construction and the first line of 
\ref{K1}(b) is a consequence of Definition~\ref{defn: bundles ET over VT}(4). 

\ref{K2} The morphisms were defined in Section~\ref{subsubsection: morphisms}. \ref{K2}(a)--(f) are easy consequences of the definition. The unique morphism property is a consequence of the commutativity of \eqref{eqn: four charts first version}.

\ref{K3} is immediate. 

\ref{K4} is immediate from the definition of the morphisms.  

\ref{K5} follows from Lemma~\ref{lemma: covers}.

\ref{K6}--\ref{K8} The symmetries are built into the construction.  The interior charts satisfy Definition~\ref{defn: choice of interior chart}($\mathcal{I}$3). The boundary charts are defined using the interior charts and the neck length function, and the latter is geometrically defined (i.e., defined using hyperbolic geometry) and commutes with marker rotations and puncture reorderings.

\ref{SC1} The gluing maps $\mathfrak{G}_{\arr T}^\en$ and $\widetilde{\mathfrak{G}}_{\arr T}^\en$ were defined in Section~\ref{subsubsection: gluing maps}. \ref{SC1}(a) follows from Theorem~\ref{thm: gluing, new version} and \ref{SC1}(c) follows from Theorem~\ref{thm: iterated gluing, new version}. The first statement of \ref{SC1}(d) follows from the second statement of Theorem~\ref{thm: gluing}, and the second statement  of \ref{SC1}(d) follows from Gromov-Hofer compactness. Finally, \ref{SC1}(b) is a consequence of the proof of gluing in Section~\ref{subsection: defn of gluing map} and in particular the fact that the error (given by the left-hand side of \eqref{Theta pm} in the case of a two-level gluing) goes to zero as the gluing parameters go to $\infty$. 

{\coblu \ref{SC2} The symmetries are built into the construction.}

\subsection{Cobordisms} \label{subsection: cob case}

In this subsection we describe the modifications needed for the cobordism maps in Section~\ref{chain map}. 

Let $(\widehat W= W\cup ([1,\infty)\times M_+)\cup ((-\infty,-1]\times M_-),\widehat \alpha)$ be the completion of the compact Liouville cobordism $(W,\alpha)$ from $(M_+,\alpha_+)$ to $(M_-,\alpha_-)$ as in Section \ref{subsection: almost complex structures}. Given $L_+\leq L_-$, let $J$ be an almost complex structure on $\widehat W$ which restricts to $J_+$ at the positive end and to $J_-$ at the negative end. Suppose $(\widehat \alpha, J)$ is an $(L_+,L_-)$-simple pair and $J$ is $(L_+,L_-)$-end-generic.

\subsubsection{Mixed curves} \label{subsubsection: mixed curves}

We first define a slightly unusual, infinite-dimensional, space $\mathcal{M}_J^{\op{ind}=k}(\dot F', \widehat{W}; \gamma; \bs\gamma_+\sqcup \bs\gamma_-)$ of $J$-holomorphic maps $u$.  Strictly speaking $u$ does not need to be $J$-holomorphic, provided (iii) and (iv) below hold; in other words, it suffices to keep track of the data of the ends of $u$ and the Fredholm index of the completion $\tilde u$, defined below.

Let $\gamma$ be a Reeb orbit for $\alpha_+$ and let $\bs\gamma_+=(\gamma_{+,1},\dots,\gamma_{+,l_+})$ and $\bs\gamma_-=(\gamma_{-,1},\dots,\gamma_{-,l_-})$ be ordered tuples of Reeb orbits for $\alpha_+$ and $\alpha_-$.
Let 
\begin{align*}
    \mathcal{M}_J^{\op{ind}=k}(\dot F', \widehat{W}; \gamma; \bs\gamma_+\sqcup \bs\gamma_-)& = \{\mbox{$J$-holomorphic maps $u: (\dot F',j)\to (\widehat{W},J)$}\\
    & \qquad \qquad \mbox{of ``Fredholm index'' $k$ satisfying (i)--(iv)}\}, 
\end{align*}
where: 
\be
\item[(i)] $\dot F=F-\{p\}-\bf p_+ - \bf p_-$, where $(F,j)$ is a closed Riemann surface and $\{p\}$, ${\bf p}_+=\{p_{+,1},\dots, p_{+,l_+}\}$, and ${\bf p}_-=\{p_{-,1},\dots, p_{-,l_-}\}$ are disjoint sets of punctures;
\item[(ii)] $\dot F'$ is $\dot F$ minus the union of closed disk neighborhoods $N(p_{+,i})$ of $p_{+,i}$ that are mutually disjoint and disjoint from the other punctures;
\item[(iii)] $u$ is asymptotic to $\gamma$ at the positive end near $p$ and to $\bs\gamma_-$ at the negative end near $\bf p_-$; 
\item[(iv)] there exists an oriented parametrization $\R/\mathcal{A}_{\alpha_+}(\gamma_{+,i})\Z$ of $\bdry N(p_{+,i})$ (here we are using the orientation as the boundary of $N(p_{+,i})$ and simple coordinates for $\gamma_{+,i}$) such that 
$$u|_{\bdry N(p_{+,i})}: \R/\mathcal{A}_{\alpha_+}(\gamma_{+,i})\Z\to [1,\infty)\times \R/\mathcal{A}_{\alpha_+}(\gamma_{+,i})\Z \times D_\delta$$ 
is $\delta$-close to the map $t\mapsto (s_i,t,0)$ for some $s_i\in [2,\infty)$,
\ee
and the ``Fredholm index'' of $u$ is the Fredholm index of the completion $\tilde u$, obtained from $u$ by attaching cylindrical ends that are close to $(-\infty,s_i]\times \gamma_{+,i}\subset (-\infty,s_i]\times M_+$ and asymptotic to $\gamma_{+,i}$ as $s\to -\infty$. {\em The curve $\tilde u$ does not exist in $\widehat W$.}

As before, we choose an ordering $\mathcal{M}_1,\mathcal{M}_2,\dots,\mathcal{M}_\rho$ of moduli spaces where each $\mathcal{M}_i$ is one of
\begin{align*}
\mathcal{M}_i^J&=\mathcal{M}_J^{\op{ind}=k_i}(\dot F_i, \widehat{W}; \gamma_{i,+};\bs\gamma_{i,-}),\\ \mathcal{M}_i^{J_\pm}&=\mathcal{M}_{J_\pm}^{\op{ind}=k_i}(\dot F_i,\R\times M_\pm;\gamma_{i,+};\bs\gamma_{i,-}),\\
\mathcal{M}_i^{J,\op{Mixed}}&=\mathcal{M}_J^{\op{ind}=k_i}(\dot F_i', \widehat{W}; \gamma_i; \bs\gamma_{i,+}\sqcup \bs\gamma_{i,-}).
\end{align*}
Here $\dot F_i$ is a connected planar surface, $\mathcal{A}_{\alpha_\pm}(\gamma_{i,+})\leq L_\pm$, and each component of each level of a building of $\bdry \mathcal{M}_i$ is either a trivial cylinder or in $\mathcal{M}_j$ with $j<i$ (for the moment we take the definition of $\bdry \mathcal{M}_i$ to be the set of curves that are close to breaking into curves in $\mathcal{M}_i^J, \mathcal{M}_i^{J_\pm},\mathcal{M}_i^{J,\op{Mixed}}$). The moduli spaces $\mathcal{M}_i^J, \mathcal{M}_i^{J_\pm},\mathcal{M}_i^{J,\op{Mixed}}$ are said to be {\em of $\op{(Cob)}$ type, $\op{(Symp)}$ type, and $\op{(Mixed)}$ type,} respectively.

Let us assume that Condition (H) holds.

\subsubsection{CH trees of $\op{(Cob)}$ type}

\begin{defn} [CH tree of $\op{(Cob)}$ type] \label{defn: cob type}
A {\em CH tree $T$ of $\op{(Cob)}$ type} is defined in the same way as a CH tree of $\op{(Symp)}$ type, with the following extra data:
\begin{enumerate}
\item There is a splitting $V(T)=V^{\op{Cob}}(T) \sqcup V^{\op{Symp}}(T)\sqcup V^{\op{Mixed}}(T)$ and the elements of $V^{\op{Cob}}(T), V^{\op{Symp}}(T), V^{\op{Mixed}}(T)$ are called {\em $\op{(Cob)}$ vertices, $\op{(Symp)}$ vertices, and $\op{(Mixed)}$ vertices}, respectively.
\item Each maximal directed path of the direct tree passes through precisely one $\op{(Cob)}$ vertex or at least one $\op{(Mixed)}$ vertex.  Every $\op{(Mixed)}$ vertex has a directed path to some $\op{(Cob)}$ vertex.
\item A $\op{(Cob)}$ vertex (resp.\ a $\op{(Symp)}$ vertex sitting above some $\op{(Cob)}$ or $\op{(Mixed)}$ vertex, a $\op{(Symp)}$ which does not sit above any $\op{(Cob)}$ or $\op{(Mixed)}$ vertex, a $\op{(Mixed)}$ vertex) is labeled by an index corresponding to a moduli space of type $\mathcal{M}_i^J$ (resp\, $\mathcal{M}_i^{J_+}$, $\mathcal{M}_i^{J_-}$, $\mathcal{M}_i^{J,\op{Mixed}}$).
\end{enumerate}
\end{defn}


\begin{defn}
Let $T$ be a CH tree of {\em $\op{(Cob)}$ type} and let $S$ be a subtree of $T$.
\be
\item $S$ is a {\em good subtree of $T$} if it has no free edges and is either a $1$-vertex tree with a $\op{(Mixed)}$ vertex or has at least one glued edge.
\item $S$ is {\em of $\op{(Symp)}$ type} if $S$ has only $\op{(Symp)}$ vertices.   
\item $S$ is {\em of $\op{(Cob)}$ type} if $S$ has at least one $\op{(Cob)}$ vertex and every $\op{(Cob)}$ vertex of $T$ that is reached by a directed path from a vertex in $S$ is also a vertex of $S$. (Note that it is possible for the root of $S$ to be a $\op{(Mixed)}$ vertex.)
\item $S$ is {\em incomplete} if it is neither of $\op{(Symp)}$ type nor of $\op{(Cob)}$ type. 
\ee
\nom[ClS]{$Cl(S)$}{Closure of a good subtree $S$ of $T$}
A disjoint union $S=\sqcup_i S_i$ of good subtrees of $T$ is a {\em good subforest of $T$}. It is of {$\op{(Symp})$ (resp.\ $\op{(Cob})$) type} if all the $S_i$ are of $\op{(Symp)}$ (resp.\ $\op{(Cob)}$) type. The {\em closure $Cl(S)$} of a good subforest $S$ is the smallest good subforest in $T$ containing $S$ such that all of its trees are subtrees of $\op{(Symp)}$ or $\op{(Cob)}$ type. 
\end{defn}

\begin{defn}
An {\em isomorphism} of CH trees of $\op{(Cob)}$ type is defined in the same way as an isomorphism of CH trees of $\op{(Symp)}$ type, except that we additionally require that $\op{(Symp)}$ vertices be taken to $\op{(Symp)}$ vertices, $\op{(Cob)}$ vertices to $\op{(Cob)}$ vertices, and $\op{(Mixed)}$ vertices to $\op{(Mixed)}$ vertices.
\end{defn}

\n
{\em Main difference between the $\op{(Symp)}$ and $\op{(Cob)}$ cases.}  In the $\op{(Symp)}$ case, $\nl_e$ for all the glued edges $e\in G(T)$ are independent.  The main difference with the $\op{(Cob)}$ case is that for any two directed paths $\gamma_1$, $\gamma_2$ from the top vertex to a $\op{(Cob)}$ vertex, $\nl_{\gamma_1}$ and $\nl_{\gamma_2}$ are dependent, where
\begin{equation}
\nl_\gamma:= \sum_{e\in G(\gamma)} \nl_e.
\end{equation}
This motivates the definitions in the following subsections.

\subsubsection{Decomposition of $\mathcal{M}_i$ of $\op{(Cob)}$ type}

For $\mathcal{M}_i$ of $\op{(Cob)}$ type, we define the subset $\mathcal{M}_i^\dagger$ consisting of $[\F,u]$ for which
\be
\item[($\dagger$)] there exists a path from the positive puncture of $\ddot F$ to a negative puncture which does not pass through an annular component $A$ of $\op{Thin}_\varepsilon(\ddot F, g_{(\ddot F,j)})$ with $\nl\geq \mathcal{L}$. (By Assumption (H), $\ddot F= \dot F$.) 
\ee 
Note that this does not preclude all annular components $A$ with $\nl\geq \mathcal{L}$.  In fact, for $[\F,u]\in \mathcal{M}_i^\dagger$, let $\{A_i\}_i$ be the set of annular components of $\op{Thin}_\varepsilon(\ddot F, g_{(\ddot F,j)})$ with $\nl\geq \mathcal{L}$. Then the restriction of $u$ to the component of $\ddot F- \sqcup_i A_i$ containing the positive puncture can be viewed as an element of some $\mathcal{M}_i^{J,\op{Mixed}}$ by ($\dagger$). 

We then view $\mathcal{M}_i^\dagger$ as the union $\cup_{T'} \mathcal{M}_{T'}^\dagger$, where the tree $T'$ encodes some topological data of $[\F,u]$ and is constructed as follows: For each component $C$ of $\ddot F- \sqcup_i A_i$, view the restriction $u|_C$ as an element of type $\mathcal{M}_{j(C)}^J$, $\mathcal{M}_{j(C)}^{J_\pm}$, or $\mathcal{M}_{j(C)}^{J,\op{Mixed}}$ and convert $C$ into a vertex with label $j(C)$.  Also convert each $A_i$ into a glued edge.  The negative ends of $u$ become the free edges. By the previous paragraph, the root of $T'$ must be a $\op{(Mixed)}$ vertex. 


We now consider the contraction $T\git S$  of a CH tree $T$ of $\op{(Cob)}$ type along a good subforest $S=\sqcup_i S_i$ of $T$.  For each $S_i$ we replace $S_i$ by a vertex which is labeled $\tau(S_i)$.  There are two cases to consider: 
\begin{enumerate}
\item If $S_i$ is of $\op{(Symp)}$ (resp.\ $\op{(Cob)}$) type, then $\tau(S_i)$ is a label for a $\op{(Symp)}$ (resp.\ $\op{(Cob)}$) vertex as in Definition~\ref{defn: contraction}(1).
\item If $S_i$ is incomplete, then $\tau(S_i)$ is a label for a $\op{(Mixed)}$ vertex corresponding to $\mathcal{M}_{T'}^\dagger$ whose topological type is that of a curve obtained from pregluing the curves $[\F_v,u_v]$, $v\in V(S_i)$, according to the prescription given by $G(S_i)$.

\end{enumerate}
The rest of Definition~\ref{defn: contraction} carries over.

\subsubsection{Kuranishi charts} \label{subsubsection: Kuranishi charts in cob case}

Definition~\ref{defn: semi-global Kuranishi structures Symp} carries over to the definition of a semi-global Kuranishi structure of $\op{(Cob)}$ type after modifying \ref{SC1} and \ref{SC2} in Section~\ref{subsubsection: semi-global Kuranishi str Cob case}.  In this subsection we describe the Kuranishi charts and morphisms of the semi-global Kuranishi structure $\mathscr{K}^{(L_+,L_-)} (\widehat \alpha,J)$ of $\op{(Cob)}$ type, leaving the analogs of \ref{SC1} and \ref{SC2} to Sections~\ref{subsubsection: gluing maps, cob case} and \ref{subsubsection: semi-global Kuranishi str Cob case}.
\nom[Kurcob]{$\mathscr K^{(L_+,L_-)}(\widehat \alpha,J)$}{Semi-global Kuranishi structure in the $\op{(Cob)}$ case; $L_+$ and $L_-$ are the action thresholds for the positive and negative ends $\alpha_+,\alpha_-$, $\widehat \alpha$ is the completed Liouville form, and $J$ is the almost complex structure}

When $\mathcal{M}_i$ is of $\op{(Symp)}$ type, the interior semi-global Kuranishi chart $\mathcal{C}_i= (\pi_i: \E_i\to \V_i, \overline\bdry_i,\psi_i,{\coblu G_i})$ for $\K_i\subset \mathcal{M}_i/\R$ satisfies Definition~\ref{defn: choice of interior chart} from Section~\ref{subsection: trimming}; in particular $\K_i$ (resp.\ $\V_i$) consists of $[\F,u]$ in $\mathcal{M}_i/\R$ (resp.\ in a neighborhood of $\mathcal{M}_i/\R$) for which there are no annular components in $\op{Thin}_\varepsilon(\ddot F, g_{(\ddot F,j)})$ or all annular components have $\nl\leq \mathcal{L}-\varepsilon'''$ (resp.\ $\nl < \mathcal{L}$).  

When $\mathcal{M}_i$ is of $\op{(Cob)}$ type, then we simply replace $\mathcal{M}_i/\R$ by $\mathcal{M}_i$ in Definition~\ref{defn: choice of interior chart} to obtain a Kuranishi chart $\mathcal{C}_i= (\pi_i: \E_i\to \V_i, \overline\bdry_i,\psi_i,{\coblu G_i})$ for $\K_i\subset \mathcal{M}_i$.

If $T$ is a CH tree of $\op{(Symp)}$ type, then the corresponding charts were given in Definition~\ref{defn: bundles ET over VT}.  Now let $T$ be a CH tree of $\op{(Cob)}$ type. Then the following definition is the analog of Definition~\ref{defn: bundles ET over VT}:

\begin{defn}[$\V_{T}$, $\V_{T}^{\en/2}$, $\V_{T}^\en$, $\K_{T}$, $\pi_{T}: \E_{T} \to \V_{T}$, $\pi_{T}^\en: \E_{T}^\en \to \V_{T}^\en$]
First suppose that the root vertex $v$ of $T$ is a $\op{(Mixed)}$ vertex and $Cl(\{v\})^\star=T$. Then:
\be
\item $\K_{T}$ is the compact subset of $\mathcal{M}_{T\git T}$ consisting of $[\F,u]$ such that $\nl_e(u)\geq \mathcal{L}-\varepsilon''+\tfrac{\varepsilon'''}{2}$ for all $e\in G(T)$ and $\nl\leq \mathcal{L}-\tfrac{\varepsilon'''}{2}$ for all other $\varepsilon$-thin annuli of $g_{(\ddot F,j)}$.  

\item $\V_{T}$ (resp.\ $\V_{T}^{\en/2}$, $\V_{T}^\en$) is the set of equivalence classes of maps $[\F,u]$ in a neighborhood of $\mathcal{M}_{T\git T}$ such that $\nl_e(u)> \mathcal{L}-\varepsilon''$ (resp.\ $>\mathcal{L}-\varepsilon''-\tfrac{\varepsilon'''}{2}$, $>\mathcal{L}-\varepsilon''-\varepsilon'''$) for all $e\in G(T)$, $\nl(u)<\mathcal{L}$ (resp.\ $<\mathcal{L}+\tfrac{\varepsilon'''}{2}$, $< \mathcal{L}+\varepsilon'''$) for all other $\varepsilon$-thin annuli of $g_{(\ddot F,j)}$, and $\overline\bdry_J u\in E'_{(\F,u)}$. 

\item For $[\F,u]\in \V_{T}^\en$, as in Equation~\eqref{equation: defn of E prime} we are setting
$$E'_{(\F,u)}:=\oplus_C E^{\ell,\varepsilon}_{(\mathcal{F},u),C},$$
where $C$ ranges over the boundary components of the $\varepsilon$-thin annuli corresponding to the edges of $G(T)$ and the $\varepsilon$-thin cusps of $g_{(\ddot F,j)}$, and $E^{\ell,\varepsilon}_{(\mathcal{F},u),C}$ is as in Section~\ref{subsection: the bundle E'}.  The bundles $\pi_{T}:\E_{T}\to \V_{T}$ and $\pi_{T}^\en: \E_{T}^\en \to \V_{T}^\en$ have fibers $E'_{(\F,u)}$.
\ee

In general view $V^{\op{Mixed}}(T)$ as a subforest $S$ of $T$ and write $Cl(S)=\sqcup_j \widetilde S_j$. Recalling $\widetilde S_j^\star$ from Section~\ref{subsubsection: some more gluing maps}, let 
$$\mathcal{G}_{\delta,T}:=\mathcal{G}_\delta(\{\V_{\widetilde S_j^\star}^{\text{en}}\}_{j},\{\V_i^{\text{en}}\}_{i\in V(T)-V(Cl(S))})$$ 
be the set of equivalence classes of maps that are $\delta$-close to breaking into representatives of $\V_{\widetilde S_j^\star}^{\text{en}}$ and $\V_i^{\text{en}}$, where the gluing is prescribed by $G(T)$, let $\E'$ be the bundle over $\mathcal{G}_{\delta,T}$ whose fiber $E'_{(\F,u)}$ over $(\F,u)$ is given by (3) above, and let 
$$\mathcal{G}^{\E'}_{\delta,T}:=\mathcal{G}_\delta^{\E'}(\{\V_{\widetilde S_j^\star}^{\text{en}}\}_{j},\{\V_i^{\text{en}}\}_{i\in V(T)-V(Cl(S))})=\{[\F,u]\in \mathcal{G}_{\delta,T}~|~ \overline\bdry_J u \in \E'\}.$$
\be
\item[(4)]  $\V_{T}$ is the set of $[\F,u]\in \mathcal{G}^{\E'}_{\delta,T}$ such that $\nl_e(u)> \mathcal{L}-\varepsilon''$ for all $e\in G(T)$ and $\nl(u)<\mathcal{L}$ for all other $\varepsilon$-thin annuli of $g_{(\ddot F,j)}$. $V_T^{\en/2}$ and $V_T^\en$ are defined analogously.
\item[(5)] $\pi_{T}: \E_{T} \to \V_{T}$ and $\pi_{T}^\en: \E_{T}^\en \to \V_{T}^\en$ are restrictions of $\E'\to \mathcal{G}^{\E'}_{\delta,T}$ to $\V_{T}$ and $\V_{T}^\en$.
\ee
{\coblu We take $\V_T$, $\V_T^{\en/2}$, $\V_T^\en$, $\K_T$, $\pi_T:\E_T\to \V_T$, and $\pi_T^\en:\E_T^\en\to \V_T^\en$ to all be $G_{T\git T}(T)$-invariant.}
\end{defn}

The metric $d_T$ is as in Definition~\ref{defn: bundles ET over VT} and the morphisms are defined in the same way as in the $\op{(Symp)}$ case.

\subsubsection{Examples}

We give some examples which illustrate the corners (see Figure~\ref{fig: trees}).  Let $T$ be a CH tree of $\op{(Cob)}$ type. 
Given a collection of glued edges $e_1,\dots,e_k$ of $T$, let
\begin{equation}
\nl_{e_1\dots e_k}:=\min(\nl_{e_1},\dots,\nl_{e_k}).
\end{equation}

\begin{figure}[ht]
\begin{center}
\psfragscanon
\psfrag{a}{\tiny $\op{(A)}$}
\psfrag{b}{\tiny $\op{(B)}$}
\psfrag{c}{\tiny $\op{(C)}$}
\psfrag{A}{\tiny $e_1$}
\psfrag{B}{\tiny $e_2$}
\psfrag{C}{\tiny $e_3$}
\psfrag{D}{\tiny $e_4$}
\includegraphics[width=9cm]{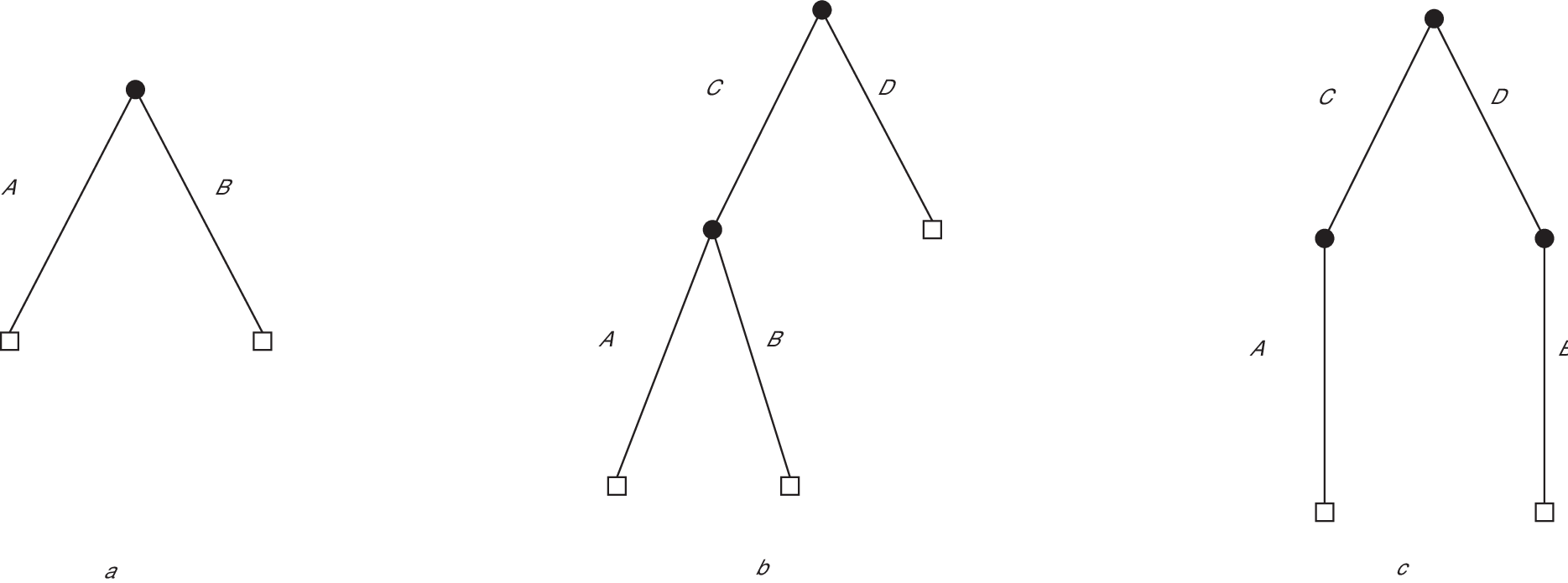}
\end{center}
\caption{The $\op{(Symp)}$ vertices are given by solid dots and the $\op{(Cob)}$ vertices by open squares.} \label{fig: trees}
\end{figure}

\s\n (A) Let $T$ be CH tree given in Figure~\ref{fig: trees}(A), i.e.,
$$V^{\op{Symp}}(T)=\{i_3\},~~ V^{\op{Cob}}(T)=\{i_1,i_2\},$$
$$G(T)=\{e_1=(i_3,i_1), e_2=(i_3,i_2)\}.$$
Let $S'=(\{i_1,i_2\}, \{e_1\})$ and $S''=(\{i_2,i_3\}, \{e_2\})$ be incomplete subtrees. See Figure~\ref{fig: corner3bis} for the image of $\mathcal{M}_{T\git T}$ in the $(\nl_{e_1},\nl_{e_2})$-coordinate plane.
\begin{figure}[ht]
\begin{center}
\psfragscanon
\psfrag{A}{\tiny $\nl_{e_1}$}
\psfrag{B}{\tiny $\nl_{e_2}$}
\psfrag{d}{\tiny $\nl_{e_1}=\mathcal L-\varepsilon''$}
\psfrag{g}{\tiny $\nl_{e_2}=\mathcal L-\varepsilon''$}
\psfrag{E}{\tiny $\varnothing$}
\psfrag{F}{\tiny $e_2$}
\psfrag{G}{\tiny $e_1$}
\psfrag{H}{\tiny $e_1,e_2$}
\psfrag{i}{\tiny $\infty$}
\includegraphics[width=4.5cm]{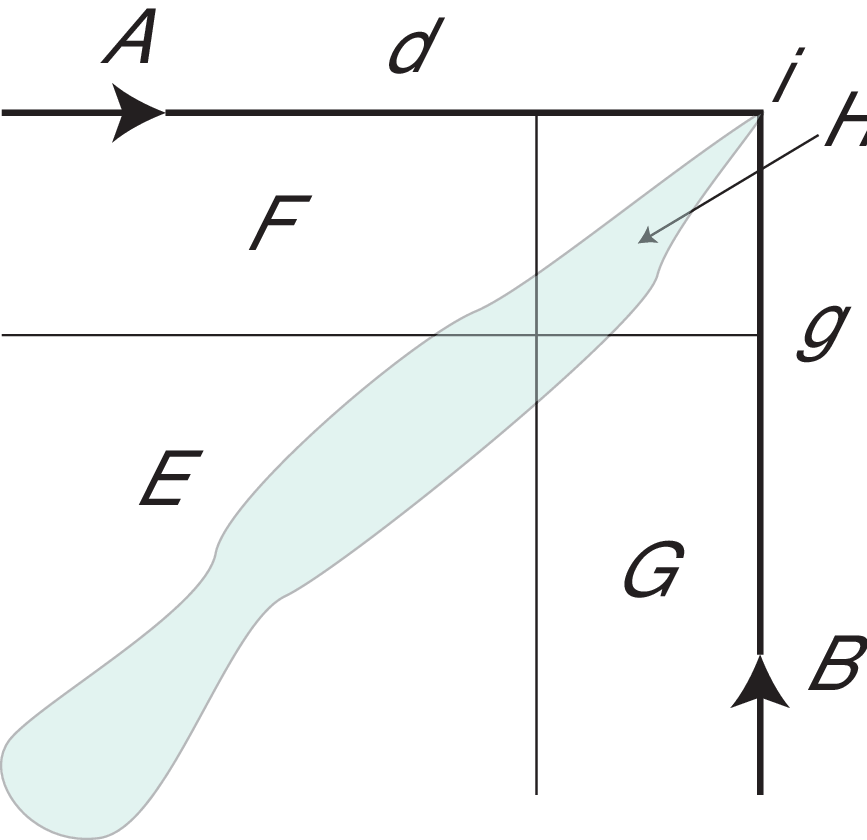}
\end{center}
\caption{The shaded region is the image of $\mathcal{M}_{T\git T}$, which is at a bounded distance from $\nl_{e_1}=\nl_{e_2}$. The labels for the regions indicate that we have stabilized using the asymptotic eigenfunctions corresponding to the given edges. The regions labeled $e_1$ and $e_2$ correspond to $T\git S''$ and $T\git S'$, respectively.} \label{fig: corner3bis}
\end{figure}
Note that $\nl_{e_1}=\nl_{e_2}$ ``in the large''. 

\s\n 
(B) Let $T$ be the CH tree given in Figure~\ref{fig: trees}(B), i.e.,
$$V^{\op{Symp}}(T)=\{i_4,i_5\},~~ V^{\op{Cob}}(T)=\{i_1,i_2,i_3\},$$
$$G(T)=\{e_1=(i_4,i_1), e_2=(i_4,i_2), e_3=(i_5,i_4), e_4=(i_5,i_3)\}.$$
We first map $\mathcal{M}_{T\git T}$ to the $(\nl_{e_1},\dots,\nl_{e_4})$-coordinate plane.  Note that, ``in the large'' we have relations $\nl_{e_1}=\nl_{e_2}$ and $\nl_{e_1}+\nl_{e_3}=\nl_{e_4}$. Hence $\nl_{e_3}$ and any one of $\nl_{e_1}, \nl_{e_2},\nl_{e_4}$ determine the other neck length functions. Note that $\nl_{e_4}> \nl_{e_1e_2}$ when we are sufficiently close to the corner and hence $\nl_{e_1e_2e_4}=\nl_{e_1e_2}$ ``in the large".

\s\n
(C) Let $T$ be the CH tree given in Figure~\ref{fig: trees}(C), i.e.,
$$V^{\op{Symp}}(T)=\{i_3,i_4,i_5\},~~ V^{\op{Cob}}(T)=\{i_1,i_2\},$$
$$G(T)=\{e_1=(i_3,i_1), e_2=(i_4,i_2), e_3=(i_5,i_3), e_4=(i_5,i_4)\}.$$
When we map $\mathcal{M}_{T\git T}$ to the $(\nl_{e_1},\dots,\nl_{e_4})$-coordinate plane, then the image of $\mathcal{M}_{T\git T}$ is a bounded distance away from $\nl_{e_1}+\nl_{e_3}=\nl_{e_2}+\nl_{e_4}$.

\subsubsection{Gluing maps} \label{subsubsection: gluing maps, cob case}

We describe the gluing maps in the $\op{(Cob)}$ case.  

Let $T$ be a CH tree of $\op{(Cob)}$ type with $|V(T)|\geq 2$. Suppose $S$ is a good subforest which is the nonempty union of glued edges such that 
$$\arr T= \arr T(S)=(T_1 \sqcup \dots \sqcup T_{m(S)}, \iota^{E(T)}, \iota^{V(T)}, \theta), \quad m(S)=|G(S)|+1,$$ 
is a degeneration such that, for each $T_i$, $Cl(T_i^{\op{trim}})=T_i^{\op{trim}}$, where $T_i^{\op{trim}}$ is the subtree of $T$ consisting of the vertices and glued edges of $T_i$, and $\iota^{E(T)}, \iota^{V(T)},$ and $\theta$ are the same maps as defined in Definition~\ref{def: degeneration}. 
Let $$E_{T}:=\{ e\in G(S)~|~ \mbox{$t(e)$ is the root of $T_i$ of $(\op{Cob})$ type}\}$$ 
and
\begin{gather*}
    \widetilde{\W}_{\arr T}^{\en/2}:= \V_{T_1}^{\en/2} \times \dots \times \V_{T_{m(S)}}^{\en/2}\times (R-\epsilon_0,\infty)^{m(S)-1},
\end{gather*}
where $R>0$ is large and $\epsilon_0>0$ is small.
We will set $R=\mathcal{L}$ and $\epsilon_0= \varepsilon''+\tfrac{\varepsilon'''}{2}$ as before.  Also let $$\E^{\en/2}_{\arr T}= \op{pr}^*_{T_1} \E^{\en/2}_{T_1}\oplus \dots \oplus\op{pr}^*_{T_{m(S)}} \E^{\en/2}_{T_{m(S)}},$$ where $\op{pr}_{T_k}:\widetilde{\W}^{\en/2}_{\arr T}\to \V^{\en/2}_{T_k}$ is the $k$th projection map. 

We then define $\W_{\arr T}^{\en/2}\subset \widetilde{\W}_{\arr T}^{\en/2}$ as follows: 
Let ${\bf v}\in \V_{T_1}^{\en/2} \times \dots \times \V_{T_{m(S)}}^{\en/2}$ and 
\begin{gather*}
({\bf I}',{\bf I}):=(I_1',\dots,I'_{|E_{T}|},I_1,\dots,I_{|G(S)-E_T|})\in (R-\epsilon_0,\infty)^{m(S)-1}.
\end{gather*}
Here the $I_i'$ are the gluing parameters corresponding to the edges in $E_{T}$ and the $I_i$ are the gluing parameters corresponding to the edges in $G(S)-E_T$.  Then $\W_{\arr T}^{\en/2}$ is the set of ``preglueable" $({\bf v},{\bf I}', {\bf I})\in \widetilde{\W}_{\arr T}^{\en/2}$, i.e., such that the pregluing $u_*({\bf v},{\bf I}',{\bf I})$ as in Definition~\ref{defn: pregluing} exists.  Observe that if $({\bf v},{\bf I}', {\bf I})\in \W_{\arr T}^{\en/2}$, then for any $r\in \{1,\dots, |E_{T}|\}$ the gluing parameters $I_r',{\bf I}$ determine the remaining gluing parameters $I'_j$, keeping in mind that the positions of curves of $\op{(Cob)}$ type cannot be translated up/down in the $\R$-direction.  Summarizing, we have the following:

\begin{lemma} \label{lemma: properties of W arr T}
The manifold $\W_{\arr T}^{\en/2}$ admits an embedding $\W^{\en/2}_{\arr T}\to \widetilde{\W}^{\en/2}_{\arr T}$ and is covered by open subsets of the form $\mathcal{U}\times (R-\epsilon_0,\infty)^{1+|G(S)-E_T|}$, where $\mathcal{U}\subset \times_i \V_{T_i}^{\en/2}$ and the coordinates on $(R-\epsilon_0,\infty)^{1+|G(S)-E_T|}$ are $(R_r',{\bf R})$ for some $r$.
\end{lemma}

We then have the gluing map
\begin{align}
\label{equation: gluing map in cob case} 
\mathfrak G_{\arr T}^{\en}: \W_{\arr T}^{\en/2} \to \mathcal{G}_{\delta,T}^{\E'},
\end{align}
obtained by the constructing the pregluing $u_*({\bf u},{\bf R}', {\bf R})$ and inverting all the errors simultaneously at each step of the contraction mapping as in Section~\ref{subsubsection: description of G123}.

\subsubsection{Semi-global Kuranishi structure in the $\op{(Cob)}$ case}  \label{subsubsection: semi-global Kuranishi str Cob case}

The definition of the semi-global Kuranishi structure in the $\op{(Symp)}$ case (Definition~\ref{defn: semi-global Kuranishi structures Symp}) carries over to the $\op{(Cob)}$ case after replacing \ref{SC1} by \ref{SC1prime} and {\coblu the gluing maps $\mathfrak G_{\arr T}^{\en}$ and $\widetilde{\mathfrak G}_{\arr T}^{\en}$ that are used in \ref{SC2} by those in \ref{SC1prime}.}

Let $T$, $S$, and $\E_{\arr T}^{\en/2}\to \widetilde{\W}^{\en/2}_{\arr T}$ be as defined in Section~\ref{subsubsection: gluing maps, cob case}. 

\be
\labitem{(SC1')}{SC1prime}
There exist global constants $R>0$ and $\epsilon_0>0$ small, ``intermediate'' {\coblu $G_{T\git T}(T)$-invariant} open subsets $\V_T^{\en/2}$ for all $T$ satisfying 
$$\V_T\subset \V_T^{\en/2}\subset \V_T^\en,  \quad \overline\V_T\subset \V_T^{\en/2}, \quad \overline\V_T^{\en/2}\subset \V_T^{\en},$$ 
and embeddings $\W^{\en/2}_{\arr T}\to \widetilde{\W}^{\en/2}_{\arr T}$ satisfying Lemma~\ref{lemma: properties of W arr T}, such that, for any CH tree $T$ with $|V(T)|\geq 2$ and a good subforest $S$ of $T$ such that $G(S)\not=\varnothing$, there exists a $C^1$-bundle map $(\widetilde{\mathfrak G}_{\arr{T}}^{\en}, \mathfrak G_{\arr{T}}^{\en})$ {\coblu between bundles of the same rank}
$$
\begin{tikzcd}
\E_{\arr{T}}^{\en/2} \arrow[r,"\widetilde{\mathfrak{G}}_{\arr{T}}^{\en}"] \arrow[d] & \E_T^{\en} \arrow[d]\\
\W_{\arr T}^{\en/2}   \arrow[r, "\mathfrak{G}_{\arr{T}}^{\en}"] &  \V_T^{\en},
\end{tikzcd}
$$
called the {\em gluing map} and which satisfies the following:
\begin{enumerate}
\item $\mathfrak{G}_{\arr{T}}^{\en}$ is a $C^1$-map which is a homeomorphism onto its image.
\item $\widetilde{\mathfrak G}_{\arr{T}}^{\en} \circ (\overline\partial_{T_1}^{\en}, \dots,\overline \partial_{T_m}^{\en})|_{\W_{\arr T}^{\en/2} }$ and $\overline \partial_T^{\en} \circ \mathfrak G_{\arr{T}}^{\en}$
are $C^1$-close and their difference goes to zero as $\op{min}_{j=1}^{m-1} I_j \to \infty$, where $I_j$ is the coordinate for the $j$th $(R-\epsilon_0, \infty)$ factor corresponding to an edge $e_j\in E(S)$. 
\item Let $S_j$ be a good subforest of $T_j$ with $G(S_j)\not=\varnothing$ which we view as a subforest of $T$. 
If $v_i \in \V_{T_i}^{\en/2}$, $(\mathbf{v}_{\arr{T_j}(S_j) },\mathbf{I}_{\arr{T_j} (S_j)})\in \W_{\arr{T_j}(S_j)}^{\en/2}$, and $\mathbf I_{\arr{T}(S)}\in (R-\epsilon_0, \infty)^{m(S)-1}$, such that 
$$(v_1,\dots,v_{j-1},\mathbf v_{\arr{T_j}(S_j)},v_{j+1}, \dots, v_m,\mathbf I_{\arr{T_j}(S_j)}, \mathbf I_{\arr{T}(S)})\in \W_{\arr T(S\cup S_j)}^{\en/2},$$
then there exists 
$$\mathbf I_{\arr{T}(S)}'=\mathbf I_{\arr{T}(S)}'(v_1,\dots,v_{j-1},\mathbf v_{\arr{T_j}(S_j)},v_{j+1}, \dots, v_m,\mathbf I_{\arr{T_j}(S_j)},\mathbf I_{\arr{T}(S)}) \in (R-\epsilon_0, \infty)^{m(S)-1}$$ 
such that the simultaneous and iterated gluing maps 
\begin{gather*}
    \mathfrak G_{\arr{T} (S\cup S_j)}^{\en}(v_1,\dots,v_{j-1},\mathbf v_{\arr{T_j}(S_j)},v_{j+1}, \dots, v_m,\mathbf I_{\arr{T_j}(S_j)}, \mathbf I_{\arr{T}(S)}),\\
    \mathfrak G_{\arr{T}(S)}^{\en} (v_1,\dots,v_{j-1}, \mathfrak G_{\arr{T_j}(S_j)}^{\en}(\mathbf{v}_{\arr{T_j}(S_j) }, \mathbf{I}_{\arr{T_j} (S_j)}), v_{j+1},\dots , v_m, \mathbf{I}'_{\arr{T}(S)} )
\end{gather*}
are $C^1$-close with error $\to 0$ and $\mathbf I_{\arr{T}(S)}'$ and $\mathbf I_{\arr{T}(S)}$ are $C^1$-close with error $\to 0$ as the minimum of all the components of $\mathbf I_{\arr{T_j}(S_j)}, \mathbf I_{\arr{T}(S)} \to \infty$. 
\NoIndent{Let $\arr T= \arr T(T_{\op{max}})$, where $T_{\op{max}}$ is the maximal good subforest of $T$ such that if $\arr T(T_{\op{max}})=T_1\sqcup \dots \sqcup T_m = $, then $Cl(T_i)=T_i$ (in $T$) for each $T_i$. Then we denote
$$\mathfrak G^{R_0, T} :=  \mathfrak G_{\arr{T}}^\en( \{({\bf v}, {\mathbf I}_{\arr T})\in \W_{\arr{T}}^{\en/2} ~|~ {\mathbf I}_{\arr T}\in (R-\epsilon_0, \infty)^{|G(T_{\op{max}})|}-(R-\epsilon_0, R_0)^{|G(T_{\op{max}})|}).$$}
\item $\mathfrak G^{R-\epsilon_0, T}\supset \overline\V_T$ and $\overline \V_T -  \mathfrak G^{R_0, T}$ is compact for any $R_0 \geq R-\epsilon_0$.
\ee
\end{enumerate}

The following is analogous to Theorem~\ref{thm: gluing map smooth} and is proved in the same way:

\begin{thm} \label{thm: gluing Cob version}
\ref{SC1prime}(a) holds for $\mathfrak G_{\arr T}^{\en}$ and $\op{Im} \mathfrak{G}_{\arr T}^{\en}\supset \V_{T}$.
\end{thm}

The analog of Theorem~\ref{thm: iterated gluing} becomes:

\begin{thm} \label{thm: iterated gluing Cob version}
\ref{SC1prime}(c) holds for the collection of gluing maps $\{\mathfrak G_{\arr T}^\en\}_{\arr T}$.
\end{thm}

In \ref{SC2}, $\mathfrak{G}_{\arr{T}}^{\en}$ is replaced by a map $\W_{\arr T}^{\en/2} \to \V_T^{\en}$.

\s\n

\subsubsection{Definition of the multisection ${\mathfrak s}^\en_T$} \label{subsubsection: defn of multisection Cob type}

The goal of this subsection is to give the definition of a multisection $\mathfrak{S}=\{\mathfrak{s}^\en_T: \V^\en_T\to \E^\en_T\}_T$ of a semi-global Kuranishi structure $\mathscr{K}$ of $\op{(Cob)}$ type and to give a construction of a specific $\mathfrak{S}$, following Section~\ref{subsubsection: defn of sT}.

\begin{defn}[Multisection of $\mathscr{K}$ of $\op{(Cob)}$ type] \label{defn: multisection of Kur Cob type}
A {\em multisection of $\mathscr{K}$ of $\op{(Cob)}$ type} is a collection $\mathfrak{S}=\{\mathfrak{s}^\en_T: \V^\en_T\to \E^\en_T\}_T$ of slight enlargements of multisections satisfying (1) and (3) of Definition~\ref{def: multisection of Kur} and 
\begin{enumerate}
    \item[(2')] there exists $R_0\gg 0$ such that for each good subforest $S$ of $T$ with $S\not=\varnothing$, if we write $\arr{T}=\arr{T}(S)=T_1\sqcup \dots \sqcup T_{m(S)}$, then on 
    $$\mathfrak G_{\arr{T}}^\en(\W_{\arr{T}}^{\en/2}\cap \{({\bf v},{\bf I})~|~ {\bf I}\in (R_0,\infty)^{m(S)-1}\}),$$
$$\widetilde{\mathfrak G}_{\arr{T}}^\en \circ (\mathfrak s_{T_1}^\en, \dots, \mathfrak s_{T_{m(S)}}^\en) \circ (\mathfrak G^\en_{\arr{T}})^{-1}\quad \mbox{and} \quad \mathfrak s_T^\en$$ 
are $C^1$-close and their difference goes to zero as $\min_{j=1}^{m(S)-1} I_j \to \infty$, where $(I_1,\dots, I_{m(S)-1})$ are the $(R_0,\infty)^{m(S)-1}$-coordinates.
\end{enumerate}
\end{defn}

We assume that \hyperlink{H}{Condition (H)} from Section~\ref{subsubsection: boundary strata} holds; the modifications used in the $\chi\geq 0$ case are analogous.  Let $0<\varepsilon'''\ll \varepsilon''$ as in Section~\ref{subsection: trimming}. Let $T$ be a CH tree of $\op{(Cob)}$ type.

We define the gluing maps $\mathfrak{G}^\en_{T,S}$ as in Section~\ref{subsubsection: some more gluing maps}, where $S$ is a good subforest of $T$ and $Cl(S)=\sqcup_j \widetilde S_j$. Let
\begin{gather*}
    \V^\en_{T,S}=(\times_j \V^\en_{\widetilde S_j^\star})\times (\times_{i\in V(T)-V(Cl(S))} \V^\en_i),\\ \mathcal{G}_{\delta,T,S}^{\E'}=\mathcal{G}_\delta^{\E'}(\{\V_{\widetilde S_j^\star}^\en\}_j,\{\V_i^\en\}_{i\in V(T)-V(Cl(S))}),
\end{gather*} 
and $\W^{\en/2}_{T,S}$ be the subset of preglueable $({\bf v}, {\bf I})\in \V^{\en/2}_{T,S}\times(\mathcal{L}-\varepsilon''-\tfrac{\varepsilon'''}{2},\infty)^{|G(T)|-|G(Cl(S))|}$,
There is a gluing map 
$$\mathfrak{G}_{T,S}^\en: \W^{\en/2}_{T,S}\to \mathcal{G}_{\delta,T,S}^{\E'}, \quad \widetilde{\mathfrak{G}}^\en_{T,S}: \E_{T,S}^{\en/2}\to \E_T^\en$$
defined in the same way as in Definition~\ref{equation: gluing map in cob case} and Section~\ref{subsubsection: gluing maps}.

\begin{defn}[Multisection ${\mathfrak s}^\en_T$, $\op{(Cob)}$ case] \label{defn: s for cobordisms} 
Suppose $T$ satisfies:
\be
\item[($\dagger\dagger$)] Either $T$ is a $1$-vertex tree or the root vertex $v$ of $T$ is a $\op{(Mixed)}$ vertex and $Cl(\{v\})^\star=T$.
\ee
The multisection $\mathfrak{s}_{T}^\en$ on $\V_T^\en$ is arbitrary, subject to consistency with $\mathfrak{s}_{T'}^\en$ on adjacent regions satisfying ($\dagger\dagger$) and invariance under $G_{T\git T}$.  

Assume that ${\mathfrak s}_{T'}^\en$ has been defined for all $T'$ with $c(T')< c(T)$.  Then the multisection ${\mathfrak s}_T^\en$ on $\E_T^\en\to \V_T^\en$ is the melding of the following multisections with respect to
$$\nl_T= (\nl_e)_{\{e\in G(T)\}}: \V_T^\en\to [\mathcal{L}-\varepsilon''-\varepsilon''',\infty)^{|G(T)|}$$
and the collections $\{R_{T,S,j}\}$ and $\{R'_{T,S,j}\}$, $j=1,2$, from Section~\ref{subsubsection: defn of sT}:
\begin{enumerate}
    \labitem{(R1)}{R1} On $\nl_T^{-1}(R'_{T,S,1})$, take the stabilization $\op{stab}_{T\git S,T}({\mathfrak s}_{T\git S}^\en)$ of ${\mathfrak s}_{T\git S}^\en$.
    \labitem{(R2')}{R2'} On $\nl_T^{-1}(R'_{T,S,2})$, where $T\git S$ does not satisfy ($\dagger \dagger$), take the restriction of $(\widetilde {\mathfrak G}_{T,S}^\en)_* \mathfrak s^\en_{T,S}:=  \widetilde{\mathfrak G}_{T,S}^\en \circ\mathfrak s^\en_{T,S}\circ (\mathfrak G_{T,S}^\en)^{-1}$, where 
$$\mathfrak s^\en_{T,S}=(\mathfrak s_{\widetilde S_j^\star}^\en)_j\times (\mathfrak s_{i}^\en)_{i\in V(T)-V(S)}.$$
    \labitem{(R3')}{R3'}  On $\nl_T^{-1}(R'_{T,S,2})$, where $T\git S$ satisfies ($\dagger \dagger$), take an arbitrary multisection satisfying invariance under marker rotations and puncture reorderings.
\end{enumerate}
{\coblu When melding the multisections, we need to choose a consistent collection of branches that agree on the overlaps as we approach the ``boundary".}
\end{defn}

Observe that, in \ref{R2'}, we have $\mathcal{L}+\varepsilon''+\varepsilon'''>\nl_e > \mathcal{L}$ for all $e\in G(S)$, which in turn implies that $\nl_e$ is bounded above for all $e\in G(cl(S))$. 

The analog of Lemma~\ref{lemma: overlaps} also holds in the $\op{(Cob)}$ case.

Finally, the definition of $\mathcal{Z}(\mathscr{K}(\mathcal{M}_i),\mathfrak S)$ is as given in Equation~\eqref{eqn: defn of Z sub i} and Proposition~\ref{prop: tomatoes} holds for $\mathcal{M}_i$ of $\op{(Cob)}$ type after replacing $\mathcal{M}_i/\R$ and $\op{vdim} \mathcal{M}_i/\R$ by $\mathcal{M}_i$ and $\op{vdim}\mathcal{M}_i$ and taking one of  $\mathcal{Z}_{j(T)}$ or $\mathcal{Z}_{k(T)}$ to be of $\op{(Cob)}$ type and the other to be of $\op{(Symp)}$ type.


\cb





\subsection{Chain homotopy} \label{subsection: chain homotopy case}

Let $X$ be a compact manifold with corners and let $\{(\widehat W^\tau=\widehat W,\widehat \alpha^\tau)\}_{\tau\in X}$ be a family of completed exact symplectic cobordisms from $(M_+,\alpha_+)$ to $(M_-,\alpha_-)$. For each $\tau \in X$, let $J^\tau$ be an almost complex structure on $\widehat W^\tau$ that restricts to $J_\pm$ at the positive/negative ends of $\widehat W^\tau$. We assume that $(\widehat \alpha^\tau, J^\tau)$ is an $(L_+,L_-)$-simple pair, $J^\tau$ is $(L_+,L_-)$-end-generic. In the special case of interest $X=[0,1]$, we additionally assume that $(\widehat \alpha^\tau, J^\tau)$ is independent of $\tau$ on each of $\tau\in[0,\epsilon]$ and $\tau\in[1-\epsilon,1]$, where $\epsilon>0$ is small.

We use the ordering of moduli spaces from Section~\ref{subsubsection: mixed curves}.   Let $\mathcal{M}_i^\tau$ be $\mathcal{M}_i$ with respect to $J^\tau$, let
$$\mathcal{M}_i^{X}=\{[\F,u] \in  \mathcal{M}_i^{\tau}~|~\tau\in X\},$$
and \cb let $\Pi_i:\mathcal{M}_i^{X}\to X$ \cb be the projection that sends $\mathcal{M}_i^\tau$ to $\tau$.  We use $\varepsilon>0$ sufficiently small and $\ell>0$ sufficiently large so that Theorem~\ref{thm: construction of L-transverse subbundle} holds simultaneously for all $J^\tau$, $\tau\in X$.

\subsubsection{Semi-global Kuranishi structure in the $\op{(Par)}$ case} \label{subsubsection: semi-global Kuranishi Par case}

\nom[Kurnom]{$\mathscr{K}^{(L_+,L_-)}(\{(\widehat \alpha^\tau,J^\tau)\}_{\tau\in X})$}{Semi-global Kuranishi structure in $\op{(Par)}$ case, parametrized by $X$}
The semi-global Kuranishi structure $\mathscr{K}^X=\mathscr{K}^{(L_+,L_-)}(\{(\widehat \alpha^\tau,J^\tau)\}_{\tau\in X})$ {\em of para\-me\-trized $\op{(Par)}$ type} is the parame\-trized version of a semi-global Kuranishi structure of $\op{(Cob)}$ type.
If $T$ is a CH tree of $\op{(Cob)}$ type, then let
$$\pi_T^{X}: \E_T^{X}\to \V_T^{X}$$
be the parametric version of the obstruction bundle $\pi_T^{\tau}:\E_T^{\tau}\to \V_T^{\tau}$ corresponding to $\tau$; the rest of the definition including the slight enlargements as given in Sections~\ref{subsubsection: Kuranishi charts in cob case}, \ref{subsubsection: gluing maps, cob case}, and \ref{subsubsection: semi-global Kuranishi str Cob case} can be done parametrically. In particular, the $\overline\bdry$-section is the parametric version of $\overline\bdry_{J^\tau}$. When $X=[0,1]$, we assume that $\pi_T^{[0,1]}$ satisfies the following conditions:
\begin{enumerate}
\item[($\tau_1$)] it is independent of $\tau$ on $[0,\epsilon]$ and on $[1-\epsilon,1]$; and
\item[($\tau_2$)] it agrees with $\pi_T^{\tau=0}$ and $\pi_T^{\tau=1}$, which we assume have already been constructed for the chain maps corresponding to $J^{\tau=0}$ and $J^{\tau=1}$.

\end{enumerate}

For the boundary structure we have the parametrized version of \ref{SC1prime},
with the modification of replacing the product of charts by fiber product.
More precisely, we explain the modification using the case when $T$ is a CH tree of $\op{(Cob)}$ type with one $\op{(Symp)}$ vertex $i_3$, two $\op{(Cob)}$ vertices $i_1$ and $i_2$, and two gluing edges $e_1$ and $e_2.$
Let $\arr T = i_1 \sqcup i_2 \sqcup i_3$ be a degeneration of $T$.
The ``pregluable" set $\W_{\arr T}^{[0,1]}$ is diffeomorphic to 
$$\V^{[0,1]}_{\{i_1,i_2\}} \times \V_{i_3} \times [R,\infty),$$
where $$\V^{[0,1]}_{\{i_1,i_2\}} = \V_{i_1}^{[0,1]} {}_{\pi_{i_1}^{[0,1]}}\times_{\pi_{i_2}^{[0,1]}} \V_{i_2}^{[0,1]},$$
and $R> 0$ is some constant. 
We denote the diffeomorphism by $\Psi.$
Note that by Theorem~\ref{thm: construction of L-transverse subbundle}, the projection $\pi_{i_k}^{[0,1]}: \V_{i_k}^{[0,1]} \to [0,1]$ is submersive for $k\in \{1,2\},$ and hence $\V^{[0,1]}_{\{i_1,i_2\}}$ is smooth.

\cb

\subsubsection{Definition of the multisection $\mathfrak{S}$} \label{subsubsection: defn of multisection Par case}

The definition of the multisection $\mathfrak{S}=\{ \mathfrak s_T^{[0,1]}: \V_T^{[0,1]}\to \E_T^{[0,1]} \}$ of  $\mathscr{K}^{[0,1]}$
will follow a parametric version of Definition~\ref{defn: s for cobordisms} but will require a slight twist in order to achieve sufficient transversality. We continue using the example of Section~\ref{subsubsection: semi-global Kuranishi Par case} to explain the twist.
Consider the gluing maps $\widetilde{\mathfrak G'}_{\arr T}^{[0,1]}, \mathfrak {G'}_{\arr T}^{[0,1]}$ obtained from precomposing the parametrized version of $\widetilde{\mathfrak G}_{\arr T}^\en, \mathfrak G_{\arr T}^\en$ in \ref{SC1prime} with $\Psi^{-1}$ and dropping the notation $\en.$
We have the following commutative diagram
$$
\begin{tikzcd}
\E^{[0,1]}_{\{i_1,i_2\}} \times \E_{i_3} \arrow[r,"\widetilde{\mathfrak{G}'}^{[0,1]}_{\arr{T}}"] \arrow[d] & \E_T \arrow[d]\\
\V^{[0,1]}_{\{i_1,i_2\}} \times \V_{i_3} \times [R,\infty)   \arrow[r, "\mathfrak{G'}_{\arr{T}}^{[0,1]}"] &  \V_T.
\end{tikzcd}
$$
We denote by $\mathfrak s_{\{i_1,i_2\}}^{[0,1]}$ the multisection of $\E_{\{i_1,i_2\}}^{[0,1]} \to \V_{\{i_1,i_2\}}^{[0,1]}.$
As usual, we require 
\be
    \item $$\widetilde{\mathfrak G'}_{\arr{T}} \circ (\mathfrak s_{\{i_1,i_2\}}^{[0,1]}, \mathfrak s_{i_3}) \circ (\mathfrak G'_{\arr{T}})^{-1}\quad \mbox{and} \quad \mathfrak s_T$$ 
are $C^1$-close and their difference goes to zero as $I \in [R,\infty) \to \infty$,
    \item multisection $\mathfrak s_{\{i_1,i_2\}}^{[0,1]} = \mathfrak s_{i_1}^{[0,1]} \times \mathfrak s_{i_2}^{[0,1]} |_{\V^{[0,1]}_{\{i_1,i_2\}}}$ 
when $\tau \in [0,\epsilon]\cup [1-\epsilon, 1],$
and 
    \item $\mathfrak s_{\{i_1,i_2\}}^{[0,1]}$ is a transversal (to $\overline \partial$) section (not necessarily of product type).

\ee 

Note that one can choose a generic pair $(\mathfrak s_{i_1}^{[0,1]}, \mathfrak s_{i_2}^{[0,1]})$ so that $\mathfrak s_{\{i_1,i_2\}}^{[0,1]}$ is of product-type, i.e., $\mathfrak s_{\{i_1,i_2\}}^{[0,1]} = \mathfrak s_{i_1}^{[0,1]} \times \mathfrak s_{i_2}^{[0,1]} |_{\V^{[0,1]}_{\{i_1,i_2\}}}$ and $\mathfrak s_{\{i_1,i_2\}}^{[0,1]}$ is transversal. But then the pair $(\mathfrak s_{i_1}^{[0,1]}, \mathfrak s_{\{i_1,i_2\}}^{[0,1]})$ is not a generic pair, which results in trouble in the construction of semi-global Kuranishi structures in higher energy strata.
Our choice of relaxing the multisection to be non-product type makes the construction of multisection in the $\op{(Cob)}$ case carry over to the $\op{(Par)}$ case straightforwardly,
but it is also  the reason why the chain homotopy map $K$ defined later is only a linear map.

\cb
The definition of the weighted branched submanifold $\mathcal{Z}^{[0,1]}(\mathscr{K}(\mathcal{M}_i),\mathfrak S)$ is as given in Equation~\eqref{eqn: defn of Z sub i} and Proposition~\ref{prop: tomatoes} holds for $\mathcal{M}_i^{[0,1]}$ of $\op{(Cob)}$ type after replacing $\mathcal{M}_i/\R$ and $\op{vdim} \mathcal{M}_i/\R$ by $\mathcal{M}_i^{[0,1]}$ and $\op{vdim}\mathcal{M}_i^{[0,1]}= \op{vdim} \mathcal{M}_i +1$  and taking one of  $\mathcal{Z}_{j(T)}$ or $\mathcal{Z}_{k(T)}$ to be of $\op{(Par)}$ type and the other to be of $\op{(Symp)}$ type.

\section{Contact homology} \label{section: contact homology}

\subsection{The contact homology dga} \label{chain complex}

\subsubsection{Definitions} \label{subsection: contact homology dga definitions}

Let $(M^{2n+1},\xi)$ be a closed contact manifold. Given $L>0$, let $\alpha$ be a contact form for $\xi$ and $J$ be an almost complex structure on $\mathbb R \times M$ such that $(\alpha,J)$ is an $L$-simple pair.  For the above data we constructed a semi-global Kuranishi structure 
$$\Kur^L(\alpha,J)=\{ \Kur^L(\alpha, J; \gamma_+; \bs \gamma_-; A)~|~ \mathcal A_\alpha ( \gamma_+),\mathcal A_\alpha(\bs \gamma_-)\leq L, A\in H_2(M;\Z)  \},$$
in Section~\ref{section: construction of semi-global}, together with a multisection $\mathfrak S = \{{\mathfrak s}_T\}_T$ that is $C^1$-close to the $0$-section and such that $\overline\bdry_T$ is transverse to $\mathfrak s_T$ for all $T$; see Section~\ref{subsubsection: grading} for how to assign $A\in H_2(M;\Z)$ to $[\F,u]$. Let $\mathcal Z(\Kur^L(\alpha, J;\gamma_+; \bs \gamma_-;A), \mathfrak S)$ be the weighted branched manifold for the full subcategory $\Kur^L(\alpha, J;\gamma_+; \bs \gamma_-;A)$ of $\Kur^L(\alpha,J)$ and the multisection $\mathfrak S$.

Let
$$\mathfrak A^L=\mathfrak A^L(M,\alpha, J, \mathscr{K}^L(\alpha,J),\mathfrak S,  \mathcal{S}^{L}_{\alpha})$$
be the unital graded commutative algebra freely generated by $\mathcal P_\alpha^{L,\op{good}}$ over the group algebra $\Q[H_2(M;\Z)]$.  The generators of the group algebra $\Q[H_2(M;\Z)]$ will be written as $e^A$, where $A\in H_2(M;\Z)$.  In Section~\ref{subsubsection: grading} we explain how to assign a $\Q$-grading $|\cdot |$ to each Reeb orbit in $\mathcal P_\alpha^{L}$ and to each $e^A$, $A\in H_2(M;\Z)$, and give a definition of $\mathcal{S}^L_\alpha$.

\nom[A]{$\mathfrak A^L=\mathfrak A^L(M,\alpha, J, \mathscr{K}^L(\alpha,J),\mathfrak S,   \mathcal{S}^{L}_{\alpha})$}{Contact homology unital graded commutative differential algebra}

We also define a $\Z/2$-grading $|\cdot|_0$ on $\mathcal{P}_\alpha^L$:  Given $\gamma\in  \mathcal{P}_\alpha^L$, we set
$$|\gamma|_0\equiv \mu_\tau(\gamma)+n-3 \mbox{ mod } 2,$$
where $\mu_\tau(\gamma)$ is the Conley-Zehnder index with respect to the trivialization $\tau$.  Observe that the parity of $\mu_\tau(\gamma)$ does not depend on the choice of $\tau$.

\begin{defn}
An ordered tuple of Reeb orbits $\bs \gamma$ is {\em good} if all of its entries are good Reeb orbits. Otherwise, $\bs\gamma$ is {\em bad}.
\end{defn}

\nom[HC]{$HC(M,\alpha, J, \Kur^L(\alpha,J), \mathfrak S, \mathcal{S}^L_\alpha)$}{Contact homology algebra for the data $(M,\alpha, J, \Kur^L(\alpha,J), \mathfrak S, \mathcal{S}^L_\alpha)$}
Next we define the differential $\partial: \mathfrak A^L \to \mathfrak A^L$. This makes $\mathfrak A^L$ into a dga and we denote its homology by $HC(M,\alpha, J, \Kur^L(\alpha,J), \mathfrak S, \mathcal{S}^L_\alpha)$ or $HC^L$. Fix an ordering $\vartheta$ of $\mathcal P^{L}_\alpha$ by the action as usual.  We then set
 
\begin{equation} \label{eqn: partial}
\partial \gamma_+=\sum_{\bs\gamma_-}\sum_{A\in H_2(M;\Z)} \frac{d_{\gamma_+,\bs \gamma_-,A}}{m_{\bs \gamma_-}} e^A  \gamma_{-,k}^{i_k} \dots \gamma_{-,1}^{i_1},
\end{equation}
\cb
where the sums are taken over all good $\vartheta$-sorted $\bs \gamma_-$ and $A\in H_2(M;\Z)$ with  $|\gamma_+|-(|e^A|+|\bs \gamma_-|)=1$.  A sorted $\bs\gamma_-$ is written as
\begin{equation}\label{eqn: gamma-}
\bs \gamma_-=(\underbrace{\gamma_{-,1},\dots, \gamma_{-,1}}_{i_{1} \text{ copies }},\dots, \underbrace{\gamma_{-,k},\dots, \gamma_{-,k}}_{i_{k} \text{ copies }}),
\end{equation}
$d_{\gamma_+,\bs \gamma_-, A}$ is the weighted signed count of elements in $\mathcal Z(\Kur^L(\alpha, J;\gamma_+; \bs \gamma_-;A), \mathfrak S)$, \cb and
$$m_{\bs \gamma_-}:=\Pi_{j=1}^k ( i_j ! m(\gamma_{-,j})^{i_j}).$$
(If $|\gamma_+|-(|e^A|+|\bs \gamma_-|)\neq 1$, then we set $d_{\gamma_+,\bs \gamma_-,A}=0$.) In Section~\ref{subsubsection: grading} we explain how to define the class $A$ corresponding to an approximate $J$-holomorphic map $u$ from $\gamma_+$ to $\bs\gamma_-$ and in Section~\ref{subsubsection: orientations} we explain how to orient the moduli spaces. Then we extend $\partial$ to $\mathfrak A^L$ using the graded Leibniz rule, namely,
$$\partial (\gamma \gamma')=(\partial \gamma) \gamma'+(-1)^{|\gamma|_0}\gamma (\partial  \gamma').$$

{\em Sometimes we write an ordered tuple of Reeb orbits multiplicatively. Moreover, when there is no confusion, we do not distinguish an ordered tuple of Reeb orbits from the monomial associated to it in $\mathfrak A^L$.}  

If $\bs\gamma_- = \gamma_{-,1}^{i_1} \dots  \gamma_{-,k}^{i_k}$, then we denote the order-reversed monomial by $\bs\gamma_- ^\rev = \gamma_{-,k}^{i_k} \dots \gamma_{-,1}^{i_1}.$

\begin{rmk}
In \eqref{eqn: partial}, we use $\bs\gamma_-^\rev$ instead of $\bs\gamma_-$. This is due to the fact that we use the coherent orientation from \cite{BM}; see the proof of Proposition~\ref{dga}. \footnote{\coblu We later find out that this is not necessary if one uses the choice of coherent orientation from \cite{HT2}. \cb }
\end{rmk} 
\cb

\subsubsection{Grading and $\Q[H_2(M;\Z)]$-coefficients} \label{subsubsection: grading}

We follow the discussion in \cite{Bo}.  Let us write $H_1(M;\Z)= F\oplus T$, where $F$ is the free part and $T$ is the torsion part.  Pick representatives $C_1,\dots, C_a$ of a basis for $F$ and representatives $D_1,\dots, D_b$ of a minimal generating set for $T$, and fix a trivialization $\tau$ for $\xi$ along each $C_i$ and $D_i$. Given $\gamma\in \mathcal{P}^L_\alpha$, let us write $[\gamma]=[\sum c_i C_i +\sum d_i D_i]\in H_1(M;\Z)$, where $d_i$ are the smallest nonnegative coefficients for the $D_i$.  Then choose a surface $S_\gamma$ such that $\bdry S_\gamma=\gamma-\sum_i c_i C_i -\sum_i d_i D_i$ and use $S_\gamma$ to extend the trivializations $\tau$ along ${C_i}$ and ${D_i}$ to all of $S_\gamma$ and hence to $\gamma$.  If $[D_i]$ has order $m_i$, then choose a spanning surface $S_{m_i D_i}$ for $m_i D_i$. This gives a trivialization $\tau'$ along $m_i D_i$ and we let $w_i$ be the rotation number of $\tau$ with respect to $\tau'$ along $m_i D_i$.  We then set
$$|\gamma|= \mu_\tau(\gamma)+2\sum_{i=1}^b\frac{w_i d_i}{m_i} + n-3\in \Q ,$$
where $\mu_\tau$ is the Conley-Zehnder index with respect to $\tau$.  We also set
$$|e^A|=-2\langle c_1(\xi),A\rangle.$$

The collection
$$\mathcal{S}^L_\alpha=\{S_\gamma\}_{\gamma\in \mathcal{P}^L_\alpha}\cup \{S_{m_i D_i}\}_{i=1,\dots, b}$$
will be called a {\em complete set of trivializing surfaces for $\mathcal{P}^L_\alpha$}.
\nom[Surface]{$\mathcal{S}^L_\alpha$}{Complete set of trivializing surfaces}

Given a map $u$ from $\gamma$ to $\bs\gamma_-$, we can cap off the projection $\pi_M\circ u$ of $u$ to $M$ along $\gamma$ by $S_\gamma$ and along $\gamma_{-,j}$ by $S_{\gamma_{-,j}}$ and further cap off any extraneous $m_i D_i$ by $S_{m_i D_i}$, which gives a closed surface $A$. Note that if
$$\bdry S_\gamma=\gamma-\sum_i c_i C_i -\sum_i d_i D_i,$$
where the $d_i$ are the smallest nonnegative coefficients, then
$$\sum_j \bdry S_{\gamma_{-,j}}=\sum_j\gamma_{-,j} -\sum_i c_i C_i -\sum_i d_i D_i + \zeta,$$
where $\zeta$ is a sum of $m_i D_i$.


\subsubsection{Orientations} \label{subsubsection: orientations}

Following the orientation conventions of \cite{BM}, we explain how to assign a sign to
$$\llbracket \F, u \rrbracket\in \mathcal Z(\Kur^L(\alpha, J;\gamma_+; \bs \gamma_-;A), \mathfrak S).$$ 

\s\n
{\em Step 1.}  For any $(\F, u)\in \mathcal{B}_{\widetilde{\mathcal U}}$, let $L_{(\F,u)}$ be the full linearized $\overline\bdry_J$-operator and
$$\op{det}L_{(\F, u)}=\wedge^{\op{top}}\ker L_{(\F, u)} \otimes \wedge^{\op{top}}(\op{coker}L_{(\F, u)})^*$$
be its determinant line. 

An {\em orientation} $\mathfrak o(\F, u)$ of $(\F,u)$ is an orientation $\mathfrak o (\op{det}L_{(\F, u)})$ of $\op{det}L_{(\F, u)}$, i.e., an equivalence class of nonzero vectors of $\op{det}L_{(\F, u)}$, where the equivalence relation is $v\sim v'$ if there exists $c>0$ such that $v=cv'$.

According to \cite{BM}, an orientation $\mathfrak o(\F,u)$ can be chosen in a continuous and coherent way for all maps $(\F,u)$ in all the $\mathcal{B}_{\widetilde{\mathcal U}}$, where ``coherent" means:
\begin{enumerate} [label = (O\arabic*)]

\item \label{O1} if we exchange the $i$th and $(i+1)$st negative punctures of $(\F,u)$, then the sign changes by $(-1)^{|\gamma_{-,i}|_0\cdot|\gamma_{-,i+1}|_0}$, where $u$ is asymptotic to $\gamma_{-,j}$ near the $j$th negative puncture $p_{-,j}$ (\cite[Theorem 2]{BM});
\item \label{O2} the gluing/pregluing and disjoint union maps preserve orientations up to a specific sign change which arises from reordering the punctures; in particular, if $(\mathcal{F}_2,u_2)$ above is glued/preglued to $(\mathcal{F}_1,u_1)$ below along the last puncture of $u_2$, and the negative ends of the (pre-)glued curve are ordered by using the ordering for $u_2$ first, followed by the ordering of $u_1$, then there is no sign correction (\cite[Corollary 10]{BM});
\item \label{O3} if $u$ is asymptotic to $\gamma$ near the  puncture $p$ and $m(\gamma)>1$, then cyclically rotating the asymptotic marker at $p$ through an angle of $\frac{2\pi}{m(\gamma)}$ preserves the orientation if and only if $\gamma$ is good (\cite[Theorem 3]{BM});
\item \label{O4} precomposing with an automorphism of the domain which preserves the punctures and asymptotic markers preserves the orientation  (\cite[Proposition 11]{BM}).
\end{enumerate}

\begin{rmk}
A coherent orientation exists for {\em all maps $(\mathcal{F},u)$ in all the $\mathcal{B}_{\widetilde{\mathcal{U}}}$} and results of \cite{BM}, namely Theorems 2 and 3 and Corollaries 10 and 11, while stated only for moduli spaces of $J$-holomorphic maps, easily generalize to the case of semi-global Kuranishi structures.
\end{rmk}

\s\n  
{\em Step 2.} 
We explain how the above choices of $\mathfrak{o}(\F,u)$ induce an orientation on the total space of $\pi:\E\to \V$ of a semi-global Kuranishi chart. We treat the slightly easier $\op{(Cob)}$ case; the modifications in the $\op{(Symp)}$ case can be done using \eqref{eqn: R factor} below. Given $(\F, u)\in  V$, let $L_{(\mathcal{F},u)}$ be the full linearized $\overline \partial_J$-operator at $(\mathcal{F},u)$ as before and $L_{(\mathcal{F},u)}^E : T_{(\mathcal{F},u)}V\to E_{(\mathcal{F},u)}$ its restriction to $T_{(\F,u)}V=L^{-1}_{(\F,u)}(E_{(\F,u)})$. By \cite[Section 8.1.1]{FO3}, the orientation $\mathfrak o(\op{det} L_{(\F,u)})$ induces an orientation $\mathfrak o(\op{det} L_{(\F, u)}^E)$ and, equivalently, an orientation ${\mathfrak o}(E)$ on the total space of $E$.  The group action $\Gamma$ on $E$ preserves ${\mathfrak o}(E)$ by (O4).  Hence the orientation descends to the total space $\E$.

\s\n 
{\em Step 3.} We explain how to orient $\mathcal Z = \mathcal{Z}(\Kur^L(\alpha,J; \gamma_+; \bs\gamma_-;A), \mathfrak S)$: Let us write $\R\times \mathcal{Z}$ for the set of $[\F,u]$ such that $\llbracket F,u\rrbracket\in \mathcal{Z}$.  Let ${\mathfrak s}=[s_1,\dots,s_m]$ be a liftable multisection of $(E\to V,\Gamma)$ that is transverse to $\overline \partial _J$.  For $(\F, u)\in \overline \partial_J^{-1}(s_i)$, $i \in \{1,\dots,m\}$, the tangent space $T_{(\F,u)} (\R\times\mathcal{Z})$, before quotienting by domain automorphisms, can be identified with the kernel of the map 
$$L_{(\F,u)}^E - Ds_i(\F,u): T_{(\mathcal{F},u)}V\to E_{(\mathcal{F},u)},$$ 
where $Ds_i(\F,u)$ is the derivative of $s_i$ in the fiber direction (projected using some connection) and $L_{(\F,u)}^E - Ds_i(\F,u)$ is surjective by the tranversality of $\overline\bdry_J$ and $s_i$. Hence we orient $\R\times\mathcal{Z}$ using 
$$\mathfrak{o}(L_{(\F,u)}^E - Ds_i(\F,u))\simeq \mathfrak{o}(L_{(\F,u)}^E).$$

Finally we orient $\mathcal Z$ at $\llbracket \F,u\rrbracket \in \mathcal Z$ using the canonical isomorphism 
\begin{equation} \label{eqn: R factor}
    \R \otimes \wedge^{\text{top}} (T_{\llbracket \F,u\rrbracket} \mathcal Z)  \cong  \wedge^{\text{top}} (T_{(\mathcal F, u)}(\R\times \mathcal{Z}) ) \cong  \wedge^{\text{top}} (T_{(\mathcal F, u)}\overline \partial_J^{-1}(s_i))
\end{equation}
and a choice of orientation of $\R$.
This definition of $\mathfrak o(\mathcal Z)$ is well-defined by \ref{O4}, and \ref{O1} and \ref{O3} translate to $\mathfrak o(\mathcal Z)$ in a straightforward manner. The following is the reformulation of \ref{O2} in this setting:
\begin{enumerate} [label = (O2')]
    \item \label{O2'}  Consider the gluing map $\mathcal Z_2 \times \mathcal Z_1 \to \bdry^{\mathcal L}\mathcal Z_3$ as in Proposition~\ref{prop: tomatoes}, where $\mathcal{Z}_1$ and $\mathcal{Z}_2$ are $0$-dimensional, representatives $(\F_2,u_2)$ and $(\F_1,u_1)$ of elements of $\mathcal Z_2$ above and $\mathcal Z_1$ below are glued along the last negative puncture of $(\F_2,u_2)$, and the ends of the glued curve are ordered by using the ordering for $u_2$ first, followed by the ordering of $u_1$.  Then the gluing map maps the product orientation of $\mathcal Z_2 \times \mathcal Z_1$ to the boundary orientation of $\bdry^{\mathcal L}\mathcal Z_3$.
\end{enumerate}
This can be seen by observing that before quotienting out the $\R$-translation, if we translate representatives $(\F_2,u_2)$ and $(\F_1,u_1)$ of $\mathcal Z_2$ and $\mathcal Z_1$ in the same direction, then the glued curve also translates in that direction, whereas if we translate the representatives in the opposite directions, say towards each other, then the glued curve moves 
toward the interior of $\mathcal Z_3$; see \cite[Section 9.1.2]{BH} for a similar calculation.

\cb

\subsubsection{Examples}

To illustrate the definition of the differential, we consider two examples (here we use $\Q$-coefficients): 

1. Suppose that $\gamma_+, \gamma_{-,1}, \gamma_{-,2}\in \mathcal P_\alpha^{L,\op{good}}$ and that $\gamma_{-,1}, \gamma_{-,2}$ are simple.
Also suppose that the $\vartheta$-sorted moduli space
$$\mathcal M:= \mathcal M (\gamma_+;\underbrace{\gamma_{-,1},\dots,\gamma_{-,1}}_{3 \text{ copies}}, \underbrace{\gamma_{-,2},\dots,\gamma_{-,2}}_{5 \text{ copies}})$$
is regular and that $\llbracket\F,u\rrbracket\in \mathcal M/\mathbb R$ is simple and is counted as $+1$ with respect to the coherent orientation.
Then by relabeling the punctures in the domain of $u$ we get $3!5!$ elements in $ \mathcal M /\mathbb R$ and $3!5!$ elements in
$$\mathcal M(\gamma_+;\underbrace{\gamma_{-,2},\dots,\gamma_{-,2}}_{5 \text{ copies}}, \underbrace{\gamma_{-,1},\dots,\gamma_{-,1}}_{3 \text{ copies}})/\mathbb R;$$
the latter moduli space will not be counted since it is not $\vartheta$-sorted. 
Depending on whether $|\gamma_{-,1}|_0$ and $|\gamma_{-,2}|_0$ are both odd, the total contribution to $\partial \gamma_+$ from the image of $u$ (modulo $\R$-translations) is either zero or $\gamma_{-,1}^3\gamma_{-,2}^5$.

2. Suppose that $\gamma_+,\gamma'_+,\gamma_- \in \mathcal P_\alpha^{L,\op{good}}$ and that $\gamma_-$ is simple. Also suppose that $\mathcal M(\gamma_+;\gamma_-,\gamma_-)$ is regular and that $[\F, u]\in \mathcal M(\gamma_+;\gamma_-,\gamma_-)$ is a  branched double cover of a simple curve $[\F', u']\in \mathcal M(\gamma'_+;\gamma_-)$. Since reordering the negative punctures of $[\F,u]$ gives the same curve but changes the sign by $(-1)^{|\gamma_-|_0\cdot |\gamma_-|_0}$, the existence of the coherent orientation system implies that $|\gamma_-|_0$ must be even and that $\llbracket\F, u\rrbracket$ contributes $\pm 1$ (depending on the coherent orientation) to $\mathcal M(\gamma_+;\gamma_-,\gamma_-)/\mathbb R$ and contributes $\pm \frac{1}{2} \gamma_-^2$ to $\partial \gamma_+$.


\subsubsection{$\bdry^2=0$} \label{subsubsection: d squared is zero}

\begin{prop} \label{dga}
$\partial$ is well-defined and $\mathfrak A^L $ is a dga with differential given by $\partial$.
\end{prop}

\begin{proof}
By the construction of the semi-global Kuranishi structure, Proposition~\ref{prop: tomatoes}, and the above discussion on coherent orientations, if $|\gamma_+|-(|e^A|+|\bs\gamma_-|)=1$ and $\gamma_+, \bs\gamma_-$ are good, then $\mathcal{Z}(\Kur^L(\alpha, J;\gamma_+; \bs \gamma_-;A),\mathfrak S)$ is a compact \cb oriented weighted branched $0$-dimensional manifold.  Hence $\partial$ is well-defined. We will abbreviate
$$\mathcal Z(\Kur(\gamma_+;\bs\gamma_-)):=  \mathcal Z(\Kur^L(\alpha, J;\gamma_+;\bs \gamma_-;A),\mathfrak S),$$
assume that the homology classes have been chosen appropriately, and suppress homology classes from the notation. We will also write $\mathcal Z^*(\Kur(\gamma_+;\bs\gamma_-))$, where $*$ is a modifier such as ``$\op{ind}=k$''. \cb

To show $\partial^2=0$, as usual we identify the terms in $\partial^2$ as the signed weighted count of the boundary components of certain oriented compact weighted branched $1$-dimensional manifolds.

Let us write $\mathcal P_{\alpha}^{L,\op{good}} = \{\gamma_1, \dots ,\gamma_m\}$, written in $\vartheta$-sorted order. 
For any $\vartheta$-sorted $\bs \gamma=\gamma_1^{c_1}\dots\gamma_m^{c_m}$ with $c_1,\dots,c_m\geq 0$, applying the definition of $\bdry$ from Equation~\eqref{eqn: partial} we obtain:
\begin{gather}
\label{eqn: coefficient-1}  m_{\bs\gamma}\langle \partial^2 \gamma_+, \bs\gamma^\rev  \rangle = \sum_{A_1, A_2 \in H_2(M;\Z)} \sum_{\bs \beta\subseteq \bs\gamma}  \sum_{k=1}^m\sum_{q=0}^{c_k-b_k}(-1)^{a+b} \frac{d_{\gamma_+;(\bs\gamma-\bs\beta)\sqcup \gamma_k}}{m_{(\bs\gamma-\bs\beta)\sqcup \gamma_k}}\frac{d_{\gamma_k;\bs \beta}}{m_{\bs \beta}} e^{A_1 + A_2} m_{\bs\gamma}\\
\nonumber = \sum_{A_1, A_2 \in H_2(M;\Z)} \sum_{\bs \beta\subseteq \bs\gamma}\sum_{k=1}^m\sum_{q=0}^{c_k-b_k}(-1)^{a+b}
{c_1\choose b_1}\dots{c_m \choose b_m}\frac{d_{\gamma_+;(\bs\gamma-\bs\beta)\sqcup \gamma_k}d_{\gamma_k;\bs \beta}}{m(\gamma_k)(c_k-b_k+1)} e^{A_1 + A_2},
\end{gather}
where $\bs\beta=\gamma_1^{b_1}\dots\gamma_m^{b_m}$ is a $\vartheta$-sorted subtuple of $\bs \gamma$ with $0\leq b_i\leq c_i$ for $i=1,\dots,m$, and {\em the summands depend on the variable $q$ only through $a$ and $b$, defined below.}
Now we explain the notation appearing in the summands of the above formula:
\begin{gather*}
\bs \gamma^\rev =\gamma_m^{c_m}\dots \gamma_1^{c_1}, \qquad
\bs\beta^\rev= \gamma_m^{b_m} \dots \gamma_1^{b_1},\\
(\bs \gamma-\bs\beta)\sqcup \gamma_k:=\gamma_1^{c_1-b_1}\dots \gamma_{k}^{c_k-b_{k}+1}\dots\gamma_{m}^{c_m-b_m},\\
(\bs \gamma^\rev -\bs\beta^\rev)\sqcup \gamma_k:= \gamma_{m}^{c_m-b_m}\dots \gamma_{k}^{c_k-b_{k}+1}\dots \gamma_1^{c_1-b_1},\\
(\bs\gamma-\bs\beta) \sqcup_{\gamma_k,q}\bs\beta:= \gamma_1^{c_1-b_1}\dots\gamma_k^q(\gamma_1^{b_1}\dots\gamma_m^{b_m}) \gamma_k^{c_k-b_k-q}\dots\gamma_m^{c_m-b_m},\\
(\bs\gamma^\rev-\bs\beta^\rev) \sqcup_{\gamma_k,q}\bs\beta^\rev:= \gamma_m^{c_m-b_m} \dots \gamma_k^{c_k-b_k-q} (\gamma_m^{b_m} \dots \gamma_1^{b_1}) \gamma_k^q \dots \gamma_1^{c_1-b_1},
\end{gather*}
and then $(-1)^b$ is given by the equation $$\bs\gamma^\rev=(-1)^b(\bs\gamma^\rev-\bs\beta^\rev) \sqcup_{\gamma_k,q}\bs\beta^\rev$$ as supercommutative monomials and $a=|\gamma_m^{c_m-b_m} \dots \gamma_k^{c_k-b_k-q}|_0$.

We claim that, for each $\bs\beta\subseteq\bs\gamma$ and $k\in \{1,\dots,m\}$ with $|\gamma_k|-(|e^{A_1}|+|\bs\beta|)=1$ and $|\gamma_+|-(|e^{A_2}|+|(\bs \gamma-\bs\beta)\sqcup \gamma_k|)=1$ and $q\in \{0,1,\dots,c_k-b_k\}$, there exist $\mathcal L \gg 0$ and a map 
$$G_{\bs\beta,k,q}:  \mathcal Z(\Kur(\gamma_k;\bs\beta)) \times  \mathcal Z(\Kur(\gamma_+;(\bs\gamma-\bs\beta)\sqcup \gamma_k))  \times O \to  \bdry^\mathcal L \mathcal Z(\Kur(\gamma_+;\bs \gamma)), $$
where
$O$ is defined below and 
$\bdry^\mathcal L$ is the $\mathcal L$-boundary defined in Proposition~\ref{prop: tomatoes}.

The proof of the claim essentially follows from Proposition~\ref{prop: tomatoes}, but we work out the combinatorics more carefully with signs.  The set $O$ is the set of all isomorphism classes of two-vertex CH trees that contract to the one-vertex tree corresponding to the moduli space $\mathcal M(\gamma_+; \bs\gamma),$ 
but we give a more explicit equivalent description.
For any $k\in\{1,\dots,m\}$, $q\in \{0,1,\dots,c_m-b_m\}$, and
$$(\llbracket\F_1,u_1\rrbracket,\llbracket\F_2, u_2\rrbracket)\in \mathcal Z(\Kur(\gamma_k;\bs \beta))\times \mathcal Z(\Kur(\gamma_+; (\bs\gamma-\bs \beta) \sqcup \gamma_k)),$$
we glue $(\F_1,u_1)$ and $(\F_2, u_2)$ so that the positive puncture of $\F_1$ is glued to the $(q+1)$th  negative puncture of $\F_2$ amongst those that converge to $\gamma_k$ and denote the glued curve by $(\F', u')$. Note that the negative punctures of $\F'$ have a natural ordering  $\tilde o$ coming from the gluing so that $\llbracket\F'_{\tilde o}, u'\rrbracket\in  \mathcal Z(\Kur(\gamma_+;(\bs\gamma-\bs\beta) \sqcup_{\gamma_k,q} \bs\beta))$. Next let $o$ be an ordering of the negative punctures of $\mathcal{F}'$ such that:
\begin{enumerate}
\item the induced orderings of the negative punctures of $\mathcal{F}_2$ (minus the negative end that is glued) and $\mathcal{F}_1$ agree with the initial orderings of the negative punctures of $\mathcal{F}_1$ and $\mathcal{F}_2$; and
\item the surface $\mathcal{F}'$ with the ordering $o$, denoted by $\mathcal{F}'_o$, has $\vartheta$-sorted negative punctures;
\end{enumerate}
and let $O$ be the set of such orderings $o$. It is immediate from the definition that $|O|={c_1 \choose b_1}\dots{c_m \choose b_m}$.  We then define
$$G_{\bs\beta,k,q}(\llbracket\F_1,u_1\rrbracket,\llbracket\F_2, u_2\rrbracket, o):=\llbracket\F'_o, u=u'\rrbracket.$$

We now claim that the weighted signed count of the boundary components of the $1$-dimensional branched manifold $\bdry^{\mathcal L}\mathcal Z(\Kur(\gamma_+;\bs \gamma))$ is given by the second line of Equation~\eqref{eqn: coefficient-1}. 
This is a consequence of the following observations:
\begin{enumerate}
\item By \ref{O3}, the components of $ \mathcal Z(\Kur(\gamma_+;\bs \gamma))$ that involve bad orbits do not contribute to the count.

\item To apply \ref{O2'}, we first reorder the negative punctures of $\F_2$ so that the $(q+1)$st puncture that converges to $\gamma_k$ becomes the last puncture. More precisely, we apply the following three steps:
\be
\item First consecutively swap this puncture with the negative puncture that is right after it. By \ref{O1}, this results in a sign change of  $$(-1)^{|\gamma_k|_0 \cdot |\gamma_k^{c_k-b_k-q} \cdots \gamma_m ^{c_m - b_m}|_0}=(-1)^{a\cdot |\gamma_k|_0}.$$
\item Then glue $(\F_1, u_1)$ and $(\F_2, u_2)$. By \ref{O2'} there is no additional sign correction from the gluing.
\item Finally, consecutively swap the block of negative punctures that came from $\F_1$ with the negative puncture right before it until the block takes the place of the $(q+1)$st negative puncture of $\F_2$. This step incurs a sign change of 
$(-1)^{a\cdot|\gamma_m^{b_m} \dots \gamma_1^{b_1}|_0}.$
\ee 
Hence the total sign change from the gluing is 
$$(-1)^{a\cdot (|\gamma_k|_0 + |\gamma_m^{b_m} \dots \gamma_1^{b_1}|_0)  }  = (-1)^a,$$
since $\partial$ changes the degree by $-1$. The glued curve is denoted by $(\F'_{\tilde o}, u')$, where $\tilde o$ is the ordering of the negative ends of $\F'$ coming from gluing.

\item We reorder the negative punctures of $\F'_{\tilde o}$ to $\F'_o$ for some $o \in O,$ and then we have $\llbracket\F'_o, u'\rrbracket\in  \mathcal Z(\Kur(\gamma_+;\bs\gamma-))$. Note that the exact element of $o \in O$ does not matter in terms of orientations. This is because in Equation~\eqref{eqn: coefficient-1} we only need to consider the case when there is no $k\in \{1, \cdots, m\}$ such that $c_k > 1$ and $|\gamma_k|_0 = 1.$ Then swapping two adjacent negative punctures that converge to the same orbit does not change the sign. Therefore, the reordering produces a sign, which is the same as the sign obtained from changing $(\bs\gamma-\bs\beta) \sqcup_{\gamma_k,q}\bs\beta$ to $\bs\gamma$, and which in turn is the same as the sign $(-1)^b$ obtained from changing $(\bs\gamma^\rev-\bs\beta^\rev) \sqcup_{\gamma_k,q}\bs\beta^\rev$ to $\bs\gamma^\rev$.

\item $\op{Im}G_{\bs\beta,k,q}=\op{Im}G_{\bs\beta,k,q'}$ for any $q,q'\in \{0,1,\dots,c_m-b_m\}$ and this cancels the term $c_k-b_k+1$ in the denominator of the second line of Equation~\eqref{eqn: coefficient-1}.
\item There are $m(\gamma_k)$ ways to simultaneously change the asymptotic markers of $\F_1$ and $\F_2$ such that $G_{\bs\beta,k,q}(\llbracket\F_1,u_1\rrbracket,\llbracket\F_2,u_2\rrbracket,o)$ does not change; this cancels the term $m(\gamma_k)$ in the denominator of the second line of Equation~\eqref{eqn: coefficient-1}.
\end{enumerate} \vskip-.2in
\end{proof}

\subsection{dga morphisms induced by cobordisms} \label{chain map}

Let $(\widehat W,\widehat \alpha)$ be the completion of the compact Liouville cobordism $(W,\alpha)$ from $(M_+,\alpha_+)$ to $(M_-,\alpha_-)$ as in Section \ref{subsection: cob case}. Given $L_+\leq L_-$, let $J$ be an almost complex structure on $\widehat W$ which restricts to $J_+$ at the positive end and to $J_-$ at the negative end.
Suppose $(\widehat \alpha, J)$ is an $(L_+,L_-)$-simple pair and $J$ is $(L_+,L_-)$-end-generic.

Using the semi-global Kuranishi structures $\Kur^{L_\pm}:=\Kur^{L_\pm}(\alpha_\pm,J_\pm)$, multisections $\mathfrak S_\pm = \{{\mathfrak s}_{T_\pm}\}_{T_\pm}$, and complete sets $\mathcal{S}^{L_\pm}_{\alpha_\pm}$ of trivializing surfaces for $\mathcal{P}^{L_\pm}_{\alpha_\pm}$, we can construct the dga's
$$\mathfrak A^{L_\pm}:=\mathfrak A^{L_\pm}(M_\pm,\alpha_\pm,J_\pm,\Kur^{L_\pm}, \mathfrak S_\pm, \mathcal{S}^{L_\pm}_{\alpha_\pm})$$
as in Section \ref{chain complex}.  Here we are also using orderings $\vartheta_\pm$ of $\mathcal{P}^{L_\pm}_{\alpha_\pm}$.
We can also construct a semi-global Kuranishi structure
$$\Kur^{(L_+,L_-)}:=\Kur^{(L_+,L_-)}(\widehat \alpha,J)$$ 
that is compatible with $\Kur^{L_\pm}$, together with a collection of multisections $\mathfrak S = \{{\mathfrak s}_T\}_T$. We denote by the above collection of data by:
$$\mathcal C:=(M_\pm, L_\pm, \alpha_\pm, J_\pm, \Kur^{L_\pm}, \mathfrak S_\pm,  \mathcal{S}^{L_\pm}_{\alpha_\pm}, \widehat W, \widehat \alpha, J, \Kur^{(L_+,L_-)}, \mathfrak S).$$

We now define the (unital) dga morphism
\nom[Phi]{$\Phi$}{Contact homology dga morphism from $\mathfrak A^{L_+}$ to $\mathfrak A^{L_-}$, after tensoring with an appropriate coefficient system}
$$\Phi:\mathfrak A^{L_+} \otimes_{\Q[H_2(M_+;\Z)]}\Q[H_2(W;\Z)]\to \mathfrak A^{L_-}\otimes_{\Q[H_2(M_-;\Z)]}\Q[H_2(W;\Z)].$$
Here we are viewing $\mathfrak{A}^{L_\pm}$ as $\Q[H_2(W;\Z)]$-modules via the algebra maps
$$\Q[H_2(M_\pm;\Z)]\to \Q[H_2(W;\Z)]$$
induced by $H_2(M_\pm;\Z)\to H_2(W;\Z)$.   
For any $\gamma_+ \in \mathcal P_{\alpha_+}^{L_+}$, let
\begin{equation}\label{eqn: chain map}
\Phi(\gamma_+)=\sum_{\bs \gamma_-}\sum_{A\in H_2(W;\Z)} \frac{p_{\gamma_+,\bs \gamma_-,A}}{m_{\bs\gamma_-}} e^A \gamma_{-,k}^{i_k} \dots \gamma_{-,1}^{i_1},
\end{equation} 
where \cb the sums are taken over all good $\vartheta_-$-sorted $\bs\gamma_-$ written in the form of Equation \eqref{eqn: gamma-} and $A\in H_2(W;\Z)$ with $|\gamma_+|-(|e^A|+|\bs\gamma_-|)=0$, and $p_{\gamma_+,\bs\gamma_-,A}$ is the weighted signed count of elements in $\mathcal{Z}(\Kur^{(L_+,L_-)}(\widehat \alpha, J; \gamma_+;\bs \gamma_-;A),\mathfrak S)$. (If $|\gamma_+|\neq |e^A|+|\bs\gamma_-|$, \cb then we set $p_{\gamma_+,\bs\gamma_-,A}=0$.) The homology class $A$ can be obtained from $(\mathcal{F},u)\in \mathcal{Z}(\Kur^{(L_+,L_-)}(\widehat \alpha, J; \gamma_+;\bs \gamma_-),\mathfrak S)$ by capping off using $\mathcal{S}^{L_\pm}_{\alpha_\pm}$ as in Section~\ref{subsubsection: grading}.  $\Phi$ can then be extended to $\mathfrak A^{L_+}$ as an algebra homomorphism.

\begin{prop}
$\Phi$ is a well-defined (unital) dga morphism.
\end{prop}

\begin{proof}
The proof is similar to that of Proposition \ref{dga}.
\end{proof}

Let  $\Phi_*: HC^{L_+}\to HC^{L_-}$ be the induced map on homology.

\begin{prop}
Let $(\widehat{W}, \widehat \alpha)$ be a trivial symplectic cobordism.
Given $L_+ \leq L_-$, let $J$ be an $\R$-translation invariant almost complex structure on $\widehat W$ such that $(\widehat \alpha, J)$ is an $(L_+, L_-)$-simple pair.  Then $\Phi_* = \op{id}$ with respect to semi-global Kuranishi structures defined using $J$, multisections, and complete sets of trivializing surfaces. 
\end{prop}

\begin{proof}
The $\R$-invariant almost complex structure $J$ is automatically $(L_+, L_-)$-end-generic, with the exception of possibly branched covers of trivial cylinders. There are three types of moduli spaces $\overline{\mathcal{M}}_J(\gamma_+;\bs\gamma_-;A)$ in the cobordism $\widehat W$: (i) $\mathcal{A}(\gamma_+)> \mathcal{A}(\bs\gamma_-)$, (ii) $\mathcal{A}(\gamma_+)= \mathcal{A}(\bs\gamma_-)$ and the curves are (possibly multiple levels of) branched covers of trivial cylinders, and (iii) trivial cylinders. 

In order to construct Kuranishi structures $\Kur^{(L_+,L_-)}(\widehat \alpha, J; \gamma_+;\bs \gamma_-;A)$ with respect to $(\widehat W,J)$ using our prescription, it suffices to puncture (iii) trivial cylinders in view of Remark~\ref{rmk: trivial cylinder}, since curves of type (ii) have $\chi < 0$ and do not need to be punctured. Let $J'$ be a sufficiently small perturbation of $J$ such that $J'$ is $(L_+, L_-)$-simple, $(L_+,L_-)$-end-generic, and compatible with $J=J_\pm$ at the positive and negative ends. Since (iii) is regular and a sufficiently small perturbation of $J$ to $J'$ preserves regularity, there exists a bijection between trivial cylinders $u:\R\times S^1\to (\widehat W,J)$ and $J'$-holomorphic cylinders $u': \R\times S^1 \to (\widehat W,J')$ with the same ends for $J'$ sufficiently close to $J$.  Hence may take the punctures $\mathbf{q}$ for $u$ and $\mathbf{q}'$ for $u'$ to be at the same $s$-height.\footnote{Roughly speaking, $\Kur^{(L_+,L_-)}(\widehat \alpha, J; \gamma_+;\bs \gamma_-;A)$ and  $\Kur^{(L_+,L_-)}(\widehat \alpha, J'; \gamma_+;\bs \gamma_-;A)$, as well as their multisections, can be taken to be ``close", so we can use either to compute $\Phi_*$.}

Let $\mathcal Z = \mathcal{Z}(\Kur^{(L_+,L_-)}(\widehat \alpha, J; \gamma_+;\bs \gamma_-;A),\{{\mathfrak s'_T}\}_T)$.  The key point is to take all the multisections $\mathfrak s'_T$  --- besides the multisection corresponding to (iii), which we can take to be $0$ --- to be perturbations of $s$-invariant multisections. Then in Cases (i) and (ii) $\mathcal Z$ is empty when $\op{vdim}=0$, since we are effectively computing $\mathcal Z(\Kur^L(\alpha,J;\gamma_+,\bs\gamma_-;A),\{{\mathfrak s_T}\}_T)$ in $\op{(Symp)}$ when $\op{vdim}=0$.  In Case (iii), $\mathcal Z$ consists of only one element for each $\gamma_+$, the trivial cylinder.  

After a sign count that is left to the reader, we obtain $\Phi_* = \op{id}$.
\end{proof}

\cb
\subsection{Definition of contact homology}

We are now in a position to define the contact homology algebra $HC(\mathcal D)$.  Given a closed cooriented contact manifold $(M,\xi)$,  let
\begin{align*}
\mathcal D=&(\alpha, \{  L_i, \varphi_i, J_i, \Kur^{L_i}(\alpha_i,J_i), \mathfrak S_i, \mathcal{S}^{L_i}_{\alpha_i}, \widehat \alpha_{i,i+1},J_{i,i+1}, \\
& \qquad  \Kur^{(L_i,L_{i+1})}(\widehat \alpha_{i,i+1},J_{i,i+1}), \mathfrak S_{i,i+1}\}_{i\in \mathbb N})
\end{align*}
be a collection that consists of the following data:
\begin{itemize}
\item a nondegenerate contact form $\alpha$ for $\xi$;
\item an increasing sequence $\{L_i\}_{i\in \mathbb N}$ such that $L_i\to \infty$;
\end{itemize}
and, for each $i\in \mathbb N$,
\begin{itemize}
\item a positive function $\varphi_i$ on $M$ which is close to $1$, such that $\alpha_i:=\varphi_i \alpha$ is an $L_i$-simple contact form and $\varphi_i<\varphi_{i+1}$;
\item an almost complex structure $J_i$ on $\mathbb R\times M$ such that $(\alpha_i,J_i)$ is an $L_i$-simple pair;
\item a semi-global Kuranishi structure $\Kur^{L_i}(\alpha_i,J_i)$ and multisections $\mathfrak S_i$;
\item a complete set $\mathcal{S}^{L_i}_{\alpha_i}$ of trivializing surfaces for $\mathcal{P}^{L_i}_{\alpha_i}$;
\item a completion $(\widehat W_{i,i+1}=\R\times M,\widehat \alpha_{i,i+1})$ of a compact Liouville cobordism  from $(M,\alpha_i)$ to $(M,\alpha_{i+1})$; here $\widehat \alpha_{i,i+1}:= \varphi_{i,i+1}\alpha$, where $\varphi_{i,i+1}$ interpolates from a positive multiple of $\varphi_i$ to $\varphi_{i+1}$;
\item an almost complex structure $J_{i,i+1}$ on $\widehat W_{i,i+1}$ that tames $d\widehat \alpha_{i,i+1}$, restricts to $J_i$ at the positive end and to $J_{i+1}$ at the negative end, and is $(L_i,L_{i+1})$-end-generic;
\item a semi-global Kuranishi structure $\Kur^{(L_i,L_{i+1})}(\widehat \alpha_{i,i+1},J_{i,i+1})$ that is compatible with $\Kur^{L_i}(\alpha_i,J_i)$ and $\Kur^{L_{i+1}}(\alpha_{i+1},J_{i+1})$ and multisections $\mathfrak S_{i,i+1}$.
\end{itemize}

Denote the map induced by the cobordism $(\widehat W_{i,i+1},\widehat \alpha_{i,i+1}, J_{i,i+1})$ and the semi-global Kuranishi structure $\Kur^{(L_i,L_{i+1})}(\widehat \alpha_{i,i+1}, J_{i,i+1}))$ and multisections $\mathfrak S_{i,i+1}$ by
\begin{align*}
\Phi_{i*}: HC^{L_i}&(\alpha_i,J_i; \Kur^{L_i}(\alpha_i,J_i), \mathfrak S_i) \to \\
&HC^{L_{i+1}}(\alpha_{i+1},J_{i+1};\Kur^{L_{i+1}}(\alpha_{i+1},J_{i+1}), \mathfrak S_{i+1}).
\end{align*}
Finally, we define the contact homology algebra by
$$H(\mathcal D):=\lim_{\longrightarrow} HC^{L_i}(\alpha_i,J_i; \Kur^{L_i}(\alpha_i,J_i), \mathfrak S_i),$$
where the directed system is constructed using $\Phi_{i*}$. From now on we will often suppress the multisections from the notation.

\subsection{Invariance of contact homology} \label{chain homotopy}

In this section we prove Propositions~\ref{prop: chain homotopy} and \ref{prop: composition} and use them to prove that $HC(\mathcal D)$ is an invariant of the contact manifold $(M,\xi)$ in Section~\ref{subsubsection: proof of invariance}.  

Let $\{(\widehat W^\tau=\widehat W,\widehat \alpha^\tau)\}_{0\leq \tau \leq1}$ be a family of completed compact Liouville cobordisms from $(M_+,\alpha_+)$ to $(M_-,\alpha_-)$. For each $\tau \in[0,1]$, let $J^\tau$ be an almost complex structure on $\widehat W^\tau$ that restricts to $J_\pm$ at the positive/negative ends of $\widehat W^\tau$. We assume that $(\widehat \alpha^\tau, J^\tau)$ is an $(L_+,L_-)$-simple pair, $J^\tau$ is $(L_+,L_-)$-end-generic, and $(\widehat \alpha^\tau, J^\tau)$ is independent of $\tau$ on each of $\tau\in[0,\epsilon]$ and $\tau\in[1-\epsilon,1]$, where $\epsilon>0$ is small.

Let $\Kur^{L_\pm}:=\Kur^{L_\pm}(\alpha_\pm,J_\pm)$ be the semi-global Kuranishi structures constructed for $(M_\pm,\alpha_\pm, J_\pm)$, let $\mathfrak S_\pm$ be the multisections on $\Kur^{L_\pm}$, and let 
$$\mathfrak A^{L_\pm}=\mathfrak A^{L_\pm}(M_\pm,\alpha_\pm,J_\pm,\Kur^{L_\pm}, \mathfrak {S}_\pm, \mathcal{S}^{L_\pm}_{\alpha_\pm})$$
as before.

\subsubsection{Chain homotopy}

For $\tau=0,1$, we have two collections of data
$$\mathcal C^\tau=(M_\pm, L_\pm, \alpha_\pm, J_\pm, \Kur^{L_\pm},  \mathcal{S}^{L_\pm}_{\alpha_\pm},\widehat W^\tau, \widehat \alpha^\tau, J^\tau, \Kur^{(L_+,L_-)})$$
as in Section \ref{chain map} and chain maps
$$\Phi^\tau:\mathfrak A^{L_+} \otimes_{\Q[H_2(M_+;\Z)]}\Q[H_2(W;\Z)]\to \mathfrak A^{L_-}\otimes_{\Q[H_2(M_-;\Z)]}\Q[H_2(W;\Z)],$$
defined using the data $\mathcal C^\tau$.

\begin{prop}[Chain homotopy] \label{prop: chain homotopy}
There is a degree $1$ $\Q[H_2(W;\Z)]$-module map
$$\Omega_K: \mathfrak A^{L_+} \otimes_{\Q[H_2(M_+;\Z)]}\Q[H_2(W;\Z)]\to \mathfrak A^{L_-}\otimes_{\Q[H_2(M_-;\Z)]}\Q[H_2(W;\Z)]$$
such that $\Phi^1-\Phi^0=\partial_- \circ\Omega_K - \Omega_K\circ\partial_+.$ In particular, $\Phi^0_*=\Phi^1_*$.
\end{prop}

\begin{proof}
The discussion of Section~\ref{subsection: chain homotopy case} yields a semi-global Kuranishi structure
$$\Kur^{(L_+,L_-)}(\{(\widehat \alpha^\tau,J^\tau)\}_{\tau\in[0,1]})$$
so that it is independent of $\tau$ on each of $\tau\in[0,\epsilon]$ and $\tau\in[1-\epsilon,1]$ and agrees with $\Kur^{(L_+,L_-)}(\widehat \alpha^\tau,J^\tau)$ for $\tau=0,1$, as well as a collection $\mathfrak S^{[0,1]} = \{{\mathfrak s}_T^{[0,1]}\}_T$ of multisections. We are using orderings $\vartheta_\pm$ for $\mathcal{P}^{L_\pm}_{\alpha_\pm}$.

We define $\Omega_K$ as follows:  First let $\Omega_K(1)=0$.  For any $k$-tuple $\bs\gamma_+$ of good Reeb orbits, we set  
\begin{equation}
\Omega_K(\bs\gamma_+)=\sum_{\bs \gamma_-}\sum_{A\in H_2(W;\Z) } \frac{\kappa_{\bs\gamma_+,\bs \gamma_-,A}}{m_{\bs\gamma_{-,1}}  \dots m_{\bs\gamma_{-,k}}} e^A  \bs\gamma_{-,k}^\rev \dots \bs\gamma_{-,1}^\rev,
\end{equation} 
where the sums are taken over all $\bs\gamma_-$ which is a $k$-tuple of $\vartheta_-$-sorted tuples of good Reeb orbits, i.e.,
$$\bs\gamma_-=(\bs\gamma_{-,1},\dots,\bs\gamma_{-,k})$$ and
\begin{equation}
\bs\gamma_{-,j}^\rev=( \underbrace{\gamma_{-,j,r_j},\dots, \gamma_{-,j,r_j}}_{i_{j,r_j} \text{ copies }},\dots, \underbrace{\gamma_{-,j,1},\dots, \gamma_{-,j,1}}_{i_{j,1} \text{ copies }}),
\end{equation} \cb
$A\in H_2(W;\Z)$ with $|\bs\gamma_+|-(|e^A|+|\bs\gamma_-|)=-1$, and $\kappa_{\bs\gamma_+,\bs\gamma_-,A}$ is the weighted signed count of the elements in
$$\mathcal{Z}^{[0,1]}(\Kur(\bs\gamma_+,\bs\gamma_-,A)):= \mathcal Z( \Kur^{(L_+,L_-)}(\{(\widehat \alpha^\tau, J^\tau)\}_{\tau\in[0,1]};\bs\gamma_+;\bs\gamma_-;A), \mathfrak S^{[0,1]}).$$
In what follows we will usually suppress $A$ from the notation.

Let $\Phi^0,\Phi^1$ be two dga morphisms
$$\mathfrak A^{L_+} \otimes_{\Q[H_2(M_+;\Z)]}\Q[H_2(W;\Z)]\to \mathfrak A^{L_-}\otimes_{\Q[H_2(M_-;\Z)]}\Q[H_2(W;\Z)],$$
induced by the cobordism $(\widehat W^\tau, \widehat \alpha^\tau, J^\tau)$, $\tau=0,1$, using the semi-global Kuranishi structure $\Kur^{(L_+,L_-)}( \widehat \alpha^\tau,J^\tau)$. We can verify that $\Omega_K$ is well-defined and satisfies
\begin{equation}\label{chain homotopy equation}
\Phi^1-\Phi^0=\partial_-\circ \Omega_K - \Omega_K\circ\partial_+
\end{equation}
as in the proof of Proposition \ref{dga}. The key point is to observe that, with the perturbation from Section~\ref{subsection: chain homotopy case}, each end of a $1$-dimensional (and hence $\op{ind}=0$) branched manifold $\mathcal{Z}^{[0,1]}(\Kur(\gamma_+,\bs\gamma_-))$ corresponds to:
\be
\item[(i)] a term contributing to $\Phi^0$ or $\Phi^1$;
\item[(ii)] the gluing of an element of some $\mathcal{Z}^{[0,1]}(\Kur(\gamma_+,\bs\gamma_-'))$ of index $-1$ corresponding to $\Omega_K$ and an element of some $\mathcal{Z}(\Kur(\bs\gamma_-',\bs\gamma_-))$ corresponding to $\bdry_-$ where all but one of the components is a trivial cylinder; or
\item[(iii)] the gluing of an element of some $\mathcal{Z}(\Kur(\gamma_+,\bs\gamma_+'))$ of index $1$ corresponding to $\bdry_+$ and an element of some $\mathcal{Z}^{[0,1]}(\Kur(\bs\gamma_+',\bs\gamma_-))$ of index $-1$ corresponding to $\Omega_K$. (In particular, if $\op{ind}< -1$, then $\mathcal{Z}^{[0,1]}(\Kur(\bs\gamma_+',\bs\gamma_-))=\varnothing$ and $\mathcal{Z}^{[0,1]}(\Kur(\gamma_+,\bs\gamma_-'))=\varnothing$.)
\ee

The signs follow from the proof of Proposition \ref{dga} and \cite[Section 9.2.2]{BH}. In particular, the sign difference of the two terms on right-hand side of Equation~\eqref{chain homotopy equation} comes from the fact that for the Case (ii) gluing, translating a representative of an element in $\mathcal{Z}(\Kur(\gamma_+,\bs\gamma_+'))$ along the positive $\R$-direction moves the glued curve towards boundary; for the Case (iii) gluing, translating a representative of an element in $\mathcal{Z}(\Kur(\bs\gamma_-',\bs\gamma_-))$ along the positive $\R$-direction moves the glued curve away from the boundary.  \cb
\end{proof}

\subsubsection{Proof of invariance} \label{subsubsection: proof of invariance}

For $(i,j)=(1,2), (i,j)=(2,3)$, or $(i,j)=(1,3)$, let $(\widehat W_{ij}, \widehat \alpha_{ij})$ be the completion of a Liouville cobordism $(W_{ij}, \alpha_{ij})$ from $(M_i,\alpha_i)$ to $(M_j,\alpha_j)$ such that $(W_{12},\alpha_{12},J_{12})$ and $(W_{23},\alpha_{23},J_{23})$ can be glued to give $(W_{13},\alpha_{13},J_{13})$ and let
$$\mathcal C_{ij}=(M_i,M_j, L_i,L_j, \alpha_i,\alpha_j, J_i,J_j, \Kur^{L_i},\Kur^{L_j}, \mathfrak{S}_i,\mathfrak{S}_j,  \mathcal{S}^{L_i}_{\alpha_i}, \mathcal{S}^{L_j}_{\alpha_j}, \widehat W_{ij}, \widehat \alpha_{ij}, J_{ij}, \Kur^{(L_+, L_-)}_{ij},\mathfrak{S}_{ij} )$$
be a collection of data as in Section \ref{chain map} for $(\widehat W_{ij}, \widehat \alpha_{ij})$.  Using the data $\mathcal C_{ij}$ we obtain a chain map
$$\Phi_{ij}: \mathfrak A^{L_i}\otimes_{\Q[H_2(M_i;\Z)]}\Q[H_2(W_{ij};\Z)]\to \mathfrak A^{L_j}\otimes_{\Q[H_2(M_j;\Z)]}\Q[H_2(W_{ij};\Z)].$$

\begin{prop}[Composition of chain maps] \label{prop: composition}
There exists a degree 1 map
$$\Omega_K:\mathfrak A^{L_1}\otimes_{\Q[H_2(M_1;\Z)]}\Q[H_2(W_{13};\Z)]\to \mathfrak A^{L_3} \otimes_{\Q[H_2(M_3;\Z)]}\Q[H_2(W_{13};\Z)]$$
such that, after the appropriate base changes (but keeping the same notation),
$$\Phi_{13}-\Phi_{23}\Phi_{12}=\partial\circ \Omega_K - \Omega_K\circ\partial.$$
In particular, $\Phi_{23*}\Phi_{12*}=\Phi_{13*}.$
\end{prop}

\begin{proof}
The proof is almost the same as that of Proposition \ref{dga}.
\end{proof}

Finally we are ready to prove the following:

\begin{thm}
The contact homology algebra $HC(\mathcal D)$ is independent of the choices $\mathcal D$ and hence is an invariant of $(M,\xi)$.
\end{thm}

\begin{proof}
Given two choices of data $\mathcal D$ and $\mathcal D'$, there exist two direct limits
\begin{align*}
HC(\mathcal D)&=\lim_{\longrightarrow} HC^{L_i}(\alpha_i,J_i, \Kur^{L_i}(\alpha_i,J_i),\mathfrak{S}_i),\\
HC(\mathcal D')&=\lim_{\longrightarrow} HC^{L'_i}(\alpha'_i,J'_i, \Kur^{L'_i}(\alpha'_i,J'_i),\mathfrak{S}_i'),
\end{align*}
where the maps for $\mathcal{D}$ are written as $\Phi_{i*}$ and the maps for $\mathcal{D}'$ are written as $\Phi_{i*}'$.
We abbreviate:
\begin{align*}
HC^{L_i}&:=HC^{L_i}(\alpha_i,J_i, \Kur^{L_i}(\alpha_i,J_i),\mathfrak{S}_i),\\
HC^{L'_i}&:=HC^{L'_i}(\alpha'_i,J'_i, \Kur^{L'_i}(\alpha'_i,J'_i),\mathfrak{S}_i').
\end{align*}
For each $i$, let $j(i)$ be the smallest $j$ such that $L_i<L'_j$.
Given a Liouville cobordism from $(M,\alpha_i,J_i)$ to $(M,\alpha'_{j(i)},J'_{j(i)})$ with Kuranishi data $\Kur^{L_i}(\alpha_i, J_i), \mathfrak{S}_i$ and $\Kur^{L_{j(i)}'}(\alpha_{j(i)}',J_{j(i)}'),\mathfrak{S}'_{j(i)}$, one can define a map $\Psi_{i*}: HC^{L_i} \to HC^{L'_{j(i)}}$ which satisfies
$$\Psi_{i+1*}\Phi_{i*}=\Phi_{j(i+1)-1*}'\dots \Phi_{j(i)+1*}' \Phi_{j(i)*}' \Psi_{i*}.$$
Similarly, we obtain maps $\Theta_{j*}:HC^{L'_j}\to HC^{L_{i(j)}}$ for each $j\in \mathbb N$ that satisfy $$\Theta_{j+1*}\Phi'_{j*}=\Phi_{i(j+1)-1*}\dots\Phi_{i(j)+1*}\Phi_{i(j)*}\Theta_{j*}.$$
Hence we obtain two direct limit maps
$$\Psi_*:  HC(\mathcal D)\to HC(\mathcal D'), \quad \Theta_*: HC(\mathcal D') \to HC(\mathcal D).$$
Note that, for each $k\in \mathbb N$, the two maps
$$\Theta_{j(k)*}\Psi_{k*} \text{ and } \Phi_{i(j(k))-1*}\dots \Phi_{k+1*} \Phi_{k*}: HC^{L_k}\to HC^{L_{i(j(k))}}$$
are the same. Therefore $\Theta_* \Psi_*=\op{id}$; similarly $\Psi_* \Theta_*=\op{id}$.
\end{proof}

\section*{Appendix}

In the Appendix we prove Lemma \ref{shenyang} and Lemma \ref{lemma: in Wkp and C1}. 
We first start with some basic facts about Sobolev spaces. 

\begin{lemma} \label{fg}
Suppose that $p=2$ and $k\geq3$.  If $g,h\in W^{k,p}_{\op{loc}}(\R^2)$, then $g\cdot h \in W^{k,p}_{\op{loc}}(\R^2)$.
\end{lemma}

\begin{proof}
We prove that $ (g\cdot h)^{(m)}\in L^p_{\op{loc}}(\R^2)$ for any $0\leq m \leq k$. Let $\Omega\subset \R^2$ be a domain with smooth boundary and compact closure. For $m=0$, we have $g,h\in C^0(\Omega)$\footnote{More precisely, $g|_\Omega,h|_\Omega$ are in $C^0(\Omega)$; by abuse of notation we often write $g,h$ for $g|_\Omega,h|_\Omega$.} by the Sobolev embedding theorem, which states that $\|f\|_{C^r(\Omega)}\leq C_\Omega\|f\|_{W^{k,p}(\Omega)}$ for some constant $C_\Omega$ if $k-\frac{n}{p}=k-\frac{2}{p}>r$.  Hence $g\cdot h\in C^0(\Omega)$ and $g\cdot h \in L^p(\Omega)$. For $m>0$, we have:
$$\|(g\cdot h)^{(m)}\|_{L^p(\Omega)}\leq \sum_{i=0}^m \binom{m}{i}\|g^{(i)}\cdot h^{(m-i)}\|_{L^p(\Omega)}.$$
Observe that one of $i$ or $m-i$ is $\leq \frac{m}{2}\leq \frac{k}{2}$.  This implies that $g^{(i)}\in C^0(\Omega)$ or $h^{(m-i)}\in  C^0(\Omega)$ since $g,h\in C^{k-1}(\Omega)$ by the Sobolev embedding theorem. Hence $\|(g\cdot h)^{(m)}\|_{L^p(\Omega)}<\infty.$
\end{proof}

\begin{lemma} \label{g(h)}
Suppose that $p=2$ and $k\geq 3$. If $G\in C^{\infty}(\mathbb R^2, \mathbb R)$ and $H\in W^{k,p}_{\op{loc}}(\mathbb R^2, \mathbb R^2)$,
then $G\circ H\in W^{k,p}_{\op{loc}}(\mathbb R^2, \mathbb R)$.
\end{lemma}

\begin{proof}
By the Sobolev embedding theorem, $H\in C^{k-1}(\Omega,\R^2)$, where $\Omega$ is as before, so $G(H)=G\circ H\in C^{k-1}(\Omega)\subset W^{k-1,p}(\Omega)$. This in particular implies that $G(H)\in L^p(\Omega)$. By the same argument, we also have $G'(H)\in W^{k-1,p}(\Omega)$.  Hence $(G(H))'=G'(H)\cdot H' \in W^{k-1,p}(\Omega)$ by Lemma \ref{fg}.  The lemma then follows.
%
%
%
%
%
\end{proof}

We assume that $p=2$ and $k\geq 4$ (see Remark~\ref{rmk: values of k and p}). We focus on a positive puncture $p$ of $\dot F$ around which the maps that we are interested in converge to $\gamma_{+}$. Let $\dot D\subset \dot F$ be an open disk about $p$ and let $(\sigma,\tau)$ be smooth coordinates on $\dot D$. For any $u$ of the class $W^{k+1,p}$ that is sufficiently close to a $J$-holomorphic curve, there is a map $$\phi_u:\dot D\to \mathbb R\times S^1$$ defined by
$$(\sigma,\tau)\mapsto (s,t)=(s\circ u(\sigma, \tau), t\circ u(\sigma, \tau)),$$
where $(\sigma,\tau)$ is a cylindrical holomorphic coordinate of $\dot D$,
$s\circ u$ is the $s$-coordinate of $u$, and $t\circ u$ is the $t$-coordinate of $u$. Then $(s,t)$, viewed as a function of $(\sigma, \tau)$, is of class $W^{k+1,p}_{\op{loc}}.$

\begin{proof}[Proof of Lemma \ref{shenyang}]
It suffices to show that $I(\cdot, s_0)$ is $C^1$ for each $s_0>T'$.

Given $(\mathcal F,u)\in \mathcal B'_{(\mathcal F_0, u_0)}$ and $\xi \in W^{k+1,p}_{\delta}(\dot F, u^{*}T(\mathbb R \times M))$, let $\{u^{\ell}\}_{-\epsilon< \ell<\epsilon}$ be a smooth path in $\mathcal B'_{(\mathcal F_0, u_0)}$ with $u^{0}=u$ and $\frac{d}{d\ell}\big|_{\ell=0}u^\ell =\xi$. Let us decompose 
$\xi=(\xi^{s,t},\xi^{x,y})$ into the $(s,t)$- and $(x,y)$-components and write $\eta^\ell=\eta+\ell\xi^{x,y}$ for the $(x,y)$-component of $u^\ell$.  Then
\begin{align*}
dI((\mathcal F, & u),s_0)(\xi, 0)=\tfrac{d}{d\ell}\big|_{\ell=0}I((\mathcal F, u^\ell),s_0)\\
=& \int_0^{\mathcal{A}(\gamma_+)}\tfrac{d}{d\ell}\big|_{\ell=0}\alpha(t,(\eta^\ell\circ \phi^{-1}_{u^\ell})(s_0,t))(\bdry_t+ d(\eta^\ell\circ \phi^{-1}_{u^\ell})(s_0,t)(\bdry_t)) dt.
\end{align*}

Recalling that
$$\alpha=(c(\gamma)+Q(x,y))dt+ \sum_{i=1}^n (x_idy_i-y_idx_i)$$
near $\R\times \gamma_+$, where $Q$ is quadratic in $x$ and $y$, we see that the terms in the integrand (after the term $\tfrac{d}{d\ell}\big|_{\ell=0}$) are quadratic in $\eta^\ell\circ \phi^{-1}_{u^\ell}$ or bilinear in $(\eta^\ell\circ \phi^{-1}_{u^\ell}, \nabla(\eta^\ell\circ \phi^{-1}_{u^\ell}))$.  Since we can use the Sobolev inequality, $C^0$-bounds in the integrand are sufficient and we obtain
\begin{align*}
|dI((\mathcal F, u),s_0)(\xi,0)|\leq \varphi(\|u\|_{W^{k+1,2}})\cdot \|\xi\|_{W^{k+1,2}},
\end{align*}
where $k\geq 4$ (i.e., we are allowed two derivatives) and $\varphi(\|u\|_{W^{k+1,2}})$ is a continuous function that depends on $\|u\|_{W^{k+1,2}}$.
Hence $I$ is differentiable. Similarly, we can show that $I$ is $C^1$. \cb
\end{proof}


In order to prove Lemma \ref{lemma: in Wkp and C1}, we first treat the case $\chi(\dot F)<0$. 
To simplify the notation in Equation~\eqref{f tilde}, we write $\widetilde{f}(s,t)$ for $\widetilde{f}^{\gamma_{\pm,i}}_j(s,t)$, $\rho_{s_{+}}(s)$ for $\frac{\partial \beta^\pm_{s_{\pm,i}}}{\partial s}(s)$ (recall $\rho_{s_+}(s)$ has compact support), and $f(t)$ for $f^{\gamma_{\pm,i}}_j(t)$.
With respect to the coordinates $s,t,x_1,\dots,x_n,y_1,\dots,y_n$,
Equation~\eqref{f tilde} can be written as
\begin{align} \label{section}
\widetilde{f}(s,t):=\rho_{s_{+}}(s)f(t)\otimes \pi_j (ds-idt).
\end{align}

First observe that $\rho_{s_{+}}(s)$ and $f(t)$ 
are smooth functions of $(s,t)$, so viewed as functions of $(\sigma, \tau)$, they are in class $W^{k+1,p}_{\op{loc}}$ by Lemma \ref{g(h)}. Next
the projection $\pi_j$ is smooth, so $\pi_j(ds-idt)\in W_{\op{loc}}^{k+1,p}(\dot F, \wedge^{0,1}T^{*}\dot F)$. Hence $\widetilde f\in W^{k+1,p}_{\op{loc}}(\dot F, \wedge^{0,1}u^{*}T\widehat W)$ by Lemma \ref{fg}, and this proves the first part of Lemma \ref{lemma: in Wkp and C1}, namely $ E^{\ell,\varepsilon}_{(\F,u)}\subset \mathcal{E}_{(\F,u)}$ for all $(\F,u)\in \mathcal{N}(K)$.

For $(j, v)\in \mathcal N(K)$ with $v:\dot F\to \widehat W$ smooth, there exists a small neighborhood $\mathcal D=\mathcal V \times \mathcal W$ of $(j, v)$ inside $\mathcal N(K)$ with $\mathcal V\subset \widetilde{\mathcal U}$ and $\mathcal W\subset \{ \exp (v,\xi)~|~ \xi \in W_\delta^{k+1,p}(\dot F, v^{*}T\widehat W)\}$.
We can trivialize $\mathcal E|_{\mathcal D}$ by identifying
$$\mathcal E_{(j',u)}=W^{k,p}_\delta (\dot F, \wedge^{0,1}_{j'} u^{*}T\widehat W)$$
with
$$\mathcal E_{(j,v)}=W^{k,p}_\delta (\dot F, \wedge ^{0,1}_j v^{*}T\widehat W)$$
for any $(j',u)\in \mathcal D$ via the exponential map. More precisely, we define
$$\Psi: W^{k,p}_\delta (\dot F, \wedge^{0,1}_{j'} u^{*}T\widehat W) \to W^{k,p}_\delta (\dot F, \wedge ^{0,1}_j v^{*}T\widehat W)$$
$$\zeta=\eta \otimes \alpha  \mapsto \op{Par}_{\xi} \eta \otimes \pi_j \alpha,$$
where $\xi=(\exp_v)^{-1} u$, $\eta\in W^{k,p}_\delta(\dot F, u^{*}T\widehat W)$, $\alpha\in W^{k,p}_\delta(\dot F, \wedge^{0,1}T^{*}F)$,
$$\op{Par}_\xi: W^{k,p}_\delta(\dot F, u^{*}T\widehat W)\to W^{k,p}_\delta(\dot F, v^{*}T\widehat W)$$ is the parallel transport along the path $\{ \exp(v,(1-t)\xi)\}_{0\leq t \leq 1}$, and $$\pi_j: W^{k,p}_\delta (\dot F, \wedge^{1}T^{*}F)\to W^{k,p}_\delta (\dot F, \wedge^{0,1}_jT^{*}F)$$
is the projection with respect to $g_{(\dot F,j)}$.

Let $\widetilde f$ be the section of $\mathcal E|_{\mathcal{N}(K)} \to \mathcal N(K)$ with
$$\widetilde f (j',u)(s,t)=\rho_{s_{+}}(s)f(t)\otimes \pi_{j'} (ds-idt),$$
where $(s,t)$ are viewed as coordinates of $\dot F$ via $\phi_u$ as usual.   We show that $\widetilde f$ is a $C^1$-section of $\mathcal E|_{\mathcal N(K)}\to \mathcal N(K)$. In the next several paragraphs we calculate the derivative $d\widetilde f(j',u)$.

We first calculate $d\widetilde f (j',u)(0,\xi)$, where $\{u_{l_1}\}_{-\epsilon\leq l_1 \leq \epsilon}$ is a smooth path in $\mathcal W$ with $u_0=u$ and $\xi= \frac{d}{dl_1}|_{l_1=0}u_{l_1}$:
\begin{align} \label{eqn: part 1}
d\widetilde f (j',u)(0,\xi)=&  \tfrac{d}{dl_1}|_{l_1=0} \widetilde f (j',u_{l_1}) \cb \\
\label{eqn: part 1A}=& \tfrac{d}{dl_1} \rho_{s_{+}(l_1)}(s_{l_1})|_{l_1=0} f(t) \otimes \pi_{j'}(ds-idt)\\
\label{eqn: part 1B} &+\rho_{s_+(0)}(s) \tfrac{d}{dl_1} f(t_{l_1})|_{l_1=0}\otimes \pi_{j'}(ds-idt)\\
\label{eqn: part 1C} &+\rho_{s_+(0)}(s)f(t)\otimes \pi_{j'}\tfrac{d}{dl_1}(ds_{l_1}-idt_{l_1})|_{l_1=0},
\end{align}
where $(s_{l_1}, t_{l_1})=\phi_{u_{l_1}}(\sigma,\tau)$.

Until the end of the Appendix, $W^{k,p}_\delta$ means $W^{k,p}_\delta$ with domain $\dot D$.  We claim that $d\widetilde f(j',u)(0,\xi)\in W^{k,p}_\delta$. Note that the weight $\delta$ can be ignored since $d\widetilde f(j',u)(0,\xi)$ has compact support. Since $\xi \in W^{k+1,p}_\delta$,
\begin{itemize}
\item $\tfrac{d}{dl_1}s_{l_1}|_{l_1=0},\tfrac{d}{dl_1}t_{l_1}|_{l_1=0}$, $\tfrac{d}{dl_1} f(t_{l_1})|_{l_1=0}\in W^{k+1,p}_\delta$; and
\item $\tfrac{d}{dl_1}(ds_{l_1}-idt_{l_1})|_{l_1=0}\in W^{k,p}_\delta$.
\end{itemize}
Next we consider the term $\tfrac{d}{dl_1} \rho_{s_{+}(l_1)}(s_{l_1})|_{l_1=0}$.  If we write $\rho_{s_+}(s)$ as $\rho(s_+,s)$, then $\rho(\cdot, \cdot)$ is smooth in both variables. Observe that
$$\tfrac{d}{dl_1} \rho(s_{+}(l_1),s_{l_1})|_{l_1=0}=\tfrac{\partial \rho}{\partial s_+}(s_+(0),s)\tfrac{d}{dl_1} s_+(l_1)|_{l_1=0}+\tfrac{\partial \rho}{\partial s}(s_+(0),s)\tfrac{d}{dl_1} s_{l_1} |_{l_1=0}.$$

\begin{lemma} \label{s_+C1}
$s_+:\mathcal N(K)\to \mathbb R$ is a $C^1$-map.
\end{lemma}

The proof of this lemma will be postponed until the very end. By Lemma \ref{s_+C1}, $\tfrac{d}{dl_1} \rho(s_{+}(l_1),s_{l_1})|_{l_1=0}\in W^{k+1,p}_\delta$.  It follows that $d\widetilde f (j',u)(0,\xi)\in W^{k,p}_\delta$.

Next we calculate $d\widetilde f (j',u)(\mathfrak j,0)$ with $\mathfrak j\in T_{j'}\widetilde{ \mathcal U}$. Let $\{j_{l_2}\}_{-\epsilon\leq l_2\leq \epsilon}$ be a smooth path in $\widetilde{\mathcal U}$ with $j_0=j'$ and $\frac{d}{dl_2}|_{l_2=0}j_{l_2}=\mathfrak j$. Then
\begin{align} \label{eqn: part 2}
d\widetilde f (j',u)(\mathfrak j,0)=& \tfrac{d}{dl_2}|_{l_2=0} \widetilde f (j_{l_2},u)\\
\label{eqn: part 2A} =&\tfrac{d}{dl_2} \rho_{s_{+}(l_2)}(s)|_{l_2=0} f(t) \otimes \pi_{j'}(ds-idt)\\
\label{eqn: part 2B} &+\rho_{s_{+}(0)}(s) f(t) \otimes \tfrac{d}{dl_2} \pi_{j_{l_2}}(ds-idt)|_{l_2=0}.
\end{align}
Since $\pi_j$ depends smoothly on $j$, we have $\tfrac{d}{dl_2}\pi_{j_{l_2}}(ds-idt)|_{l_2=0}\in W^{k,p}_{\op{loc}}$. Again we write $\rho_{s_{+}}(s)$ as $\rho(s_+,s)$ and obtain
$$\tfrac{d}{dl_2} \rho(s_{+}(l_2),s)|_{l_2=0}=\tfrac{\partial \rho}{\partial s_+}(s_+(0), s)\tfrac{d}{dl_2} s_+(l_2)|_{l_2=0}.$$
By Lemma \ref{s_+C1}, $\tfrac{d}{dl_2} \rho(s_{+}(l_2),s)|_{l_2=0}\in W^{k+1,p}_\delta$. Hence $d\widetilde f (j',u)(\mathfrak j,0)\in W^{k,p}_\delta$.

We now bound the $W^{k,p}_\delta$-norm of each term in Equations~\eqref{eqn: part 1} and \eqref{eqn: part 2}:
\begin{itemize}
\item \eqref{eqn: part 1A} is bounded above by $C\|\xi\|_{W^{k+1,p}}\cdot C(j')\cdot \|gu\|_{W^{k+1,p}}$;
\item \eqref{eqn: part 1B} by $C\|\xi\|_{W^{k+1,p}}\cdot C(j')\cdot \|gu\|_{W^{k+1,p}}$;
\item \eqref{eqn: part 1C} by $C \cdot C(j') \cdot \|\xi\|_{W^{k+1,p}}$;
\item \eqref{eqn: part 2A} by $C\|\mathfrak j\|\cdot C(j')\cdot \|gu\|_{W^{k+1,p}}$; and
\item \eqref{eqn: part 2B} by $C\|\mathfrak j\|\cdot \|gu\|_{W^{k+1,p}}$.
\end{itemize}
Here $C(j')$ is a positive continuous function of $j'$ and $g$ is a smooth cutoff function which equals one on $s_+(v)-c\leq s\leq s_+(v)+c$ for some $c>0$ and has compact support on $\dot D$.
Hence
$$\|d \widetilde f (j',u)(\mathfrak j,\xi)\|_{W^{k,p}_\delta} \leq C(C(j') + C(j') \cdot \|gu\|_{W^{k+1,p}}+\|gu\|_{W^{k+1,p}}) (\| \mathfrak j\| + \|\xi\|_{W^{k+1,p}_\delta}),$$
which implies that $\widetilde f$ is a differentiable section of $\mathcal E|_{\mathcal{N}(K)}\to \mathcal N(K)$. Almost the same calculation shows that $\widetilde f$ is in $C^1$.

\begin{proof}[Proof of Lemma~\ref{s_+C1}]
Let $r_{j'}:\dot F\to \mathbb R$ be the injectivity radius map, i.e, $r_{j'}$ maps $x\in \dot F$ to the injectivity radius $\op{inj}(x)$ with respect to $g_{(\dot F,j)}$. The map $r_{j}$ is smooth on the ends of $\dot F$ and also smoothly depends on $j'$.  For any $(j',u)\in \mathcal D$, consider the map
\begin{equation}
R_{(j',u)}: [R,\infty) \to \mathbb R^+, \quad s\mapsto r_{j'}(\phi^{-1}_u(s,0)),
\end{equation}
where $R\gg 0$ and depends on $(j',u)$.
Note that $R_{(j',u)}$ is invertible near the positive end.

When $\chi(\dot F)<0$, we have
\begin{align} \label{s_+ explicit}
s_+(j',u)=R_{(j',u)}^{-1}(\varepsilon).
\end{align}
Differentiating the equation
\begin{equation}
r_{j'}(\phi^{-1}_{u_{l_1}}(s_+(j',u_{l_1}),0))=\varepsilon
\end{equation}
with respect to $l_1$ and rearranging, we obtain:
\begin{align}
ds_+(j',u)(0,\xi)=\tfrac{d}{dl_1}s_+(j',u_{l_1})|_{l_1=0}= -\frac{d r_{j'}(z)\circ (d\phi_u (z))^{-1} (\xi^{s,t}(z)) }{dr_{j'}(z)\circ (d\phi_u(z))^{-1} (\partial_s) },
\end{align}
where $\partial_s$ is a vector field on $\mathbb R^+ \times S^1$, $\xi^{s,t}(z)$ is the $(s,t)$-component of $\xi$ at $z$,  and $z=\phi_u^{-1}(s_+(j',u),0).$  Here $\frac{d}{dl_1} \phi^{-1}_{u_{l_1}}(s_+(j',u),0))|_{l_1=0}=(d\phi_u (z))^{-1} (\xi^{s,t}(z))$.  Note that all terms are evaluated at the point $z$ and $z$ continuously depends on $(j',u)$, so $d r_{j'}(z)\circ (d\phi_u (z))^{-1}$ and $dr_{j'}(z)\circ (d\phi_u(z))^{-1} (\partial_s)$ depend continuously on $(j',u)$.  Hence $ds_+(j',u)(0,\cdot)$ is a bounded linear map from $W_\delta^{k+1,p}(\dot F, v^{*}T\widehat W)$ to $\mathbb R$ which depends continuously on $(j',u)$.

Next, differentiating the equation
\begin{equation}
r_{j_{l_2}}(\phi^{-1}_{u}(s_+(j_{l_2},u),0))=\varepsilon
\end{equation}
with respect to $l_2$, we obtain
\begin{align}
ds_+(j',u)(\mathfrak j,0)=\tfrac{d}{dl_2}s_+(j_{l_2},u)|_{l_2=0}= -\frac{{d}{dl_2} r_{j_{l_2}}(z)|_{l_2=0}}{dr_{j'}(z)\circ(d\phi_u(z))^{-1}(\bdry_s)}.
\end{align}
Since $\frac{d}{dl_2} r_{j_{l_2}}(z)|_{l_2=0}$ depends continuously on $(j',u)$,  $ds_+(j',u)(\cdot,0)$ is a bounded linear map from $T_{j'}\mathcal V$ to $\mathbb R$ which depends continuously on $(j',u)$. This proves that $s_+$ is a $C^1$-map when $\chi(\dot F)<0$.

When $\chi(\dot F)=0$ or $1$, $s_+$ is still given by Equation~\eqref{s_+ explicit}, but the metric $g_{(\ddot F,j')}$ is now a function of $u$. More precisely, for any $u=(s,t,\eta)$, define
$$I_u:[R,\infty)\to \mathbb R^+,\quad s' \mapsto \mathcal{A}_\alpha((t,\eta(s,t))|_{s=s'})$$
with $R\gg 0$ and then
$$s_+'(u)=I_u^{-1}(\varepsilon').$$
\cb Recall that the set $\mathbf{q}=(q_1,\dots,q_{2m(\gamma_{+})})$ of additional punctures that we put on $\dot F$ satisfies
$$q_k=\phi_u^{-1}\big(s_+', \tfrac{k}{2m(\gamma_{+})}\mathcal A_\alpha(\gamma_{+})\big)\in \dot F.$$
Then $g_{(\ddot F,j')}$ is smooth function of $\mathbf q$.  An argument similar to that of the $\chi(\dot F)<0$ case shows that $s_+$ is a $C^1$-map.
\end{proof}

\end{document}